\documentclass{memo-l}

\usepackage{amssymb}
\usepackage{amstext}
\usepackage{amsmath}
\usepackage{amscd}
\usepackage{amsthm}
\usepackage{amsfonts}
\usepackage{enumerate}
\usepackage{graphicx}
\usepackage{latexsym}
\usepackage{mathrsfs}
\usepackage[all,color]{xy}
\usepackage{mathtools}
\xyoption{all}
\usepackage[usenames,dvipsnames]{xcolor}

\usepackage{pstricks}
\usepackage{lscape}
\usepackage{comment}

\newtheorem{theorem}{Theorem}[chapter]

\newtheorem{corollary}[theorem]{Corollary}
\newtheorem{lemma}[theorem]{Lemma}
\newtheorem{proposition}[theorem]{Proposition}
\newtheorem{definition-proposition}[theorem]{Definition-Proposition}
\newtheorem{problem}[theorem]{Problem}

\theoremstyle{definition}
\newtheorem{definition}[theorem]{Definition}

\newtheorem{setting}[theorem]{Setting}
\newtheorem{remark}[theorem]{Remark}
\newtheorem{example}[theorem]{Example}
\newtheorem{observation}[theorem]{Observation}

\newcommand{\mm}{{\mathfrak{m}}}
\newcommand{\nn}{{\mathfrak{n}}}
\newcommand{\pp}{{\mathfrak{p}}}
\newcommand{\qq}{{\mathfrak{q}}}
\renewcommand{\AA}{\mathcal{A}}
\newcommand{\BB}{\mathcal{B}}
\newcommand{\CC}{\mathcal{C}}
\newcommand{\CCC}{\mathsf{C}}
\newcommand{\DDD}{\mathsf{D}}
\newcommand{\KKK}{\mathsf{K}}
\newcommand{\OO}{{\mathcal O}}
\newcommand{\PP}{\mathcal{P}}

\newcommand{\TT}{\mathcal{T}}
\newcommand{\UU}{\mathcal{U}}
\newcommand{\VV}{\mathcal{V}}

\newcommand{\XX}{\mathcal{X}}
\newcommand{\YY}{\mathcal{Y}}
\renewcommand{\a}{\vec{a}}
\renewcommand{\b}{\vec{b}}
\renewcommand{\c}{\vec{c}}
\newcommand{\de}{\vec{\delta}}
\renewcommand{\l}{\vec{\ell}}

\newcommand{\s}{\vec{s}}
\newcommand{\ttt}{\vec{t}}
\renewcommand{\v}{\vec{v}}
\newcommand{\x}{\vec{x}}
\newcommand{\y}{\vec{y}}
\newcommand{\z}{\vec{z}}
\newcommand{\w}{\vec{\omega}}

\newcommand{\A}{\mathbb{A}}
\newcommand{\D}{\mathbb{D}}
\newcommand{\E}{\mathbb{E}}
\renewcommand{\L}{\mathbb{L}}
\newcommand{\Z}{\mathbb{Z}}
\newcommand{\Q}{\mathbb{Q}}
\renewcommand{\P}{\mathbb{P}}
\newcommand{\R}{{\mathbf R}}

\newcommand{\X}{\mathbb{X}}

\newcommand{\ca}{\operatorname{ca}\nolimits}
\newcommand{\bo}{\operatorname{b}\nolimits}
\newcommand{\sg}{\operatorname{sg}\nolimits}
\newcommand{\rank}{\operatorname{rank}\nolimits}
\newcommand{\depth}{\operatorname{depth}\nolimits}
\newcommand{\pd}{\mathop{{\rm proj.dim}}\nolimits}
\newcommand{\soc}{\operatorname{soc}\nolimits}
\newcommand{\Ext}{\operatorname{Ext}\nolimits}
\newcommand{\Tor}{\operatorname{Tor}\nolimits}
\newcommand{\Hom}{\operatorname{Hom}\nolimits}
\newcommand{\rad}{\operatorname{rad}\nolimits}
\newcommand{\End}{\operatorname{End}\nolimits}

\newcommand{\gl}{\mathop{{\rm gl.dim\,}}\nolimits}
\newcommand{\op}{\operatorname{op}\nolimits}
\newcommand{\RHom}{\mathbf{R}\strut\kern-.2em\operatorname{Hom}\nolimits}
\newcommand{\RshHom}{\mathbf{R}\strut\kern-.2em\mathscr{H}\strut\kern-.3em\operatorname{om}\nolimits}
\newcommand{\shHom}{\mathscr{H}\strut\kern-.3em\operatorname{om}\nolimits}
\newcommand{\Lotimes}{\mathop{\stackrel{\mathbf{L}}{\otimes}}\nolimits}
\newcommand{\Image}{\operatorname{Im}\nolimits}
\newcommand{\Kernel}{\operatorname{Ker}\nolimits}
\newcommand{\Cokernel}{\operatorname{Cok}\nolimits}
\newcommand{\Spec}{\operatorname{Spec}\nolimits}
\newcommand{\Supp}{\operatorname{Supp}\nolimits}

\newcommand{\bu}{\bullet}
\newcommand{\ci}{0}

\DeclareMathOperator{\moduleCategory}{\mathsf{mod}} \renewcommand{\mod}{\moduleCategory}
\DeclareMathOperator{\Mod}{\mathsf{Mod}}
\DeclareMathOperator{\proj}{\mathsf{proj}}

\DeclareMathOperator{\thick}{\mathsf{thick}}
\DeclareMathOperator{\tri}{\mathsf{tri}}
\DeclareMathOperator{\coh}{\mathsf{coh}}
\DeclareMathOperator{\vect}{\mathsf{vect}}
\DeclareMathOperator{\lb}{\mathsf{line}}

\DeclareMathOperator{\qgr}{\mathsf{qgr}}
\DeclareMathOperator{\CM}{\mathsf{CM}}
\DeclareMathOperator{\MF}{\mathsf{MF}}
\DeclareMathOperator{\add}{\mathsf{add}}

\newcommand{\cut}{\ar@{.}}
\newcommand{\rel}{\ar@{.}}

\numberwithin{equation}{chapter}
\numberwithin{section}{chapter}

\newcommand{\DD}{\mathcal{D}}
\renewcommand{\SS}{\mathcal{S}}
\newcommand{\ZZ}{\mathcal{Z}}

\DeclareMathOperator{\Qcoh}{\mathsf{Qcoh}}

\SelectTips{cm}{}
\newcommand{\mdots}{\raisebox{4.5pt}{\rotatebox{84}{$\ddots$}}}

\makeindex

\begin{document}

\frontmatter

\title[Geigle-Lenzing complete intersections]{Representation theory of Geigle-Lenzing complete intersections}


\author[Herschend]{Martin Herschend}
\address{M. Herschend: Uppsala University, Department of Mathematics, Box 480, 751 06 Uppsala, Sweden / Graduate School of Mathematics, Nagoya University, Chikusa-ku, Nagoya, 464-8602, Japan}
\email{martin.herschend@math.uu.se / martinh@math.nagoya-u.ac.jp}
\thanks{The first author was partially supported by JSPS Grant-in-Aid for Young Scientists (B) 24740010.}

\author[Iyama]{Osamu Iyama}
\address{O. Iyama: Graduate School of Mathematics, Nagoya University, Chikusa-ku, Nagoya, 464-8602, Japan}
\email{iyama@math.nagoya-u.ac.jp}
\urladdr{http://www.math.nagoya-u.ac.jp/~iyama/}
\thanks{The second author was partially supported by JSPS Grant-in-Aid for Scientific Research (B) 24340004, (B) 16H03923, (C) 23540045, (S) 22224001 and (S) 15H05738.}

\author[Minamoto]{Hiroyuki Minamoto}
\address{H. Minamoto:Department of Mathematics and Information Sciences,
Faculty of Science,1-1 Gakuen-cho, Nakaku, Sakai, Osaka 599-8531, Japan }
\email{minamoto@mi.s.osakafu-u.ac.jp}
\thanks{The third author was partially supported by JSPS Grant-in-Aid for Challenging 
Exploratory Research 50527885.}

\author[Oppermann]{Steffen Oppermann}
\address{S. Oppermann: Institutt for matematiske fag, NTNU, 7491 Trondheim, Norway}
\email{steffen.oppermann@ntnu.no}
\thanks{The final author was partially supported by NFR Grant 250056.}

\date{}



\dedicatory{Dedicated to Helmut Lenzing on the occasion of his seventy-fifth birthday}

\begin{abstract}
Weighted projective lines, introduced by Geigle and Lenzing in 1987, are important objects in representation theory. They have tilting bundles, whose endomorphism algebras are the canonical algebras introduced by Ringel. The aim of this paper is to study their higher dimensional analogs.
First, we introduce a certain class of commutative Gorenstein rings $R$ graded by abelian groups $\mathbb{L}$ of rank $1$, which we call Geigle-Lenzing complete intersections. We study the stable category $\underline{\mathsf{CM}}^{\mathbb{L}}R$ of Cohen-Macaulay representations, which coincides with the singularity category $\mathsf{D}^{\mathbb{L}}_{\rm sg}(R)$. We show that $\underline{\mathsf{CM}}^{\mathbb{L}}R$ is triangle equivalent to $\mathsf{D}^{\rm b}(\mathsf{mod} A^{\rm CM})$ for a finite dimensional algebra $A^{\rm CM}$, which we call the CM-canonical algebra. As an application, we classify the $(R,\mathbb{L})$ that are Cohen-Macaulay finite. We also give sufficient conditions for $(R,\mathbb{L})$ to be $d$-Cohen-Macaulay finite in the sense of higher Auslander-Reiten theory.
Secondly, we study a new class of non-commutative projective schemes in the sense of Artin-Zhang, i.e.\ the category $\mathsf{coh}\mathbb{X}=\mathsf{mod}^{\mathbb{L}}R/\mathsf{mod}^{\mathbb{L}}_0R$ of coherent sheaves on the Geigle-Lenzing projective space $\mathbb{X}$. Geometrically this is the quotient stack $\mathbb{X}=[X/G]$ for $X={\rm Spec}\,R\setminus\{R_+\}$ and $G={\rm Spec}\,k[\mathbb{L}]$. We show that $\mathsf{D}^{\rm b}(\mathsf{coh}\mathbb{X})$ is triangle equivalent to $\mathsf{D}^{\rm b}(\mathsf{mod} A^{\ca})$ for a finite dimensional algebra $A^{\rm ca}$, which we call a $d$-canonical algebra. We study when $\mathbb{X}$ is $d$-vector bundle finite, and when $\mathbb{X}$ is derived equivalent to a $d$-representation infinite algebra in the sense of higher Auslander-Reiten theory.
Our $d$-canonical algebras provide a rich source of  $d$-Fano and $d$-anti-Fano algebras in non-commutative algebraic geometry. We also observe Orlov-type semiorthogonal decompositions of $\mathsf{D}_{\rm sg}^{\mathbb{L}}(R)$ and $\mathsf{D}^{\bo}(\mathsf{coh}\mathbb{X})$.
\end{abstract}

\maketitle

\tableofcontents

%

\chapter*{Introduction}

\section{Geigle-Lenzing complete intersections and projective spaces}

Weighted projective lines, introduced by Geigle-Lenzing in 1987 \cite{GL}, are important objects in representation theory, see e.g.\ \cite{CK,K,L1,L,Me}.
They are a certain class of complete intersection rings of dimension two, and the ones with the easiest representation theory (called `domestic') give rise to simple surface singularities, which are the most basic class in Cohen-Macaulay representation theory \cite{He,A2,Es}.
They have connections with various branches of mathematics,
e.g.\ preprojective algebras in representation theory, McKay correspondence in algebraic geometry, Brieskorn links in differential geometry, Kontsevich's homological mirror symmetry conjecture \cite{KST1,KST2}.

One of the important properties of weighted projective lines is that
they have tilting bundles, whose endomorphism algebras are the canonical algebras
introduced by Ringel \cite{Rin}, and hence we have a derived equivalence between
weighted projective lines and canonical algebras.
They have been widely studied in representation theory. For example, it is known that
any hereditary abelian category with a tilting object is derived equivalent to either a weighted projective line or a path algebra of an acyclic quiver \cite{H2}.

The aim of this paper is to introduce a higher dimensional analog of the weighted projective
lines of Geigle-Lenzing and the canonical algebras of Ringel, and to study them systematically by using the standard tools in representation theory of finite dimensional algebras \cite{ARS,ASS,H1} and Cohen-Macaulay rings \cite{Y,LW,DG} (see also \cite{DR,J,GK,He,A2,Es,Kn,BGS,EH}).
Below we explain the new objects which will be introduced in this paper, and postpone stating the precise results to the next sections.
We introduce
\emph{Geigle-Lenzing} (\emph{GL}) \emph{complete intersections}
\begin{equation}\label{show R}
R=k[T_0, \ldots, T_d, X_1, \ldots, X_n ]/(X_i^{p_i}-\ell_i(T_0,\ldots,T_d)\mid 1\le i\le n),
\end{equation}
where $p_1,\ldots,p_n$ are positive integers and $\ell_1,\ldots,\ell_n$ are linear forms in $k[T_0,\ldots,T_d]$ in general position.
This is a complete intersection of dimension $d+1$, and it is graded by an abelian group
\begin{equation}\label{show L}
\L = \langle \x_1,\ldots,\x_n, \c \rangle / \langle p_i \x_i - \c \mid 1\le i \le n\rangle
\end{equation}
of rank one, which may contain torsion elements, as follows:
\[\deg T_j:= \c\ \mbox{ and }\ \deg X_i:=\x_i\ \mbox{ for any $i$ and $j$.}\]
This provides us with the category $\CM^{\L}R$ of $\L$-graded maximal Cohen-Macaulay $R$-modules, which is Frobenius and satisfies Auslander-Reiten-Serre duality.
Its projective-injective objects are precisely the $\L$-graded projective $R$-modules, and the stable category $\underline{\CM}^{\L}R$ is canonically triangle equivalent to the singularity category $\DDD_{\sg}^{\L}(R)$.
We give a triangle equivalence
\[\underline{\CM}^{\L}R\simeq\DDD^{\bo}(\mod A^{\rm CM})\]
with a finite dimensional $k$-algebra $A^{\rm CM}$, which we call the \emph{CM-canonical algebra}. This equivalence is new even in the classical case $d=1$.
As an application, we classify GL complete intersections $(R,\L)$ that are Cohen-Macaulay finite.
Also we study properties of $A^{\rm CM}$ in detail.

On the other hand, our GL complete intersection $(R,\L)$ provides us with
a non-commutative projective scheme $\X$ in the sense of Artin-Zhang \cite{AZ}.
The category of coherent sheaves on a \emph{Geigle-Lenzing} (\emph{GL})
\emph{projective space} $\X$ is defined as the quotient category
\begin{equation}\label{show coh X}
\coh\X=\mod^{\L}R/\mod^{\L}_0R,
\end{equation}
of the abelian category $\mod^{\L}R$ of finitely generated $\L$-graded $R$-modules
by its Serre subcategory $\mod^{\L}_0R$ consisting of finite dimensional modules.
Geometrically $\X$ is the quotient stack
\begin{equation}\label{show X}
\X=[X/G]
\end{equation}
for the scheme $X=\Spec R\setminus\{R_+\}$ and the group scheme $G=\Spec k[\L]$
acting on $X$.

The category $\coh\X$ is an abelian category of global dimension $d$ satisfying
Auslander-Reiten-Serre duality.
In the case $d=1$, these $\X$ are precisely the weighted projective lines of Geigle-Lenzing.
We give a triangle equivalence
\[\DDD^{\bo}(\coh\X)\simeq\DDD^{\bo}(\mod A^{\ca})\]
with a finite dimensional $k$-algebra $A^{\ca}$, which we call a \emph{$d$-canonical algebra}.
In the case $d=1$, these are precisely the canonical algebras of Ringel.
We classify GL projective spaces $\X$ that are vector bundle finite. Also we study properties of $A^{\ca}$ in detail.

Our GL complete intersections are divided into three disjoint classes depending on the
sign of their $a$-invariants: \emph{Fano}, \emph{Calabi-Yau} and \emph{anti-Fano}.
In the case $d=1$, these correspond to the famous trichotomy of weighted projective lines:
\emph{domestic}, \emph{tubular} and \emph{wild}.
It is well-known that the following conditions are equivalent for $d=1$.
\begin{itemize}
\item $(R,\L)$ is domestic (or equivalently, Fano).
\item $(R,\L)$ is Cohen-Macaulay finite.
\item $\X$ is vector bundle finite.
\item $\underline{\CM}^{\L}R$  is triangle equivalent to $\DDD^{\bo}(\mod kQ)$ for a Dynkin quiver $Q$.
\item $\DDD^{\bo}(\coh\X)$ is triangle equivalent to $\DDD^{\bo}(\mod k\widetilde{Q})$ for an extended Dynkin quiver $\widetilde{Q}$.
\end{itemize}
In this case, Veronese subring $R^{(\w)}$ of $R$ is a simple surface singularity,
whose stable Auslander algebra and the Auslander algebra are isomorphic to
the preprojective algebras of $kQ$ and $k\widetilde{Q}$, respectively.

In this paper, we will study the corresponding statements for arbitrary $d$.
One of the main problems is when $\underline{\CM}^{\L}R$ or $\DDD^{\bo}(\coh\X)$ is triangle equivalent to $\DDD^{\bo}(\mod A)$ for a finite dimensional $k$-algebra with $\gl A\le d$, or equivalently, when it has a \emph{$d$-tilting object} (that is, a tilting object $T$ whose
endomorphism algebra has global dimension at most $d$).
This property is important in the context of higher dimensional Auslander-Reiten theory \cite{I1,I2}.
In fact, if $\underline{\CM}^{\L}R$ has a $d$-tilting object, then $(R,\L)$ is \emph{$d$-Cohen-Macaulay finite}, that is, the category $\CM^{\L}R$ has a nice $d$-cluster tilting subcategory.
Similarly, if $\DDD^{\bo}(\coh\X)$ has a $d$-tilting bundle, then $\X$ is \emph{$d$-vector bundle finite}, that is, the category $\vect\X$ has a nice $d$-cluster tilting subcategory.
Moreover, the endomorphism algebra of a $d$-tilting sheaf in $\DDD^{\bo}(\coh\X)$ is a \emph{$d$-representation infinite algebra} in the sense of \cite{HIO}.
Such algebras naturally appear in algebraic geometry as observed recently by Buchweitz-Hille \cite{BuH}.
For instance, if $n\le d+1$, then the $d$-canonical algebra $A^{\ca}$ is $d$-representation infinite.
We show that the existence of $d$-tilting objects/bundles implies that $(R,\L)$ is Fano, and construct $d$-tilting objects/bundles in many cases when $(R,\L)$ is Fano.

We also study GL projective spaces in the context of non-commutative algebraic geometry, where Nakayama functors of finite dimensional algebras play the role of canonical bundles over algebraic varieties.
Recently the third author introduced the classes of $d$-Fano algebras and $d$-anti-Fano algebras \cite{M,MM}. 
We show that $d$-canonical algebras are $d$-Fano algebras (respectively, $d$-anti-Fano algebras) if and only if $(R,\L)$ is Fano (respectively, anti-Fano).
We introduce the notion of Cohen-Macaulay sheaves and vector bundles on $\X$
as a new approach to the study of non-commutative projective schemes.

\section{Our Results on Geigle-Lenzing complete intersections}

Let $(R,\L)$ be a GL complete intersection as in \eqref{show R} and \eqref{show L}.
Then $R$ is in fact a complete intersection of Krull dimension
$d+1$. Without loss of generality, we may assume $p_i\ge2$
for each $i$ (Observation~\ref{Weights 1}).
The dualizing element of $R$, given by
\[\w:=(n-d-1)\c-\sum_{i=1}^n\x_i\in\L,\]
plays a key role in this paper since it gives the $a$-invariant (also known as Gorenstein parameter) of $R$.
Using the degree map $\delta:\L\to\Q$ given by $\delta(\x_i)=\frac{1}{p_i}$ and $\delta(\c)=1$, GL complete intersections are divided into the following 3 classes
depending on the sign of the degree $\delta(\w)=n-d-1-\sum_{i=1}^n\frac{1}{p_i}$ of $\w$:
\[\begin{array}{|c||c|c|c|}\hline
\delta(\w)&<0&=0&>0\\ \hline
(R,\L)&\mbox{Fano}&\mbox{Calabi-Yau}&\mbox{anti-Fano}\\ \hline
d=1&\mbox{domestic}&\mbox{tubular}&\mbox{wild}\\ \hline
\end{array}\]
In the well-studied case of $d=1$, these 3 classes are called domestic, tubular and wild
from a point of view of their representation type.

The first aim of this paper is to study the category $\CM^{\L}R$ of $\L$-graded (maximal) Cohen-Macaulay $R$-modules for the GL complete intersection $(R,\L)$.
Since $R$ is Gorenstein, $\CM^{\L}R$ forms a Frobenius
category and the stable category $\underline{\CM}^{\L}R$ forms a triangulated
category. By a classical result due to Buchweitz \cite{Bu}, we have the following
basic property (see Theorem~\ref{Buchweitz Eisenbud}):
\begin{itemize}
\item There exists a triangle equivalence
$\DDD^{\L}_{\sg}(R)\simeq\underline{\CM}^{\L}R$,
\end{itemize}
where $\DDD^{\L}_{\sg}(R):=\DDD^{\bo}(\mod^{\L}R)/\KKK^{\bo}(\proj^{\L}R)$
is the singularity category of $R$ \cite{O}.
On the other hand, we show the following ring theoretic properties of $R$ as an $\L$-graded
ring (see Defintion~\ref{define L-isolated singularity}). 

\begin{theorem}[Theorems~\ref{L-factorial L-domain} and \ref{L-isolated singularity}]
Any GL complete intersection $R$ is an $\L$-factorial $\L$-domain and has
$\L$-isolated singularities. 
\end{theorem}

As an application, we have the following basic property (see Theorem~\ref{AR duality}):
\begin{itemize}
\item (\emph{Auslander-Reiten-Serre duality}) There exists a functorial isomorphism
$\Hom_{\underline{\CM}^{\L}R}(X,Y)\simeq D\Hom_{\underline{\CM}^{\L}R}(Y,X(\w)[d])$
for any $X,Y\in\CM^{\L}R$.
\end{itemize}

One of the powerful approaches to study a given triangulated category $\TT$ (e.g.\ $\underline{\CM}^{\L}R$)
is to construct a triangle equivalence with the bounded derived category $\DDD^{\bo}(\mod A)$
of some finite dimensional $k$-algebra $A$, which one can study by using the well-studied
methods in representation theory. This is equivalent to finding a tilting object in $\TT$ when
$\TT$ is algebraic (e.g.\ $\underline{\CM}^{\L}R$).

Our first main result in this paper shows that this is always possible for the case of
the stable category $\underline{\CM}^{\L}R$:

\begin{theorem}[Theorem~\ref{CM tilting}]\label{CM tilting in introduction}
For any GL complete intersection $(R,\L)$, there is a finite dimensional $k$-algebra $A^{\rm CM}$ and a triangle equivalence
\[\underline{\CM}^{\L}R\simeq\DDD^{\bo}(\mod A^{\rm CM}).\]
In fact, $\underline{\CM}^{\L}R$ has a certain tilting object $T^{\rm CM}$ satisfying $\underline{\End}^{\L}_R(T^{\rm CM})=A^{\rm CM}$.
\end{theorem}

This is new even in the classical case $d=1$, while this is known for
the hypersurface case $n=d+2$ by Futaki-Ueda \cite{FU} and Kussin-Lenzing-Meltzer \cite{KLM} ($d=1$).

We call the algebra $A^{\rm CM}$ the \emph{CM-canonical algebra}.
We give the following explicit formula for global dimension of $A^{\rm CM}$
by using Tate's DG algebra resolutions for complete intersections \cite{Tat}.

\begin{theorem}[Theorem~\ref{gl.dim of A^CM}]\label{gl.dim of A^CM in introduction}
Let $(R,\L)$ be a GL complete intersection of dimension $d+1$ with weights $p_1,\ldots,p_n$. Assume $p_i\ge2$ for all $i$.
\begin{itemize}
\item[(a)] $A^{\rm CM}=0$ if and only if $n\le d+1$.
\item[(b)] If $n\ge d+2$, then $\gl A^{\rm CM}=2(n-d-2)+\#\{i\mid p_i\ge3\}$.
\end{itemize}
\end{theorem}

We also give an explicit presentation of $A^{\rm CM}$
in terms of a quiver with relations (see Theorem~\ref{endomorphism convex}).
In particular, for the hypersurface case $n=d+2$, $A^{\rm CM}$ has the following
simple description (see Corollary~\ref{presentation of CM canonical}):
\begin{itemize}
\item If $n=d+2$, then $A^{\rm CM}\simeq\bigotimes_{i=1}^nk\A_{p_i-1}$,
where $k\A_{p_i-1}$ is the path algebra of the equioriented quiver of type $\A_{p_i-1}$.
\end{itemize}
This gives immediately a version of Kn\"orrer periodicity (see Corollary~\ref{presentation of CM canonical}).

\medskip
We say that a tilting object $U$ in $\underline{\CM}^{\L}R$ is \emph{$d$-tilting}
if the endomorphism algebra $\underline{\End}^{\L}_R(U)$ has global dimension at most
$d$. We will study the following problem:

\begin{itemize}
\item When does $\underline{\CM}^{\L}R$ have a $d$-tilting object?
\end{itemize}

We give the following necessary condition.

\begin{theorem}[Theorem~\ref{d-tilting imply Fano}]\label{d-tilting imply Fano in introduction}
Assume $d\ge1$. If $\underline{\CM}^{\L}R$ has a $d$-tilting object, then $(R,\L)$ is Fano.
\end{theorem}

Observe however, that by Example~\ref{Fano no d-tilting} the convese of Theorem~\ref{d-tilting imply Fano in introduction} is not true.

\medskip
Now we study when $(R,\L)$ is \emph{Cohen-Macaulay finite}
(=\emph{CM finite}) in the sense that there are only finitely many isomorphism 
classes of indecomposable objects in $\CM^{\L}R$ up to degree shift.
In the classical case $d=1$, it is well-known that $(R,\L)$ is CM finite
if and only if $(R,\L)$ is domestic.
In this paper, we give the following `Buchweitz-Greuel-Schreyer-type' classification of GL complete intersections
that are CM finite as an application of Theorem~\ref{CM tilting in introduction} above.

\begin{theorem}[Theorem~\ref{characterize RF}]\label{characterize RF in introduction}
Let $(R,\L)$ be a GL complete intersection of dimension $d+1$ with weights $p_1,\ldots,p_n$.
Assume $p_i\ge2$ for all $i$. Then $(R,\L)$ is CM finite if and only if one of the 
following conditions hold.  
\begin{itemize}
\item $n\le d+1$.
\item $n=d+2$, and $(p_1,\ldots,p_n)=(2,\ldots,2,p_n)$, $(2,\ldots,2,3,3)$, $(2,\ldots,2,3,4)$
or $(2,\ldots,2,3,5)$ up to permutation.
\end{itemize}
\end{theorem}

In these cases, we describe the Auslander-Reiten quiver of $\CM^{\L}R$ explicitly (see Theorem~\ref{AR quiver for CM finite}).

Theorem~\ref{characterize RF in introduction} tells us that CM finiteness is a very strong restriction, and in that case, $\underline{\CM}^{\L}R$ is triangle equivalent to $\underline{\CM}^{\L'}R'$ for
some domestic GL complete intersection $(R',\L')$ of dimension two ($d=1$) by Kn\"orrer periodicity.

\medskip
Higher dimensional Auslander-Reiten theory provides a possible approach
to study the category $\CM^{\L}R$ even when $(R,\L)$ is not CM finite (even in the case
when it is CM wild).
A key notion is `$d$-CM finiteness' obtained by replacing the category $\CM^{\L}R$
by a $d$-cluster tilting subcategory:
A full subcategory $\CC$ of $\CM^{\L}R$ is called \emph{$d$-cluster tilting}
if $\CC$ is a functorially finite subcategory of $\CM^{\L}R$ such that
\begin{eqnarray*}
\CC&=&\{X\in\CM^{\L}R\mid\forall i\in\{1,2,\ldots,d-1\}\ \Ext_{\mod^{\L}R}^i(\CC,X)=0\}\\
&=&\{X\in\CM^{\L}R\mid\forall i\in\{1,2,\ldots,d-1\}\ \Ext_{\mod^{\L}R}^i(X,\CC)=0\}
\end{eqnarray*}
(see Section~\ref{section: preliminaries 2} for details). In this case $\CC$ satisfies
$\CC(\w)=\CC$.
We say that $(R,\L)$ is \emph{$d$-Cohen-Macaulay finite} (=\emph{$d$-CM finite}) if 
there exists a $d$-cluster tilting subcategory $\CC$ of $\CM^{\L}R$ such that there
are only finitely many isomorphism classes of indecomposable objects in $\CC$ up to degree shift.
In the classical case $d=1$, $1$-CM finiteness coincides with classical
CM finiteness since $\CM^{\L}R$ is the unique $1$-cluster tilting subcategory
of $\CM^{\L}R$. 
We will study the following problem:

\begin{itemize}
\item When is $(R,\L)$ $d$-CM finite?
\end{itemize}

We show that $d$-tilting objects are closely related to $d$-cluster tilting subcategories, as their names suggest. In fact,
we give the following sufficient condition for $(R,\L)$ to be $d$-CM finite.

\begin{theorem}[Theorem~\ref{construct dCT}]\label{construct dCT in introduction}
If $\underline{\CM}^{\L}R$ has a $d$-tilting object, then $(R,\L)$ is $d$-CM finite.
\end{theorem}

An important class of \emph{higher preprojective algebras} (see
Definition \ref{define preprojective algebra}) appear naturally in this context.
When $U$ is a $d$-tilting object in $\underline{\CM}^{\L}R$, then we have an isomorphism
\begin{equation*}
\underline{\End}_R^{\L/\Z\w}(U)\simeq\Pi_{d+1}(\underline{\End}_R^{\L}(U))
\end{equation*}
for the $(d+1)$-preprojective algebra $\Pi_{d+1}(\underline{\End}_R^{\L}(U))$ of $\underline{\End}_R^{\L}(U)$ (Theorem~\ref{d-tilting imply Fano}).
This is a generalization of the well-known fact that the stable Auslander algebras of simple surface
singularities are isomorphic to the preprojective algebras of Dynkin type.
We will give a similar isomorphism \eqref{Auslander=preprojective} below for $d$-tilting bundles on
GL projective spaces $\X$.

Moreover $d$-CM finiteness is closely related to existence of non-commutative crepant resolutions 
(=NCCRs) in the sense of Van den Bergh \cite{V}. 
When $(R, \L)$ is not Calabi-Yau and $d\ge1$, we consider a $\Z$-graded ring defined as the Veronese subring
\[R^{(\w)}=\bigoplus_{\ell\in\Z}R_{\ell\w}\]
of $R$. In the classical case $d=1$, $R^{(\w)}$ is a simple surface singularity if $(R,\L)$ is domestic \cite[Proposition 8.4]{GL2}.
It is well-known that simple surface singularities are Cohen-Macaulay finite, and hence has NCCRs.
We give the following natural higher dimensional analog of these results.

\begin{theorem}[Theorem~\ref{construct NCCR}]\label{construct NCCR in introduction}
Assume that $(R,\L)$ is not Calabi-Yau and $d\ge1$. If $(R,\L)$ is $d$-CM finite, then $R^{(\w)}$ has an NCCR.
\end{theorem}

Investigating the derived equivalence class of the algebras $A^{\rm CM}$ in
Theorem~\ref{CM tilting in introduction}, we have
the following list of GL complete intersections that have $d$-tilting objects
(though $T^{\rm CM}$ itself is not necessarily $d$-tilting).

\begin{theorem}[Theorem~\ref{Main1}]\label{Main1 in introduction}
If one of the following conditions are satisfied, then $\underline{\CM}^{\L}R$ has a $d$-tilting object, $(R,\L)$ is $d$-CM finite and $R^{(\w)}$ has an NCCR.
\begin{itemize}
\item $n\le d+1$.
\item $n=d+2\ge2$ and $(p_1,p_2)=(2,2)$.
\item $n=d+2\ge3$ and $(p_1,p_2,p_3)=(2,3,3)$, $(2,3,4)$ or $(2,3,5)$.
\item $n=d+2\ge4$ and $(p_1,p_2,p_3,p_4)=(3,3,p_3,p_4)$ with $p_3,p_4\in\{3,4,5\}$.
\item $\#\{i\mid p_i=2\}\ge3(n-d)-4$.
\end{itemize}
\end{theorem}

For examples of $d$-cluster tilting subcategories, see Examples~\ref{semisimpleAR2} and \ref{A2*A3 CT subcategory}.

For GL hypersurfaces (i.e.\ $n=d+2$), we use matrix factorizations to study their Cohen-Macaulay 
representations. In particular, for GL hypersurfaces $(R^j=S^j/(f_j),\L_j)$ with $j=1,2$ and 
$(R=(S^1\otimes_kS^2)/(f_1+f_2),\L)$, the tensor products of matrix factorizations
give a bifunctor
\[-\otimes_{\rm MF}-\colon\underline{\CM}^{\L_1}R^1\times\underline{\CM}^{\L_2}R^2\to
\underline{\CM}^{\L}R.\]
This is a shadow of the ordinary tensor product $\otimes_k$ in the following sense.

\begin{theorem}[Theorem~\ref{tensor of MF is tensor}]
We have the following commutative diagram up to natural isomorphism.
\[\xymatrix@R=1em{
\DDD^{\bo}(\mod^{\L_1}R^1)\times\DDD^{\bo}(\mod^{\L_2}R^2)\ar[r]\ar[d]^{-\otimes_k-}&\DDD_{\sg}^{\L_1}(R^1)\times\DDD_{\sg}^{\L_2}(R^2)\ar@{-}[r]^\sim\ar[d]^{-\otimes_k-}&\underline{\CM}^{\L_1}R^1\times\underline{\CM}^{\L_2}R^2\ar[d]^{-\otimes_{\rm MF}-}\\
\DDD^{\bo}(\mod^{\L}R)\ar[r]&\DDD_{\sg}^{\L}(R)\ar@{-}[r]^\sim&\underline{\CM}^{\L}R
}\]
\end{theorem}

To prove this, we show that the tensor product $\otimes_{\rm MF}$ is nothing but
the Cohen-Macaulay approximation of the ordinary tensor product $\otimes_k$
(Proposition~\ref{tensor product is CM approximation}).

\section{Our Results on Geigle-Lenzing projective spaces}
Let $\X$ be a GL projective space of dimension $d$ and $\coh\X$ the category of coherent
sheaves on $\X$ as given in \eqref{show X} and \eqref{show coh X}
(see Chapter~\ref{section: GL space} for details).
Note that, for the case $n=0$, $\X$ is the projective space $\P^d$, and for the case $d=1$, $\X$ is a weighted projective line of Geigle-Lenzing.

The following are some basic properties of $\coh\X$ (see Theorem~\ref{Serre} for details).
\begin{itemize}
\item $\coh\X$ is a Noetherian abelian category with global dimension $d$.
\item (\emph{Auslander-Reiten-Serre duality}) There exists a functorial isomorphism
$\Ext^d_{\X}(X,Y)\simeq D\Hom_{\X}(Y,X(\w))$ for any $X,Y\in\coh\X$.
\end{itemize}

The trichotomy (Fano, Calabi-Yau, anti-Fano) of GL complete intersections $(R,\L)$
is characterized by ampleness of the automorphisms $(-\w)$ and $(\w)$ of
$\coh\X$ in the sense of Artin-Zhang \cite{AZ} (see Definition~\ref{define ample},
Theorem~\ref{AZ ampleness}).

We will introduce two important full subcategories of $\coh\X$:
One is the category $\vect\X$ of vector bundles on $\X$, and the other is its full
subcategory $\lb\X$ of direct sums of line bundles (see Section~\ref{section: vector bundles} for details).
Our categorical approach enables us to study the relationship between Cohen-Macaulay
representations of GL complete intersections and coherent sheaves on GL projective spaces: We have a fully faithful functor $\pi\colon\CM^{\L}R\to\vect\X$, which induces an equivalence
$\pi:\proj^{\L}R\simeq\lb\X$. The first functor is also an equivalence in the classical case $d=1$,
but $\vect\X$ is much bigger than $\pi(\CM^{\L}R)$ for $d>1$. In fact we have the following description of $\pi(\CM^{\L}R)$ inside $\vect\X$ (Proposition~\ref{arith CM}):
\begin{eqnarray}\label{show pi(CM)}
\pi(\CM^{\L}R)&=&\{X\in\vect\X\mid\forall i\in\{1,2,\ldots,d-1\}\ \Ext_{\X}^i(\lb\X,X)=0\}\\ 
&=&\{X\in\vect\X\mid\forall i\in\{1,2,\ldots,d-1\}\ \Ext_{\X}^i(X,\lb\X)=0\}.\notag
\end{eqnarray}
In the context of projective geometry, the objects in $\pi(\CM^{\L}R)$ are called \emph{arithmetically Cohen-Macaulay bundles} (e.g.\ \cite{CH,CMP}).

We say that a GL projective space $\X$ is \emph{vector bundle finite} (=\emph{VB finite}) if 
there are only finitely many isomorphism classes of indecomposable objects in $\vect\X$ up to degree shift.

\begin{theorem}[Theorem~\ref{VB finite}]
A GL projective space $\X$ is VB finite if and only if $d=1$ and $\X$ is Fano (or equivalently, domestic).
\end{theorem}

Similarly to $d$-CM finiteness of GL complete intersections, we say that a GL projective space $\X$ is \emph{$d$-vector bundle finite} (=\emph{$d$-VB finite}) if there exists a 
$d$-cluster
tilting subcategory $\CC$ of $\vect\X$ (see Section~\ref{section: preliminaries 2} for details)
such that there are only finitely many isomorphism classes of indecomposable
objects in $\CC$ up to degree shift. We will study the following problem:

\begin{itemize}
\item When is $\X$ $d$-VB finite?
\end{itemize}

Thanks to the similarity between the equality \eqref{show pi(CM)} and our definition of $d$-cluster tilting subcategories, we obtain the following observation.

\begin{theorem}[Theorem~\ref{CT subcategory}]\label{CT subcategory in introduction}
Let $\X$ be the GL projective space corresponding to $(R,\L)$.
Then $d$-cluster tilting subcategories of $\CM^{\L}R$ are precisely
$d$-cluster tilting subcategories of $\vect\X$ containing $\lb\X$.
In particular, if $(R,\L)$ is $d$-CM finite, then $\X$ is $d$-VB finite.
\end{theorem}

Therefore $\X$ is $d$-VB finite for the cases listed in Theorem~\ref{Main1 in introduction}.
Another immediate consequence of Theorem~\ref{CT subcategory in introduction}
is the following result, where the `if' part generalizes Horrocks' splitting criterion
for vector bundles on the projective space $\P^d$ \cite[2.3.1]{OSS}.

\begin{corollary}[Corollary~\ref{line is CT subcategory}]
Let $\X$ be a GL projective space of dimension $d$ with weights $p_1,\ldots,p_n$.
Assume $p_i\ge 2$ for all $i$.
Then $\lb\X$ is a $d$-cluster tilting subcategory of $\vect\X$ if and only if $n\le d+1$.
\end{corollary}

One of the important properties of weighted projective lines is the existence of
tilting bundles, whose endomorphism algebras are Ringel's canonical algebras.
We generalize this result by showing that any GL projective space $\X$
has a tilting bundle.

\begin{theorem}[Theorem~\ref{tilting}]
For any GL projective space $\X$, there is a finite dimensional $k$-algebra $A^{\ca}$ and a triangle equivalence
\[\DDD^{\bo}(\coh\X)\simeq\DDD^{\bo}(\mod A^{\ca}).\]
In fact, $\DDD^{\bo}(\coh\X)$ has a tilting bundle $T^{\ca}=\bigoplus_{\x\in[0,d\c]}\OO(\x)\in\lb\X$ satisfying $\End_{\X}(T^{\ca})=A^{\ca}$.
\end{theorem}

This was known by Geigle-Lenzing \cite{GL} for $d=1$, by Beilinson \cite{Be} for $n=0$,
by Baer \cite{Ba} for $n\le d+1$, and by Ishii-Ueda \cite{IU} for $n=d+2$.
The whole statement has been shown independently by Lerner and the second author \cite{IL}
in the context of `GL orders' on $\P^d$.

We call the algebra $A^{\ca}$ a \emph{$d$-canonical algebra}.
We give the following explicit formula of global dimension of $A^{\ca}$.

\begin{theorem}[Theorem~\ref{global dimension}]\label{global dimension in introduction}
Let $\X$ be a GL projective space of dimension $d$ with weights $p_1,\ldots,p_n$.
Assume $p_i\ge 2$ for all $i$. Then we have
\[\gl A^{\ca}=\left\{\begin{array}{cc}
d&n \le d+1,\\
2d&n>d+1.
\end{array}\right.\]
Moreover, in the first case, $A^{\ca}$ is a $d$-representation infinite algebra.
\end{theorem}

We give an explicit presentation of $A^{\ca}$ in terms of a quiver
with relations (see Theorem~\ref{endomorphism convex} and
Section~\ref{section: basic properties of canonical}). In particular, we
show that, if $n \le d+1$, then $A^{\ca}$ is a $d$-representation infinite
algebra of type $\widetilde{\A}$ as introduced in \cite{HIO} (see Theorem
\ref{Atilde theorem}).
We use the tilting bundle $T^{\ca}$ to describe the Coxeter polynomial of $\coh\X$ explicitly (see Theorem~\ref{thm.coxeter_polynomial}).

More generally, we study the endomorphism algebras $\End_{\X}(V)$ of arbitrary
tilting bundles $V$ on $\X$. For any such $V$ there is a strong 
relationship between the GL projective space $\X$ and the endomorphism algebra $\End_{\X}(V)$. 
For example, we will show the following result, where the notions of 
$d$-Fano and $d$-anti-Fano algebras (see Definition~\ref{define Fano algebra}) 
were recently introduced by the third author \cite{M,MM} in non-commutative 
algebraic geometry.

\begin{theorem}[Theorem~\ref{trichotomy}]
Let $V$ be a tilting bundle on a GL projective space $\X$. Then 
$\X$ is Fano (respectively, anti-Fano) if and only if $\End_{\X}(V)$
is a $d$-Fano (respectively, $d$-anti-Fano) algebra.
\end{theorem}

One of our main problems to study is the following: 
\begin{itemize} 
\item When is $\coh\X$ derived equivalent to a $d$-representation infinite algebra? 
\end{itemize} 
For example, this is the case if $n\le d+1$ by Theorem~\ref{global dimension in introduction}.
Again, we say that a tilting object $V$ in $\DDD^{\bo}(\coh\X)$ is \emph{$d$-tilting}
if the endomorphism algebra $\End_{\DDD^{\bo}(\coh\X)}(V)$ has global dimension at most $d$.
In this case, by Proposition~\ref{preliminaries for end}, it is precisely $d$.
Then $d$-tilting sheaves give rise to $d$-representation infinite algebras thanks to the following result of Buchweitz-Hille \cite{BuH}.

\begin{proposition}[Proposition~\ref{d-tilting imply Fano2}]\label{d-tilting imply Fano2 in introduction}
If $V$ is a $d$-tilting sheaf on $\X$, then $\End_{\X}(V)$ is $d$-representation infinite.
\end{proposition}

Therefore our question simplifies to the following more accessible one: 
\begin{itemize} 
\item When does $\X$ have a $d$-tilting sheaf/bundle? 
\end{itemize} 
We give the following necessary condition as in Theorem~\ref{d-tilting imply Fano in introduction}.

\begin{proposition}[Proposition~\ref{d-tilting imply Fano2}]\label{d-tilting imply Fano2 in introduction2}
If $\DDD^{\bo}(\coh\X)$ has a $d$-tilting object, then $\X$ is Fano.
\end{proposition}

Moreover we give the following sufficient condition for $\X$ to be $d$-VB finite
in terms of tilting theory, as in Theorem~\ref{construct dCT in introduction}.

\begin{theorem}[Theorem~\ref{construct dCT2}]\label{construct dCT2 in introduction}
If $\X$ has a $d$-tilting bundle $V$, then $\X$ is $d$-VB finite.
\end{theorem}

Now we explain a more explicit version of Theorem~\ref{construct dCT2 in introduction}.
A key role is played by the notion of a \emph{slice} in a $d$-cluster tilting 
subcategory $\UU$ of $\vect\X$, which is an object $V$ in $\UU$ such that 
$\UU=\add\{V(\ell\w)\mid\ell\in\Z\}$ and $\Hom_\UU(V,V(\ell\w))=0$ holds for all $\ell>0$
(see Definition~\ref{define slice} for details). 
We prove the following relationship between $d$-tilting bundles and $d$-cluster tilting subcategories.

\begin{theorem}[Theorem~\ref{tilting-cluster tilting})\ (tilting-cluster tilting correspondence]\label{tilting-cluster tilting in introduction}
Let $\X$ be a GL projective space. Then $d$-tilting bundles on $\X$ are precisely slices in $d$-cluster tilting  subcategories of $\vect\X$.
\end{theorem}

To prove this theorem, for a $d$-representation infinite algebra $\Lambda$,
we introduce an extension-closed subcategory $\VV_\Lambda$ of $\DDD^{\bo}(\mod\Lambda)$,
which plays a role of the category $\vect\X$.
We prove that $\VV_\Lambda$ has an explicitly described $d$-cluster tilting subcategory
if the $(d+1)$-preprojective algebra of $\Lambda$ is left graded coherent (Theorem~\ref{d-RI has d-CT}). 


We give the following characterizations of $d$-tilting bundles on $\X$
contained in the subcategory $\CM^{\L}R$ of $\vect\X$ by applying Theorem~\ref{construct NCCR in introduction}.

\begin{corollary}[Theorem~\ref{tilting-cluster tilting 2}]\label{tilting-cluster tilting 2 in introduction}
Let $\X$ be a GL projective space corresponding to $(R,\L)$.
For $V\in\CM^{\L}R$, the following conditions are equivalent.
\begin{itemize}
\item[(a)] $V$ is a $d$-tilting bundle on $\X$.
\item[(b)] $V$ is a slice in a $d$-cluster tilting subcategories of $\CM^{\L}R$.
\item[(c)] $V^{(\w)}$ gives an NCCR of $R^{(\w)}$ such that $\End_{R^{(\w)}}(V^{(\w)})_i=0$ for all $i<0$.
\end{itemize}
In this case, there are isomorphisms
\begin{equation}\label{Auslander=preprojective}
\End_{R^{(\w)}}(V^{(\w)})\simeq\Pi
\end{equation}
of $\Z$-graded algebras, where $\Pi$ is the $(d+1)$-preprojective algebra of $\End_{\X}(V)$.
\end{corollary}

The isomorphism \eqref{Auslander=preprojective} is a generalization of the well-known isomorphism
between the Auslander algebras of simple surface singularities
and the preprojective algebras of extended Dynkin type.
A similar picture already appeared in \cite{AIR} in a different setting.

We are expecting that $d$-tilting objects in $\underline{\CM}^{\L}R$
always lift to $d$-tilting bundles on $\X$. In fact, by
using Theorems~\ref{tilting-cluster tilting in introduction} and
\ref{Main1 in introduction}, we give the following result.

\begin{theorem}[Corollary~\ref{example of d-tilting bundle}]\label{Main2 in introduction}
Let $\X$ be a GL projective space with weights $p_1,\ldots,p_{d+2}$.
If one of the following conditions are satisfied, then $\X$ has a $d$-tilting bundle, and 
therefore $\X$ is $d$-VB finite (if $d\ge1$) and derived equivalent to a $d$-representation infinite algebra.
\begin{itemize}
\item $d \geq 0$ and $(p_1, p_2) = (2,2)$.
\item $d\ge1$ and $(p_1,p_2,p_3)$ is one of $(2,3,3)$, $(2,3,4)$ or $(2,3,5)$.
\item $d\ge2$ and $(p_1,p_2,p_3,p_4)=(3, 3, p_3, p_4)$ with $p_3,p_4\in\{3,4,5\}$.
\end{itemize}
\end{theorem}

Some of our results can be summarized as follows.
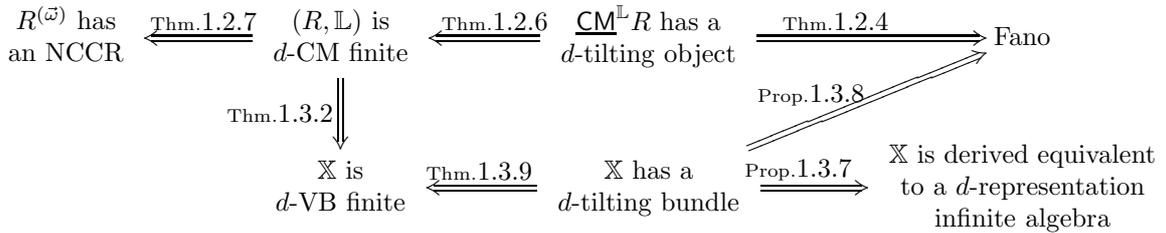
\[\xymatrix@C=4em@R=2em{
{\begin{array}{c}\mbox{$R^{(\w)}$ has}\\ \mbox{an NCCR}\end{array}}\\
{\begin{array}{c}\mbox{$(R,\L)$ is}\\ \mbox{$d$-CM finite}\end{array}}\ar@{=>}[d]_{\rm Thm.\ref{CT subcategory in introduction}}\ar@{=>}[u]^{\rm Thm.\ref{construct NCCR in introduction}}&
{\begin{array}{c}\mbox{$\underline{\CM}^{\L}R$ has a}\\ \mbox{$d$-tilting object}\end{array}}\ar@{=>}[l]_{\rm Thm.\ref{construct dCT in introduction}}
\ar@{=>}[r]^{\rm Thm.\ref{d-tilting imply Fano in introduction}}&\mbox{Fano}\\
{\begin{array}{c}\mbox{$\X$ is}\\ \mbox{$d$-VB finite}\end{array}}&
{\begin{array}{c}\mbox{$\X$ has a}\\ \mbox{$d$-tilting bundle}\end{array}}
\ar@{=>}[l]_(.55){\rm Thm.\ref{construct dCT2 in introduction}}
\ar@{=>}[r]^(.4){\rm Prop.\ref{d-tilting imply Fano2 in introduction}}
\ar@{=>}[ru]^{\rm Prop.\ref{d-tilting imply Fano2 in introduction2}}
&{\begin{array}{c}\mbox{$\X$ is derived equivalent}\\
\mbox{to a $d$-representation}\\
\mbox{infinite algebra}
\end{array}}
}\]
We end this chapter by the following question.

\begin{problem}\label{strongest hope 2}
How are the conditions in the above diagram related to each other?
\end{problem}

It is well-known that, in the classical case $d=1$, all the conditions in the above diagram are equivalent.
The upper right implication is strict for $d=2$ as we remarked above.
We suspect that many of these conditions are still equivalent,
but we do not have a precise conjecture.

\medskip
We record the following formulas for the group $\L$.

\begin{center}
\fbox{ \begin{minipage}{12cm}
\noindent
{\bf Generators and relations.} $\x_i$ (for $1 \leq i \leq n$), $\c$; $p_i \x_i = \c$.

\medskip
\noindent
{\bf Dualizing element.} $\w = - (d+1) \c + \sum_{i = 1}^n (p_i - 1) \x_i$ \\
\hspace*{1cm} regular case $n = d+1$: $\w = - \sum_{i=1}^{d+1} \x_i$ \\
\hspace*{1cm} hypersurface case $n = d+2$: $\w = \c - \sum_{i=1}^{d+2} \x_i$

\medskip
\noindent
{\bf Dominant element.} $\de = d \c + 2 \w = (n-d-2) \c + \sum_{i=1}^n (p_i - 2) \x_i$ \\
\hspace*{1cm} regular case $n = d+1$: $\de = - \c + \sum_{i=1}^{d+1} (p_i - 2) \x_i$ \\
\hspace*{1cm} hypersurface case $n = d+2$: $\de = \sum_{i=1}^{d+2} (p_i - 2) \x_i$
\end{minipage} }
\end{center}

\medskip\noindent
{\bf Acknowledgements }
The authors thank Helmut Lenzing for a number of valuable comments
on our work.
They thank Ragnar-Olaf Buchweitz and Lutz Hille for kindly explaining their work \cite{BuH} and stimulating discussions.
They thank Mitsuyasu Hashimoto, Akira Ishii, Boris Lerner, Atsushi Takahashi, Ryo Takahashi, Masataka Tomari, Kazushi Ueda, Yuji Yoshino and Keiichi Watanabe for stimulating discussions and valuable comments.



\mainmatter
%
%
%


\chapter{Preliminaries}\label{section: preliminaries}

In this chapter, we introduce basic notions which will be used throughout the paper.
For general background, we refer to \cite{ASS,ARS,Rin} on representation theory of
finite dimensional algebras,
to \cite{AHK,H1} on tilting theory, to \cite{Y,LW} on Cohen-Macaulay representation theory,
and to \cite{BH} on commutative ring theory.

Throughout this paper, we denote by $k$ an arbitrary field, and by $D$ the $k$-dual, that is $D(-)=\Hom_k(-,k)$. All modules are left modules.
The composition of $f:X\to Y$ and $g:Y\to Z$ is denote as $fg:X\to Z$.
For a ring $\Lambda$, we denote by $\Mod\Lambda$ the category of $\Lambda$-modules, by $\mod\Lambda$ the category of finitely generated $\Lambda$-modules,
and by $\proj\Lambda$ the category of finitely generated projective $\Lambda$-modules.
For an abelian group $G$ and a $G$-graded $k$-algebra $\Lambda$, we denote by $\Mod^G\Lambda$, $\mod^G\Lambda$, and $\proj^G\Lambda$ the $G$-graded versions.

For a class $\XX$ of objects in an additive category $\CC$, we denote by $\add_{\CC}\XX$
or $\add\XX$ the full subcategory of $\CC$ consisting of direct summands of finite
direct sums of objects in $\XX$.

For full subcategories $\XX$ and $\YY$ of a triangulated category $\TT$, we denote by $\XX*\YY$
the full subcategory of $\TT$ consisting of objects $Z$ such that there 
exists a triangle $X\to Z\to Y\to X[1]$ with $X\in\XX$ and $Y\in\YY$.

We denote by $\CCC(-)$, $\KKK(-)$, and $\DDD(-)$ the category of complexes, the homotopy
category, and the derived category, respectively. By $\CCC^{\rm b}(-)$,
$\KKK^{\rm b}(-)$ and $\DDD^{\rm b}(-)$ we mean the bounded versions.
We denote by $(\DDD^{\le0}(-),\DDD^{\ge0}(-))$ the standard t-structure in the derived category.


\section{Triangulated categories and tilting theory}\label{section: triangulated category}
Let us start with recalling some basic notions in triangulated categories.
Throughout this chapter, let $\TT$ be a triangulated category with a suspension
functor $[1]$.

We call a full subcategory $\UU$ of $\TT$ \emph{triangulated} if it
is closed under cones and $[\pm1]$. If $\UU$ is also closed under direct
summands, we call it \emph{thick}.
For a subcategory $\CC$ of $\TT$, we denote by $\thick\CC$
or $\thick_{\TT}\CC$ (respectively, $\tri\CC$ or $\tri_{\TT}\CC$) the smallest thick 
(respectively, triangulated) subcategory of $\TT$ containing $\CC$.

The following observation can be checked easily.

\begin{observation}\label{thick lemma} 
We have $\thick\CC=\add(\tri\CC)$.
\end{observation}

\begin{definition}[Tilting object]
We say that an object $U\in\TT$ is \emph{tilting} (respectively, \emph{silting}) if
$\Hom_{\TT}(U,U[i])=0$ for all $i\neq0$ (respectively, $i>0$) and $\TT=\thick U$.
\end{definition}

For example, for any ring $A$, the bounded homotopy category $\KKK^{\bo}(\proj A)$
of finitely generated projective $A$-modules has a tilting object $A$.
Moreover a converse of this statement holds under reasonable 
assumptions: We call a triangulated category \emph{algebraic} if it is 
triangle equivalent to the stable category of a Frobenius category,
and \emph{idempotent-complete} if any idempotent endomorphism 
corresponds to a direct summand.
We say that a fully faithful triangle functor $F:\TT\to\TT'$ is an
\emph{equivalence up to direct summands} if, for any object $X\in\TT'$, 
there exists an object $Y\in\TT$ such that $X$ is a direct summand
of $F(Y)$.

\begin{proposition}\cite[4.3]{Ke}\cite[2.14]{Ki}\label{tilting theorem}
Let $\TT$ be an algebraic triangulated category with a tilting
object $U$. Then there exists a triangle equivalence
$F:\TT\to\KKK^{\bo}(\proj\End_{\TT}(U))$ up to direct summands.
In particular, if $\TT$ is idempotent complete, then $F$ is a triangle equivalence.
\end{proposition}


We say that two finite dimensional algebras $\Lambda$ and $\Gamma$ 
are \emph{derived equivalent} if one of the following equivalent conditions hold:
\begin{itemize}
\item There exists a triangle equivalence
$\KKK^{\bo}(\proj\Lambda)\simeq\KKK^{\bo}(\proj\Gamma)$.
\item There exists a triangle equivalence
$\DDD^{\bo}(\mod\Lambda)\simeq\DDD^{\bo}(\mod\Gamma)$.
\item There exists a triangle equivalence 
$\DDD(\Mod\Lambda)\simeq\DDD(\Mod\Gamma)$.
\item There exists a tilting object $U$ in $\KKK^{\bo}(\proj\Lambda)$ such that  $\End_{\KKK^{\bo}(\proj\Lambda)}(U)\simeq\Gamma$.
\end{itemize}
The following observations are basic.

\begin{proposition}\label{derived category and global dimension}
\begin{itemize}
\item[(a)] A finite dimensional $k$-algebra $\Lambda$ has finite global 
dimension if and only if the natural functor
$\KKK^{\bo}(\proj\Lambda)\to\DDD^{\bo}(\mod\Lambda)$ is an 
equivalence.
\item[(b)] Assume that finite dimensional $k$-algebras $\Lambda$ and $\Gamma$ are derived equivalent.
Then $\Lambda$ has finite global dimension if and only if so does $\Gamma$.
\end{itemize}
\end{proposition}

Let us recall the notion of Serre functors in triangulated categories.

\begin{definition}[Serre functor]
Let $\TT$ be a $k$-linear and Hom-finite triangulated category.
A \emph{Serre functor} of $\TT$ is an autoequivalence $S:\TT\to\TT$ such
that there exists a functorial isomorphism $\Hom_{\TT}(X,Y)\simeq D\Hom_{\TT}(Y,SX)$ for $X,Y\in\TT$.
\end{definition}

It is easy to show that Serre functors of $\TT$ are unique up to natural isomorphism.

For example, if $X$ is a smooth projective variety of dimension 
$d$ with a canonical sheaf $\omega$, then $\DDD^{\bo}(\coh X)$ has 
a Serre functor $\omega[d]\otimes_X-$.
The following basic result by Happel gives another typical example of Serre functors.

\begin{proposition}\label{Nakayama Serre}\cite{H1}
Let $\Lambda$ be a finite dimensional $k$-algebra of finite global dimension.
Then the \emph{Nakayama functor}
\[\nu:=(D\Lambda)\Lotimes_\Lambda-:\DDD^{\bo}(\mod\Lambda)\to\DDD^{\bo}(\mod\Lambda)\]
gives a Serre functor of $\DDD^{\bo}(\mod\Lambda)$.
\end{proposition}

The following elementary observation is useful to calculate the global dimension.

\begin{observation}\label{gl dim}
Let $\Lambda$ be a finite dimensional $k$-algebra of finite global 
dimension. Then
\begin{eqnarray*}
\gl\Lambda&=&\sup\{i\in\Z\mid\Ext_\Lambda^i(D\Lambda,\Lambda)\neq0\}\\
&=&\sup\{i\in\Z\mid\Hom_{\DDD^{\bo}(\mod\Lambda)}(\Lambda,\nu^{-1}(\Lambda)[i])\neq0\}.
\end{eqnarray*}
\end{observation}

\begin{definition}[Calabi-Yau triangulated categories]\label{define CY}
Let $\TT$ be a triangulated category with a Serre functor $S$.
We say that $\TT$ is \emph{fractionally Calabi-Yau of dimension $\frac{m}{\ell}$}
(or simply \emph{$\frac{m}{\ell}$-Calabi-Yau}) for integers $\ell>0$ and $m$ if there exists an isomorphism $S^\ell\simeq[m]$ of functors $\TT\to\TT$.
Observe that $\frac{m}{\ell}$-Calabi-Yau implies $\frac{mi}{\ell i}$-Calabi-Yau for all positive integers $i$, but the converse is not true in general.

We say that a finite dimensional $k$-algebra $\Lambda$ with finite 
global dimension is  \emph{fractionally Calabi-Yau of dimension $\frac{m}{\ell}$} (or \emph{$\frac{m}{\ell}$-Calabi-Yau})
if $\DDD^{\bo}(\mod\Lambda)$ is.
\end{definition}

We give a few examples.

\begin{example}\label{Dynkin is CY}
\begin{itemize}
\item[(a)] \cite{MY} Let $kQ$ be a path algebra of Dynkin quiver. Then $kQ$ is
$\frac{h-2}{h}$-Calabi-Yau for the Coxeter number $h$ of $Q$:
\[\begin{array}{|c|c|c|c|c|}\hline
\A_n&\D_n&\E_6&\E_7&\E_8\\ \hline
n+1&2(n-1)&12&18&30\\ \hline
\end{array}\]
\item[(b)] \cite{HI1} Assume that $\Lambda_i$ is $\frac{m_i}{\ell_i}$-Calabi-Yau for $i=1,2$. If $\Lambda_1\otimes_k\Lambda_2$ has finite 
global dimension, then $\Lambda_1\otimes_k\Lambda_2$ is $\frac{m}{\ell}$-Calabi-Yau for $\ell:={\rm l.c.m.}(\ell_1,\ell_2)$ and $m:=\frac{\ell m_1}{\ell_1}+\frac{\ell m_2}{\ell_2}$.
\end{itemize}
\end{example}

The following observations are easy.

\begin{proposition}\label{fractional CY and global dimension}
Let $\Lambda$ be a finite dimensional $k$-algebra of finite global dimension $n$.
\begin{itemize}
\item[(a)]  $\nu^{-1}(\DDD^{\ge0}(\mod\Lambda))\subset\DDD^{\ge0}(\mod\Lambda)$ and
$\nu(\DDD^{\le0}(\mod\Lambda))\subset\DDD^{\le0}(\mod\Lambda)$ hold.
\item[(b)] $\nu^{-1}(\DDD^{\le0}(\mod\Lambda))\subset\DDD^{\le n}(\mod\Lambda)$ and
$\nu(\DDD^{\ge0}(\mod\Lambda))\subset\DDD^{\ge-n}(\mod\Lambda)$ hold.
\item[(c)] If $n\neq0$ and $\DDD^{\bo}(\mod\Lambda)$ is $\frac{m}{\ell}$-Calabi-Yau, then $\ell>1$ and $\frac{m}{\ell-1}\le n$ holds.
\end{itemize}
\end{proposition}

\begin{proof}
(a)(b) Both are elementary (see \cite[5.4(a)]{I3} for (b)).

(c) If $\ell=1$, then $D\Lambda=\nu(\Lambda)=\Lambda[m]$ holds. Thus $m=0$ holds, and $\Lambda$ is selfinjective. Since $\Lambda$ has finite global dimension $n$, we have $n=0$, a contradiction to our assumption. Thus $\ell>1$ holds.
Now we have
\[\Lambda[m]=\nu^{\ell}(\Lambda)=\nu^{\ell-1}(D\Lambda)\in\nu^{\ell-1}(\DDD^{\ge0}(\mod\Lambda))\subset\DDD^{\ge -(\ell-1)n}(\mod\Lambda)\]
by (b). Hence $-m\ge-(\ell-1)n$ holds, and the assertion follows.
\end{proof}

Later we use the following general observation.

\begin{proposition}\label{FPT lemma}
Let $\TT$ be a triangulated category, $\SS$ a thick subcategory of $\TT$ and $\pi:\TT\to\TT/\SS$ the natural functor.
\begin{itemize}
\item[(a)] We have a bijection between thick subcategories of $\TT$ containing $\SS$ and thick subcategories of $\TT/\SS$ given by $\UU\mapsto\add\pi(\UU)$.
\item[(b)] For a thick subcategory $\UU$ of $\TT$, we have $\thick_{\TT}\{ \SS,\UU \} = \TT$ if and only if $\thick_{\TT/\SS}\UU=\TT/\SS$.
\end{itemize}
\end{proposition}

\begin{proof}
(a) It is easy to check that $\pi(\UU)$ is a triangulated subcategory of $\TT/\SS$.
By Observation~\ref{thick lemma}, we have that $\add\pi(\UU)$ is a thick subcategory of $\TT/\SS$.

For a thick subcategory $\VV$ of $\TT/\SS$, it is clear that
$\pi^{-1}(\VV):=\{X\in\TT\mid \pi(X)\in\VV\}$ is a thick subcategory 
of $\TT$ containing $\SS$. It is easy to check that these 
correspondences are mutually inverse.

(b) This is immediate from (a).
\end{proof}

\section{Higher dimensional Auslander-Reiten theory}\label{section: preliminaries 2}

Let us start with recalling the following basic notion.

\begin{definition}[Functorially finiteness]\cite{AS}
Let $\AA$ be an additive category and $\CC$ a full subcategory of $\AA$.
For an object $A\in\AA$, we say that a morphism $f:C\to A$ is a
\emph{right $\CC$-approximation} of $A$ if $C\in\CC$ and the map
$f:\Hom_{\AA}(C',C)\to\Hom_{\AA}(C',A)$ is surjective for any $C'\in\CC$.
If any object in $\AA$ has a right $\CC$-approximation, then we say that
$\CC$ is a \emph{contravariantly finite} subcategory of $\AA$.
Dually we define a \emph{left $\CC$-approximation} and a 
\emph{covariantly finite subcategory}.
We say that $\CC$ is \emph{functorially finite} if it is contravariantly
finite and covariantly finite.
\end{definition}

The notion of a $d$-cluster tilting subcategory is central in higher dimensional
Auslander-Reiten theory. Note that it is also called a
\emph{maximal $(d-1)$-orthogonal} subcategory.

\begin{definition}[$d$-cluster tilting subcategories]\cite{I1,I2,IY}
Let $\AA$ be an abelian category, $\BB$ a full extension-closed
subcategory of $\AA$ and $\CC$ a full subcategory of $\BB$.
We say that $\CC$ \emph{generates} (respectively, \emph{cogenerates}) $\BB$ if any object in $\BB$ is a factor object (respectively, subobject) of some object in $\CC$.
We say that $\CC$ is a \emph{$d$-cluster tilting} subcategory of $\BB$
if $\CC$ is a functorially finite subcategory of $\BB$ that generates and cogenerates $\BB$ such that
\begin{eqnarray*}
\CC&=&\{X\in\BB\mid\forall i\in\{1,2,\ldots,d-1\}\ \Ext_{\AA}^i(\CC,X)=0\}\\
&=&\{X\in\BB\mid\forall i\in\{1,2,\ldots,d-1\}\ \Ext_{\AA}^i(X,\CC)=0\}.
\end{eqnarray*}
Notice that the condition that $\CC$ generates and cogenerates $\BB$ is not imposed in previous references since it is automatic if $\BB$ has enough projectives and enough injectives.

Similarly, we define a \emph{$d$-cluster tilting subcategory} of a triangulated category
$\TT$ as a functorially finite subcategory $\CC$ satisfying 
\begin{eqnarray*}
\CC&=&\{X\in\TT\mid\forall i\in\{1,2,\ldots,d-1\}\ \Hom_{\TT}(\CC,X[i])=0\}\\
&=&\{X\in\TT\mid\forall i\in\{1,2,\ldots,d-1\}\ \Hom_{\TT}(X,\CC[i])=0\}.
\end{eqnarray*}
We say that an object $C$ is \emph{$d$-cluster tilting} if $\add C$ is a $d$-cluster tilting subcategory.
\end{definition}

In this paper, we apply these definitions for $\BB$ in the following settings:
\begin{itemize}
\item $\BB:=\CM^{\L}R$ (considered as a subcategory of $\AA:=\mod^{\L}R$) or $\TT:=\underline{\CM}^{\L}R$ for a Geigle-Lenzing complete intersection $(R,\L)$ (see Section~\ref{section: CMF and dCMF}).
\item $\BB:=\vect\X$ (considered as a subcategory of $\AA:=\coh\X$) for a Geigle-Lenzing projective space $\X$ (see Section~\ref{section: vector bundles}).
\item $\TT:=\DDD^{\bo}(\mod\Lambda)$ for a finite dimensional $k$-algebra $\Lambda$.
\end{itemize}

\medskip
In the rest of this section, let $\Lambda$ be a finite dimensional $k$-algebra of finite global dimension.
For an integer $d$, we define the \emph{$d$-shifted Nakayama functor} by
\[\nu_d:=(D\Lambda)[-d]\Lotimes_{\Lambda}-:\DDD^{\bo}(\mod \Lambda)\to\DDD^{\bo}(\mod \Lambda).\]
The following classes of finite dimensional algebras are central in this paper.

\begin{definition}\label{define d-RI}\cite{I3,IO1,HIO}
Let $\Lambda$ be a finite dimensional $k$-algebra of global dimension of finite global dimension.
\begin{itemize}
\item[(a)] We call $\Lambda$ \emph{$\nu_d$-finite} if $\nu_d^i(\DDD^{\ge0}(\mod\Lambda))\subset\DDD^{\ge1}(\mod\Lambda)$ holds for some $i>0$.
\item[(b)] We call $\Lambda$ \emph{$d$-representation infinite} if
$\nu_d^{-i}(\Lambda)\in\mod \Lambda$ for all $i\ge0$.
This is equivalent to that $\nu_d^i(D\Lambda)\in\mod\Lambda$ for all
$i\ge0$ by \cite[2.9]{HIO}.
\end{itemize}
\end{definition}

Clearly any algebra with global dimension at most $d-1$ is $\nu_d$-finite.
Also it is easy to see that $d$-representation infinite algebras have global dimension $d$.

The notion of higher preprojective algebras plays an important role.

\medskip\noindent
\begin{definition}[Higher preprojective algebra]\label{define preprojective algebra}
Let $\Lambda$ be a finite dimensional $k$-algebra of finite global dimension.
The \emph{$(d+1)$-preprojective algebra} of $\Lambda$ is defined as the $\Z$-graded $k$-algebra
\[\Pi=\Pi_{d+1}(\Lambda):=\bigoplus_{\ell\in\Z}\Hom_{\DDD^{\bo}(\mod \Lambda)}(\Lambda,\nu_d^\ell(\Lambda)),\]
where the multiplication is given by
\[f\cdot g:=f\nu_d^\ell(g)\in\Hom_{\DDD^{\bo}(\mod \Lambda)}(\Lambda,\nu_d^{\ell+m}(\Lambda))\]
for any $f\in\Hom_{\DDD^{\bo}(\mod \Lambda)}(\Lambda,\nu_d^\ell(\Lambda))$ and $g\in\Hom_{\DDD^{\bo}(\mod \Lambda)}(\Lambda,\nu_d^m(\Lambda))$.

A $d$-representation infinite algebra $\Lambda$ is called
\emph{$d$-representation tame} if the center $Z$ of $\Pi$ is
a Noetherian ring and $\Pi$ is a finitely generated $Z$-module.
\end{definition}

Clearly an algebra $\Lambda$ with global dimension at most $d$ is $\nu_d$-finite if and only if $\dim_k\Pi$ is finite.
In this case, there is a systematic construction of $d$-cluster tilting subcategories
of $\DDD^{\bo}(\mod\Lambda)$:

\begin{theorem}\label{tau_d finite has d-CT}\cite[1.23]{I3}
Let $\Lambda$ be a finite dimensional $k$-algebra with global dimension at most $d$ that is
$\nu_d$-finite. Then $\DDD^{\bo}(\mod\Lambda)$ has the $d$-cluster tilting subcategory
\[\UU_\Lambda:=\add\{\nu_d^i(\Lambda)\mid i\in\Z\}.\]
\end{theorem}

Now we consider a $d$-representation infinite 
algebra $\Lambda$ with $\Pi=\Pi_{d+1}(\Lambda)$. We assume that $\Pi$ is 
\emph{left graded coherent}, that is, finitely presented $\Z$-graded 
$\Pi$-modules are closed under kernels.
Then the category $\mod^{\Z}\Pi$ of finitely presented $\Z$-graded 
$\Pi$-modules is abelian, and the category $\mod^{\Z}_0\Pi$ of finite 
dimensional $\Z$-graded $\Pi$-modules is a Serre subcategory of $\mod^{\Z}\Pi$.
Let
\[\qgr^{\Z}\Pi:=\mod^{\Z}\Pi/\mod^{\Z}_0\Pi\]
be the quotient category. This is a non-commutative projective scheme in the sense of
Artin-Zhang \cite{AZ}. Let us recall the following result of the third author.

\begin{proposition}\cite[3.7]{M} 
Let $\Lambda$ be a $d$-representation infinite algebra such that $\Pi=\Pi_{d+1}(\Lambda)$ is left graded coherent.
Then there is a triangle equivalence $\DDD^{\bo}(\mod\Lambda)\simeq\DDD^{\bo}(\qgr^{\Z}\Pi)$ which sends $\Lambda$  to $\Pi$ and induces an equivalence
\[\{X\in\DDD^{\bo}(\mod\Lambda)\mid\forall\ell\gg0\ \nu_d^{-\ell}(X)\in\mod\Lambda\}\simeq\qgr^{\Z}\Pi.\]
\end{proposition}

Now we define an extension-closed subcategory of
$\qgr^{\Z}\Pi\subset\DDD^{\bo}(\mod\Lambda)$ by
\begin{eqnarray*}
\VV_\Lambda:=\{X\in\DDD^{\bo}(\mod\Lambda)\mid\forall\ell\gg0\ 
\nu_d^{-\ell}(X)\in\mod\Lambda,\ \nu_d^\ell(X)\in(\mod\Lambda)[-d]\}.
\end{eqnarray*}
This category should be considered as the category of vector bundles on the non-commutative projective scheme $\qgr^{\Z}\Pi$.

We have the following main result in this section.

\begin{theorem}\label{d-RI has d-CT} 
Let $\Lambda$ be a $d$-representation infinite algebra such that $\Pi=\Pi_{d+1}(\Lambda)$
is left graded coherent. Then $\VV_\Lambda$ has the $d$-cluster tilting subcategory
\[\UU_\Lambda:=\add\{\nu_d^i(\Lambda)\mid i\in\Z\}.\]
\end{theorem}

To prove Theorem~\ref{d-RI has d-CT}, the following observation is 
crucial, where $*$ denotes the extension of categories (see Section~\ref{section: triangulated category}).

\begin{proposition}\label{describe coh_d}
We have $\VV_\Lambda=
(\UU_\Lambda*\UU_\Lambda[1]*\cdots*\UU_\Lambda[d-1])\cap
(\UU_\Lambda[1-d]*\cdots*\UU_\Lambda[-1]*\UU_\Lambda)$.
\end{proposition}

We start by preparing the following easy observation.

\begin{lemma}\label{long exact}
Let $X$ and $C^i$ be $\Lambda$-modules.
If $X\in C^n[-n]*\cdots*C^1[-1]*C^0$, then there exists an exact 
sequence $0\to X\to C^0\to\cdots\to C^n\to0$ of $\Lambda$-modules.
\end{lemma}

\begin{proof}
We use the induction on $n$.
Since $X\in C^n[-n]*\cdots*C^1[-1]*C^0$, there exists a triangle
\begin{equation}\label{smaller case}
Y[-1]\to X\to C^0\to Y
\end{equation}
in $\DDD^{\bo}(\mod\Lambda)$ with $Y\in C^n[1-n]*\cdots*C^1$.
Then we have
\[Y\in(C^n[1-n]*\cdots*C^1)\cap (C^0*X[1])\subset\mod\Lambda\]
by looking at cohomologies. Applying $H^0$ to the triangle
\eqref{smaller case}, we have an exact sequence
$0\to X\to C^0\to Y\to0$. On the other hand, by the induction assumption, there exists an exact sequence
$0\to Y\to C^1\to\cdots\to C^n\to0$. Combining these sequences,
we have the assertion.
\end{proof}

Now we are ready to prove Proposition~\ref{describe coh_d}.

\begin{proof}[Proof of Proposition~\ref{describe coh_d}.]
(i) We prove ``$\supset$''.

Fix $X$ in the right hand side. For any $U_i,U^i\in\UU_\Lambda$, we have for $\ell\gg0$
\begin{eqnarray*}
\nu_d^{-\ell}(U_0*U_1[1]*\cdots*U_{d-1}[d-1])
&=&(\nu_d^{-\ell}(U_0)*\nu_d^{-\ell}(U_1)[1]*\cdots*\nu_d^{-\ell}(U_{d-1})[d-1])\\
&\subset&(\mod\Lambda)*(\mod\Lambda)[1]*\cdots*(\mod\Lambda)[d-1],\\
\nu_d^{-\ell}(U^{1-d}[1-d]*\cdots*U^{-1}[-1]*U^0)
&=&(\nu_d^{-\ell}(U^{1-d})[1-d]*\cdots*\nu_d^{-\ell}(U^{-1})[-1]*\nu_d^{-\ell}(U^0))\\
&\subset&(\mod\Lambda)[1-d]*\cdots*(\mod\Lambda)[-1]*(\mod\Lambda).
\end{eqnarray*}
Therefore we have
\begin{eqnarray*}
\nu_d^{-\ell}(X)&\in&
((\mod\Lambda)*\cdots*(\mod\Lambda)[d-1])
\cap((\mod\Lambda)[1-d]*\cdots*(\mod\Lambda))\\
&=&\mod\Lambda.
\end{eqnarray*}
for $\ell\gg0$. By a similar argument, we have
$\nu_d^{\ell}(X)\in(\mod\Lambda)[-d]$ for $\ell\gg0$. Therefore 
$X\in\VV_\Lambda$.

(ii) We prove ``$\subset$''.
We only prove $\VV_\Lambda\subset\UU_\Lambda*\UU_\Lambda[1]*\cdots*\UU_\Lambda[d-1]$ since one can show $\VV_\Lambda\subset\UU_\Lambda[1-d]*\cdots*\UU_\Lambda[-1]*\UU_\Lambda$ dually.

Let $X\in\VV_\Lambda$. Then $Y:=\nu_d^{\ell}(X)[d]$ and $Z:=\nu_d^{-\ell}(X)$ belong to $\mod\Lambda$ for $\ell\gg0$. Since $\Lambda$ has global dimension $d$, we can take an injective resolution $0\to Y\to I^0\to\cdots\to I^d\to0$ of $Y$. Then
\[Y\in I^d[-d]*\cdots*I^1[-1]*I^0\]
holds. Now let
\[\PP_\Lambda:=\add\{\nu_d^{-i}(\Lambda)\mid i\ge0\}\subset\UU_\Lambda\cap\mod\Lambda\]
be the category of $d$-preprojective $\Lambda$-modules.
Then $P^i:=\nu_d^{-2\ell}(I^i)[-d]$ belongs $\PP_\Lambda$, and we have
\[Z=\nu_d^{-2\ell}(Y)[-d]\in P^d[-d]*\cdots* P^1[-1]*P^0.\]
By Lemma~\ref{long exact}, we have an exact sequence
$0\to Z\to P^0\to\cdots\to P^d\to0$ of $\Lambda$-modules,
and in particular, $Z$ is a submodule of $P^0\in\PP_\Lambda$.
By applying \cite[4.28]{HIO}, we have an exact sequence
\[0\to P_{d-1}\to\cdots\to P_0\to Z\to0\]
of $\Lambda$-modules with $P_i\in\PP_\Lambda$. Therefore
\[Z\in\PP_\Lambda*\PP_\Lambda[1]*\cdots*\PP_\Lambda[d-1]\]
holds, and hence $X=\nu_d^{\ell}(Z)\in\UU_\Lambda*\UU_\Lambda[1]*\cdots*\UU_\Lambda[d-1]$.
\end{proof}

Now we are ready to prove Theorem~\ref{d-RI has d-CT}.

\begin{proof}[Proof of Theorem~\ref{d-RI has d-CT}.]
(i) It was shown in \cite[4.2]{HIO} that $\Hom_{\DDD^{\bo}(\mod\Lambda)}(\UU_\Lambda,\UU_\Lambda[i])=0$ holds for all $i$ with $1\le i\le d-1$.

(ii) For any $X\in\VV_\Lambda$, Proposition~\ref{describe coh_d} shows that there exists a triangle
\begin{equation}\label{approximation}
U\xrightarrow{g} X\xrightarrow{f}Y\to U[1]
\end{equation}
with $U\in\UU_\Lambda$ and $Y\in\UU_\Lambda[1]*\cdots*\UU_\Lambda[d-1]$.
Since $\Hom_{\DDD^{\bo}(\mod\Lambda)}(\UU_\Lambda,\UU_\Lambda[1]*\cdots*\UU_\Lambda[d-1])=0$, we have that $g$ is a right $\UU_\Lambda$-approximation of $X$.
Thus $\UU_\Lambda$ is a contravariantly finite subcategory of $\VV_\Lambda$.

Dually one can show that $\UU_\Lambda$ is a covariantly finite subcategory of $\VV_\Lambda$.

(iii) Assume that $X\in\VV_\Lambda$ satisfies $\Hom_{\DDD^{\bo}(\mod\Lambda)}(X,\UU_\Lambda[i])=0$ for all $i$ with $1\le i\le d-1$.
Then $f=0$ holds in the triangle \eqref{approximation}. Thus $g$ is a split epimorphism, and we have $X\in\UU_\Lambda$.

Similarly one can show that if $X\in\VV_\Lambda$ satisfies $\Hom_{\DDD^{\bo}(\mod\Lambda)}(\UU_\Lambda,X[i])=0$ for all $i$ with $1\le i\le d-1$, then $X\in\UU_\Lambda$.

(iv) Fix any $X\in\VV_\Lambda$. By Proposition~\ref{describe coh_d}, there exist exact sequences
\[0\to U_{d-1}\to\cdots\to U_0\to X\to0\ \mbox{ and }\ 0\to X\to U^0\to\cdots\to U^{d-1}\to0\]
with $U_i,U^i\in\UU_\Lambda$ in the abelian category $\qgr^{\Z}\Pi$.
Therefore $\UU_\Lambda$ generates and cogenerates $\VV_\Lambda$.
\end{proof}

At the end of this section, we include the following observation on a generalization of $d$-representation infinite algebras, which will be applied
to our higher canonical algebras (see Theorem~\ref{global dimension}).

\begin{definition}\label{define almost d-RI}
Let $\Lambda$ be a finite dimensional $k$-algebra of finite global dimension.
We call $\Lambda$ \emph{almost $d$-representation infinite} if
$H^j(\nu_d^{-i}(\Lambda))=0$ holds for all $i\in\Z$ and all $j\in\Z\setminus\{0,d\}$.
\end{definition}

Clearly any $d$-representation infinite algebra is almost $d$-representation infinite. Moreover we have the following easy observations.

\begin{proposition}\label{d or 2d}
\begin{itemize}
\item[(a)] $d$-representation infinite algebras are precisely
almost $d$-representation infinite algebras of global dimension $d$
\item[(b)] An almost $d$-representation infinite algebra has global dimension $d$ or $2d$.
\end{itemize}
\end{proposition}

\begin{proof}
(a) We only have to show that any almost $d$-representation infinite algebra
$\Lambda$ of global dimension $d$ is $d$-representation infinite.
Since $\nu_d^{-i}(\Lambda)\in\DDD^{\le0}(\mod\Lambda)$ holds for any
$i\ge0$ by Proposition~\ref{fractional CY and global dimension}(b),
we have $\nu_d^{-i}(\Lambda)\in\mod\Lambda$. Thus the assertion follows.

(b) Since $\Lambda$ has finite global dimension,
$\gl \Lambda=\max\{i\ge0\mid\Ext^i_\Lambda(D\Lambda,\Lambda)\neq0\}$
holds by Observation~\ref{gl dim}. Since
\[\Ext^i_\Lambda(D\Lambda,\Lambda)=
\Hom_{\DDD(\Lambda)}(\Lambda,\nu_d^{-1}(\Lambda)[i-d])=
H^{i-d}(\nu_d^{-1}(\Lambda))\]
vanishes except $i=d$ or $2d$, we have the assertion.
\end{proof}

%
%
%

\part{Geigle-Lenzing complete intersections}

\chapter{Geigle-Lenzing complete intersections}\label{section: GL CI}

Throughout this paper we fix an arbitrary base field $k$ and an integer $d\ge-1$. 
(We assume neither $k$ to have characteristic zero nor $k$ to be algebraically closed.)

\section{The definition and basic properties}\label{section: GL CI 1}
We start with the polynomial algebra
\[C:=k[T_0, \ldots, T_d]\]
in $d+1$ variables. For $n\ge0$, we choose $n$ linear forms 
\[ \ell_i =\ell_i(T_0,\ldots,T_d) = \sum_{j=0}^d \lambda_{ij} T_j\in C,\]
with $\lambda_{ij}\in k$.
We also fix an $n$-tuple $(p_1,\ldots,p_n)$ of positive integers called \emph{weights}. Let
\[ S := C[X_1,\ldots, X_n] = k[T_0, \ldots, T_d, X_1, \ldots, X_n ] \]
be the polynomial algebra in $d+n+1$ variables and
\[ h_i := X_i^{p_i}- \ell_i\in S. \]
Now we consider the factor $k$-algebra
\[R := S / (h_i \mid 1 \leq i \leq n).\]
In the case $d=-1$, we let $C=k$ and $\ell_1=\cdots=\ell_n=0$. Thus
\begin{equation}\label{R for d=-1}
R=k[X_1,\ldots,X_n]/ (X_i^{p_i} \mid 1 \leq i \leq n).
\end{equation}

\medskip\noindent
{\bf The grading group.}
Let $\L$ be the abelian group generated by the symbols $\x_1,\ldots,\x_n,\c$, 
modulo relations $p_i \x_i = \c$ for any $1 \leq i \leq n$:
\begin{eqnarray*}
\L := \langle \x_1,\ldots,\x_n, \c \rangle / \langle p_i \x_i - \c \mid 1\le i \le n\rangle.
\end{eqnarray*}
We regard $S$ as an $\L$-graded $k$-algebra by
\[\deg T_j:= \c\ \mbox{ and }\ \deg X_i:=\x_i\]
for any $i$ and $j$. Since $h_i$ is homogeneous (of degree $\c$) for all $i$, we can regard $R$ as an $\L$-graded $k$-algebra.

\medskip\noindent
{\bf Geigle-Lenzing complete intersection.}
We call the pair $(R,\L)$ a \emph{weak Geigle-Lenzing} (\emph{GL})
\emph{complete intersection} associated with $\ell_1,\ldots,\ell_n$ and
$p_1,\ldots,p_n$.
It is in fact a complete intersection of dimension $d+1$
as we will see in Proposition~\ref{R is CI} below.

We call $R$ \emph{Geigle-Lenzing} (\emph{GL}) \emph{complete intersection} if our
linear forms $\ell_1,\ldots,\ell_n$ are in general position in the following sense:
\begin{itemize}
\item Any set of at most $d+1$ of the polynomials $\ell_i$ is linearly independent.
\end{itemize}
In the rest we assume that $R$ is a weak GL complete intersection.

\medskip
Let $\L_+$ be the submonoid of $\L$ generated by all $\x_i$'s and $\c$.
We equip $\L$ with the structure of a partially ordered set:
$\x\ge\y$ if and only if $\x-\y\in\L_+$. Then $\L_+$ consists of all
elements $\x\in\L$ satisfying $\x\ge0$. We denote intervals in $\L$ by
\[[\x,\y] := \{\z \in\L \mid \x \le \z \le \y\}.\]
We collect some basic observations.

\begin{observation}\label{basic results on L}
\begin{itemize}
\item[(a)] Any element $\x\in\L$ can be written uniquely as
\[\x=\sum_{i=1}^na_i\x_i+a\c\]
with $0\le a_i<p_i$ and $a\in\Z$. We call this presentation the \emph{normal form} of $\x$.
\item[(b)] $\L$ is an abelian group of rank one. It is torsion free
if and only if $p_1,\ldots,p_n$ are pairwise coprime.
\item[(c)] We have $R_{\x}\neq0$ if and only if $\x\in\L_+$ if and only if $a\ge0$
in the normal form in (a). Therefore $R$ is positively graded in this sense.
\end{itemize}
\end{observation}

\begin{observation}[Weights 1]\label{Weights 1}
Adding a linear form $\ell_{n+1}$ with weight $p_{n+1} = 1$ changes neither $\L$ nor $R$, since the new variable $X_{n+1}$ is expressed as a linear
combination of $T_j$'s by the relation $X_{n+1} = \ell_{n+1}(T_0,\ldots,T_d)$.
Thus we may freely add or remove hyperplanes with weights $1$.

Therefore we can assume that
\begin{itemize}
\item $p_i\ge2$ for all $i$ with $1\le i\le n$
\end{itemize}
without loss of generality by removing all hyperplanes with weights $1$.
\end{observation}

\begin{observation}[Normalization]\label{Normalization}
Now we assume that $(R,\L)$ is a GL complete intersection.
Then the group ${\rm GL}(d+1, k)$ acts on $S$ by acting on the linear span of the variables $T_i$. Transforming coordinates in this way we may assume
\[\ell_i(T_0,\ldots,T_d) = \begin{cases} T_{i-1}& \text{ if }\ 1 \leq i \leq \min \{d+1, n\}, \\
\sum_{j=0}^d\lambda_{ij}T_j & \text{ if }\ \min \{d+1, n\} < i \leq n. \end{cases}.\]
Then we obtain the relations $h_i =  X_i^{p_i} - T_{i-1}$ for
$1 \leq i \leq \min\{d+1,n\}$. Therefore the variables $T_i$ with
$0 \leq i \leq \min\{d,n-1\}$ are superfluous in the presentation of $R$, and we may write
\begin{equation}\label{normalized form}
R = \begin{cases} k[X_1, \ldots, X_n, T_n, \ldots, T_d] & \text{ if } n \leq d+1, \\
k[X_1, \ldots, X_n] / (X_i^{p_i} - \sum_{j=1}^{d+1}\lambda_{i,j-1}X_j^{p_j} \mid d+2 \leq i \leq n) & \text{ if } n \geq d+2. \end{cases}
\end{equation}
\end{observation}

In the second case of \eqref{normalized form}, our assumption that $\ell_1,\ldots,\ell_n$ are in general
position is equivalent to that all minors (including non-maximal ones) of
the $(n-d-1)\times(d+1)$ matrix
\[[\lambda_{i,j-1}]_{d+2\le i\le n,\ 1\le j\le d+1}\]
have non-zero determinants. In the case $d=1$, this means that the $n$ points
\[(1:0),\ (0:1),\ (\lambda_{30}:\lambda_{31}),\ldots,(\lambda_{n0}:\lambda_{n1})\]
in $\P^1$ are mutually distinct. If $k$ is algebraically closed, then 
we can normalize the relation in \eqref{normalized form} for $i=d+2$ as
\[X_{d+2}^{p_{d+2}}=X_1^{p_1}+X_2^{p_2}+\cdots+X_{d+1}^{p_{d+1}}.\]
This presentation is widely used for $d=1$.

Let us observe that our weak GL complete intersections can be obtained by
the following elementary construction,
for which the name `root construction' was used for stacks \cite{AGV}.
It also appeared in the context of weighted projective varieties in \cite{LO}.

\begin{observation}[Root construction]\label{root construction}
Let $G$ be an abelian group, and $A$ a commutative $G$-graded ring.
For a non-zero homogeneous element $\ell\in A$ of degree $g\in G$
and a positive integer $p$, let
\[G(\ell,p):=(G\oplus\Z)/\langle(g,-p)\rangle\ \mbox{ and }\
A(\ell,p):=A[X]/(X^p-\ell).\]
Then $A(\ell,p)$ is a $G(\ell,p)$-graded ring by $\deg X:=(0,1)$ and
$\deg a:=(g',0)$ for any $a\in A_{g'}$ with $g'\in G$.
We call this process to construct $(A(\ell,p),G(\ell,p))$ from $(A,G)$
a \emph{root construction}. Clearly,
\begin{itemize}
\item $A(\ell,p)$ is a free $A$-module with a basis $\{X^i\mid0\le i<p\}$.
\end{itemize}
Our weak GL complete intersection $(R,\L)$ can be obtained
by applying root construction iteratively to the polynomial ring
$C=k[T_0,\ldots,T_d]$ with the standard $\Z$-grading.
In fact, let
\begin{eqnarray*}
(R^0,\L^0)&:=&(C,\Z),\\
(R^i,\L^i)&:=&(R^{i-1}(\ell_i,p_i),\L^{i-1}(\ell_i,p_i))
\end{eqnarray*}
for $1\le i\le n$. Then one can easily check that $(R,\L)=(R^n,\L^n)$ holds.
\end{observation}

The following simple observation is quite useful.

\begin{proposition}\label{Zc Veronese}
\begin{itemize}
\item[(a)] $C$ is the $(\Z\c)$-Veronese subalgebra of $R$, that is,
$C=\bigoplus_{a\in\Z}R_{a\c}$.
\item[(b)] $R$ is a free $C$-module of rank $p_1p_2\cdots p_n$
with a basis $\{X_1^{a_1}X_2^{a_2}\cdots X_n^{a_n}\mid 0\le a_i<p_i\}$.
\item[(c)] Let $\x=\sum_{i=1}^na_i\x_i+a\c$ be a normal form of $\x\in\L$.
Then the multiplication map $X_1^{a_1}X_2^{a_2}\cdots X_n^{a_n}:C_{a\c}\to R_{\x}$
is bijective.
\end{itemize}
\end{proposition}

\begin{proof}
(b) Since $R$ is obtained form $C$ by applying root construction 
iteratively, the assertion is clear.
 
(a)(c) Immediate from (b).
\end{proof}

\begin{definition}[Regular sequences]
Let $(a_1,\ldots,a_\ell)$ be a sequence of homogeneous elements in
$S$ whose degrees are in $\L_+\setminus\{0\}$.
For $M\in\mod^{\L}S$,
we say that $(a_1,\ldots,a_\ell)$ is an \emph{$M$-regular sequence} 
\cite{BH} if the multiplication map
\[a_i:M/M(a_1,\ldots,a_{i-1})\to M/M(a_1,\ldots,a_{i-1})\]
is injective for any $1\le i\le\ell$.
\end{definition}

Any permutation of an $M$-regular sequence is again an $M$-regular sequence. For any positive integers $q_1,\ldots,q_\ell$, the sequence $(a_1,\ldots,a_\ell)$ is $M$-regular if and only if $(a_1^{q_1},\ldots,a_\ell^{q_\ell})$ is $M$-regular.

We prepare the following easy observations.

\begin{lemma} \label{prop.regularseq1}
\begin{itemize}
\item[(a)] $(h_1,\ldots,h_n)$ is an $S$-regular sequence.
\item[(b)] Let $f_0,\ldots,f_d$ be linearly independent linear forms in $C$.
Then $(h_1, \ldots, h_n, f_0, \ldots, f_d)$ is an $S$-regular sequence, and 
$(f_0, \ldots, f_d)$ is an $R$-regular sequence.
\item[(c)] Let $i_0, \ldots, i_s \in \{1, \ldots, n\}$ and
$f_{s+1}, \ldots f_d \in C$ be linear forms. If $\ell_{i_0}, \ldots, \ell_{i_s}$,
$f_{s+1}, \ldots, f_d$ are linearly independent in $C$, then
$(X_{i_0}, \ldots, X_{i_s}, f_{s+1}, \ldots, f_d)$ is an $R$-regular sequence.
\end{itemize}
\end{lemma}

\begin{proof}
(b) Let $S'$ be the $k$-subalgebra of $S$ generated by
$T_0,\ldots,T_d,X_1^{p_1},\ldots,X_n^{p_n}$.
Then $S'$ is a polynomial algebra with these variables. Since
$h_1, \ldots, h_n, f_0, \ldots, f_d$ are linearly independent linear forms in $S'$,
they form an $S'$-regular sequence.
Since $S$ is a free $S'$-module of finite rank, we have the assertion.

(a) Immediate from (b).

(c) The latter assertion in (b) implies that $(\ell_{i_0}=X_{i_0}^{p_{i_0}}, \ldots,
\ell_{i_s}=X_{i_s}^{p_{i_s}}, f_{s+1}, \ldots, f_d)$ is an $R$-regular sequence.
\end{proof}

Immediately we have the following observations.

\begin{proposition}\label{R is CI}
Let $R$ be a weak GL complete intersection.
\begin{itemize}
\item[(a)] $R$ is a complete intersection of dimension $d+1$. Also $R_\mm$ is a complete intersection of dimension $d+1$ for all maximal ideals $\mm$ of $R$.
\item[(b)] Assume that $R$ is a GL complete intersection and $p_i\ge 2$ for all $i$.
\begin{itemize}
\item[$\bullet$] $R$ is regular if and only if $n\le d+1$.
\item[$\bullet$] $R$ is a hypersurface if and only if $n\le d+2$.
\end{itemize}
\end{itemize}
\end{proposition}

\begin{proof}
(a) This is immediate from Lemma~\ref{prop.regularseq1}(a).

(b) By Observation~\ref{Normalization}, the number of minimal generators
of the maximal ideal $R_+=\bigoplus_{\x>0}R_{\x}$ of $R$ is
$\max\{d+1,n\}$ even after localizing at $R_+$. Thus the assertion follows.
\end{proof}

\begin{definition}[Dualizing element]\label{define w}
From Proposition~\ref{R is CI}, it follows that $R$ is a Gorenstein ring. Thus
\[\Ext_R^i(k,R)=\left\{\begin{array}{cc}
0    &i\neq d\\
k(-\w)&i=d
\end{array}\right.\]
holds for some element $\w\in\L$, which is called the \emph{dualizing element}
(also known as \emph{$a$-invariant} \cite{BH}, \emph{Gorenstein parameter}) of $(R,\L)$. In this case, $\omega_R:=R(\w)$ is
called the \emph{canonical module} of $(R,\L)$.\end{definition}

We have the following explicit formula for $\w$.

\begin{proposition}\label{a-invariant}
The dualizing element of $(R,\L)$ is given by
\[\w=(n-d-1)\c-\sum_{i=1}^n\x_i\in\L.\]
\end{proposition}

\begin{proof}
Since $(S,\L)$ is a polynomial ring, its $a$-invariant is given
by minus the sum of the degrees of 
all variables, i.e.\ $a(S):=-(d+1)\c-\sum_{i=1}^{n}\x_i$. Since $h_1,\ldots,h_n$ is an $S$-regular sequence by 
Lemma~\ref{prop.regularseq1}(a), we have
\[\w=a(S)+\sum_{i=1}^n\deg h_i=-(d+1)\c-\sum_{i=1}^{n}\x_i+n\c\]
by using standard commutative algebra \cite[3.6.14]{BH}.
\end{proof}

The following degree map and trichotomy are important.

\begin{definition}[Degree map and trichotomy]
Let $(R,\L)$ be a GL complete intersection.
We define a homomorphism $\delta:\L\to\Q$ called the \emph{degree map}
by $\delta(\x_i)=\frac{1}{p_i}$ and $\delta(\c)=1$. Using
\begin{equation}\label{degree of w}
\delta(\w)=n-d-1-\sum_{i=1}^n\frac{1}{p_i}\in\Q,
\end{equation}
we say that $(R,\L)$ is \emph{Fano} (respectively, \emph{Calabi-Yau},
\emph{anti-Fano}) if $\delta(\w)<0$ (respectively, $\delta(\w)=0$,
$\delta(\w)>0$).
\end{definition}

We will see in  Theorem~\ref{AZ ampleness} that these definitions correspond to ampleness
of the automorphisms $(\pm\w)$ of the category $\coh\X$.

For example, if $n\le d+1$, then $R$ is Fano. For the integer
\[p:={\rm l.c.m.}(p_1,p_2,\ldots,p_n),\]
$(R,\L)$ is Fano (respectively, Calabi-Yau, anti-Fano) if and only if
$p\w=q\c$ holds for an integer $q<0$ (respectively, $q=0$, $q>0$).

\begin{example}
Let $d=-1$. Then $C=k$, $\ell_1=\cdots=\ell_n=0$ and $R$ is given by \eqref{R for d=-1}.
The linear independence condition is vacuously satisfied. In this case, $\w=\sum_{i=1}^n(p_i-1)\x_i$. Assume $p_i\ge2$ for all $i$.
\begin{itemize}
\item[(a)] $(R,\L)$ is never Fano.
\item[(b)] $(R,\L)$ is Calabi-Yau if and only if $n=0$ (i.e.\ $R=k$).
\item[(c)] All other cases are anti-Fano.
\end{itemize}
\end{example}

\begin{example}
Let $d=0$. Then $C=k[T]$, and for $1\le i\le n$, we have $\ell_i(T)=\lambda_iT$ for some $\lambda_i\in k\setminus\{0\}$ by the linear independence condition. Thus $R$ has the form
\begin{eqnarray*}
R&=&k[T,X_1,\ldots,X_n]/(X_i^{p_i}-\lambda_iT\mid 1\le i\le n)\\
&=&k[X_1,\ldots,X_n]/(\lambda_i^{-1}X_i^{p_i}-\lambda_j^{-1}X_j^{p_j}\mid 1\le i<j\le n).
\end{eqnarray*}
In this case, $\w=-\c+\sum_{i=1}^n(p_i-1)\x_i$. Assume $p_i\ge2$ for all $i$.
\begin{itemize}
\item[(a)] $(R,\L)$ is Fano if and only if $n\le 1$.
\item[(b)] $(R,\L)$ is Calabi-Yau if and only if the weights are $(2,2)$.
\item[(c)] All other cases are anti-Fano.
\end{itemize}
\end{example}

\begin{example}
Let $d=1$. Assume that $p_i\ge2$ for all $i$.
\begin{itemize}
\item[(a)] There are 5 types for Fano case: $(p,q)$, $(2,2,p)$, $(2,3,3)$, $(2,3,4)$ and $(2,3,5)$, corresponding to $\widetilde{\A}_{p+q-1}$, $\widetilde{\D}_{p+2}$, 
$\widetilde{\E}_6$, $\widetilde{\E}_7$ and $\widetilde{\E}_8$. 
\item[(b)] There are 4 types for Calabi-Yau case: $(3,3,3)$, $(2,4,4)$, $(2,3,6)$ and $(2,2,2,2)$, corresponding to $\E_6^{(1,1)}$, $\E_7^{(1,1)}$, $\E_8^{(1,1)}$ and $\D_4^{(1,1)}$.
\item[(c)] All other cases are anti-Fano.
\end{itemize}
This is nothing but the classical trichotomy of \emph{domestic}, \emph{tubular} and \emph{wild} types of weighted projective lines.
\end{example}

\begin{example}\label{2-canonical}
Let $d=2$. Assume that $p_i\ge2$ for all $i$.
\begin{itemize}
\item[(a)] There are the following cases for Fano.
\begin{itemize}
\item[$\bullet$] $n\le 3$.
\item[$\bullet$] 7 infinite series for $n=4$: $(2,2,p,q)$, $(2,3,3,p)$, $(2,3,4,p)$,
$(2,3,5,p)$, $(2,3,6,p)$, $(2,4,4,p)$ and $(3,3,3,p)$ for any $p,q$.
\item[$\bullet$] 110 remaing cases for $n=4$: $(2,3,7,p)$ for $7\le p\le 41$, 
$(2,3,8,p)$ for $8\le p\le 23$, $(2,3,9,p)$ for $9\le p\le 17$,
$(2,3,10,p)$ for $10\le p\le 14$, $(2,3,11,p)$ for $11\le p\le 13$,
$(2,4,5,p)$ for $5\le p\le 19$, $(2,4,6,p)$ for $6\le p\le 11$,
$(2,4,7,p)$ for $7\le p\le 9$, $(2,5,5,p)$ for $5\le p\le 9$,
$(2,5,6,p)$ for $7\le p\le 9$, $(3,3,4,p)$ for $4 \le p \le 11$ 
and $(3,3,5,p)$ for $5 \le p \le 7$.
\item[$\bullet$] 1 infinite series for $n=5$: $(2,2,2,2,p)$ for any $p$.
\item[$\bullet$] 3 remaining cases for $n=5$: $(2,2,2,3,p)$ for $3\le p\le 5$.
\end{itemize}
\item[(b)] There are 18 cases for Calabi-Yau case:
\begin{itemize}
\item[$\bullet$] 14 cases for $n=4$: $(2,3,7,42)$, $(2,3,8,24)$, $(2,3,9,18)$, $(2,3,10,15)$, $(2,3,12,12)$, $(2,4,5,20)$, $(2,4,6,12)$, $(2,4,8,8)$, $(2,5,5,10)$, $(2,6,6,6)$, $(3,3,4,12)$, $(3,3,6,6)$, $(3,4,4,6)$ and $(4,4,4,4)$.
\item[$\bullet$] 3 cases for $n=5$: $(2,2,2,3,6)$, $(2,2,2,4,4)$ and $(2,2,3,3,3)$.
\item[$\bullet$] 1 case for $n=6$: $(2,2,2,2,2,2)$.
\end{itemize}
\item[(c)] All other cases are anti-Fano.
\end{itemize}
\end{example}

We prepare the following statement, which relates $\w$ to the interval $[0,d\c]$.

\begin{lemma}\label{order omega}
Let $\x=\sum_{i=1}^na_i\x_i+a\c$ be a normal form, that is, $0\le a_i< p_i$ and $a\in\Z$.
Then the following conditions are equivalent.
\begin{itemize}
\item[(a)] $\x \le d\c$.
\item[(b)] $a+\#\{i\mid a_i>0\} \le d$.
\item[(c)] $0\not\le\x+\w$.
\end{itemize}
In particular, we have $[0,d\c]=\{\x\in\L\mid0\le\x,\ 0\not\le\x+\w\}$.
\end{lemma}

\begin{proof}
(a)$\Leftrightarrow$(b) The normal form of $d\c-\x$ is
\[d\c-\x=\sum_{a_i>0}(p_i-a_i)\x_i+(d-a-\#\{i\mid a_i>0\})\c.\]
So $d\c-\x\ge0$ if and only if $d-a-\#\{i\mid a_i>0\}\ge0$.

(b)$\Leftrightarrow$(c) The normal form of $\x + \w$ is
\[\x + \w = \sum_{a_i\neq 0}(a_i-1) \x_i+ \sum_{a_i= 0}(p_i-1) \x_i + (a-d-1+\#\{i\mid a_i>0\})\c.\]
So $\x + \w \not\ge 0$ if and only if $a-d-1+\#\{i\mid a_i>0\}\le-1$.
\end{proof}

The quotient groups $\L/\Z\c$ and $\L/\Z\w$ play an important role.
We give some easy observations.

\begin{proposition}\label{compare [0,dc] with L/w}
\begin{itemize}
\item[(a)] $\L/\Z\c$ is isomorphic to $\prod_{i=1}^n\Z/p_i\Z$.
\item[(b)] If $n\le d+1$, then the map $[0,d\c]\to\L/\Z\w$ is bijective.
\item[(c)] If $(R,\L)$ is Fano, then the map $[0,d\c]\to\L/\Z\w$ is surjective.
\item[(d)] If $(R,\L)$ is not Calabi-Yau, then the cardinality of 
$\L/\Z\w$ is equal to the absolute value of $(p_1p_2\cdots p_n)\delta(\w)$.
\end{itemize}
\end{proposition}

\begin{proof}
(a) This is clear.

(c) Fix $\x\in\L$. Since $(R,\L)$ is Fano, we have $\ell\w<0$ for $\ell\gg0$.
Therefore there exists an integer $\ell$ such that $\x+\ell\w\ge0$ and $\x+(\ell+1)\w \not\ge 0$. This is equivalent to $\x+\ell\w\in[0,d\c]$ by Lemma~\ref{order omega}.

(b) Since $\w<0$ holds by $n\le d+1$, the integer $\ell$ satisfying $\x+\ell\w\ge0$ and $\x+(\ell+1)\w \not\ge 0$ is unique. Therefore the assertion holds.

(d) We have $\#(\L/\Z\c)=p_1p_2\cdots p_n$ by (a). Since $\delta(\c)=1$, we have
\[\frac{\#(\L/\Z\w)}{p_1p_2\cdots p_n}=\frac{\#(\L/\Z\w)}{\#(\L/\Z\c)}=\frac{|\delta(\w)|}{|\delta(\c)|}=|\delta(\w)|.\qedhere\]
\end{proof}

\begin{remark}
The converse of Proposition \ref{compare [0,dc] with L/w}(c) is not true.
For example, let $d=1$, $n=3$ and $(p_1,p_2,p_3)=(2,3,7)$, then $(R,\L)$ is
anti-Fano. In this case, $\L=\Z\w$ holds since $\x_1=21\w$, $\x_2=14\w$ 
and $\x_3=6\w$.
In particular the map $[0,d\c]\to\L/\Z\w$ is surjective.
\end{remark}

\section{$R$ is $\L$-factorial and has $\L$-isolated singularities}\label{section: GL CI 2}

Let $(R,\L)$ be a GL complete intersection associated with 
linear forms $\ell_1,\ldots,\ell_n$ and weights $p_1,\ldots,p_n$.
In this section, we give some ring theoretic properties of $(R,\L)$.
In particular, we show that, in a graded sense, $(R,\L)$ is factorial
and has isolated singularities.

Let us start with introducing some notions for graded rings.

\begin{definition}\label{define L-isolated singularity}
Let $G$ be an abelian group and $A$ a commutative Noetherian
$G$-graded ring.
\begin{itemize}
\item[(a)] We say that $A$ is a \emph{$G$-domain} if a product of non-zero 
homogeneous elements is non-zero. We say that $A$ is a \emph{$G$-field} if any non-zero homogeneous element is invertible.
\item[(b)] A homogeneous ideal $\pp$ of $A$ is \emph{$G$-prime}
(respectively, \emph{$G$-maximal}) if $A/\pp$ is a $G$-domain (respectively, $G$-field).
In this case we denote by $A_{(\pp)}$ the localization of $A$ with respect to
the multiplicative set consisting of all homogeneous elements in $A-\pp$.
We denote by $\Spec^GA$ the set of all
$G$-prime ideals of $A$.
\item[(c)] A non-zero homogeneous element $a\in A$ is a
\emph{$G$-prime element} if the principal ideal $Aa$ is a $G$-prime 
ideal of $A$. A $G$-domain $A$ is \emph{$G$-factorial} if any non-zero homogeneous
element in $A$ is a product of $G$-prime elements in $A$.
\item[(d)] We say that $A$ is \emph{$G$-regular} if $\mod^GA$ has finite 
global dimension. We say that $A$ has (at worst) \emph{$G$-isolated singularities} if 
$A_{(\pp)}$ is $G$-regular for any $\pp\in\Spec^GA$ which is not $G$-maximal.
\end{itemize}
\end{definition}

When the group is trivial $G=\{1\}$, we recover the usual notions of domain, 
field etc. 

\begin{remark}
These notions depend not only on the ring $A$ but also on the group $G$.
As a simple example, let $k$ be a field with characteristic $2$ and
$A:=k[x]/(1+x^2)=k[x]/(1+x)^2$. Then $A$ is neither a field nor regular.
On the other hand, regarding $A$ as a $(\Z/2\Z)$-graded ring by
$\deg x=1$, we have that $A$ is a $(\Z/2\Z)$-field and $(\Z/2\Z)$-regular.
\end{remark}

We start with a few easy observations.

\begin{observation}\label{observation for G-prime}
\begin{itemize}
\item[(a)] Any $G$-field is a $G$-domain. Hence any 
$G$-maximal ideal is a $G$-prime ideal.
\item[(b)] If $A$ is $G$-regular, then $A_{(\pp)}$ is $G$-regular for any
$\pp\in\Spec^GA$.
\end{itemize}
\end{observation}

\medskip
We show the following result, which generalizes \cite[1.3]{GL} for the case $d=1$.

\begin{theorem}\label{L-factorial L-domain}
A GL complete intersection $(R,\L)$ is an $\L$-domain if $d\ge0$, and an $\L$-factorial $\L$-domain if $d\ge1$. More generally, a weak GL complete intersection is an $\L$-domain (respectively, $\L$-factorial $\L$-domain) if the linear forms $\ell_1,\ldots,\ell_n$ are non-zero (respectively, pairwise linearly independent).
\end{theorem}

\begin{proof}
We use the following general argument due to Lenzing.

\begin{proposition}\label{L-factorial induction}
For a $G$-domain $A$, let $(A',G'):=(A(\ell,p),G(\ell,p))$ be a root 
construction as in Observation~\ref{root construction}. Assume $\ell\neq0$.
\begin{itemize}
\item[(a)] $A'=A[X]/(X^p-\ell)$ is a $G'$-domain.
\item[(b)] Assume that $A$ is $G$-factorial and $\ell$ is a $G$-prime element.
Then $A'$ is $G'$-factorial. Moreover, $G'$-prime elements in $A'$ are either $X$
or $G$-prime elements $a\in A$ satisfying $Aa\neq A\ell$.
\end{itemize}
\end{proposition}

\begin{proof}
Note that $A'$ is a free $A$-module of rank $p$ with a basis $\{X^i\mid 0\le i<p\}$.
Moreover, any non-zero homogeneous element in $A'$ can be written uniquely as
$aX^i$ for some non-zero homogeneous element $a\in A$ and $0\le i<p$.

(a) Let $aX^i$ and $bX^j$ be homogeneous elements in $A'$
with $a,b\in A$ and $0\le i,j<p$. Then
\[(aX^i)(bX^{j})=\left\{\begin{array}{ll}
(ab)X^{i+j}&\mbox{if $i+j<p$,}\\
(ab\ell)X^{i+j-p}&\mbox{if $i+j\ge p$.}
\end{array}\right.\]
Since $A$ is a $G$-domain and $\ell\neq0$, both $ab$ and $ab\ell$
are non-zero. Thus $(aX^i)(bX^{j})\neq0$ holds.

(b) Since $\ell\in A$ is $G$-prime, $A'/A'X=A/A\ell$ is a $G$-domain. 
Hence this is a $G'$-domain, and $X\in A'$ is $G'$-prime.

Let $a\in A$ be a $G$-prime element such that $Aa\neq A\ell$.
Since $A/Aa$ is a $G$-domain and $\ell\neq0$ in $A/Aa$ by our 
assumption, it follows from (a) that
\[A'/A'a=(A/Aa)[X]/(X^p-\ell)=(A/Aa)(\ell,p)\]
is a $G'$-domain. Thus $a$ is a $G'$-prime element in $A'$.

Now we show that $A'$ is $G'$-factorial.
Fix a non-zero homogeneous element $aX^i\in A'$ with $a\in A$ and $0\le i<p$.
Since $A$ is $G$-factorial, we can write $a=a_1\cdots a_s\ell^t$
for $G$-prime elements $a_j$ satisfying $Aa_j\neq A\ell$ and $t\ge0$.
Then $aX^i$ is a product
\[aX^i=a_1\cdots a_sX^{pt+i}\]
of $G'$-prime elements in $A'$. Thus the assertion follows.
\end{proof}

Now Theorem~\ref{L-factorial L-domain} follows immediately from
Observation~\ref{root construction} and Proposition~\ref{L-factorial induction}.
\end{proof}

If $d\ge0$, then the zero ideal $(0)$ of $R$ is an $\L$-prime ideal
by Theorem~\ref{L-factorial L-domain}. Therefore the localization $R_{(0)}$ of $R$ is
an $\L$-field, and its degree $0$ part $(R_{(0)})_0$ is a field.

\begin{definition}[Rank function]
Assume $d\ge0$. For $X\in\mod^{\L}R$, we define the \emph{rank} of $X$ by
\[\rank X:=\dim_{(R_{(0)})_0}(X_{(0)})_0.\]
\end{definition}

We need the following observations, where $K_0(\mod^{\L}R)$
is the Grothendieck group of $\mod^{\L}R$.

\begin{proposition}\label{property of rank}
Assume $d\ge0$.
\begin{itemize}
\item[(a)] We have an equivalence $(-)_0:\mod^{\L}R_{(0)}\simeq\mod(R_{(0)})_0$.
\item[(b)] $\rank$ is given by a morphism of abelian groups
\[(-)_{(0)}:K_0(\mod^{\L}R)\to K_0(\mod(R_{(0)})_0)\simeq\Z.\]
\item[(c)] For any $\x\in\L$ and any non-zero submodule $X$ of $R(\x)$, we have $\rank X=\rank R(\x)=1$.
\end{itemize}
\end{proposition}

\begin{proof}
(a) $R_{(0)}$ is strongly graded in the sense that 
$(R_{(0)})_{\x}\cdot(R_{(0)})_{-\x}=(R_{(0)})_0$ for any $\x\in\L$.
Thus the assertion follows from \cite[I.3.4]{NV2}.

(b) This is clear.

(c) For any $\x\in\L$, there exist monomials $r,s\in R$ such that 
$\x=\deg r-\deg s$. Thus we have an isomorphism $rs^{-1}:(R_{(0)})_0\simeq(R(\x)_{(0)})_0$ of $(R_{(0)})_0$-modules, and $\rank R(\x)=1$.
Since the inclusion $X\subset R(\x)$ gives an isomorphism $X_{(0)}\simeq R(\x)_{(0)}$, we have $\rank X=1$.
\end{proof}

The following notion will be used frequently.

\begin{definition}[Syzygies]
Let $d\ge-1$ be arbitrary. For $i\ge0$, the \emph{$i$-th syzygy} of $X\in\mod^{\L}R$
is defined as $\Omega^iX:=\Image f_i$, where
\begin{equation}\label{mpr of X}
\xymatrix{\cdots\ar[r]^{f_3}&P_2\ar[r]^{f_2}&P_1\ar[r]^{f_1}&P_0\ar[r]^{f_0}&X\ar[r]&0}
\end{equation}
is a minimal projective resolution of $X$ in $\mod^{\L}R$. We define a full subcategory of $\mod^{\L}R$ by
\[\Omega^i(\mod^{\L}R):=\add\{\Omega^iX,\ R(\x)\mid X\in\mod^{\L}R,\ \x\in\L\}.\]
\end{definition}

The following is a basic observation.

\begin{lemma}\label{CM vanishing}
For any $X,Y\in\mod^{\L}R$, there exists $\a\in\L$ such that
$\Hom^{\L}_R(X,\Omega^iY(\x))=0$ holds for any $i\ge0$ and $\x\in\L$
satisfying $\x\le\a$.
\end{lemma}

\begin{proof}
For $Y\in\mod^{\L}R$, let ${\rm s}(Y)=\{\x\in\L\mid Y_{\x}\neq0\}$.
Assume that the $R$-module $X$ (respectively, $Y$) is generated by homogeneous elements of degrees $\a_1,\cdots,\a_\ell$ (respectively, $\b_1,\cdots,\b_m$). Let
\[I:=\bigcup_{j=1}^m(\b_j+\L_+)\subset \L.\]
Then ${\rm s}(Y)\subset I$ holds clearly, and by an easy induction, ${\rm s}(\Omega^iY)\subset I$
holds for all $i\ge0$. We take sufficiently small $\a\in\L$ such that
$\a_j\notin I-\a$ for all $1\le j\le\ell$.
Then for all $\x\le\a$ and $1\le j\le\ell$, we have $\a_j\notin I-\x$ and hence
$\a_j\notin{\rm s}(\Omega^iY(\x))$ for all $i\ge0$. Thus $\Hom^{\L}_R(X,\Omega^iY(\x))=0$.
\end{proof}

If $d\ge0$, the syzygies satisfy the following property, which we often use.

\begin{proposition}\label{non-zero is injective}
Assume $d\ge0$. Let $X\in\Omega(\mod^{\L}R)$. 
\begin{itemize}
\item[(a)] If $\rank X=0$, then $X=0$.
\item[(b)] If $Y\in\mod^{\L}R$ has rank $0$, then $\Hom^{\L}_R(Y,X)=0$.
\item[(c)] Let $\x\in\L$. Any non-zero morphism $f:R(\x)\to X$ in $\mod^{\L}R$ is injective.
\end{itemize}
\end{proposition}

\begin{proof}
(a) Let $X\subset P$ with $P\in\proj^{\L}R$. Then we have a commutative diagram
\[\xymatrix@R1em{
0\ar[r]&X\ar[r]\ar[d]&P\ar[d]\\
&X_{(0)}\ar[r]&P_{(0)},
}\]
where $X_{(0)}=0$ holds since $\rank X=0$.
Since $R$ is an $\L$-domain, the right vertical morphism is injective. Therefore $X=0$.

(b) For $f\in\Hom^{\L}_R(Y,X)$, let $Z=\Image f$. Since $Z\in\Omega(\mod^{\L}R)$ has rank $0$, we have $Z=0$ by (a). Thus $f=0$.

(c) If $X=R(\y)$ for some $\y\in\L$, then the assertion is clear since $R$ is an $\L$-domain.
Thus the assertion holds if $X\in\proj^{\L}R$ since $X$ is a direct sum of modules of the form
$R(\y)$ with $y\in\L$.

For the general case, we take an injective morphism $g:X\to P$ with $P\in\proj^{\L}R$.
Then $fg:R(\x)\to P$ is non-zero, and hence injective. Thus $f$ is also injective.
\end{proof}

For later use, we prepare the following useful observation.

\begin{lemma}\label{(Q,X) non zero}
Assume $d\ge0$. We consider a minimal projective resolution \eqref{mpr of X} of $X\in\mod^{\L}R$.
Let $i\ge0$ and $Q$ be an indecomposable direct summand of $P_i$.
\begin{itemize}
\item[(a)] There exists an injective morphism $Q\to P_0$ in $\mod^{\L}R$.
\item[(b)] If $X\in\Omega(\mod^{\L}R)$, then there exists an injective morphism $Q\to X$ in $\mod^{\L}R$.
\end{itemize}
\end{lemma}

\begin{proof}
(a) By Proposition~\ref{non-zero is injective}(c), we only have to show $\Hom^{\L}_R(Q,P_0)\neq0$.
This is clear for $i=0$. Assume $i\ge1$. Then for any indecomposable direct summand $Q$ of $P_i$, there exists an indecomposable direct summand $Q'$ of $P_{i-1}$ such that $\Hom^{\L}_R(Q,Q')\neq0$ since \eqref{mpr of X} is a minimal projective resolution.
By our inductive hypothesis, there exists an injective morphism $Q'\to P_0$ in $\mod^{\L}R$. Thus $\Hom^{\L}_R(Q,P_0)\neq0$ holds.

(b) By (a), there exist an indecomposable direct summand $Q'$ of $P_0$ and an injective morphism $g:Q\to Q'$.
Again by Proposition~\ref{non-zero is injective}(c), the morphism $h:Q'\subset P_0\to X$ is injective.
Thus $gh:Q\to X$ is injective too.
\end{proof}

In the rest of this section, we show that GL complete 
intersections have $\L$-isolated singularities.
We need the assumption that our hyperplanes are in a general position.

A crucial role is played by the following non-commutative ring, 
which already appeared in Gabriel's classical covering theory \cite{Ga}. 

\begin{definition}\label{define covering}
Let $A$ be a $G$-graded ring and $H$ a subgroup of $G$ with finite index.
We fix a complete set $I\subset G$ of representatives of $G/H$.
We define an $H$-graded ring $A^{[H]}$ called the \emph{covering}
\cite{Ga,IL} (or \emph{quasi-Veronese subalgebra} \cite{MM}) of $A$ as
\[A^{[H]}:=\bigoplus_{h\in H}B_h,\ \ \ B_h:=\left( A_{x-y+h} \right)_{x,y\in I}\]
where the multiplication $B_h\times B_{h'}\to B_{h+h'}$ for $h,h'\in H$ is given by 
\[(a_{x,y})_{x,y\in I}\cdot (a'_{x,y})_{x,y\in I}
:=\left(\sum_{z\in I}a_{x,z}\cdot a'_{z,y}\right)_{x,y\in I}.\]
\end{definition}

Note that the (ungraded) $k$-algebra structure of $A^{[H]}$ does not depend on the choice of $I$,
while its $H$-graded $k$-algebra structure does. But it is still unique up to graded Morita equivalence.
In fact, we know from \cite[3.1]{IL} that we have an equivalence of categories
\begin{equation}\label{general graded morita}
F:\Mod^GA\simeq\Mod^HA^{[H]},
\end{equation}
which is given as follows: For $M=\bigoplus_{g\in G}M_g$ in $\Mod^GA$, define
$FM\in\Mod^HA^{[H]}$ by
\[FM:=\bigoplus_{h\in H}(FM)_h\ \ \ \mbox{where}\ \ \ (FM)_h:=(M_{x+h})_{x\in I}.\]

Now we apply this general observation to GL complete intersections.
For the subgroup $\Z\c$ of $\L$ generated by $\c$, we take the complete set 
\begin{equation}\label{choose representatives}
I:=\{\sum_{i=1}^na_i\x_i\mid 0\le a_i\le p_i-1\ (1\le i\le n)\}
\end{equation}
of representatives of $\L/\Z\c$. Then we have the corresponding
covering $R^{[\Z\c]}$ of $R$, which is a $\Z$-graded $k$-algebra.
Applying \eqref{general graded morita}, we have the following observation.

\begin{proposition}\label{graded morita for R}
We have an equivalence of categories
\[\mod^{\L}R  \simeq \mod^{\Z}R^{[\Z\c]}.\]
\end{proposition}

We are ready to prove our main result in this section.
Note that our $R$ has a unique $\L$-maximal ideal
\[R_+:=\bigoplus_{\x\in\L\setminus\{0\}}R_{\x}.\]
We denote by $R_{T_j}$ is the localization of $R$ with respect to the multiplicative set
$\{T_j^\ell\mid\ell\ge0\}$ for any $j$ with $0\le j\le d$.

\begin{theorem}\label{L-isolated singularity}
Let $(R,\L)$ be a GL complete intersection over an arbitrary field $k$.
\begin{itemize}
\item[(a)] $R$ has at worst $\L$-isolated singularities.
\item[(b)] $R_{T_j}$ is $\L$-regular for all $0\le j\le d$.
\end{itemize}
\end{theorem}

To prove part (b), we need some preparations.
After establishing (b), we can deduce (a) easily.

For the subgroup $\Z\c$ of $\L$ generated by $\c$, we take the complete 
set $I$ in \eqref{choose representatives} of representatives of $\L/\Z\c$.
Throughout we fix $j$, and define a $\Z$-graded algebra $T(j)$ by
\[T(j):=(R_{T_j})^{[\Z\c]}\]
(see Definition~\ref{define covering}).
Then we have the following observation.

\begin{proposition}\label{graded morita}
We have equivalences
\[\mod^{\L}R_{T_j} \simeq \mod^{\Z}T(j)\simeq\mod T(j)_0.\]
\end{proposition}

\begin{proof}
The first equivalence follows from \eqref{general graded morita}.
Since $T_j$ is an invertible element of degree $1$ in $T(j)$,
we have $T(j)_\ell=T_j^\ell T(j)_0$ for any $j\in\Z$.
Thus $T(j)$ is strongly graded in the sense that $T(j)_\ell\cdot 
T(j)_{-\ell}=T(j)_0$ holds for any $\ell\in\Z$. By \cite[I.3.4]{NV2},
we have the second equivalence.
\end{proof}

Now we give a description of $T(j)_0$ in terms of a tensor product.
For $p>0$ and a ring $A$ with an element $a\in A$,
we define a subring of the full matrix ring ${\rm M}_p(A)$ by
\[{\rm T}_p(A,a):=\left[\
\begin{array}{ccccc}
A&(a)&\cdots&(a)&(a)\\
A&A&\cdots&(a)&(a)\\
\vdots&\vdots&\ddots&\vdots&\vdots\\
A&A&\cdots&A&(a)\\
A&A&\cdots&A&A
\end{array}\right].\]
For the polynomial ring $C=k[T_0,\ldots,T_d]$,
note that $(C_{T_j})_0=k[T_0/T_j,\ldots,T_d/T_j]$ holds.
We have the following explicit description of $T(j)_0$, which is an analog
of a description of $R^{[\Z\c]}$ given in \cite[3.7]{IL}.

\begin{lemma}
We have $T(j)_0\simeq{\rm T}_{p_1}((C_{T_j})_0,\ell_1/T_j)\otimes\cdots\otimes
{\rm T}_{p_n}((C_{T_j})_0,\ell_n/T_j)$, where
the tensor products are over $(C_{T_j})_0$. 
\end{lemma}

\begin{proof}
If $\x=\sum_{i=1}^na_i\x_i+a\c$ is in normal form, then we have $(R_{T_j})_{\x}=(\prod_{i=1}^nX_i^{a_i})T_j^a(C_{T_j})_0$.

For $\x=\sum_{i=1}^na_i\x_i$ and $\y=\sum_{i=1}^nb_i\x_i$ in $I$, let 
$\epsilon_i:=0$ if $a_i\ge b_i$ and $\epsilon_i:=1$ otherwise. Then
$\x-\y$ has the normal form $\sum_{i=1}^n(a_i-b_i+\epsilon_ip_i)\x_i-(\sum_{i=1}^n\epsilon_i)\c$, and we have
\begin{eqnarray*}
(R_{T_j})_{\x-\y}&=&(\prod_{i=1}^n(X_i^{a_i-b_i+\epsilon_ip_i}/T_j^{\epsilon_i}))(C_{T_j})_0\\
&=&((X_1^{a_1-b_1+\epsilon_1p_1}/T_j^{\epsilon_1})(C_{T_j})_0)\otimes\cdots\otimes
((X_n^{a_n-b_n+\epsilon_np_n}/T_j^{\epsilon_n})(C_{T_j})_0),
\end{eqnarray*}
Therefore we have an isomorphism
\begin{eqnarray*}
T(j)_0&=&\bigotimes_{i=1}^n\left[
\begin{array}{ccccc}
(C_{T_j})_0&(X_i^{p_i-1}/T_j)(C_{T_j})_0&\cdots&(X_i^2/T_j)(C_{T_j})_0&(X_i/T_j)(C_{T_j})_0\\
X_i(C_{T_j})_0&(C_{T_j})_0&\cdots&(X_i^3/T_j)(C_{T_j})_0&(X_i^2/T_j)(C_{T_j})_0\\
\vdots&\vdots&\ddots&\vdots&\vdots\\
X_i^{p_i-2}(C_{T_j})_0&X_i^{p_i-3}(C_{T_j})_0&\cdots&(C_{T_j})_0&(X_i^{p_i-1}/T_j)(C_{T_j})_0\\
X_i^{p_i-1}(C_{T_j})_0&X_i^{p_i-2}(C_{T_j})_0&\cdots&X_i(C_{T_j})_0&(C_{T_j})_0
\end{array}\right]\\
&\simeq&{\rm T}_{p_1}((C_{T_j})_0,X_1^{p_1}/T_j)\otimes\cdots\otimes
{\rm T}_{p_n}((C_{T_j})_0,X_n^{p_n}/T_j)
\end{eqnarray*}
of $k$-algebras.
\end{proof}

\begin{lemma}\label{global dimension d}
$T(j)_0$ has global dimension $d$.
\end{lemma}

\begin{proof}
$T(j)_0$ is a $(C_{T_j})_0$-algebra which is a free $(C_{T_j})_0$-module of finite rank.
It suffices to show that for any maximal ideal $\mm$ of $(C_{T_j})_0$, the global dimension
of $(T(j)_0)_{\mm}$ is $d$. Let
\[I_{\mm}:=\{i\mid 1\le i\le n,\ \ell_i/T_j\in\mm\}.\]
For any $i\notin I_{\mm}$, we have that $\ell_i/T_j$ is a unit in $((C_{T_j})_0)_{\mm}$ and hence
${\rm T}_{p_i}(((C_{T_j})_0)_{\mm},\ell_i/T_j)={\rm M}_{p_i}(((C_{T_j})_0)_{\mm})$ holds.
Thus we have
\[(T(j)_0)_{\mm}\simeq
\bigotimes_{i=1}^n{\rm T}_{p_i}(((C_{T_j})_0)_{\mm},\ell_i/T_j)
\simeq{\rm M}_p(\bigotimes_{i\in I_{\mm}}{\rm T}_{p_i}(((C_{T_j})_0)_{\mm},\ell_i/T_j))\]
for $p:=\prod_{i\notin I_{\mm}}p_i$.
Since $\ell_1,\ldots,\ell_n$ are in general position, $(\ell_i/T_j)_{i\in I_{\mm}}$ is
a regular sequence of $((C_{T_j})_0)_{\mm}$. Thus the global dimension of
$\bigotimes_{i\in I_{\mm}}{\rm T}_{p_i}(((C_{T_j})_0)_{\mm},\ell_i/T_j)$
is $d$ by \cite[2.12]{IL}, and we have the assertion.
\end{proof}

Now we are ready to prove Theorem~\ref{L-isolated singularity}.

\begin{proof}[Proof of Theorem~\ref{L-isolated singularity}.]
(b) The statement follows from Propositions~\ref{graded morita} and \ref{global dimension d}.

(a) Fix $\pp\in\Spec^{\L}R\setminus\{R_+\}$. We will show
that $R_{(\pp)}$ is $\L$-regular.
Since $\pp\neq R_+$, there exists $j$ such that $T_j\notin\pp$.
Clearly $R_{T_j}\pp$ is an $\L$-prime ideal of $R_{T_j}$ such that $R_{(\pp)}=(R_{T_j})_{(R_{T_j}\pp)}$.
Since $R_{T_j}$ is $\L$-regular by (b), we have that
$R_{(\pp)}=(R_{T_j})_{(R_{T_j}\pp)}$ is $\L$-regular by Observation
\ref{observation for G-prime}. Thus the assertion follows.
\end{proof}

If we forget the $\L$-grading, then our GL complete intersection can have non-isolated singularities, e.g.
$R=k[X_1,X_2,X_3]/(X_1^p+X_2^p+X_3^p)$ and $k$ has characteristic $p$.
This doesn't happen in characteristic zero by the following observation.

\begin{proposition}\label{isolated singularity}
Let $(R,\L)$ be a GL complete intersection over a perfect field $k$,
such that the weights are non-zero in $k$.
Then $R_{\pp}$ is a regular local ring for any 
$\pp\in\Spec R\setminus\{R_+\}$. In particular $R$ has at worst isolated singularities.
\end{proposition}

\begin{proof}
If $n\le d+1$, then the assertion follows from Proposition~\ref{R is CI}(b).
Assume $n\ge d+2$. By Observation~\ref{Normalization}, we have
$R=S'/(X_i^{p_i}-\sum_{j=1}^{d+1}\lambda_{i,j-1}X_j^{p_j}\mid d+2\le i\le n)$
for $S'=k[X_1,\ldots,X_n]$, where all maximal minors of the
$(n-d-1)\times n$ matrix
\[L:=[\lambda_{i,j-1}|-I_{n-d-1}]_{d+2\le i\le n,\ 1\le j\le d+1}\]
have non-zero determinants. The Jacobian matrix is given by
\[M:=L\cdot{\rm diag}(p_1X_1^{p_1-1},\ldots,p_{n}X_{n}^{p_{n}-1}).\]
By the Jacobian criterion \cite[16.20]{E2}, the singular
locus of $R$ is given by $V(J)\cap\Spec R$, where $J$ is the ideal of
$S'$ generated by all maximal minors of $M$
and $V(J):=\{\pp\in\Spec S'\mid\pp\supset J\}$.
Since $p_i\neq0$ in $k$ for any $i$ and all maximal minors of $L$
have non-zero determinants, we have
\[J=(\prod_{i\in I}X_i^{p_i-1}\mid I\subset\{1,\ldots,n\},\ |I|=n-d-1).\]
Therefore it is easy to check that $\pp\in\Spec S'$ contains $J$ if and only if $\pp$ contains
at least $d+2$ elements from $\{X_1,\ldots,X_n\}$. Thus
\[V(J)=\bigcup_{I\subset \{1,\ldots,n\},\ |I|=d+2}V((X_i\mid i\in I))\]
holds. On the other hand, for any subset $I$ of $\{1,\ldots,n\}$
with $|I|\ge d+1$, we have $R/(X_i\mid i\in I)\in\mod^{\L}_0R$
by Lemma~\ref{prop.regularseq1}(c), and hence $V((X_i\mid i\in I))=\{R_+\}$ 
holds. In particular, the singular locus of $R$ is contained in $\{R_+\}$.
\end{proof}

\section{Tate's DG algebra resolutions}

In this section, we use the following presentation of GL complete intersections.

\begin{setting}\label{minimal presentation}
Let $(R,\L)$ be a GL complete intersection.
Applying Observation~\ref{Weights 1} to $(R,\L)$, we may assume that $p_i\ge2$ for all $1\le i\le n$.
By Observation~\ref{Normalization} we may assume that 
\[R =\begin{cases}
k[X_1,\ldots,X_n,X_{n+1},\ldots,X_{d+1}]&n\le d+1,\\
k[X_1, \ldots, X_n] / (X_i^{p_i} - \sum_{j=1}^{d+1}\lambda_{i,j-1}X_j^{p_j} \mid d+2 \leq i \leq n)&n\ge d+2.
\end{cases}.\]
For the case $n\le d+1$, by convention, $\x_i=\c$ and $X_i=T_{i-1}$ for $n+1\le i\le d+1$.
\end{setting}

Let $G$ be an abelian group and $A$ a $G$-graded ring. Assume that a $G$-graded $A$-module $X$ has a minimal projective resolution $\cdots\to P^{-2}\to P^{-1}\to P^0\to X\to0$ such that
$P^{-i}=\bigoplus_{g\in G}R(-g)^{\oplus b_{i,g}}$ for some $b_{i,g}\in\Z_{\ge0}$.
Then the \emph{$G$-Poincar\'{e} series} of $X$ is define by
\[{\rm P}^A_X(u):=\sum_{i\ge0}\left(\sum_{g\in G}b_{i,g}g\right)u^i\in k[G][[u]].\]
We prove the following results by establishing an $\L$-graded version of Tate resolutions \cite{Tat}.

\begin{theorem}\label{L-Poincare}
Let $(R,\L)$ be a GL complete intersection in Setting~\ref{minimal presentation}.
Fix $\ell_i\in[1,p_i-1]$ for $1\le i\le n$. Let $F=k[X_1,\ldots,X_n]/(X_1^{\ell_1},\ldots,X_n^{\ell_n})\in\mod^{\L}R$
and let $\cdots \to P^{-2}\to P^{-1}\to P^{0}\to F\to0$ be a minimal projective resolution of $F$.\begin{itemize}
\item[(a)] The $R$-module $F$ has the $\L$-Poincar\'{e} series
\[{\rm P}^R_F(u)=\begin{cases}
{\displaystyle\prod_{i=1}^{d+1}(1+(\ell_i\x_i)\cdot u)}&\mbox{ if }\ n\le d+1,\\
\frac{\displaystyle\prod_{i=1}^n(1+(\ell_i\x_i)\cdot u)}{\displaystyle(1-\c\cdot u^2)^{n-d-1}}&\mbox{ if }\ n\ge d+2,
\end{cases}.\]
where in the first case, let $\x_i=\c$ and $\ell_i=1$ for all $n+1\le i\le d+1$ by convention.
\item[(b)] Hence we have
\begin{eqnarray*}
P^{-r}=\begin{cases}
{\displaystyle\bigoplus_{\substack{I \subset [1,d+1],\ |I|=r}} R\left(-\sum_{i \in I} \ell_i\x_i\right)}&\mbox{ if }\ n\le d+1,\\
{\displaystyle\bigoplus_{\substack{I \subset [1,n],\ r-|I|\in 2\Z_{\ge0}}} R\left(\left(\frac{|I|-r}{2}\right)\c - \sum_{i \in I} \ell_i\x_i\right)^{\oplus{{n-d-2+\frac{r-|I|}{2}}\choose n-d-2}}}&\mbox{ if }\ n\ge d+2.
\end{cases}
\end{eqnarray*}
\item[(c)] We have
\[\gl(\mod^{\L}R)=\left\{\begin{array}{ll}
d+1&\mbox{ if }\ n\le d+1,\\
\infty&\mbox{ if }\ n\ge d+2.
\end{array}\right.\]
\item[(d)] If $n\ge d+2$, then ${\displaystyle\lim_{r\to\infty}\frac{|P^{-r}|}{r^{n-d-2}}=
\frac{2^{d+1}}{(n-d-2)!}}$ holds, where $|P^{-r}|$ is the number of indecomposable
direct summands of $P^{-r}$.\end{itemize}
\end{theorem}

To prove Theorem~\ref{L-Poincare}, we need to introduce some terminology.

We fix an abelian group $G$. For a $(\Z\times G)$-graded module $V$, the $\Z$-degree is written as superscript the $G$-degree is written as subscript
\[V=\bigoplus_{(i,g)\in\Z\times G}V^i_g=\bigoplus_{i\in\Z}V^i=\bigoplus_{g\in G}V_g.\]
We write $|v|=|v|_{\Z}=i$ if $v\in V^i$, and $|v|_G=g$ if $v\in V_g$.

We need the following $G$-graded analogue of differential graded rings \cite{KM,LPWZ}.

\begin{definition}
A \emph{graded-commutative differential graded ring with Adams $G$-grading} (or \emph{$G$-gc-DG ring} in short) is a pair $(A,d)$ of an associative $(\Z\times G)$-graded ring $A$ and a morphism $d:A\to A$ of $(\Z\times G)$-graded abelian groups of degree $(1,0)$ satisfying $d^2=0$ and the following conditions.
\begin{itemize}
\item (Leibniz rule) $d(ab)=d(a)b+(-1)^{|a|}ad(b)$ holds for any $a,b\in A$.
\item (graded-commutativity) $ab=(-1)^{|a||b|}ba$ for any $a,b\in A$, and $a^2=0$ holds for any $a\in A$ such that $|a|$ is odd.
\end{itemize}
In this case, the total cohomology $H(A)$ of $A$ has a natural structure of a $(\Z\times G)$-graded ring.
\end{definition}

A classical example is given by Koszul complexes.

\begin{example}\label{Koszul example}
Let $A$ be a commutative $G$-graded ring, and $t_1,\ldots,t_r$ elements of $A$ which are homogeneous.
We regard the Koszul complex $B=A\langle T_1,\ldots,T_r\rangle$ of $A$ with respect to $(t_1,\ldots,t_r)$ as a $G$-gc-DG ring as follows.
\begin{itemize}
\item $B$ is an exterior algebra over $A$: $T_iT_j+T_jT_i=0$ and $T_i^2=0$ for any $i,j$.
\item $B^0=A$, the degree of $T_i$ is $(-1,|t_i|_G)\in\Z\times G$, and the differential is given by $d(T_i)=t_i$. 
\end{itemize}
If $(t_1,\ldots,t_r)$ is an $A$-regular sequence, then $H(B)=H^0(B)\simeq A/(t_1,\ldots,t_r)$ as $G$-graded rings.
\end{example}

The following construction given in \cite{Tat} plays an important role.

\begin{example}\label{killing cycle}
Let $(A,d)$ be a $G$-gc-DG ring, and $t\in A^{r-1}_g$ an element such that $d(t)=0$ and $r$ is even.
We define a $G$-gc-DG ring $(B,d)=(A\langle T;t\rangle,d)$ as follows:
\begin{itemize}
\item $B$ is a $(\Z\times G)$-graded free $A$-module with basis $\{T^{(i)}\mid i\ge0\}$. The multiplication is given by $T^{(i)}T^{(j)}={i+j\choose i}T^{(i+j)}$ and $aT^{(i)}=T^{(i)}a$ for any $i,j\ge0$ and $a\in A$.
\item $T^{(i)}$ has degree $(ir,ig)$ for any $i\ge0$.
\item The differential is given by $d(T^{(i)})=tT^{(i-1)}$ for any $i>0$.
\end{itemize}
It was shown in \cite[Theorem 2]{Tat} that, if the sequence $H(A)\xrightarrow{t\cdot}H(A)\xrightarrow{t\cdot}H(A)$ is exact and $A^i=0$ holds for any $i>0$, then $H(B)=H(A)/tH(A)$ holds.
\end{example}

Our Theorem~\ref{L-Poincare} is a special case of the following $G$-graded version of Tate resolutions \cite[Theorem 4]{Tat}, \cite{Gu} (see also \cite[Proposition 1.5.4]{GuL}, \cite[Theorem 6.1.8]{Av}).

\begin{theorem}\label{G-Poincare}
Let $A$ be a commutative $G$-graded ring, and let $I\subset J$ the ideals of $A$ generated by homogeneous $A$-regular sequences $(a_1,\ldots,a_r)$ and $(t_1,\ldots,t_n)$ respectively.
\begin{itemize}
\item[(a)] There exists a $G$-gc-DG ring $B$ which gives a projective resolution of the $G$-graded $(A/I)$-module $A/J$.
\item[(b)] Assume that $A$ has a unique $G$-maximal ideal ${\mathfrak m}$ (Definition~\ref{define L-isolated singularity}). If $I\subset{\mathfrak m}J$ and $J\subset{\mathfrak m}$ holds, then $B$ can be chosen in a way that it gives a minimal projective resolution of $A/J$ and we have
\[{\rm P}^{A/I}_{A/J}(u)=\frac{\prod_{i=1}^n(1+|t_i|_G\cdot u)}{\prod_{j=1}^r(1-|a_j|_G\cdot u^2)}.\]
\end{itemize}
\end{theorem}

\begin{proof}
Let $C=A\langle T_1,\ldots,T_n\rangle$ be the Koszul complex of $A$ with respect to $t_1,\ldots,t_n$. 
It is a $G$-gc-DG ring such that $T_i$ has degree $(-1,|t_i|_G)\in\Z\times G$ and satisfies $d(T_i)=t_i$ (Example~\ref{Koszul example}).
Since $(t_1,\ldots,t_n)$ is an $A$-regular sequence, $H(C)=H^0(C)=A/J$ holds. Let
\[\overline{A}:=A/I\ \mbox{ and }\ \overline{C}:=C\otimes_A\overline{A}.\]
For each $(\Z\times G)$-graded free $\overline{A}$-module $F=\bigoplus_{(i,g)\in\Z\times G}(\overline{A}(-g)[i])^{\oplus b_{i,g}}$, we consider the \emph{$G$-Poincar\'{e} series} ${\rm P}_{F}(u):=\sum_{i\in\Z}(\sum_{g\in G}b_{i,g} g)u^i\in k[G][[u]]$. Clearly we have
\begin{equation}\label{P_overlineB}
{\rm P}_{\overline{C}}(u)=\prod_{i=1}^n(1+|t_i|_G\cdot u).
\end{equation}
Write $a_j=\sum_{i=1}^n c_{ji}t_i$ for $c_{ji}\in A$, and let $s_j=\sum_{i=1}^nc_{ji}T_i\in C$. Then $s_j$ has degree $(-1,|a_j|_G)$. Since $(a_1,\ldots,a_r)$ is an $A$-regular sequence, \cite[Theorem 3]{Tat} gives an isomorphism
\[H(\overline{C})\simeq H(C)\langle\overline{s}_1,\ldots,\overline{s}_r\rangle=(A/J)\langle\overline{s}_1,\ldots,\overline{s}_r\rangle\]
of $(\Z\times G)$-graded rings, where the right-hand side is the exterior algebra over $A/J$.

Now we define inductively $G$-gc-DG rings $B^{[j]}$ such that $B^{[j]}$ is a $(\Z\times G)$-graded free $\overline{A}$-module and $H(B^{[j]})=(A/J)\langle\overline{s}_{j+1},\ldots,\overline{s}_r\rangle$ is the exterior algebra over $A/J$.
Let $B^{[0]}:=\overline{C}$. Once $B^{[j-1]}$ is defined, we apply Example~\ref{killing cycle} to define $B^{[j]}:=B^{[j-1]}\langle S_{j};\overline{s}_{j}\rangle$ such that $S_j$ has degree $(-2,|a_j|_G)$ and satisfies $d(S_j^{(i)})=\overline{s}_{j}S_j^{(i-1)}$ for any $i>0$.
Then the desired conditions are satisfied.

In particular, the $G$-gc-DG ring $B:=B^{[r]}$ satisfies $H(B)=A/J$. Thus assertion (a) holds.

It remains to prove (b). Since $J\subset{\mathfrak m}$, we have $d(\overline{C})\subset{\mathfrak m}\overline{C}$. Since $I\subset{\mathfrak m}J$ holds, we may choose $c_{ji}\in{\mathfrak m}$ and $\overline{s}_j\in{\mathfrak m}\overline{C}$. Thus $d(B^{[j]})\subset{\mathfrak m}B^{[j]}$ holds inductively, and therefore the first assertion follows.
On the other hand, we clearly have
\begin{equation}\label{P_C^j}
{\rm P}_{B^{[j]}}(u)={\rm P}_{B^{[j-1]}}(u)\cdot(1+|a_j|_G\cdot u^2+(2|a_j|_G)\cdot u^4+(3|a_j|_G)\cdot u^6+\cdots)=\frac{{\rm P}_{B^{[j-1]}}(u)}{1-|a_j|_G\cdot u^2}.
\end{equation}
By \eqref{P_overlineB} and \eqref{P_C^j}, we have the desired equality for ${\rm P}^{\overline{A}}_{A/J}(u)={\rm P}_{B^{[r]}}(u)$.
\end{proof}

We are ready to prove Theorem~\ref{L-Poincare}.

\begin{proof}[Proof of Theorem~\ref{L-Poincare}]
(a) Let $G:=\L$. As in Setting \ref{minimal presentation}, if $n\le d+1$, then let
\[A:=k[X_1,\ldots,X_{d+1}]\supset J:=(X_1^{\ell_1},\ldots,X_n^{\ell_n},X_{n+1},\ldots,X_{d+1})\supset I:=0,\]
and otherwise let
\[A:=k[X_1,\ldots,X_n]\supset J:=(X_1^{\ell_1},\ldots,X_n^{\ell_n})\supset I:=(X_i^{p_i} - \sum_{j=1}^{d+1}\lambda_{i,j-1}X_j^{p_j} \mid d+2 \leq i \leq n).\]
Applying Theorem~\ref{G-Poincare}(b), we obatin the assertion.
In fact, $S$ has a unique $\L$-maximal ideal $S_+=\bigoplus_{\x>0}S_{\x}$, and $J\subset S_+$ holds since $\ell_i\ge1$ for any $i$. Moreover $I\subset S_+J$ holds since $\ell_i\le p_i-1$.

(b) This is straightforward from (a).

(c) If $n\ge d+2$, then $k\in\mod^{\L}R$ has infinite projective dimension by (a),
and hence $\gl(\mod^{\L}R)=\infty$.
If $n\le d+1$, then any simple objects in $\mod^{\L}R$ has projective dimension $d+1$ by (a).
This implies $\gl(\mod^{\L}R)=d+1$ (e.g. \cite[Proposition 2.2]{IR}).

(d) By (b), 
$|P^{-r}|=\sum_{I\subset[1,n],\ r-|I|\in2\Z_{\ge0}}{n-d-2+\frac{r-|I|}{2}\choose n-d-2}$.
If $r$ is sufficiently large, then the highest term with respect to $r$ is
\[\sum_{I\subset[1,n],\ r-|I|\in2\Z_{\ge0}}\frac{(\frac{r}{2})^{n-d-2}}{(n-d-2)!}=2^{n-1}\frac{(\frac{r}{2})^{n-d-2}}{(n-d-2)!}=\frac{2^{d+1}r^{n-d-2}}{(n-d-2)!}.\]
Thus the assertion follows.
\end{proof}

\section{$I$-canonical algebras}
Let $(R,\L)$ be a GL complete intersection. We will frequently use the following terminology.

\begin{definition}[Convex, Upset]
We say that a subset $I$ of $\L$ is \emph{convex} if for any $\x,\y,\z\in\L$ such that $\x\le\y\le\z$ and $\x,\z\in I$, we have $\y\in I$.
We say that a subset $I$ of $\L$ is an \emph{upset} if $I+\L_+\subset I$
holds. Moreover we say that an upset is \emph{non-trivial} if it is neither $\L$ nor $\emptyset$.
\end{definition}

We introduce a class of finite dimensional algebras, which play an important role in this paper.

\begin{definition}[$I$-canonical algebras]\label{define I-canonical}
Let $(R,\L)$ be a GL complete intersection.
For a finite subset $I$ of $\L$, we define a $k$-algebra
\[A^I:=(R_{\x-\y})_{\x,\y\in I}\]
in a similar way to Definition~\ref{define covering}.
Namely the multiplication of $A^I$ is given by
\[(r_{\x,\y})_{\x,\y\in I}\cdot(r'_{\x,\y})_{\x,\y\in I}:=
(\sum_{\z\in I}r_{\x,\z}\cdot r'_{\z,\y})_{\x,\y\in I}.\]
We call $A^I$ the \emph{$I$-canonical algebra}.
\end{definition}

\medskip
We give the first properties of $I$-canonical algebras.

\begin{proposition}\label{I-canonical has finite global dimension}
Let $I$ be a finite subset of $\L$.
\begin{itemize}
\item[(a)] The $k$-algebra $A^I$ has finite global dimension.
In particular, we have $\KKK^{\bo}(\proj A^I)\simeq\DDD^{\bo}(\mod A^I)$.
\item[(b)] We have isomorphisms of $k$-algebras
$A^{-I}\simeq (A^{I})^{\op}$ and $A^{I+\x}\simeq A^{I}$ for any $\x\in\L$.
\item[(c)] For any $\x,\y\in\L$, we have an isomorphism
$A^{[\x,\y]}\simeq(A^{[\x,\y]})^{\op}$ of $k$-algebras.
\end{itemize}
\end{proposition}

\begin{proof}
(a) All the diagonal entries of $A^I$ are $R_{\x,\x}=R_0=k$.
If $R_{\x,\y}\neq0$, then $\x\ge\y$ holds.
These observations imply that $A^I$ has finite global dimension.

(b) These are clear.

(c) This follows from (b) since  $-[\x,\y] + \x+\y =[\x,\y]$ holds.
\end{proof}

Throughtout this paper, we use the following subcategories.

\begin{definition}\label{define mod^I}
Let $(R,\L)$ be a GL complete intersection. For a subset $I$ of $\L$, let
\begin{eqnarray*}
\mod^IR&:=&\{X=\bigoplus_{\x\in\L}X_{\x}\in\mod^{\L}R\mid\forall\x\in\L\setminus I,\ X_{\x}=0\},\\
\mod^I_0R&:=&\mod^IR\cap\mod^{\L}_0R,\\
\proj^IR&:=&\add\{R(-\x)\mid\x\in I\}. 
\end{eqnarray*}
\end{definition}

If $I$ is finite, then $\mod^IR=\mod^I_0R$ holds clearly.
We have the following elementary properties.

\begin{proposition}\label{equivalences for I-canonical}
Let $I$ be a finite subset of $\L$.
\begin{itemize}
\item[(a)] We have an equivalence $\proj^IR\simeq\proj A^I$.
\item[(b)] If $I$ is convex, then we have an equivalence $\mod^IR\simeq\mod A^I$.
\end{itemize}
\end{proposition}

\begin{proof}
(a) Since $\Hom_R^{\L}(R(-\x),R(-\y))=R_{\x-\y}$ holds for any $\x,\y\in\L$, we have the assertion.

(b) This is an analog of \eqref{general graded morita}.
The equivalence is given by $M=\bigoplus_{\x\in I}M_{\x}\mapsto (M_{\x})_{\x\in I}$.
\end{proof}

Now we have the following quiver presentations of $I$-canonical algebras.

\begin{theorem}\label{endomorphism convex}
Let $(R,\L)$ be a GL complete intersection in Setting~\ref{minimal presentation}.
For a finite convex subset $I$ of $\L$, the $I$-canonical algebra $A^I$ is presented
by the quiver $Q^I$ defined by
\begin{itemize}
\item $Q^I_0 = I$,
\item $Q^I_1 = \{x_i : \x \to \x + \x_i \mid 1\le i \le \max\{n,d+1\}, \, \x \in I \cap (I-\x_i)\}$
\end{itemize}
with the following relations:
\begin{itemize}
\item $x_ix_j - x_jx_i :  \x \to \x + \x_i + \x_j $, where $1\le i < j \le \max\{n,d+1\}$ and
$\x \in I \cap (I-\x_i-\x_j)$,
\item $x_i^{p_i}-\sum_{j=1}^{d+1}\lambda_{i,j-1}x_j^{p_j}: \x \to \x + \c$, where
$d+2 \le i \le n$ and $\x \in I \cap (I-\c)$.
\end{itemize}
\end{theorem}

\begin{proof}
The vertices of $Q^I$ naturally corresponds to the primitive idempotents of $A^I$.
The arrow $x_i$ of $Q^I$ corresponds to the genertor $X_i$ of $R$ for $i$ with $1\le i \le \max\{n,d+1\}$.
Thus we have a morphism $kQ^I\to A^I$ of $k$-algebras, which is
surjective since $I$ is convex.
Clearly the commutativity relations $x_ix_j=x_jx_i$ are satisfied in $A^I$.
Also the relations
$X_i^{p_i} = \sum_{j=1}^{d+1}\lambda_{i,j-1}X_j^{p_j}$ in $R$ correspond to
the relations $x_i^{p_i}=\sum_{j=1}^{d+1}\lambda_{i,j-1}x_j^{p_j}$ in $A^I$.
Thus we have a surjective morphism $B^I\to A^I$ of
$k$-algebras, where $B^I$ is the factor algebra of $kQ^I$ by
these relations.
This is an isomorphism since it clearly induces an isomorphism 
$B^Ie_{\x}\simeq A^Ie_{\x}=\bigoplus_{\y\in I\cap(\x+\L_+)}R_{\y}$
for any $\x\in Q^I_0$. Therefore we have the assertion.
\end{proof}

We give an example explaining Theorem~\ref{endomorphism convex}.

\begin{example}
Consider the case $d=2$, $n = 4$ and $(p_1,p_2,p_3,p_4) = (3,4,5,7)$. The set $I = [0,\x_2+2\x_3] \cup [0,3\x_3+2\x_4] \cup [0, \c]$ is convex. Hence we can apply Theorem~\ref{endomorphism convex} to get the following description of $A^I$, where relations are indicated by dotted lines.
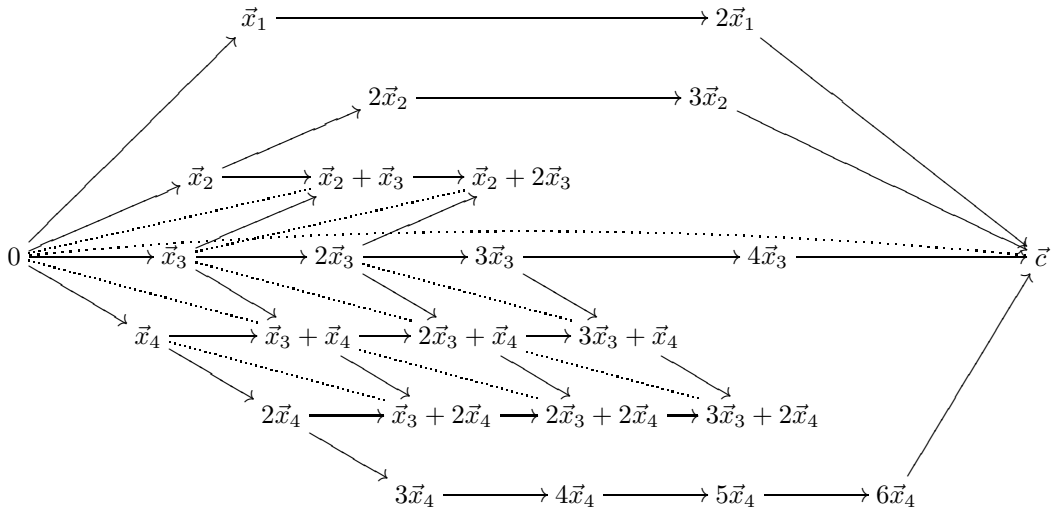
\[
\begin{xy} 0;<60pt,0pt>:<10pt,30pt>:: 
(0,0) *+{0} ="0",
(1,3) *+{\x_1} ="1",
(4,3) *+{2\x_1} ="11",
(6.4,0) *+{\c} ="c",
(1,1) *+{\x_2} ="2",
(2,2) *+{2\x_2} ="22",
(4,2) *+{3\x_2} ="222",
(1,0) *+{\x_3} ="3",
(2,0) *+{2\x_3} ="33",
(3,0) *+{3\x_3} ="333",
(4.7,0) *+{4\x_3} ="3333",
(2,1) *+{\x_2+\x_3} ="23",
(3,1) *+{\x_2+2\x_3} ="233",
(1,-1) *+{\x_4} ="4",
(2,-2) *+{2\x_4} ="44",
(3,-3) *+{3\x_4} ="444",
(4,-3) *+{4\x_4} ="4444",
(5,-3) *+{5\x_4} ="44444",
(6,-3) *+{6\x_4} ="444444",
(2,-1) *+{\x_3+\x_4} ="34",
(3,-2) *+{\x_3+2\x_4} ="344",
(3,-1) *+{2\x_3+\x_4} ="334",
(4,-2) *+{2\x_3+2\x_4} ="3344",
(4,-1) *+{3\x_3+\x_4} ="3334",
(5,-2) *+{3\x_3+2\x_4} ="33344",
"0", {\ar"1"},
"1", {\ar"11"},
"11", {\ar"c"},
"0", {\ar"2"},
"3", {\ar"23"},
"33", {\ar"233"},
"2", {\ar"22"},
"22", {\ar"222"},
"222", {\ar"c"},
"0", {\ar"3"},
"3", {\ar"33"},
"2", {\ar"23"},
"23", {\ar"233"},
"33", {\ar"333"},
"4", {\ar"34"},
"34", {\ar"334"},
"334", {\ar"3334"},
"44", {\ar"344"},
"344", {\ar"3344"},
"3344", {\ar"33344"},
"333", {\ar"3333"},
"3333", {\ar"c"},
"0", {\ar"4"},
"4", {\ar"44"},
"3", {\ar"34"},
"34", {\ar"344"},
"33", {\ar"334"},
"334", {\ar"3344"},
"333", {\ar"3334"},
"3334", {\ar"33344"},
"44", {\ar"444"},
"444", {\ar"4444"},
"4444", {\ar"44444"},
"44444", {\ar"444444"},
"444444", {\ar"c"},
"0", {\rel"23"},
"3", {\rel"233"},
"0", {\rel"34"},
"3", {\rel"334"},
"33", {\rel"3334"},
"4", {\rel"344"},
"34", {\rel"3344"},
"334", {\rel"33344"},
"0", {\rel@/^0.8pc/"c"},
\end{xy}
\]
\end{example}

Now we calculate global dimension of $I$-canonical algebras when $I$ is an interval by using Theorem~\ref{L-Poincare}.

\begin{theorem}\label{gl.dim of A^I}
Let $(R,\L)$ be a GL complete intersection in Setting~\ref{minimal presentation}.
Let $\x\in\L_+$ be an element with normal form $\x=\sum_{i=1}^na_i\x_i+a\c$, and $I:=[0,\x]\subset\L$. Then
\[\gl A^I=\gl(\mod^IR)=\begin{cases}
\min\{d+1,\ \#\{i\mid a_i\neq0\}+a\}&\mbox{ if }\ n\le d+1,\\
\#\{i\mid a_i\neq0\}+2a&\mbox{ if }\ n\ge d+2.\end{cases}\]
\end{theorem}

To prove this, we prepare the following observation.

\begin{lemma}\label{gl.dim of mod^IR}
Let $\cdots\to P^{-2}\to P^{-1}\to P^0\to k\to0$
be a minimal projective resolution of $k$ in $\mod^{\L}R$. For $r\ge0$, let
$J_r=\{\y\in\L\mid R(-\y)\in\add P^{-r}\}$. Then 
\begin{equation*}
\gl(\mod^IR)=\sup\{r\ge0\mid J_r\cap I\neq\emptyset\}.
\end{equation*}
\end{lemma}

\begin{proof}
Any object in $\mod^IR$ has a finite filtration by simple objects $k(-\z)$ with $\z\in I$.
Therefore $\gl(\mod^IR)=\max\{\pd_{\mod^IR}(k(-\z))\mid \z\in I\}$ holds.
For any $\z\in I$, we have a minimal projective resolution
\[\cdots\to (P^{-2}(-\z))_I\to (P^{-1}(-\z))_I\to (P^0(-\z))_I\to k(-\z)\to0\]
of $k(-\z)$ in $\mod^IR$. Thus we have
\[\gl(\mod^IR)=\sup\{r\ge0\mid (J_r+\z)\cap I\neq\emptyset\ \mbox{ for some}\ \z\in I\}.\]
If $\y+\z\in I$ holds for $\y\in J_r$ and $\z\in I$, then $0\le\y\le\y+\z\le\x$ and hence $\y\in J_r\cap I$.
Thus the assertion follows.
\end{proof}

\begin{proof}[Proof of Theorem~\ref{gl.dim of A^I}]
Recall from Proposition~\ref{equivalences for I-canonical}(b) that we have an equivalence
$\mod A^I\simeq\mod^IR$. Thus $\gl A^I=\gl(\mod^IR)$ holds.
By Lemma~\ref{gl.dim of mod^IR}, $\gl(\mod^IR)$ is equal to 
$\sup\{r\ge0\mid J_r\cap I\neq\emptyset\}$ for $J_r$ defined there.
Let $m:=\#\{i\mid a_i\neq0\}$.

First, we consider the case $n\ge d+2$.
We prove $\sup\{r\ge0\mid J_r\cap I\neq\emptyset\}=m+2a$. 
Applying Theorem~\ref{L-Poincare}(b) to $\ell_1=\cdots=\ell_n=1$ and $F=k$, we have
\[J_r=\left\{\sum_{i\in I}x_i+b\c\mid I\subset\{1,\ldots,n\},\ b\ge0,\ \#I=r-2b\right\}.\]
Thus $J_{m+2a}\cap I\neq\emptyset$ holds since $\sum_{a_i\neq0}\x_i+a\c$ belongs to $J_{m+2a}\cap I$.
Define a map $\gamma:\L\to\Z$ by $\gamma(\y)=\#\{i\mid\ell_i\neq0\}+2\ell$ for the normal form $\y=\sum_{i=1}^n\ell_i\x_i+\ell\c$.
This is order preserving, and satisfies $\gamma(J_r)=r$ and $\gamma(\x)=m+2a$. Therefore if $r>m+2a$, then $J_r\cap I=\emptyset$ holds. Thus the desired equality holds.

Next, we consider the case $n\le d+1$.
We prove $\sup\{r\ge0\mid J_r\cap I\neq\emptyset\}=\min\{d+1,m+a\}$.
Let $\x_i=\c$ for any $n+1\le i\le d+1$. Again applying Theorem~\ref{L-Poincare}(b) to $\ell_1=\cdots=\ell_n=1$ and $F=k$, we have
\[J_r=\{\sum_{i\in I}\x_{i}\mid I\subset\{1,\ldots,d+1\},\ \#I=r\}.\]
If $r>d+1$, then $J_r=\emptyset$ clearly holds.
Define a map $\gamma:\L\to\Z$ given by $\gamma(\y)=\#\{i\mid\ell_i\neq0\}+\ell$ for the normal form $\y=\sum_{i=1}^n\ell_i\x_i+\ell\c$. This is order preserving, and satisfies $\gamma(J_r)=r$ and $\gamma(\x)=m+a$. Therefore if $r>m+a$, then $J_r\cap I=\emptyset$ holds.

If $d+1\le m+a$, then $\sum_{i=1}^{d+1}\x_i$ belongs to $J_{d+1}\cap I$.
If $d+1>m+a$, then we choose any subset $S$ of $\{1,\ldots,d+1\}$ containing $\{1\le i\le n\mid a_i\neq0\}$ such that $\#S=m+a$. Then $\sum_{i\in S}\x_i$ belongs to $J_{m+a}\cap I$.
Thus the desired equality holds.
\end{proof}

%
%
%


\chapter[Cohen-Macaulay representations on GL complete intersections]{Cohen-Macaulay representations on Geigle-Lenzing complete intersections}\label{section: CM GL CI}

Let $(R,\L)$ be a Geigle-Lenzing (GL) complete intersection associated with linear forms $\ell_1,\ldots,\ell_n$ and weights $p_1,\ldots,p_n$ over an arbitrary field $k$. We study Cohen-Macaulay representations of $(R,\L)$.
Recall that, for a commutative Noetherian local ring $(A,\mm)$, the \emph{depth} $\depth_AX$ of $X\in\mod A$ is the maximal length of $X$-regular sequences, or equivalently, $\min\{i\ge0\mid\Ext^i_A(A/\mm,X)\neq0\}$ \cite[1.2.8]{BH}.

\begin{definition-proposition}\label{define CM}
For an integer $i$ with $0\le i\le d+1$, we call $X\in\mod R$ a \emph{Cohen-Macaulay $R$-module of dimension $i$} if the following equivalent conditions are satisfied.
\begin{itemize}
\item[(a)] $\Ext^j_R(X,R)=0$ for all $j\neq d+1-i$.
\item[(b)] For each maximal ideal $\mm$ of $R$, $\Ext^j_{R_\mm}(X_\mm,R_\mm)=0$ for all $j\neq d+1-i$.
\item[(c)] For each maximal ideal $\mm$ of $R$, $\depth_{R_\mm}X_\mm=\dim_{R_\mm}X_\mm=i$ or $X_\mm=0$.
\end{itemize}
If $X\in\mod^{\L}R$, then the following conditions are also equivalent to the ones above.
\begin{itemize}
\item[(b$'$)] For $\mm=R_+$, $\Ext^j_{R_\mm}(X_\mm,R_\mm)=0$ for all $j\neq d+1-i$.
\item[(c$'$)] For $\mm=R_+$, $\depth_{R_\mm}X_\mm=\dim_{R_\mm}X_\mm=i$ or $X_\mm=0$.
\end{itemize}
We define full subcategories of $\mod^{\L}R$ by
\begin{eqnarray*}
\CM^{\L}_iR&:=&\{X\in\mod^{\L}R\mid\mbox{$X$ is a Cohen-Macaulay $R$-module of dimension $i$}\},\\
\CM^{\L}R&:=&\CM^{\L}_{d+1}R.
\end{eqnarray*}
We simply call objects in $\CM^{\L}R$ ($\L$-graded maximal) \emph{Cohen-Macaulay} $R$-modules.
\end{definition-proposition}

\begin{proof}
For $j\neq d+1-i$, let $E:=\Ext^j_R(X,R)$. Then $\Ext^j_{R_\mm}(X_\mm,R_\mm)=E_\mm$ holds.

(a)$\Leftrightarrow$(b) Since $E=0$ is equivalent to $E_\mm=0$ for all maximal ideal $\mm$ of $R$, the assertion follows.

(b)$\Leftrightarrow$(c) (respectively, (b$'$)$\Leftrightarrow$(c$'$)) This is \cite[3.5.11]{BH}.

(a)$\Leftrightarrow$(b$'$) It suffices to show $\Leftarrow$. Since $X\in\mod^{\L}R$, then $E\in\mod^{\L}R$. For any homogeneous element $a\in E_{\x}$, there exists $r\in R\setminus R_+$ such that $ra=0$. Since $r_0\neq0$ and $r_0a=(ra)_{\x}=0$, we have $a=0$.
\end{proof}

\section{Basic properties of $\CM^{\L}R$}
The stable category \cite{ABr} defined as follows is fundamental
in representation theory.

\begin{definition}[Stable category]
We denote by $\underline{\mod}^{\L}R$ the \emph{stable category} of
$\mod^{\L}R$ \cite{ABr}. That is, $\underline{\mod}^{\L}R$ has the same 
objects as $\mod^{\L}R$, and the morphism set is given by
\[\Hom_{\underline{\mod}^{\L}R}(X,Y)=\underline{\Hom}^{\L}_R(X,Y):=
\Hom^{\L}_R(X,Y)/P(X,Y)\]
for any $X,Y\in\mod^{\L}R$, where $P(X,Y)$ is the submodule of $\Hom^{\L}_R(X,Y)$
consisting of morphisms that factor through objects in $\proj^{\L}R$.
The full subcategory $\underline{\CM}^{\L}R$ of $\underline{\mod}^{\L}R$
corresponding to the full subcategory $\CM^{\L}R$ of $\mod^{\L}R$ plays 
an important role in this paper.
\end{definition}

The following are some of the basic properties in Cohen-Macaulay representation theory:

\begin{theorem}\label{AR duality}
\begin{itemize}
\item[(a)] \emph{(Auslander-Reiten-Serre duality)} We have a functorial isomorphism for any $X,Y\in\CM^{\L}R$:
\[\underline{\Hom}_{\mod^{\L}R}(X,Y)\simeq D\Ext_{\mod^{\L}R}^d(Y,X(\w)).\]
\item[(b)] $\CM^{\L}R$ has almost split sequences, that is, for any indecomposable non-projective object $X\in\CM^{\L}R$, there exists an almost split sequence
\[0\to\tau X\to Y\to X\to0\]
in $\CM^{\L}R$, where $\tau=\Omega^{1-d}(-)(\w)$ is the Auslander-Reiten translation.
\item[(c)] \emph{(Auslander-Buchweitz approximation)}
For any $C\in\mod^{\L}R$, there exists exact sequences
\[0\to Y_C\to X_C\xrightarrow{f_C} C\to0\ \mbox{ and }\ 0\to C\xrightarrow{f^C} Y^C\to X^C\to0\]
in $\mod^{\L}R$ such that $X_C,X^C\in\CM^{\L}R$, $Y_C$ and $Y^C$
have finite projective dimension, $f_C$ is right minimal and $f^C$ is left minimal. Moreover $f_C$ is a minimal right $(\CM^{\L}R)$-approximation.
\item[(d)] $\CM^{\L}R$ is a functorially finite subcategory of $\mod^{\L}R$.
\end{itemize}
\end{theorem}

\begin{proof}
(a)(b) $R$ has at worst $\L$-isolated singularities by Theorem~\ref{L-isolated singularity}.
Thus the assertions follow from a general result in \cite{AR} (see also \cite{IT}).

(c) The argument in \cite{AB} works in the $\L$-graded setting.

(d) It is basic that $\Ext^i_R(X,Y)=0$ holds for all $i>0$, $X\in\CM^{\L}R$ and $Y\in\mod^{\L}R$
such that $Y$ has finite projective dimension.
Thus the morphism $X_C\to C$ in (c) gives a right $(\CM^{\L}R)$-approximation of $C$,
and $\CM^{\L}R$ is a contravariantly finite subcategory.

This also implies covariantly finiteness of $\CM^{\L}R$ as follows:
For any $X\in\mod^{\L}R$, let $X^*=\Hom_R(X,R)\in\mod^{\L}R$.
Let $a:Y\to X^*$ be a right $(\CM^{\L}R)$-approximation of $X^*$.
It is easily checked that the composition
\[X\stackrel{\epsilon_X}{\longrightarrow} X^{**}\stackrel{a^*}{\longrightarrow}Y^*\]
of the evaluation map $\epsilon_X$ and $a^*$ gives a left
$(\CM^{\L}R)$-approximation of $X$.
\end{proof}

Let us also recall some basic results on the structure of the stable category 
$\underline{\CM}^{\L}R$ as a triangulated category. We call the 
quotient category
\[\DDD_{\sg}^{\L}(R):=\DDD^{\bo}(\mod^{\L}R)/\KKK^{\bo}(\proj^{\L}R)\]
the \emph{singularity category} of $R$ \cite{Bu,O}.
We have the following results due to Happel, Auslander and Reiten, 
Buchweitz and Eisenbud.

\begin{theorem}\label{Buchweitz Eisenbud}
\begin{itemize}
\item[(a)] $\CM^{\L}R$ is a Frobenius category whose projective objects
are $\proj^{\L}R$, and $\underline{\CM}^{\L}R$ is a triangulated category.
\item[(b)] $\underline{\CM}^{\L}R$ has a Serre functor $S:=(\w)[d]$.
\item[(c)] The composition $\CM^{\L}R\subset\DDD^{\bo}(\mod^{\L}R)\to\DDD_{\sg}^{\L}(R)$ induces a triangle equivalence
\[\underline{\CM}^{\L}R\simeq\DDD_{\sg}^{\L}(R).\]
This gives a triangle functor $\rho:\DDD^{\bo}(\mod^{\L}R)\to\underline{\CM}^{\L}R$. Moreover $X_C\simeq\rho(C)\simeq\Omega X^C$ holds in $\underline{\CM}^{\L}R$ in Theorem~\ref{AR duality}(c).
\item[(d)] If $n=d+2$, then we have an isomorphism $[2]\simeq(\c)$ of functors $\underline{\CM}^{\L}R\to\underline{\CM}^{\L}R$.
\end{itemize}
\end{theorem}

\begin{proof}
(a) Since $R$ is Gorenstein, $\CM^{\L}R$ is a Frobenius category.
Therefore its stable category $\underline{\CM}^{\L}R$ is a triangulated category
by a general result by Happel \cite{H1}.

(b) This is immediate from Auslander-Reiten-Serre duality in
Theorem~\ref{AR duality}(a).

(c) The first assertion is a classical result by Buchweitz \cite{Bu} (see also \cite{IYa}). The last assertion is clear from $Y_C\simeq0\simeq Y^C$ in $\DDD_{\sg}^{\L}(R)$.

(d) $R$ is a hypersurface for $n=d+2$ by Proposition~\ref{R is CI}(b).
Therefore this is a well-known
result for matrix factorizations \cite{E,Y}.
\end{proof}

We characterize when the stable category $\underline{\CM}^{\L}R$ is zero.

\begin{proposition}\label{when stable category is zero}
Assume that $p_i\ge2$ for all $i$. Then the following conditions are equivalent.
\begin{itemize}
\item[(a)] $\CM^{\L}R=\proj^{\L}R$ (or equivalently, $\underline{\CM}^{\L}R=0$).
\item[(b)] $\mod^{\L}R$ has finite global dimension (or equivalently, global dimension $d+1$).
\item[(c)] $n\le d+1$.
\end{itemize}
\end{proposition}

\begin{proof}
(a) is equivalent to $\DDD^{\bo}(\mod^{\L}R)=\KKK^{\bo}(\proj^{\L}R)$ 
by Theorem~\ref{Buchweitz Eisenbud}(c).
This is clearly equivalent to (b) and (c) by Theorem~\ref{L-Poincare}.
\end{proof}

We also characterize when the stable category $\underline{\CM}^{\L}R$ is fractionally Calabi-Yau (see Definition~\ref{define CY}).

\begin{corollary}\label{CM CY}
Assume that $p_i\ge2$ for all $i$. Then the triangulated category $\underline{\CM}^{\L}R$ is fractionally Calabi-Yau if and only if one of the following conditions holds, where $p:={\rm l.c.m.}(p_1,\ldots,p_n)$.
\begin{itemize}
\item $n\le d+1$ holds. In this case $\underline{\CM}^{\L}R=0$.
\item $n=d+2$ holds. In this case $\underline{\CM}^{\L}R$ is
$\frac{p(d+2\delta(\w))}{p}$-Calabi-Yau.
\item $(R,\L)$ is Calabi-Yau. In this case $\underline{\CM}^{\L}R$ is
$\frac{dp}{p}$-Calabi-Yau.
\end{itemize}
In particular, if $n=d+2$, then $(R,\L)$ is Fano (respectively, Calabi-Yau,
anti-Fano) if and only if the fractional Calabi-Yau dimension of
$\underline{\CM}^{\L}R$ is less than (respectively, equal to, more than) $d$.
\end{corollary}

\begin{proof}
By Theorem~\ref{Buchweitz Eisenbud}(b), the Serre functor of 
$\underline{\CM}^{\L}R$ is given by $S:=(\w)[d]$.
Clearly $p\w=p\delta(\w)\c$ holds.
First we show the `if' part.
If $n\le d+1$, then $\underline{\CM}^{\L}R=0$ by Proposition~\ref{when stable category is zero}. If $n=d+2$, then
\begin{eqnarray*}
S^{p}=(p\w)[pd]=(p\delta(\w)\c)[pd]=[p(d+2\delta(\w))]
\end{eqnarray*}
holds, where we used $(\c)=[2]$ from Theorem~\ref{Buchweitz Eisenbud}(d).
Thus $\underline{\CM}^{\L}R$ is $\frac{p(d+2\delta(\w))}{p}$-Calabi-Yau.
If $(R,\L)$ is Calabi-Yau, then we have $p\w=0$ and
\[S^p=(p\w)[dp]=[dp].\]
Thus $\underline{\CM}^{\L}R$ is $\frac{dp}{p}$-Calabi-Yau.

Next we show the `only if' part.
Assume that $\underline{\CM}^{\L}R$ is $\frac{\ell}{m}$-Calabi-Yau.
Then $[m-\ell d]=S^\ell[-\ell d]=(\ell\w)$ holds.
If $m=\ell d$, then $\w\in\L$ is a torsion element and hence $(R,\L)$ is Calabi-Yau.
If $m\neq\ell d$, then the sequence $|P^{-r}|$ in Theorem~\ref{L-Poincare}(d) must be bounded and hence $n\le d+2$ holds.
\end{proof}

For $X\in\mod^{\L}R$, we consider the \emph{support} 
\[\Supp^{\L}X:=\{\pp\in\Spec^{\L}R\mid X_{(\pp)}\neq0\}.\]
The following observation is elementary.

\begin{lemma}\label{basic for graded}
$X\in\mod^{\L}R$ belongs to $\mod_0^{\L}R$ if and only if
$\Supp^{\L}X\subset\{R_+\}$ if and only if $X_{T_j}=0$ for all $j$ with $0\le j\le d$.
\end{lemma}

\begin{proof}
If $X\in\mod_0^{\L}R$, then it is annihilated by a power of $T_j$ and hence $X_{T_j}=0$ holds for all $j$.
If $X_{T_j}=0$ holds for all $j$, then $\pp\in\Supp^{\L}X$ has to contain all $T_j$ and hence $\pp=R_+$ holds.

It remains to show the first equivalence.
Note that, for $\pp\in\Spec^{\L}R$, the $R$-module $R/\pp$ belongs to $\mod^{\L}_0R$ if and only if $\pp=R_+$.
By a standard commutative algebra, there is a filtration 
$X_0=0\subset X_1\subset\cdots\subset X_\ell=X$ such that
$X_i/X_{i-1}\simeq(R/\pp_i)(\a_i)$ for $\pp_i\in\Spec^{\L}R$ and 
$\a_i\in\L$ in $\mod^{\L}R$ for any $i$ with $1\le i\le\ell$.
In this case, we have $\Supp^{\L}R=\bigcup_{i=1}^{\ell}V(\pp_i)$
for $V(\pp_i):=\{\qq\in\Spec^{\L}R\mid\pp_i\subset\qq\}$.

Then $X\in\mod^{\L}_0R$ holds if and only if $R/\pp_i\in\mod^{\L}_0R$
holds for any $1\le i\le\ell$ if and only if $\pp_i=R_+$ for any $1\le i\le\ell$ 
if and only if $\Supp^{\L}R\subset\{R_+\}$.
\end{proof}

The following notion is central in Auslander-Reiten theory for Cohen-Macaulay modules \cite{A,Y}.

\begin{definition-proposition}\label{locally free at punctured spectrum}
We say that $M\in\mod^{\L}R$ is \emph{locally free on the punctured spectrum} if the following equivalent conditions are satisfied.
\begin{itemize}
\item[(a)] For any $\pp\in\Spec^{\L}R\setminus\{R_+\}$, we have $M_{(\pp)}\in\proj^{\L}R_{(\pp)}$.
\item[(b)] For any $j$ with $0\le j\le d$, we have $M_{T_j}\in\proj^{\L}R_{T_j}$.
\item[(c)] $\Ext^i_R(M,R)\in\mod^{\L}_0R$ for any $i>0$.
\item[(d)] For any $X\in\mod^{\L}R$ and any $i>0$, the $R$-modules 
$\underline{\Hom}_R(M,X)$ and $\Ext^i_R(M,X)$ belong to $\mod_0^{\L}R$.
\end{itemize}
\end{definition-proposition}

\begin{proof}
(d)$\Rightarrow$(c) Clear.

(c)$\Rightarrow$(b) Since $R_{T_j}$ is $\L$-regular by
Theorem~\ref{L-isolated singularity}, 
the $R_{T_j}$-module $M_{T_j}$ has finite projective dimension.
On the other hand, by Lemma~\ref{basic for graded}, we have
$\Ext^i_{R_{T_j}}(M_{T_j},R_{T_j})\simeq\Ext^i_R(M,R)_{T_j}=0$ for
any $i>0$. Thus $M_{T_j}$ must be a projective $R_{T_j}$-module.

(b)$\Rightarrow$(a) Since $\pp\neq R_+$, there exists $0\le j\le d$ such that $T_j\notin\pp$.
Since $M_{T_j}\in\proj^{\L}R_{T_j}$, we have $M_{(\pp)}\in\proj^{\L}R_{(\pp)}$.

(a)$\Rightarrow$(d)
We only show $\Ext^i_R(M,X)\in\mod^{\L}_0R$ since the other
assertion can be shown similarly. For any $\pp\in\Spec^{\L}R\setminus\{R_+\}$,
we have $M_{(\pp)}\in\proj^{\L}R$ by (a).
Thus $\Ext^i_R(M,X)_{(\pp)}=\Ext^i_{R_{(\pp)}}(M_{(\pp)},X_{(\pp)})=0$
holds. Hence $\Supp^{\L}\Ext^i_R(M,X)\subset\{R_+\}$ holds, and the assertion
follows from Lemma~\ref{basic for graded}.
\end{proof}

Since $R$ has $\L$-isolated singularities, we have the following useful
property of $\L$-graded Cohen-Macaulay $R$-modules.

\begin{proposition}\label{CM is locally free}
Any object in $\CM^{\L}R$ is locally free on the punctured spectrum.
\end{proposition}

\begin{proof}
Since $M\in\CM^{\L}R$, we have $\Ext^i_R(M,R)=0$ for any $i>0$.
Thus the condition (c) in Definition-Proposition~\ref{locally free at punctured spectrum} is satisfied.
\end{proof}

We need the following analog of 
\cite[A.2]{KMV}\cite[A.2]{O2}\cite{S}\cite[2.4]{T1}.

\begin{proposition}\label{generating derived category}
Let $R$ be a GL complete intersection. Then we have
\[\DDD^{\bo}(\mod^{\L}R)=\thick\{\proj^{\L}R,\mod^{\L}_0R\}\ \mbox{ and }\ \underline{\CM}^{\L}R=\thick(\rho(\mod^{\L}_0R)),\]
where $\rho:\DDD^{\bo}(\mod^{\L}R)\to\underline{\CM}^{\L}R$
is the triangle functor in Theorem~\ref{Buchweitz Eisenbud}(c).
\end{proposition}

We give a simple proof following Takahashi's method \cite[3.4, 4.1]{T2}.

\begin{proof}
By Proposition~\ref{FPT lemma}(b), it suffices to prove the second equality.
For $X\in\DDD^{\bo}(\mod^{\L}R)$, let
\[\Lambda:=\underline{\End}_R(\rho(X))=
\bigoplus_{\x\in\L}\underline{\Hom}^{\L}_R(\rho(X),\rho(X)(\x)).\]
By Proposition~\ref{CM is locally free}, we have $\Lambda\in\mod^{\L}_0R$.
We take $\ell\gg0$ such that $T_j^\ell$ annihilates $\Lambda$ for all $0\le j\le d$.
Then $R/(T_0^\ell)\Lotimes_R\cdots\Lotimes_RR/(T_d^\ell)=R/(T_0^\ell,\ldots,T_d^\ell)$ holds,
and $R/(T_0^\ell,\ldots,T_d^\ell)\Lotimes_RX$ belongs to $\DDD^{\bo}(\mod_0^{\L}R)$.

Now let $R^j:=R/(T_0^\ell,\ldots,T_{j-1}^\ell)$ and $X^j:=R^j\Lotimes_RX$ for 0$\le j\le d+1$.
Then $X^0=X$ and $X^{d+1}\in\DDD^{\bo}(\mod_0^{\L}R)$ hold.
The short exact sequence $0 \to R^j(-\ell\c)\stackrel{T_{j}^\ell}{\longrightarrow} R^j \to R^{j+1} \to 0$ gives rise to a triangle
\[X^j(-\ell\c)\stackrel{T_{j}^\ell}{\longrightarrow}X^j\stackrel{f_j}{\longrightarrow} X^{j+1}\to X^j(-\ell\c)[1].\]
Since $\rho(T_j^\ell)=0$, we have that $\rho(f_j)$ is a split monomorphism in 
$\underline{\CM}^{\L}R$. Hence $\rho(X^j)\in\add\rho(X^{j+1})$ holds for any $j$.
Since $X^{d+1}\in\DDD^{\bo}(\mod_0^{\L}R)$, we have inductively
$\rho(X)\in\thick\rho(\mod^{\L}_0R)$.
\end{proof}

We end this section by giving a description of the Auslander-Reiten quiver of $\CM^{\L}R$.

\begin{definition}[Auslander-Reiten quiver]
The \emph{Auslander-Reiten quiver} $\mathfrak{A}(\CM^{\L}R)$ of $(R,\L)$ is defined as follows:
\begin{itemize}
\item The set of vertices are the isomorphism classes of indecomposable objects in $\CM^{\L}R$.
\item For indecomposable objects $X,Y\in\CM^{\L}R$, let 
\begin{eqnarray*}
&D_X:=\End_R^{\L}(X)/\rad\End_R^{\L}(X),\ {\rm Irr}(X,Y):=\rad_{\CM^{\L}R}(X,Y)/\rad_{\CM^{\L}R}^2(X,Y),&\\
&d_{XY}:=\dim_{D_X}{\rm Irr}(X,Y)\ \mbox{and}\ 
d'_{XY}:=\dim{\rm Irr}(X,Y)_{D_Y}.&
\end{eqnarray*}
\item We draw a valued arrow $X\xrightarrow{(d_{XY},d'_{XY})}Y$ if $d_{XY}$ (or equivalently,
$d'_{XY}$) is non-zero. The valuation $(1,1)$ is usually omitted.
\end{itemize}
We define the \emph{stable Auslander-Reiten quiver} $\mathfrak{A}(\underline{\CM}^{\L}R)$
of $(R,\L)$ by removing all vertices $R(\x)$ with $\x\in\L$ from ${\mathfrak A}(\CM^{\L}R)$.
\end{definition}

It is well-known that $\mathfrak{A}(\CM^{\L}R)$ describes the terms of
minimal right (respectively, left) almost split morphisms in $\CM^{\L}R$ in the following sense.

\begin{proposition}\label{properties of AR quiver}
Let $X,Y\in\CM^{\L}R$ be indecomposable.
\begin{itemize}
\item[(a)] The minimal right (respectively, left) almost split morphism of $X$ has a form
\[\bigoplus_{Y}Y^{\oplus d_{YX}}\to X\ \ (\mbox{respectively,}\ X\to\bigoplus_{Y}Y^{\oplus d'_{XY}}).\]
\item[(b)] If $X$ is non-projective, then $d_{YX}=d'_{\tau X,Y}$ and $d'_{YX}=d_{\tau X,Y}$ hold for the Auslander-Reiten translation $\tau X=\Omega^{1-d}X(\w)$.
\end{itemize}
\end{proposition}

Thus $\mathfrak{A}(\CM^{\L}R)$ has a structure of a translation quiver, and
$\mathfrak{A}(\underline{\CM}^{\L}R)$ has a structure of a stable translation quiver
since $\underline{\CM}^{\L}R$ is a triangulated category.

We give the following description of $\mathfrak{A}(\CM^{\L}R)$ from $\mathfrak{A}(\underline{\CM}^{\L}R)$, where $\rho\colon\DDD^{\bo}(\mod^{\L}R)\to\underline{\CM}^{\L}R$ is the functor given in Theorem~\ref{Buchweitz Eisenbud}(c) and we identify objects in $\underline{\CM}^{\L}R$ with objects in $\CM^{\L}R$ without non-zero projective direct summands.

\begin{theorem}\label{AR quiver in general}
Assume $n\ge d+2$ and $p_i\ge2$ for all $i$.
Then $X:=X_{R_+}$ in Theorem~\ref{AR duality}(c) is indecomposable, and $\mathfrak{A}(\CM^{\L}R)$ is obtained from $\mathfrak{A}(\underline{\CM}^{\L}R)$
by adding, for any $\x \in \L$
\begin{itemize}
\item the projective $R$-module $R(\x)$;
\item an arrow $X(\x)\to R(\x)$, where $X(\x)\simeq\rho(k)(\x)[-1]$;
\item an arrow $R(\x) \to \Omega^{d-1}X(\x-\w)$, where $\Omega^{d-1}X(\x-\w)\simeq\rho(k)(\x-\w)[-d]$.
\end{itemize}
In particular for any indecomposable projective $R$-module $R(\x)$, precisely one arrow starts at $R(\x)$ and precisely one arrow ends at $R(\x)$. These arrows have valuation $(1,1)$.
\end{theorem}

We need the following preparation, which is a version of \cite[Proposition 5.7]{ADS} in our setting.

\begin{lemma}\label{delta=0}
Assume $n\ge d+2$ and $p_i\ge2$ for all $i$. Consider the sequence 
\begin{equation}\label{CM approximation of R_+}
0\to Y\xrightarrow{g} X\xrightarrow{f}R_+\to 0
\end{equation}
given in Theorem~\ref{AR duality}(c).
Then $X$ does not have non-zero projective direct summands.
\end{lemma}

\begin{proof}
Recall that $X\in\CM^{\L}R$, $Y$ has finite projective dimension and $f$ is right minimal.
We prove the assertion by induction on $d$.

In the case $d=-1$, this is obvious since $X=R_+$.
For $d\ge0$, let $R'=R/(X_1)$ and $\L'=\L/\langle\x_1\rangle$.
Then $(R',\L')$ is a GL complete intersection with smaller dimension.
We consider the functor $\overline{(-)}=R'\otimes_R-:\mod^{\L}R\to\mod^{\L'}R'$.
Since $X_1\in\overline{R_+}$ generates a simple submodule of $\overline{R_+}$ which is not contained in the radical, we have isomorphisms
\[\overline{R_+}= R_+/R_+X_1\simeq(RX_1/R_+X_1)\oplus(R_+/RX_1)= k(-\x_1)\oplus R'_+\]
in $\mod^{\L}R$.
Note that $\Tor^R_i(R',\Omega(\mod^{\L}R))=0$ holds for all $i\ge1$ since $X_1$ acts as a non-zerodivisor on each object in $\Omega(\mod^{\L}R)$ by Theorem~\ref{L-factorial L-domain}.
Appyling $\overline{(-)}$ to \eqref{CM approximation of R_+}, we have an exact sequence
\[0=\Tor^R_1(R',R_+)\to\overline{Y}\xrightarrow{\overline{g}}\overline{X}\xrightarrow{\overline{f}}\overline{R_+}\to0.\]
It is straightforward to check that $\overline{X}\in\CM^{\L'}R'$ and $\overline{Y}$ has finite projective dimension.
Thus the morphism $\overline{f}\colon\overline{X}\to\overline{R_+}$ gives a right
$(\CM^{\L'}R')$-approximation of $\overline{R_+}\simeq k(-\x_1)\oplus R'_+$.

Since $f\colon X\to R_+$ is right minimal, the morphism $g\colon Y\to X$ belongs to the radical of the category $\mod^{\L}R$. Thus $\overline{g}\colon
\overline{Y}\to\overline{X}$ belongs to the radical of $\mod^{\L'}R'$, and hence $\overline{f}\colon\overline{X}\to\overline{R_+}$ is right minimal. By induction hypothesis, the $R'$-module $\overline{X}$ does not have non-zero 
projective direct summands, and therefore the $R$-module $X$ also
does not have non-zero projective direct summands, as desired.
\end{proof}

Now we prove Theorem~\ref{AR quiver in general}. We need Theorem~\ref{CM tilting} proved in the next section.

\begin{proof}[Proof of Theorem~\ref{AR quiver in general}]
We only have to understand the arrows from or to $R(\x)$ with $\x\in\L$. It suffices to 
understand arrows to and from $R$, and then add all degree shifts by $\L$.

Let us first consider arrows to $R$ in $\mathfrak{A}(\CM^{\L}R)$. Let $0\to Y\to X\to R_+\to0$
be an exact sequence given in Theorem~\ref{AR duality}(c). Then the composition
$X\to R_+\to R$ is a minimal right almost morphism of $R$ in $\CM^{\L}R$, and hence the 
direct summands of $X$ are precisely the sources of arrows to $R$.
By Lemma~\ref{delta=0}, $X$ does not have non-zero projective direct summands.
Since $X\simeq\rho(R_+)\simeq\Omega\rho(k)$ holds in $\underline{\CM}^{\L}R$ by Theorem~\ref{Buchweitz Eisenbud}(c), we have $\End_{\underline{\CM}^{\L}R}(X)\simeq\End_{\underline{\CM}^{\L}R}(\rho(k))\simeq\End_{\DDD^{\bo}(\mod^{\L}R)}(k)=k$ by Theorem~\ref{CM tilting}.
Thus $X$ is indecomposable, and there is a precisely one arrow $X\to R$ ending at $R$. It has valuation
$(1,1)$ since $D_X=k=D_R$.

By Proposition~\ref{properties of AR quiver}(b), there is a precisely one arrow $R \to \tau^-(X)=\Omega^{d-1}X(-\w)$ starting at $R$, and it has valuation $(1,1)$.
\end{proof}

\section{Tilting theory in the stable categories of Cohen-Macaulay modules}
\label{subsection: GL CI 2}

Let $(R,\L)$ be a Geigle-Lenzing complete intersection associated with
linear forms $\ell_1,\ldots,\ell_n$ and weights $p_1,\ldots,p_n$.
Recall from Theorem~\ref{Buchweitz Eisenbud}(c) that we have a triangle functor
\[\rho:\DDD^{\bo}(\mod^{\L}R)\to\underline{\CM}^{\L}R\]
which induces a triangle eqivalence
$\DDD^{\L}_{\sg}(R)=\DDD^{\bo}(\mod^{\L}R)/\KKK^{\bo}(\proj^{\L}R)\simeq\underline{\CM}^{\L}R$.
In this section, we show that certain subcategories of $\DDD^{\bo}(\mod^{\L}R)$
are triangle equivalent to $\underline{\CM}^{\L}R$ through $\rho$.
As an application, we obtain a triangle equivalence
$\DDD^{\bo}(\mod A^{\rm CM})\simeq\underline{\CM}^{\L}R$ for a certain finite dimensional $k$-algebra $A^{\rm CM}$.

See Definition~\ref{define mod^I} for $\mod^IR$, $\mod^I_0R$ and $\proj^IR$ for a subset $I$ of $\L$, which we will use intensively now.

\begin{observation}
Let $I$ be a subset of $\L$.
\begin{itemize}
\item[(a)] $\KKK^{\bo}(\proj^IR)=\thick_{\DDD^{\bo}(\mod^{\L}R)}\{R(-\x)\mid\x \in I\}$.
\item[(b)] If $I$ is convex, then $\DDD^{\bo}(\mod^IR)$ (respectively, $\DDD^{\bo}(\mod^I_0R)$) is triangle
equivalent to the full subcategory of $\DDD^{\bo}(\mod^{\L}R)$ consisting of all
objects whose cohomologies belong to $\mod^IR$ (respectively, $\mod^I_0R$).
\end{itemize}
\end{observation}

When $I$ is a convex subset, we have a functor
\[(-)_I:\mod^{\L}R\to\mod^IR\ \mbox{ given by }\ X_I:=\bigoplus_{\x\in I}X_{\x}.\]
Note that $X_I$ has a natural structure of an $\L$-graded $R$-module,
which is a subfactor module of $X$.

The following easy observations play an important role (see Definition~\ref{define I-canonical} for $A^I$).

\begin{theorem}\label{tilting in S^I}
Let $I$ be a finite subset of $\L$.
\begin{itemize}
\item[(a)] We have a triangle equivalence $\DDD^{\bo}(\mod A^I)\simeq \KKK^{\bo}(\proj^IR)$.
Moreover $\KKK^{\bo}(\proj^IR)$ has a tilting object
\[T^I:=\bigoplus_{\x\in I}R(-\x)\in\proj^IR\ \mbox{ such that }\ \End_R^{\L}(T^I)\simeq A^I.\]
\item[(b)] Assume that $I$ is convex. Then we have a triangle equivalence $\DDD^{\bo}(\mod A^I)\simeq\DDD^{\bo}(\mod^IR)$.
Moreover $\DDD^{\bo}(\mod^IR)$ has a tilting object
\[U^I:=\bigoplus_{\x\in I} R(-\x)_I\in\mod^IR\ \mbox{ such that }\ \End_R^{\L}(U^I)\simeq A^I.\]
\end{itemize}
\end{theorem}

\begin{proof}
(a) We have an equivalence $\proj^IR\simeq\proj A^I$ by
Proposition~\ref{equivalences for I-canonical}(a). Thus we have a triangle equivalence
\[ \KKK^{\bo}(\proj^IR)\simeq\KKK^{\bo}(\proj A^I)\]
sending $T^I$ to $A^I$, which shows that $T^I$ is a tilting object in $\KKK^{\bo}(\proj^IR)$,
and that $\End_R^{\L}(T^I) \simeq \End_{A^I}(A^I) = A^I$.
By Proposition~\ref{I-canonical has finite global dimension}(a), we have the triangle equivalences $\KKK^{\bo}(\proj A^I)\simeq\DDD^{\bo}(\mod A^I)$.

(b) We have an equivalence $\mod^IR\simeq\mod A^I$ by Proposition~\ref{equivalences for I-canonical}(b). Thus we have a triangle equivalence
\[\DDD^{\bo}(\mod^IR)\simeq\DDD^{\bo}(\mod A^I)\]
sending $U^I$ to $A^I$. Thus $\End_R^{\L}(U^I)\simeq A^I$ holds, and $U^I$ is a tilting object in $\DDD^{\bo}(\mod^IR)$ since $A^I$ has finite global dimension by Proposition~\ref{I-canonical has finite global dimension}(a).
\end{proof}

Since $R$ is Gorenstein, we have a duality
\[(-)^\star:=\RHom_R(-,R):\DDD^{\bo}(\mod^{\L}R)\xrightarrow{\sim}\DDD^{\bo}(\mod^{\L}R).\]
Since $R(\x)^\star=R(-\x)$ and $k(\x)^\star=k(-\x-\w)[-d]$, we have induced dualities
\begin{equation} \label{* for SP}
\begin{aligned}
(-)^\star & \colon \KKK^{\bo}(\proj^IR) \xrightarrow{\sim} \KKK^{\bo}(\proj^{-I}R),\\
(-)^\star & \colon \DDD^{\bo}(\mod^I_0R) \xrightarrow{\sim} \DDD^{\bo}(\mod^{-I+\w}_0R).
\end{aligned}
\end{equation}

For the next results, we need the following piece of notation. 
Let $\TT$ be a triangulated category.
Recall from the beginning of Chapter~\ref{section: preliminaries} that $\XX*\YY$
denotes the category of extensions of subcategories $\XX$ and $\YY$
of $\TT$. If $\Hom_{\TT}(\XX,\YY)=0$ holds, we write
\[\XX\perp\YY:=\XX*\YY.\]
A \emph{semiorthogonal decomposition} (or \emph{stable t-structure})
of $\TT$ is a pair of thick subcategories of $\TT$ satisfying $\TT=\XX\perp \YY$.
In this case we have triangle equivalences $\XX\simeq\TT/\YY$ and $\YY\simeq\TT/\XX$.

The following result plays a crucial role. For a subset $I \subset \L$, we write
\[I^c=\L\setminus I\ \mbox{ and }\ -I=\{-\x\mid\x\in I\}.\]
Recall that a subset $I$ of $\L$ is called \emph{upset} if $I+\L_+\subset I$.
In this case, $-I^c$ is clearly an upset too.

\begin{theorem}\label{basic embedding}
For any non-trivial upset $I$ in $\L$, we have
\[\DDD^{\bo}(\mod^{\L}R)= \KKK^{\bo}(\proj^{I^c} R) \perp (\DDD^{\bo}(\mod^IR) \cap (\DDD^{\bo}(\mod^{-I^c}R))^\star)\perp \KKK^{\bo}(\proj^I R).\]
Thus the composition
\[ \DDD^{\bo}(\mod^IR) \cap (\DDD^{\bo}(\mod^{-I^c}R))^\star \subset \DDD^{\bo}(\mod^{\L}R)
\stackrel{\rho}{\longrightarrow}\underline{\CM}^{\L}R\]
is a triangle equivalence.
\end{theorem}

For the proof, we will need the following distributive law for extensions, which is easily checked.

\begin{observation} \label{distributive}
Let $\XX$, $\YY$, and $\ZZ$ be thick subcategories of $\DD$.
If $\YY\subset\ZZ$ (respectively, $\XX\subset\ZZ$), then $(\XX*\YY)\cap\ZZ=(\XX\cap\ZZ)*\YY$
(respectively, $(\XX*\YY)\cap\ZZ=\XX*(\YY\cap\ZZ)$).
\end{observation}

Next we collect some elementary semiorthogonal decompositions. This is an $\L$-graded version of \cite[2.3]{O}.

\begin{lemma}\label{decompositions for P}
Let $I$ be a non-empty upset in $\L$.
\begin{itemize}
\item[(a)] We have
\[ \KKK^{\bo}(\proj^\L R) = \KKK^{\bo}(\proj^{I^c} R) \perp \KKK^{\bo}(\proj^{I} R) \text{ and } \DDD^{\bo}(\mod^\L R) = \KKK^{\bo}(\proj^{I^c}R) \perp \DDD^{\bo}(\mod^I R). \]
More generally, for two upsets $I \subseteq J \subseteq \L$, we have
\[ \KKK^{\bo}(\proj^J R) = \KKK^{\bo}(\proj^{J \setminus I} R) \perp \KKK^{\bo}(\proj^{I} R) \text{ and } \DDD^{\bo}(\mod^J R) = \KKK^{\bo}(\proj^{J \setminus I}R) \perp \DDD^{\bo}(\mod^I R). \]
\item[(b)] We have a triangle equivalence $\DDD^{\bo}(\mod^I R) / \KKK^{\bo}(\proj^{I} R) \simeq \underline{\CM}^{\L}R$.
\end{itemize}
\end{lemma}

\begin{proof}
(a) Clearly $\Hom_R^{\L}(\proj^{I^c}R, \proj^IR)=0$.
For any $P\in\proj^{\L}R$, we denote by $P^{I^c}$ the sub-$R$-module of $P$
generated by the subspace $\bigoplus_{\x\in I^c}P_{\x}$, and $P^I:=P/P^{I^c}$.
These give functors $(-)^{I^c}:\proj^{\L}R\to\proj^{I^c}R$,
$(-)^I:\proj^{\L}R\to\proj^IR$ and a sequence
\[0\to(-)^{I^c}\to {\rm id}\to(-)^I\to0\]
of natural transformations which is objectwise split exact.
Therefore we have induced triangle functors
$(-)^{I^c} \colon \KKK^{\bo}(\proj^{\L}R)\to\KKK^{\bo}(\proj^{I^c}R)$,
$(-)^I \colon \KKK^{\bo}(\proj^{\L}R)\to\KKK^{\bo}(\proj^IR)$ and a functorial triangle
$Q^{I^c}\to Q\to Q^I\to Q^{I^c}[1]$ for any $Q\in\KKK^{\bo}(\proj^{\L}R)$.
Thus we have the first equality.

The second equality is shown similarly, using the equivalence between $\DDD^{\bo}(\mod R)$ and the homotopy category
$\KKK^{-,\bo}(\proj^{\L}R)$ of complexes bounded above with bounded cohomologies.
The remaining equalities can be shown similarly.

(b) By (a), we have triangle equivalences
\[ \KKK^{\bo}(\proj^{I} R) = \frac{\KKK^{\bo}(\proj^\L R)}{\KKK^{\bo}(\proj^{I^c} R)} \text{ and } \DDD^{\bo}(\mod^I R) = \frac{\DDD^{\bo}(\mod^\L R)}{\KKK^{\bo}(\proj^{I^c}R)}. \]
Therefore we have
\[ \underline{\CM}^{\L}R \simeq \frac{\DDD^{\bo}(\mod^\L R)}{\KKK^{\bo}(\proj^\L R)} = \frac{\DDD^{\bo}(\mod^\L R)/\KKK^{\bo}(\proj^{I^c}R)}{\KKK^{\bo}(\proj^\L R)/\KKK^{\bo}(\proj^{I^c} R)} \simeq \frac{\DDD^{\bo}(\mod^I R)}{\KKK^{\bo}(\proj^{I} R)}. \qedhere \]
\end{proof}

\begin{proof}[Proof of Theorem~\ref{basic embedding}]
Applying Lemma~\ref{decompositions for P}(a) to the upset $-I^c$, we have
\[ \DDD^{\bo}(\mod^\L R) = \KKK^{\bo}(\proj^{-I}R) \perp \DDD^{\bo}(\mod^{-I^c} R).\]
Applying $(-)^\star$, we obtain the decomposition
\begin{eqnarray*}
\DDD^{\bo}(\mod^\L R) &=& \DDD^{\bo}(\mod^\L R)^\star = \DDD^{\bo}(\mod^{-I^c} R)^\star \perp \KKK^{\bo}(\proj^{-I}R)^{\star}\\
&=& \DDD^{\bo}(\mod^{-I^c} R)^\star \perp \KKK^{\bo}(\proj^{I}R).
\end{eqnarray*}
Since $\KKK^{\bo}(\proj^{I}R) \subseteq \DDD^{\bo}(\mod^I R)$, by Observation~\ref{distributive} we obtain
\[ \DDD^{\bo}(\mod^I R) = (\DDD^{\bo}(\mod^IR) \cap (\DDD^{\bo}(\mod^{-I^c}R))^\star) \perp \KKK^{\bo}(\proj^{I}R).\]
Now the first claim holds by Lemma~\ref{decompositions for P}(a), and the second one by 
\[\DDD^{\bo}(\mod^IR) \cap (\DDD^{\bo}(\mod^{-I^c}R))^\star\simeq\frac{\DDD^{\bo}(\mod^I R)}{\KKK^{\bo}(\proj^{I}R)}\]
and Lemma~\ref{decompositions for P}(b).
\end{proof}

We call
\[\de:=d\c+2\w\in\L\]
the \emph{dominant element} (cf.\ \cite{KLM}). Applying Definition~\ref{define I-canonical} to the interval $[0,\de]$, we define the \emph{CM-canonical algebra}
\[A^{\rm CM}:=A^{[0,\de]}.\]
Now we are ready to prove the following main result in this section.

\begin{theorem}\label{CM tilting}
Let $(R,\L)$ be a Geigle-Lenzing complete intersection.
\begin{itemize}
\item[(a)] The following composition is a triangle equivalence:
\[ \DDD^{\bo}(\mod^{[0, \de]} R) \subset \DDD^{\bo}(\mod^{\L}R) \stackrel{\rho}{\longrightarrow}\underline{\CM}^{\L}R.\]
\item[(b)] We have triangle equivalences
\[\DDD^{\bo}(\mod A^{\rm CM})\simeq \DDD^{\bo}(\mod^{[0, \de]} R)
\simeq\underline{\CM}^{\L}R\ \mbox{ such that }\ A^{\rm CM}\mapsto U^{[0,\de]}\mapsto T^{\rm CM}:=\rho(U^{[0,\de]}).\]
In particular $\underline{\CM}^{\L}R$ has a tilting object $T^{\rm CM}$.
\item[(c)] We have
\begin{eqnarray*}
\DDD^{\bo}(\mod^{[0, \de]} R) &=& \DDD^{\bo}(\mod^{\L_+}R) \cap (\DDD^{\bo}(\mod^{-\L_+^c}R))^\star,\\
\DDD^{\bo}(\mod^{\L}R) &=& \KKK(\proj^{\L_+^c} R) \perp \DDD^{\bo}(\mod^{[0, \de]} R) \perp \KKK(\proj^{\L_+} R).
\end{eqnarray*}
\end{itemize}
\end{theorem}

For the hypersurface case $n=d+2$, this result was shown by Futaki-Ueda \cite{FU} and Kussin-Lenzing-Meltzer \cite{KLM} ($d=1$) using quite different methods.
For the non-hypersurface case, Theorem~\ref{CM tilting} is new even for the case $d=1$.

\begin{proof}
We only have to prove (c). In fact, (a) follows from (c) and Theorem~\ref{basic embedding},
and (b) follows from (a) and Theorem~\ref{tilting in S^I}(b).

In the rest, we prove the statement (c).
By Lemma~\ref{order omega}, $-\w\ {\not\le}\ \x$ if and only if $\x\le d\c$.
Thus $\x\in\L_+^c+\w$ if and only if $\w\ {\not\le}\ \x$ if and only if
$\x\le \de$. Thus we have
\begin{equation}\label{dc+2w}
\L_+\cap(\L_+^c+\w)=[0,\de].
\end{equation}
As a consequence, we have
\begin{align*}
& \DDD^{\bo}(\mod^{\L_+} R) \cap (\DDD^{\bo}(\mod^{-\L_+^c}R))^\star \supset \DDD^{\bo}(\mod_0^{\L_+} R) \cap \underbrace{(\DDD^{\bo}(\mod_0^{-\L_+^c}R))^\star}_{\stackrel{\eqref{* for SP}}{=} \DDD^{\bo}(\mod_0^{\L_+^c+\w}R)} \\
& \qquad = \DDD^{\bo}(\mod_0^{\L_+ \cap (\L_+^c + \w)} R) \stackrel{\eqref{dc+2w}}{=} \DDD^{\bo}(\mod_0^{[0,\de]} R) = \DDD^{\bo}(\mod^{[0,\de]} R).
\end{align*}
To show the reverse inclusion, it is enough to show that the
composition
\[ \DDD^{\bo}(\mod^{[0,\de]} R) \subset  \DDD^{\bo}(\mod^\L R) \stackrel{\rho}{\to}\underline{\CM}^{\L}R \]
is dense. This is equivalent to $\DDD^{\bo}(\mod^\L R)=\thick \{ \mod^{[0, \de]} R,\proj^{\L} R\}$ by Proposition~\ref{FPT lemma}.
Since we have $\DDD^{\bo}(\mod^\L R)=\thick\{\mod_0^\L R,\proj^\L R\}$
by Proposition~\ref{generating derived category}, it is enough to prove
\begin{equation}\label{S^[0,de] generates}
\mod_0^\L R \subset \thick \{\mod_0^{[0,\de]} R,\proj^\L R\}.
\end{equation}
To prove this, we prepare the following simple observation.

\begin{lemma}\label{perfect}
Let $J$ be a subset of $\{1,\ldots,n\}$ with $|J|=n-d-1$ and $J^c:=\{1,\ldots,n\}\setminus J$.
Set $F:=R/(X_i^{p_i}\mid i\in J^c)$.
\begin{itemize}
\item[(a)] The $R$-module $F$ is finite dimensional and has finite projective dimension.
\item[(b)] $\soc F=k(-\sum_{i=1}^n(p_i-1)\x_i)$.
\item[(c)] $F/\soc F$ belongs to $\thick\{k(-\x)\}_{0\le \x<\sum_{i=1}^n(p_i-1)\x_i}$.
\end{itemize}
\end{lemma}

\begin{proof}
Since $|J^c|=d+1$, $\ell_i$ ($i\in J^c$) are linear independent. It follows
from Lemma~\ref{prop.regularseq1}(c) that $X_i^{p_i}$ ($i\in J^c$) forms an $R$-regular sequence. Thus the assertion (a) follows.
Moreover $F$ is a finite dimensional Gorenstein algebra whose $a$-invariant is given by
\[\w+(d+1)\c=\sum_{i=1}^n(p_i-1)\x_i.\]
This equals the degree of the socle of $F$, and the assertion (b) follows.
Since the degree 0 part of $F$ is $k$, its degree $\sum_{i=1}^n(p_i-1)\x_i$ part
has to be one dimensional. Thus the assertion (c) follows.
\end{proof}

Now we prove \eqref{S^[0,de] generates}.

(i) We show that $k(-\x)\in\thick\{\mod^{[0, \de]}R,\proj^\L R\}$ holds for any $\x\in\L_+$
by using induction with respect to the partial order on $\L_+$.
We write $\x \in \L_+$ in a normal form $\x=\sum_{i=1}^na_i\x_i+a\c$
(see Observation~\ref{basic results on L}(a)).
Let $J:=\{i\mid a_i=p_i-1\}$.

Assume $|J|+a\le n-d-2$. Then we have
\[0\le\x\le (n-d-2)\c+\sum_{i=1}^n(p_i-2)\x_i=(2n-d-2)\c-2\sum_{i=1}^n\x_i=\de,\]
and hence $k(-\x)\in \mod^{[0,\de]} R$.

Assume $|J|+a\ge n-d-1$. Then there exists a subset $J'$ of $\{1,\ldots,n\}$ such that $|J'|=n-d-1$ and $\y:=\x-\sum_{i\in J'}(p_i-1)\x_i$ belongs to $\L_+$.
By Lemma~\ref{perfect}, we have an exact sequence
\[0\to k(-\x)\to F(-\y)\to (F/\soc F)(-\y)\to0\]
in $\mod^{\L}R$ with $F(-\y)$ of finite projective dimension and $(F/\soc F)(-\y)$ belongs to $\thick\{k(-\z)\}_{\y\le \z<\x}$.
By the induction hypothesis, we have $(F/\soc F)(-\y)\in \thick \{\mod_0^{[0,\de]} R,\proj^\L R\}$ and therefore $k(-\x)\in \thick \{\mod_0^{[0,\de]} R,\proj^\L R\}$.

(ii) Similarly, by using induction with respect to the reverse of the partial order on $\L$, we can show that $k(-\x) \in \thick \{\mod_0^{[0,\de]} R,\proj^\L R\}$ holds for any $\x\in\L$.

We have finished proving \eqref{S^[0,de] generates}, which implies all other statements as we observed.
\end{proof}

\section{CM-canonical algebras and $d$-tilting objects}

In the rest of this section, we concentrate on our CM-canonical algebra $A^{\rm CM}$.
It is useful to express the dominant element in a normal form
\begin{equation}\label{another delta}
\de=(n-d-2)\c+\sum_{i=1}^n(p_i-2)\x_i.
\end{equation}
The global dimension of CM-canonical algebras is given by the following result.

\begin{theorem}\label{gl.dim of A^CM}
Let $(R,\L)$ be a Geigle-Lenzing complete intersection of dimension $d+1$ with weights $p_1,\ldots,p_n$, and $A^{\rm CM}$ the corresponding CM-canonical algebra.
Assume $p_i\ge2$ for all $i$.
\begin{itemize}
\item[(a)] $A^{\rm CM}=0$ if and only if $n \le d+1$.
\item[(b)] If $n\ge d+2$, then $\gl A^{\rm CM}=2(n-d-2)+\#\{i\mid p_i\ge3\}$.
\end{itemize}
\end{theorem}

\begin{proof}
(a) Clearly $A^{\rm CM}\neq0$ if and ony if $0\le\de$.
This is equivalent to $n\ge d+2$ since 
$\de$ has the normal form \eqref{another delta}.

(b) This follows from \eqref{another delta} and Theorem~\ref{gl.dim of A^I}.
\end{proof}

The quiver presentation of $A^{\rm CM}$ was given in Theorem~\ref{endomorphism convex}.
For the hypersurface case $n=d+2$, we give more detailed properties below,
where the statement (b) was shown in \cite[6.1]{KLM}, \cite[1.2]{FU}.
We refer to \cite{HM} for more information on the tensor product of the path algebras of type $\A$.

\begin{corollary}\label{presentation of CM canonical}
Let $(R,\L)$ be a GL complete intersection of dimension $d+1$
with $d+2$ weights $p_1,\ldots,p_{d+2}$, and $A^{\rm CM}$ the corresponding
CM-canonical algebra. Assume $p_i\ge2$ for all $i$.
\begin{itemize}
\item[(a)] $\de=\sum_{i=1}^{d+2}(p_i-2)\x_i$ holds, and the global dimension of $A^{\rm CM}$ is equal to $\#\{i\mid p_i\ge 3\}$.
\item[(b)] We have $A^{\rm CM}\simeq\bigotimes_{i=1}^{d+2}k\A_{p_i-1}$,
where $k\A_{p_i-1}$ is the path algebra of the equioriented quiver of type $\A_{p_i-1}$.
In particular, $\underline{\CM}^{\L}R$ is independent of the choice of linear forms.
\item[(c)] The Grothendieck group $K_0(\underline{\CM}^{\L}R)$
is a free abelian group of rank $\prod_{i=1}^{d+2}(p_i-1)$.
\item[(d)] \emph{(Kn\"orrer periodicity)} Let $(R',\L')$ be a Geigle-Lenzing complete
intersection of dimension $d+2$ with $d+3$ weights $2,p_1,\ldots,p_{d+2}$.
Then we have a triangle equivalence 
$\underline{\CM}^{\L}R\simeq\underline{\CM}^{\L'}R'$.
\end{itemize}
\end{corollary}

\begin{proof}
(a) This is immediate from \eqref{another delta} and Theorem~\ref{gl.dim of A^CM}.

(b) By \eqref{another delta}, $\de=\sum_{i=1}^{d+2} (p_i-2) \x_i$ holds. It is easy to check that the quiver $Q^{[0,\de]}$ coincides with the quiver of 
$\bigotimes_{i=1}^{d+2}kQ_{p_i-1}$. Moreover $A^{[0,\de]}$ has
only commutativity relations since $[0,\de]\cap([0,\de]-\c)=\emptyset$ holds
by $0\ {\not\le}\ \de-\c$. Hence the assertion follows.

(c) The assertion follows from the triangle equivalence $\DDD^{\bo}(\mod A^{\rm CM})
\simeq\underline{\CM}^{\L}R$ in Theorem~\ref{CM tilting} since the Grothendieck group
of $\bigotimes_{i=1}^{d+2}k\A_{p_i-1}$ is a free abelian group of rank $\prod_{i=1}^{d+2}(p_i-1)$.

(d) By (b), the CM-canonical algebras of $R$ and $R'$ are isomorphic.
Thus we have the desired triangle equivalence.
\end{proof}

We give a few examples of CM-canonical algebras.

\begin{example} \label{CM-canonical for d=-1}
We consider the case $d=-1$.
\begin{itemize}
\item[(a)] Let $n=1$ and $p_1 \ge 2$ arbitrary. Then by Corollary~\ref{presentation of CM canonical}(b), $U=\bigoplus_{i=1}^{p_1-1}k[x_1]/(x_1^i)$ is a tilting object in $\underline{\CM}^{\L}R$ and $A^{\rm CM} \simeq k\A_{p_1-1}$. For instance, if $p_1=6$, then $A^{\rm CM}$ is the path algebra of the following quiver:
\[
\begin{xy} 0;<4pt,0pt>:<0pt,4pt>:: 
(0,0) *+{0} ="0",
(10,0) *+{\x_1} ="1",
(20,0) *+{2\x_1} ="11",
(30,0) *+{3\x_1} ="111",
(40,0) *+{4\x_1} ="1111",
"0", {\ar"1"},
"1", {\ar"11"},
"11", {\ar"111"},
"111", {\ar"1111"},
\end{xy}
\]

\item[(b)] Let $n =2$. Then $\de=\c+\sum_{i=1}^2(p_i-2)\x_i$.
For $(p_1,p_2) = (3,4)$, the quiver of $A^{\rm CM}$ is the following:
\[
\begin{xy} 0;<3pt,-3pt>:<3pt,3pt>:: 
(0,0) *+{0} ="0",
(0,10) *+{\x_2} ="1",
(0,20) *+{2\x_2} ="11",
(0,30) *+{3\x_2} ="111",
(10,0) *+{\x_1} ="2",
(10,10) *+{\x_1+\x_2} ="12",
(10,20) *+{\x_1+2\x_2} ="112",
(10,30) *+{\x_1+3\x_2} ="1112",
(20,0) *+{2\x_1} ="22",
(20,10) *+{2\x_1+\x_2} ="122",
(20,20) *+{2\x_1+2\x_2} ="1122",
(30,20) *+{\c} ="c",
(30,30) *+{\c+\x_2} ="c1",
(30,40) *+{\c+2\x_2} ="c11",
(40,20) *+{\c+\x_1} ="c2",
(40,30) *+{\c+\x_1+\x_2} ="c12",
(40,40) *+{\c+\x_1+2\x_2} ="c112",
"0", {\ar"1"},
"1", {\ar"11"},
"11", {\ar"111"},
"2", {\ar"12"},
"12", {\ar"112"},
"112", {\ar"1112"},
"0", {\ar"2"},
"1", {\ar"12"},
"11", {\ar"112"},
"111", {\ar"1112"},
"22", {\ar"122"},
"122", {\ar"1122"},
"2", {\ar"22"},
"12", {\ar"122"},
"112", {\ar"1122"},
"111", {\ar"c"},
"1112", {\ar"c2"},
"22", {\ar"c"},
"122", {\ar"c1"},
"1122", {\ar"c11"},
"c", {\ar"c1"},
"c1", {\ar"c11"},
"c2", {\ar"c12"},
"c12", {\ar"c112"},
"c", {\ar"c2"},
"c1", {\ar"c12"},
"c11", {\ar"c112"},
\end{xy}
\]
\end{itemize}
\end{example}

\begin{example}
We consider the case $d=0$.
\begin{itemize}
\item[(a)] Let $n =2$ and $p_1$, $p_2\ge2$ arbitrary. Then
$A^{\rm CM} \simeq k\A_{p_1-1}\otimes_kk\A_{p_2-1}$. For instance, if $p_1=p_2=4$, then the quiver of $A^{\rm CM}$ is the following:
\[
\begin{xy} 0;<6pt,0pt>:<0pt,3pt>:: 
(0,0) *+{0} ="0",
(0,10) *+{\x_2} ="2",
(0,20) *+{2\x_2} ="22",
(10,0) *+{\x_1} ="1",
(10,10) *+{\x_1+\x_2} ="12",
(10,20) *+{\x_1+2\x_2} ="122",
(20,0) *+{2\x_1} ="11",
(20,10) *+{2\x_1+\x_2} ="112",
(20,20) *+{2\x_1+2\x_2} ="1122",
"0", {\ar"2"},
"1", {\ar"12"},
"11", {\ar"112"},
"2", {\ar"22"},
"12", {\ar"122"},
"112", {\ar"1122"},
"0", {\ar"1"},
"2", {\ar"12"},
"22", {\ar"122"},
"1", {\ar"11"},
"12", {\ar"112"},
"122", {\ar"1122"},
\end{xy}
\]
\item[(b)] Let $n =3$. Then $\de=\c+\sum_{i=1}^3(p_i-2)\x_i$. 
For $(p_1,p_2,p_3) = (2,2,4)$, the quiver of $A^{\rm CM}$ is the following:
\[
\begin{xy} 0;<6pt,0pt>:<0pt,1.8pt>:: 
(0,0) *+{0} ="0",
(20,10) *+{\x_3} ="3",
(40,20) *+{2\x_3} ="33",
(20,-10) *+{3\x_3} ="333",
(-7,-20) *+{\x_1} ="1",
(13,-10) *+{\x_1+\x_3} ="13",
(33,0) *+{\x_1+2\x_3} ="133",
(7,-20) *+{\x_2} ="2",
(27,-10) *+{\x_2+\x_3} ="23",
(47,0) *+{\x_2+2\x_3} ="233",
(0,-40) *+{\c} ="c",
(20,-30) *+{\c+\x_3} ="c3",
(40,-20) *+{\c+2\x_3} ="c33",
"0", {\ar"3"},
"3", {\ar"33"},
"33", {\ar"333"},
"333", {\ar"c"},
"0", {\ar"1"},
"3", {\ar"13"},
"33", {\ar"133"},
"0", {\ar"2"},
"3", {\ar"23"},
"33", {\ar"233"},
"1", {\ar"13"},
"13", {\ar"133"},
"1", {\ar"c"},
"13", {\ar"c3"},
"133", {\ar"c33"},
"2", {\ar"23"},
"23", {\ar"233"},
"2", {\ar"c"},
"23", {\ar"c3"},
"233", {\ar"c33"},
"c", {\ar"c3"},
"c3", {\ar"c33"},
\end{xy}
\]
\end{itemize}
\end{example}

\begin{example}
We consider the case $d=1$.
\begin{itemize}
\item[(a)] Let $n =3$ and $p_1$, $p_2$, $p_3\ge2$ arbitrary. Then
$A^{\rm CM} \simeq \bigotimes_{i=1}^3k\A_{p_i-1}$.
For instance, if $(p_1,p_2,p_3) = (3,3,3)$, then the quiver of $A^{\rm CM}$ is the following:
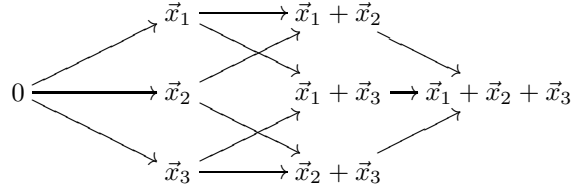
\[
\begin{xy} 0;<3pt,0pt>:<0pt,3pt>:: 
(-30,0) *+{0} ="0",
(-10,10) *+{\x_1} ="1",
(-10,0) *+{\x_2} ="2",
(-10,-10) *+{\x_3} ="3",
(10,10) *+{\x_1+\x_2} ="12",
(10,0) *+{\x_1+\x_3} ="13",
(10,-10) *+{\x_2+\x_3} ="23",
(30,0) *+{\x_1+\x_2+\x_3} ="123",
"0", {\ar"1"},
"0", {\ar"2"},
"0", {\ar"3"},
"1", {\ar"12"},
"1", {\ar"13"},
"2", {\ar"12"},
"2", {\ar"23"},
"3", {\ar"13"},
"3", {\ar"23"},
"12", {\ar"123"},
"13", {\ar"123"},
"23", {\ar"123"},
\end{xy}
\]
In fact this is derived equivalent to the canonical algebra of tubular type $(3,3,3)$ (see Lemma~\ref{tensor product}).
\item[(b)] Let $n =4$. Then $\de=\c+\sum_{i=1}^4(p_i-2)\x_i$.
For $(p_1,p_2,p_3,p_4) = (2,2,2,3)$, the quiver of $A^{\rm CM}$ is the following:
\[
\begin{xy} 0;<6pt,0pt>:<0pt,6pt>:: 
(-20,0) *+{0} ="0",
(-10,0) *+{\x_4} ="4",
(-5,15) *+{\x_1} ="1",
(-5,10) *+{\x_2} ="2",
(-5,5) *+{\x_3} ="3",
(0,0) *+{2\x_4} ="44",
(5,15) *+{\x_1+\x_4} ="14",
(5,10) *+{\x_2+\x_4} ="24",
(5,5) *+{\x_3+\x_4} ="34",
(10,0) *+{\c} ="c",
(20,0) *+{\c+\x_4} ="c4",
"0", {\ar"1"},
"0", {\ar"2"},
"0", {\ar"3"},
"0", {\ar"4"},
"1", {\ar"14"},
"1", {\ar"c"},
"2", {\ar"24"},
"2", {\ar"c"},
"3", {\ar"34"},
"3", {\ar"c"},
"4", {\ar"14"},
"4", {\ar"24"},
"4", {\ar"34"},
"4", {\ar"44"},
"44", {\ar"c"},
"c", {\ar"c4"},
"14", {\ar"c4"},
"24", {\ar"c4"},
"34", {\ar"c4"},
\end{xy}
\]
\end{itemize}
\end{example}

\begin{example}
We consider the case $d=2$.
\begin{itemize}
\item[(a)] Let $n =4$ and $p_1$, $p_2$, $p_3$, $p_4\ge2$ arbitrary. Then
$A^{\rm CM} \simeq \bigotimes_{i=1}^4k\A_{p_i-1}$. For instance, if $(p_1,p_2,p_3,p_4) = (3,3,3,3)$, then the quiver of $A^{\rm CM}$ is the following:
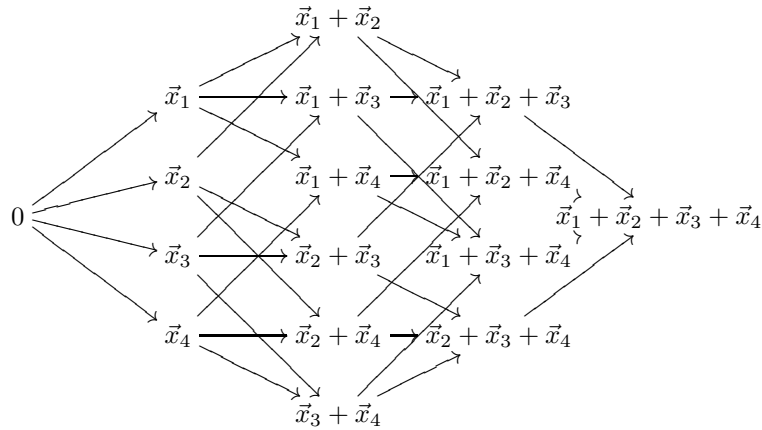
\[
\begin{xy} 0;<3pt,0pt>:<0pt,3pt>:: 
(-30,-5) *+{0} ="0",
(-10,10) *+{\x_1} ="1",
(-10,0) *+{\x_2} ="2",
(-10,-10) *+{\x_3} ="3",
(-10,-20) *+{\x_4} ="4",
(10,20) *+{\x_1+\x_2} ="12",
(10,10) *+{\x_1+\x_3} ="13",
(10,0) *+{\x_1+\x_4} ="14",
(10,-10) *+{\x_2+\x_3} ="23",
(10,-20) *+{\x_2+\x_4} ="24",
(10,-30) *+{\x_3+\x_4} ="34",
(30,10) *+{\x_1+\x_2+\x_3} ="123",
(30,0) *+{\x_1+\x_2+\x_4} ="124",
(30,-10) *+{\x_1+\x_3+\x_4} ="134",
(30,-20) *+{\x_2+\x_3+\x_4} ="234",
(55,-5) *+{\x_1+\x_2+\x_3+\x_4} ="1234",
"0", {\ar"1"},
"0", {\ar"2"},
"0", {\ar"3"},
"0", {\ar"4"},
"1", {\ar"12"},
"1", {\ar"13"},
"1", {\ar"14"},
"2", {\ar"12"},
"2", {\ar"23"},
"2", {\ar"24"},
"3", {\ar"13"},
"3", {\ar"23"},
"3", {\ar"34"},
"4", {\ar"14"},
"4", {\ar"24"},
"4", {\ar"34"},
"12", {\ar"123"},
"13", {\ar"123"},
"23", {\ar"123"},
"12", {\ar"124"},
"14", {\ar"124"},
"24", {\ar"124"},
"13", {\ar"134"},
"14", {\ar"134"},
"34", {\ar"134"},
"23", {\ar"234"},
"24", {\ar"234"},
"34", {\ar"234"},
"123", {\ar"1234"},
"124", {\ar"1234"},
"134", {\ar"1234"},
"234", {\ar"1234"},
\end{xy}
\]
\item[(b)] Let $n =5$. Then $\de=\c+\sum_{i=1}^5(p_i-2)\x_i$.
For $(p_1,p_2,p_3,p_4,p_5) = (2,2,2,2,3)$, the quiver of $A^{\rm CM}$ is the following:
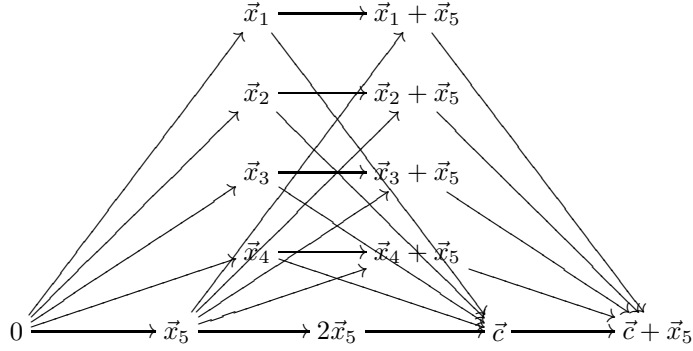
\[
\begin{xy} 0;<6pt,0pt>:<0pt,6pt>:: 
(-20,0) *+{0} ="0",
(-10,0) *+{\x_5} ="5",
(-5,20) *+{\x_1} ="1",
(-5,15) *+{\x_2} ="2",
(-5,10) *+{\x_3} ="3",
(-5,5) *+{\x_4} ="4",
(0,0) *+{2\x_5} ="55",
(5,20) *+{\x_1+\x_5} ="15",
(5,15) *+{\x_2+\x_5} ="25",
(5,10) *+{\x_3+\x_5} ="35",
(5,5) *+{\x_4+\x_5} ="45",
(10,0) *+{\c} ="c",
(20,0) *+{\c+\x_5} ="c5",
"0", {\ar"1"},
"0", {\ar"2"},
"0", {\ar"3"},
"0", {\ar"4"},
"0", {\ar"5"},
"1", {\ar"15"},
"1", {\ar"c"},
"2", {\ar"25"},
"2", {\ar"c"},
"3", {\ar"35"},
"3", {\ar"c"},
"4", {\ar"45"},
"4", {\ar"c"},
"5", {\ar"15"},
"5", {\ar"25"},
"5", {\ar"35"},
"5", {\ar"45"},
"5", {\ar"55"},
"55", {\ar"c"},
"c", {\ar"c5"},
"15", {\ar"c5"},
"25", {\ar"c5"},
"35", {\ar"c5"},
"45", {\ar"c5"},
\end{xy}
\]
\end{itemize}
\end{example}

The triangulated category $\underline{\CM}^{\L}R$ has a tilting object $T^{\rm CM}$ given in Theorem~\ref{CM tilting}.
In the rest of this section, we discuss other tilting objects in $\underline{\CM}^{\L}R$. In particular, $d$-tilting objects defined below play an important role in this paper.

First we prove the following result.

\begin{proposition}\label{property of stable tilting}
Let $U$ be a tilting object in $\underline{\CM}^{\L}R$,
and $\Lambda:=\underline{\End}^{\L}_R(U)$.
\begin{itemize}
\item[(a)] $\Lambda$ has finite global dimension.
\item[(b)] There exists a triangle equivalence $F:\DDD^{\bo}(\mod \Lambda)\simeq\underline{\CM}^{\L}R$.
\item[(c)] For any triangle equivalence $G:\DDD^{\bo}(\mod \Lambda)\simeq\underline{\CM}^{\L}R$, the diagram
\[\xymatrix{
\DDD^{\bo}(\mod \Lambda)\ar[rr]^G\ar[d]^{\nu_d}&&\underline{\CM}^{\L}R\ar[d]^{(\w)}\\
\DDD^{\bo}(\mod \Lambda)\ar[rr]^G&&\underline{\CM}^{\L}R
}\]
commutes up to natural isomorphism, where $\nu_d=(D\Lambda)[-d]\Lotimes_{\Lambda}-$ is the $d$-shifted Nakayama functor.
\end{itemize}
\end{proposition}

\begin{proof}
(b) By Proposition~\ref{tilting theorem}, we have a triangle equivalence
$F:\DDD^{\bo}(\mod\Lambda)\simeq\underline{\CM}^{\L}R$.

(a) By Theorem~\ref{CM tilting}(b), $\Lambda$ is derived equivalent to $A^{\rm CM}$, which has finite global dimension by Proposition~\ref{I-canonical has finite global dimension}(a).
Thus the assertion follows from Proposition~\ref{derived category and global dimension}(b).

(c) Both $\nu$ and $(\w)[d]$ are Serre functors of $\DDD^{\bo}(\mod \Lambda)$
and $\underline{\CM}^{\L}R$ respectively by Proposition~\ref{Nakayama Serre} and Theorem~\ref{Buchweitz Eisenbud}.
Since the Serre functor is unique up to natural isomorphism, the assertion follows.
\end{proof}

We prepare an observation which is useful to prove that a given object in $\underline{\CM}^{\L}R$ is tilting.

\begin{lemma}\label{End and stable End}
Let $U,V\in\CM^{\L}R$, $U\to I_U$ an injective hull of $U$ in $\CM^{\L}R$, and $P_V\to V$ a projective cover of $V$ in $\CM^{\L}R$.
\begin{itemize}
\item[(a)] The following conditions are equivalent.
\begin{itemize}
\item[(i)] $\underline{\Hom}^{\L}_R(U,V)=\Hom^{\L}_R(U,V)$.  
\item[(ii)] For any $\x\in\L$, either $\Hom^{\L}_R(U,R(\x))=0$ or $\Hom^{\L}_R(R(\x),V)=0$ holds.
\item[(iii)] $\Hom_R^{\L}(I_U,V)=0$.
\item[(iv)] $\Hom_R^{\L}(U,P_V)=0$.
\item[(v)] $\Hom^{\L}_R(I_U,P_V)=0$.
\end{itemize}
\item[(b)] If one of the conditions in (a) is satisfied, then $\Hom^{\L}_R(U,\Omega^iV)=0=\Hom_{\underline{\CM}^{\L}R}(U,V[-i])$ holds for any $i>0$.
\end{itemize}
\end{lemma}

\begin{proof}
(a) The implications (ii)$\Rightarrow$(iii)$+$(iv) and (v)$\Rightarrow$(i) are immediate.

(i)$\Rightarrow$(ii) Let $f\colon U\to R(\x)$ and $g\colon R(x)\to V$ be non-zero morphisms in
$\mod^{\L}R$. By Proposition~\ref{non-zero is injective}(c), the composition $fg\colon U\to V$
is non-zero, a contradiction.

(iii)$\Rightarrow$(v) Assume $\Hom^{\L}_R(I_U,P_V)\neq0$. Take a non-zero morphism $f\colon I_U\to P$ in $\mod^{\L}R$ for an indecomposable direct summand $P$ of $P_V$.
Since the composition $g\colon P\to P_V\to V$ of the inclusion $P\to P_V$ and the projective cover $P_V\to V$ is non-zero,
it is injective by Proposition~\ref{non-zero is injective}(c).
Thus the composition $fg\colon I_U\to V$ is non-zero, a contradiction.

(iv)$\Rightarrow$(v) Using the duality $\Hom_R(-,R)\colon\CM^{\L}R\simeq\CM^{\L}R$,
the assertion follows from (iii)$\Rightarrow$(v).

(b) Let $\cdots\to P_2\to P_1\to P_0\to V\to0$ be a minimal projective resolution of $V$ in $\mod^{\L}R$.
Applying Lemma~\ref{(Q,X) non zero}(a) to $X:=V$, for any indecomposable direct summand $Q$ of $P_i$ with $i\ge0$, there exists an injective morphism $Q\to P_0$ in $\mod^{\L}R$.
Since $\Hom^{\L}_R(U,P_0)=0$ holds by (a)(iv), $\Hom^{\L}_R(U,Q)=0$ holds.
Thus $\Hom^{\L}_R(U,\Omega^{i+1}V)\subset\Hom^{\L}_R(U,P_{i})=0$ holds for any $i\ge0$. In particular, $\underline{\Hom}^{\L}_R(U,V[-i-1])=0$ holds for any $i\ge0$.
\end{proof}

Let us introduce the following terminology.

\begin{definition}[$r$-tilting objects]
Let $r$ be an integer. We say that a tilting object $U$ in $\underline{\CM}^{\L}R$ is 
\emph{$r$-tilting} if $\underline{\End}^{\L}_R(U)$ has global dimension at most $r$.
\end{definition}

Throughout this paper, $d$-tilting objects play crucial roles. We pose the following question.

\begin{problem}\label{when d-tilting exists}
Let $d\ge0$. When does $\underline{\CM}^{\L}R$ have a $d$-tilting object?
\end{problem}

For the case $d=0,1$, we have the following complete answer.

\begin{example}\label{d-tilting for d=0,1}
Let $(R,\L)$ be a GL complete intersection of dimension $d+1$ with weights $p_1,\ldots,p_n$.
Assume that $p_i\ge2$ for all $i$.
\begin{itemize}
\item[(a)] Let $d=0$. Then $\underline{\CM}^{\L}R$ has a $0$-tilting object
if and only if either $n\le1$ or $n=2$ and the weights are $(2,2)$.
\item[(b)] Let $d=1$. Then $\underline{\CM}^{\L}R$ has a $1$-tilting object
if and only if $(R,\L)$ is domestic.
\end{itemize}
\end{example}

\begin{proof}
(a) Since semisimple algebras are closed under derived equivalences,
$\underline{\CM}^{\L}R$ has a $0$-tilting object if and only if $A^{\rm CM}$ is semisimple.
This is equivalent to $\de\notin\L_+$ (that is, $n\le1$) or $\de=0$ (that is, $n=2$ and the weights
are $(2,2)$).

(b) This is well-known. Also, the `only if' part follows from Theorem~\ref{d-tilting imply Fano} below,
and the `if' part follows from Proposition~\ref{d-tilting object example2} below.
\end{proof}

We prepare the following observation, which follows directly from the definition.

\begin{lemma}\label{graded and ungraded0}
Let $\a\in\L$ be an element which is not a torsion.
For any $X,Y\in\mod^{\L}R$ and $i\ge0$, we have
\[\Ext_{\mod^{\L/\Z\a}R}^i(X,Y)=\bigoplus_{\ell\in\Z}\Ext_{\mod^{\L}R}^i(X,Y(\ell\a)).\]
\end{lemma}

We give the following necessary condition, where we refer to Definition~\ref{define d-RI} for the notion of $\nu_d$-finiteness and Definition~\ref{define preprojective algebra} for $(d+1)$-preprojective algebras.

\begin{theorem}\label{d-tilting imply Fano}
Let $(R,\L)$ be a GL complete intersection with $d\ge1$.
Assume that $\underline{\CM}^{\L}R$ has a $d$-tilting object $U$.
Let $\Lambda=\underline{\End}_R^{\L}(U)$, and let $\Pi$ be
the $(d+1)$-preprojective algebra of $\Lambda$.
\begin{itemize}
\item[(a)] $(R,\L)$ is Fano.
\item[(b)] $\Lambda$ is $\nu_d$-finite, and there is an isomorphism $\Pi\simeq\underline{\End}_R^{\L/\Z\w}(U)$ of $\Z$-graded algebras.
\end{itemize}
\end{theorem}

\begin{proof}
Since $\Lambda$ has finite global dimension by Proposition~\ref{property of stable tilting},
we have $\KKK^{\bo}(\proj\Lambda)=\DDD^{\bo}(\mod \Lambda)$.

(a) Assume that $(R,\L)$ is Calabi-Yau. Then the triangulated category 
$\underline{\CM}^{\L}R$ is fractionally $\frac{dp}{p}$-Calabi-Yau by Corollary 
\ref{CM CY}. Hence $\Lambda$ has global dimension strictly bigger than $d$ by
Proposition~\ref{fractional CY and global dimension}, a contradiction.

Now we assume that $(R,\L)$ is anti-Fano.
Since $\Lambda$ has global dimension at most $d$, we have
$\nu_d^{-\ell}(\Lambda)\in\DDD^{\le0}(\mod \Lambda)$ for $\ell\ge0$
by Proposition~\ref{fractional CY and global dimension}. Hence
$\Hom_{\DDD^{\bo}(\mod \Lambda)}(\Lambda,\nu_d^{-\ell}(\Lambda)[i])=0$ holds for any $\ell\ge0$ and $i>0$.
On the other hand, for $\ell\gg0$, the element $-\ell\w\in\L$ becomes sufficiently small since $(R,\L)$ is anti-Fano. Therefore
\[\Hom_{\DDD^{\bo}(\mod \Lambda)}(\Lambda,\nu_d^{-\ell}(\Lambda)[-i])=\underline{\Hom}^{\L}_R(U,\Omega^iU(-\ell\w))=0\]
holds for any $i\ge0$ by Lemma~\ref{CM vanishing}. 
Consequently, we have $\nu_d^{-\ell}(\Lambda)=0$
for $\ell\gg0$. This is a contradiction since $\nu_d$ is an autoequivalence. 

Therefore $(R,\L)$ must be Fano.

(b) By Lemma \ref{graded and ungraded0}, we have an isomorphism
\begin{eqnarray*}
\Pi=\bigoplus_{\ell\in\Z}\Hom_{\DDD^{\bo}(\mod \Lambda)}(\Lambda,\nu_d^{-\ell}(\Lambda))\simeq\bigoplus_{\ell\in\Z}\underline{\Hom}^{\L}_R(U,U(-\ell\w))=\underline{\End}^{\L/\Z\w}_R(U).
\end{eqnarray*}
Since $\dim_k\underline{\End}_R(U)$ is finite by Proposition~\ref{CM is locally free}, so is $\dim_k\Pi=\dim_k\underline{\End}^{\L/\Z\w}_R(U)$.
Thus $\Lambda$ is $\nu_d$-finite.
\end{proof}

We give the following criterion for the tilting object $T^{\rm CM}$ given in Theorem~\ref{CM tilting} to be $d$-tilting.

\begin{proposition}\label{d-tilting object example1}
Let $(R,\L)$ be a GL complete intersection of dimension $d+1$ with weights $p_1,\ldots,p_n$.
Assume $p_i\ge2$ for all $i$.
\begin{itemize}
\item[(a)] The tilting object $T^{\rm CM}$ of $\underline{\CM}^{\L}R$ is $d$-tilting if and only if $\#\{i\ |\ p_i=2\}\ge 3(n-d)-4$. In this case, $n$ is at most $\frac{3d}{2}+2$.
\item[(b)] If $(R,\L)$ is a hypersurface, then $T^{\rm CM}$ is $d$-tilting if and only if at least two weights are two.
\end{itemize}
\end{proposition}

\begin{proof}
(a) By Theorem~\ref{gl.dim of A^CM}, $\gl A^{\rm CM}=2(n-d-2)+\#\{i\ |\ p_i\ge3\}$ holds.
Thus $T^{\rm CM}$ is $d$-tilting if and only if $d\ge 2(n-d-2)+n-\#\{i\ |\ p_i=2\}$ if and only if $\#\{i\ |\ p_i=2\}\ge 3(n-d)-4$.

(b) Immediate from (a).
\end{proof}

Let us describe $A^{\rm CM}$ in a special case of Proposition~\ref{d-tilting object example1}(a) explicitly.

\begin{example}
We consider the case $d\ge2$, $n=d+3$ and $p_i=2$ for all $1\le i\le n$. In this case, $T^{\rm CM}$ is a $d$-tilting object in $\underline{\CM}^{\L}R$, or more precisely, $\gl A^{\rm CM}=2$ holds.
Since $\de=\c$, the algebra $A^{\rm CM}$ is given by the following quiver with two relations.
\[\xymatrix@R=.5em@C=4em{
&\x_1\ar[rdd]|{X_1}\\
&\x_2\ar[rd]|{X_2}\\
0\ar[ruu]|{X_1}\ar[ru]|{X_2}\ar[rd]|{X_{n-1}}\ar[rdd]|{X_n}&\cdots&\c&X_{n-1}^2=\sum_{j=1}^{d+1}\lambda_{n-1,j-1}X_j^2\\
&\x_{n-1}\ar[ru]|{X_{n-1}}&&X_{n}^2=\sum_{j=1}^{d+1}\lambda_{n,j-1}X_j^2\\
&\x_n\ar[ruu]|{X_n}
}\]
\end{example}

Now let us recall the following well-known result.

\begin{lemma}\label{tensor product}
\begin{itemize}
\item[(a)] For $\ell=2,3,4$, the algebra $k\A_2\otimes_kk\A_\ell$ is derived equivalent to
$k\D_4$ if $\ell=2$, $k\E_6$ if $\ell=3$, and $k\E_8$ if $\ell=4$.
\item[(b)] $k\A_2\otimes_kk\A_5$ is derived equivalent to the tubular canonical algebra of type $(2,3,6)$.
\item[(c)] $k\A_3\otimes_kk\A_3$ is derived equivalent to the tubular canonical algebra of type $(2,4,4)$.
\item[(d)] $k\A_2\otimes_kk\A_2\otimes_kk\A_2$ is derived equivalent to the tubular canonical algebra of type $(3,3,3)$.
\item[(e)] For $2\le\ell\le m$, the algebra $k\A_\ell\otimes_kk\A_m$ is derived equivalent to a hereditary algebra
if and only if $(\ell,m)\in\{(2,2),(2,3),(2,4)\}$.
\end{itemize}
In particular, for the algebra $A$ in (b),(c) or (d), there are 
infinitely many indecomposable objects in $\DDD^{\bo}(\mod A)$ even up to suspension.
\end{lemma}

\begin{proof}
For (a), we refer to \cite[\S8]{KLM}.  For (b), (c) and (d), we refer to \cite[5.6]{KLM}.

(e) The `if' part follows from (a), and the `only if' part follows from the fact that $k\A_\ell\otimes_kk\A_m$ is
fractionally Calabi-Yau of dimension $\frac{\ell-1}{\ell+1}+\frac{m-1}{m+1}$ (Example~\ref{Dynkin is CY}), which must be smaller than $1$ by Proposition~\ref{fractional CY and global dimension}(c).
\end{proof}

As an application, we give more examples of GL hypersurfaces such that
$\underline{\CM}^{\L}R$ has $d$-tilting objects.

\begin{proposition}\label{d-tilting object example2}
If $n=d+2$ and one of the following conditions are satisfied, then $\underline{\CM}^{\L}R$ has a $d$-tilting object $U$.
\begin{itemize}
\item $d\geq 0$ and $(p_1,p_2)=(2,2)$.
\item $d \geq 1$ and $(p_1,p_2,p_3)=(2,3,3)$, $(2,3,4)$ or $(2,3,5)$.
\item $d \geq 2$ and $(p_1,p_2,p_3,p_4)=(3,3,p_3,p_4)$ with $p_3,p_4\in\{3,4,5\}$.
\end{itemize}
\end{proposition}

\begin{proof}
Since we have a triangle equivalence $\underline{\CM}^{\L}R\simeq
\DDD^{\bo}(\mod A^{\rm CM})$ by Theorem~\ref{CM tilting},
it suffices to show that $A^{\rm CM}$ is derived equivalent to an algebra of global dimension at most $d$.
We have $A^{\rm CM}\simeq\bigotimes_{i=1}^{n}k\A_{p_i-1}$ by Corollary~\ref{presentation of CM canonical}.

In the case $d\ge0$ and $(p_1,p_2)=(2,2)$, the algebra $A^{\rm CM}$ 
has global dimension at most $d$.

In the case $d\ge 1$ and $(p_1,p_2,p_3)=(2,3,3)$ (respectively, $(2,3,4)$, $(2,3,5)$),
the algebra $\bigotimes_{i=1}^3k\A_{p_i-1}$ is derived equivalent to a path
algebra $kQ$ of type $\D_4$ (respectively, $\E_6$, $\E_8$) by Lemma
\ref{tensor product}.
Hence $A^{\rm CM}$ is derived equivalent to
$kQ\otimes_k(\bigotimes_{i=4}^nk\A_{p_i-1})$ which has global dimension at most $d$.

In the case $d\ge2$ and $(p_1,p_2,p_3,p_4)=(3,3,p_3,p_4)$, the algebras
$k\A_2\otimes_kk\A_{p_3-1}$ and $k\A_2\otimes_kk\A_{p_4-1}$ are derived equivalent to
$k\D_4$, $k\E_6$ or $k\E_8$ by Lemma~\ref{tensor product}. Hence
$\bigotimes_{i=1}^4k\A_{p_i-1}$ is derived equivalent to an algebra $\Lambda$ with
global dimension $2$, and hence $A^{\rm CM}$ is derived equivalent to
$\Lambda\otimes_k(\bigotimes_{i=5}^nk\A_{p_i-1})$ which has global dimension at most $d$.
\end{proof}

To study Problem~\ref{when d-tilting exists} for the case $n=d+2$, the following observation is useful.

\begin{lemma}\label{type A consideration}
Let $n\ge3$ and $p_1,\ldots,p_n\ge 2$ be integers.
If $\bigotimes_{i=1}^nk\A_{p_i-1}$ is derived equivalent to an algebra $\Lambda$
with global dimension at most $n-2$, then $\sum_{i=1}^{n}\frac{1}{p_i}>1$.
\end{lemma}

\begin{proof}
By Example~\ref{Dynkin is CY}, the algebra $\bigotimes_{i=1}^nk\A_{p_i-1}$ is fractionally Calabi-Yau of dimension $\sum_{i=1}^n\frac{p_i-2}{p_i}=n-2\sum_{i=1}^n\frac{1}{p_i}$. 
Therefore by Proposition~\ref{fractional CY and global dimension}(c),
we have
\[n-2\sum_{i=1}^n\frac{1}{p_i}<\gl\Lambda\le n-2.\]
or $\Lambda$ is semisimple. In the first case we immediately have $\sum_{i=1}^{n}\frac{1}{p_i}>1$. In the second case it follows that
$p_1=\cdots=p_n=2$ and hence $\sum_{i=1}^{n}\frac{1}{p_i}=\frac{n}{2}>1$.
\end{proof}

Now we give an example which shows that $\underline{\CM}^{\L}R$ does not necessarily
have a $d$-tilting object even if $(R,\L)$ is Fano. It also shows that
the converse of Lemma~\ref{type A consideration} is not true.

\begin{example}\label{Fano no d-tilting}
Let $d=2$, $n=4$ and $(p_1,p_2,p_3,p_4)=(2,5,5,5)$ or $(2,3,9,9)$.
\begin{itemize}
\item[(a)] $(R,\L)$ is Fano.
\item[(b)] $\bigotimes_{i=1}^4\A_{p_i-1}$ is not derived equivalent to a finite dimensional algebra $A$ with $\gl A\le2$.
\item[(c)] $\underline{\CM}^{\L}R$ does not have a $2$-tilting object.
\end{itemize}
\end{example}

\begin{proof}
(a) is clear. It suffices to show (b) since (c) follows immediately from (b).
Since $A^{\rm CM}=\bigotimes_{i=1}^4\A_{p_i-1}$, it is fractionally Calabi-Yau of dimension $\frac{9}{5}$ (respectively, $\frac{17}{9}$).
Applying Proposition~\ref{fractional CY and global dimension}(c), we obtain the assertion since $2<\frac{9}{4}$ (respectively, $2<\frac{17}{8}$).
\end{proof}

The following analog of Example~\ref{d-tilting for d=0,1} will be used in
Section~\ref{Section: d-tilting bundle}.

\begin{proposition} \label{prop.base_cases}
Let $(R,\L)$ be a GL hypersurface associated with weights $p_1,\ldots,p_{d+2}$.
\begin{itemize}
\item[(a)] Let $d = -1$. Then $\underline{\CM}^{\L}R$ has a $0$-tilting object if and only if $p_1=2$.
\item[(b)] Let $d = 0$. Then $\underline{\CM}^{\L}R$ has a $1$-tilting object if and only if $(p_1, p_2)=(2,p_2)$, $(3,3)$, $(3,4)$, or $(3,5)$.
\end{itemize}
\end{proposition}

\begin{proof}
By Corollary~\ref{presentation of CM canonical}(b) we know that $\underline{\CM}^{\L}R$ is triangle equivalent to $\bigotimes_{i=1}^{d+2} k \mathbb{A}_{p_i - 1}$. Thus the assertion follows from Lemma~\ref{tensor product}(e).
\end{proof}

\section{Cohen-Macaulay finiteness and $d$-Cohen-Macaulay finiteness}\label{section: CMF and dCMF}

Let $(R,\L)$ be a Geigle-Lenzing complete intersection associated with 
linear forms $\ell_1,\ldots,\ell_n$ and weights $p_1,\ldots,p_n$.

\begin{definition}[Cohen-Macaulay finiteness]
We say that a GL complete intersection $(R,\L)$ is \emph{Cohen-Macaulay finite} 
(=\emph{CM finite}) if there are only finitely many isomorphism classes of 
indecomposable objects in $\CM^{\L}R$ up to degree shift.
\end{definition}

As an application of results in previous section, we have the following classification
of CM finite GL complete intersections.

\begin{theorem}\label{characterize RF}
Let $(R,\L)$ be a GL complete intersection of dimension $d+1$ with weights $p_1,\ldots,p_n$.
Assume that $p_i\ge2$ for all $i$.
Then $(R,\L)$ is CM finite if and only if one of the following conditions hold.
\begin{itemize}
\item $n\le d+1$.
\item $n=d+2$, and $(p_1,\ldots,p_n)=(2,\ldots,2,p_n)$, $(2,\ldots,2,3,3)$, $(2,\ldots,2,3,4)$
or $(2,\ldots,2,3,5)$ up to permutation.
\end{itemize}
\end{theorem}

\begin{proof}
If $n\le d+1$, then $(R,\L)$ is CM finite since $\CM^{\L}R=\proj^{\L}R$ holds.

Assume $n\ge d+3$. Then $X:=\Omega^{d+1}k$ is an indecomposable object in $\CM^{\L}R$ since it is $\Omega^{d+1}\rho(k)$ and $\rho(k)$ is indecomposable by Theorem~\ref{CM tilting}(a).
By Theorem~\ref{L-Poincare}(d), the $\Omega^\ell X$ with $\ell\ge0$ are pairwise non-isomorphic even up to degree shift. Thus $(R,\L)$ is not CM finite.

It remains to consider the case $n=d+2$.

\begin{lemma}\label{criterion for finiteness}
Assume $n=d+2$. Then $(R,\L)$ is CM finite if and only if
there are only finitely many isomorphism classes of indecomposable objects in 
$\DDD^{\bo}(\mod A^{\rm CM})$ up to suspension.
\end{lemma}

\begin{proof}
We have a triangle equivalence $\DDD^{\bo}(\mod A^{\rm CM})\simeq
\underline{\CM}^{\L}R$ by Theorem~\ref{CM tilting}. 
Moreover $[2]=(\c)$ holds by Theorem~\ref{Buchweitz Eisenbud}(d).
Thus the assertion holds since both $\Z/2\Z$ and $\L/\Z\c$ are finite.
\end{proof}

For the four cases with $n=d+2$ listed in Theorem~\ref{characterize RF}, it follows from
Lemma~\ref{tensor product}
that $A^{\rm CM}\simeq\bigotimes_{i=1}^nk\A_{p_i-1}$ is derived equivalent
to a path algebra of a Dynkin quiver. In particular there are only finitely many
isomorphism classes of indecomposable objects in $\DDD^{\bo}(\mod A^{\rm CM})$
up to suspension. Hence $(R,\L)$ is CM finite by Lemma~\ref{criterion for finiteness}.

Conversely, assume that the weights are not one of
$(2,\ldots,2,p_n)$, $(2,\ldots,2,3,3)$, $(2,\ldots,2,3,4)$ or $(2,\ldots,2,3,5)$.
It is easy to check that one of the following conditions hold up to permutation:
\begin{itemize}
\item $p_1\ge 3$ and $p_2\ge 6$.
\item $p_1\ge 4$ and $p_2\ge 4$.
\item $p_1\ge 3$, $p_2\ge 3$ and $p_3\ge 3$.
\end{itemize}
Thus there exists an idempotent $e$ in $A^{\rm CM}$
such that $eA^{\rm CM}e$ is isomorphic to $k\A_2\otimes_kk\A_5$,
$k\A_3\otimes_kk\A_3$ or $k\A_2\otimes_kk\A_2\otimes_kk\A_2$.
In each case, there are infinitely many isomorphism classes of 
indecomposable objects in $\DDD^{\bo}(\mod eA^{\rm CM}e)$ (and 
hence in $\DDD^{\bo}(\mod A^{\rm CM})$) even up to suspension by 
Lemma~\ref{tensor product}. By Lemma~\ref{criterion for finiteness}, we 
have that $(R,\L)$ is not CM finite.
\end{proof}

We apply Theorem~\ref{characterize RF} to the cases $d=-1,0$.

\begin{example}
Let $(R,\L)$ be a GL complete intersection of dimension $d+1$ with weights $p_1,\ldots,p_n$. 
Assume that $p_i\ge2$ for all $i$.
\begin{itemize}
\item[(a)] Let $d=-1$. Then $(R,\L)$ is CM finite if and only if $n\le 1$.
\item[(b)] Let $d=0$. Then $(R,\L)$ is CM finite if and only if $n\le1$ or $n=2$ and the weights are $(2,p_2)$, $(3,3)$, $(3,4)$ or $(3,5)$.
\end{itemize}
\end{example}

We give a description of the Auslander-Reiten quiver of $\CM^{\L}R$ when it is CM finite.

\begin{theorem}\label{AR quiver for CM finite}
Let $(R,\L)$ be a GL complete intersection in the list of Theorem~\ref{characterize RF}.
Then the Auslander-Reiten quiver ${\mathfrak A}(\CM^{\L}R)$ is given as follows.
\begin{itemize}
\item If $n\le d+1$, then the vertices are $R(\x)$ with $\x\in\L$, and the arrows are $R(\x)\to R(\x+\x_i)$ for all $x\in\L$ and $1\le i\le d+1$, where $\x_i:=\c$ for $n+1\le i\le d+1$.
\item If $n=d+2$, then ${\mathfrak A}(\CM^{\L}R)$ is given in Figure~\ref{CMfiniteFig}.
\end{itemize}
\end{theorem}

\begin{figure}
\[
\resizebox{\textwidth}{!}{
\begin{xy} 0;<14pt,0pt>:<0pt,14pt>::
(8,8) *+{\cdot} ="40",
(10,10) *+{\bullet} ="50",
(8,4) *+{} ="21",
(8.4,4.4) *+{\mdots} ="x",
(9.2,5.2) *+{\mdots} ="y",
(10,6) *+{} ="31",
(12,8) *+{\cdot} ="41",
(14,10) *+{\bullet} ="51",
(8,0) *+{\bullet} ="02",
(10,2) *+{\cdot} ="12",
(12,4) *+{} ="22",
(12.4,4.4) *+{\mdots} ="x",
(13.2,5.2) *+{\mdots} ="y",
(14,6) *+{} ="32",
(16,8) *+{\cdot} ="42",
(18,10) *+{\bullet} ="52",
(12,0) *+{\bullet} ="03",
(14,2) *+{\cdot} ="13",
(16,4) *+{} ="23",
(16.4,4.4) *+{\mdots} ="x",
(17.2,5.2) *+{\mdots} ="y",
(18,6) *+{} ="33",
(20,8) *+{\cdot} ="43",
(22,10) *+{\bullet} ="53",
(16,0) *+{\bullet} ="04",
(18,2) *+{\cdot} ="14",
(20,4) *+{} ="24",
(20.4,4.4) *+{\mdots} ="x",
(21.2,5.2) *+{\mdots} ="y",
(22,6) *+{} ="34",
(24,8) *+{\cdot} ="44",
(26,10) *+{\bullet} ="54",
(20,0) *+{\bullet} ="05",
(22,2) *+{\cdot} ="15",
(24,4) *+{} ="25",
(24.4,4.4) *+{\mdots} ="x",
(25.2,5.2) *+{\mdots} ="y",
(26,6) *+{} ="35",
(28,8) *+{\cdot} ="45",
(30,10) *+{\bullet} ="55",
(24,0) *+{(0)} ="06",
(26,2) *+{1,p-1} ="16",
(28,4) *+{} ="26",
(28.4,4.4) *+{\mdots} ="x",
(29.2,5.2) *+{\mdots} ="y",
(30,6) *+{} ="36",
(32,8) *+{1,1} ="46",
(34,10) *+{(\x_{n-1})} ="56",
(28,0) *+{(\x_n)} ="07",
(30,2) *+{\cdot} ="17",
(32,4) *+{} ="27",
(32.4,4.4) *+{\mdots} ="x",
(33.2,5.2) *+{\mdots} ="y",
(34,6) *+{} ="37",
(36,8) *+{\cdot} ="47",
(38,10) *+{\bullet} ="57",
(32,0) *+{\bullet} ="08",
(34,2) *+{\cdot} ="18",
(36,4) *+{} ="28",
(36.4,4.4) *+{\mdots} ="x",
(37.2,5.2) *+{\mdots} ="y",
(38,6) *+{} ="38",
(40,8) *+{\cdot} ="48",
(42,10) *+{\bullet} ="58",
(36,0) *+{\bullet} ="09",
(38,2) *+{\cdot} ="19",
(40,4) *+{} ="29",
(40.4,4.4) *+{\mdots} ="x",
(41.2,5.2) *+{\mdots} ="y",
(42,6) *+{} ="39",
(44,8) *+{\cdot} ="49",
(46,10) *+{\bullet} ="59",
(40,0) *+{\bullet} ="010",
(42,2) *+{\cdot} ="110",
(44,4) *+{} ="210",
(44.4,4.4) *+{\mdots} ="x",
(45.2,5.2) *+{\mdots} ="y",
(46,6) *+{} ="310",
(48,8) *+{\cdot} ="410",
(50,10) *+{\bullet} ="510",
(44,0) *+{\bullet} ="011",
(46,2) *+{\cdot} ="111",
(48,4) *+{} ="211",
(48.4,4.4) *+{\mdots} ="x",
(49.2,5.2) *+{\mdots} ="y",
(50,6) *+{} ="311",
(48,0) *+{\bullet} ="012",
(50,2) *+{\cdot} ="112",
"40", {\ar"50"},
"40", {\ar"31"},
"50", {\ar"41"},
"31", {\ar"41"},
"41", {\ar"51"},
"21", {\ar"12"},
"41", {\ar"32"},
"51", {\ar"42"},
"02", {\ar"12"},
"12", {\ar"22"},
"32", {\ar"42"},
"42", {\ar"52"},
"12", {\ar"03"},
"22", {\ar"13"},
"42", {\ar"33"},
"52", {\ar"43"},
"03", {\ar"13"},
"13", {\ar"23"},
"33", {\ar"43"},
"43", {\ar"53"},
"13", {\ar"04"},
"23", {\ar"14"},
"43", {\ar"34"},
"53", {\ar"44"},
"04", {\ar"14"},
"14", {\ar"24"},
"34", {\ar"44"},
"44", {\ar"54"},
"14", {\ar"05"},
"24", {\ar"15"},
"44", {\ar"35"},
"54", {\ar"45"},
"05", {\ar"15"},
"15", {\ar"25"},
"35", {\ar"45"},
"45", {\ar"55"},
"15", {\ar"06"},
"25", {\ar"16"},
"45", {\ar"36"},
"55", {\ar"46"},
"06", {\ar"16"},
"16", {\ar"26"},
"36", {\ar"46"},
"46", {\ar"56"},
"16", {\ar"07"},
"26", {\ar"17"},
"46", {\ar"37"},
"56", {\ar"47"},
"07", {\ar"17"},
"17", {\ar"27"},
"37", {\ar"47"},
"47", {\ar"57"},
"17", {\ar"08"},
"27", {\ar"18"},
"47", {\ar"38"},
"57", {\ar"48"},
"08", {\ar"18"},
"18", {\ar"28"},
"38", {\ar"48"},
"48", {\ar"58"},
"18", {\ar"09"},
"28", {\ar"19"},
"48", {\ar"39"},
"58", {\ar"49"},
"09", {\ar"19"},
"19", {\ar"29"},
"39", {\ar"49"},
"49", {\ar"59"},
"19", {\ar"010"},
"29", {\ar"110"},
"49", {\ar"310"},
"59", {\ar"410"},
"010", {\ar"110"},
"110", {\ar"210"},
"310", {\ar"410"},
"410", {\ar"510"},
"110", {\ar"011"},
"210", {\ar"111"},
"410", {\ar"311"},
"011", {\ar"111"},
"111", {\ar"211"},
"111", {\ar"012"},
"211", {\ar"112"},
"012", {\ar"112"},
\end{xy}
}
\]
{\centering $\CM^{\L}R$ for $(p_1,\ldots,p_{n-2},p_{n-1},p_{n}) = (2,\ldots,2,2,p)$

}

\[
\resizebox{\textwidth}{!}{
\begin{xy} 0;<16pt,0pt>:<0pt,16pt>::
(8,8) *+{\circ} ="40",
(8,6) *+{\circ} ="60",
(8,4) *+{\cdot} ="21",
(10,6) *+{\cdot} ="31",
(12,8) *+{\bullet} ="41",
(10,5) *+{\cdot} ="51",
(12,6) *+{\bullet} ="61",
(8,0) *+{\circ} ="02",
(10,2) *+{\cdot} ="12",
(12,4) *+{\cdot} ="22",
(14,6) *+{\cdot} ="32",
(16,8) *+{\circ} ="42",
(14,5) *+{\cdot} ="52",
(16,6) *+{\circ} ="62",
(12,0) *+{\bullet} ="03",
(14,2) *+{\cdot} ="13",
(16,4) *+{\cdot} ="23",
(18,6) *+{\cdot} ="33",
(20,8) *+{\bullet} ="43",
(18,5) *+{\cdot} ="53",
(20,6) *+{\bullet} ="63",
(16,0) *+{\circ} ="04",
(18,2) *+{\cdot} ="14",
(20,4) *+{\cdot} ="24",
(22,6) *+{\cdot} ="34",
(24,8) *+{\circ} ="44",
(22,5) *+{\cdot} ="54",
(24,6) *+{\circ} ="64",
(20,0) *+{(0)} ="05",
(22,2) *+{22} ="15",
(24,4) *+{\cdot} ="25",
(26,6) *+{21} ="35",
(28,8) *+{(\x_{n-1})} ="45",
(26,5) *+{12} ="55",
(28,6) *+{(\x_n)} ="65",
(24,0) *+{\circ} ="06",
(26,2) *+{\cdot} ="16",
(28,4) *+{\cdot} ="26",
(30,6) *+{\cdot} ="36",
(32,8) *+{\circ} ="46",
(30,5) *+{\cdot} ="56",
(32,6) *+{\circ} ="66",
(28,0) *+{\bullet} ="07",
(30,2) *+{11} ="17",
(32,4) *+{\cdot} ="27",
(34,6) *+{\cdot} ="37",
(36,8) *+{\bullet} ="47",
(34,5) *+{\cdot} ="57",
(36,6) *+{\bullet} ="67",
(32,0) *+{(\x_1)} ="08",
(34,2) *+{\cdot} ="18",
(36,4) *+{\cdot} ="28",
(38,6) *+{\cdot} ="38",
(40,8) *+{\circ} ="48",
(38,5) *+{\cdot} ="58",
(40,6) *+{\circ} ="68",
(36,0) *+{\bullet} ="09",
(38,2) *+{\cdot} ="19",
(40,4) *+{\cdot} ="29",
(42,6) *+{\cdot} ="39",
(44,8) *+{\bullet} ="49",
(42,5) *+{\cdot} ="59",
(44,6) *+{\bullet} ="69",
(40,0) *+{\circ} ="010",
(42,2) *+{\cdot} ="110",
(44,4) *+{\cdot} ="210",
(44,0) *+{\bullet} ="011",
"40", {\ar"31"},
"60", {\ar"51"},
"21", {\ar"31"},
"31", {\ar"41"},
"21", {\ar"51"},
"51", {\ar"61"},
"21", {\ar"12"},
"31", {\ar"22"},
"41", {\ar"32"},
"51", {\ar"22"},
"61", {\ar"52"},
"02", {\ar"12"},
"12", {\ar"22"},
"22", {\ar"32"},
"32", {\ar"42"},
"22", {\ar"52"},
"52", {\ar"62"},
"12", {\ar"03"},
"22", {\ar"13"},
"32", {\ar"23"},
"42", {\ar"33"},
"52", {\ar"23"},
"62", {\ar"53"},
"03", {\ar"13"},
"13", {\ar"23"},
"23", {\ar"33"},
"33", {\ar"43"},
"23", {\ar"53"},
"53", {\ar"63"},
"13", {\ar"04"},
"23", {\ar"14"},
"33", {\ar"24"},
"43", {\ar"34"},
"53", {\ar"24"},
"63", {\ar"54"},
"04", {\ar"14"},
"14", {\ar"24"},
"24", {\ar"34"},
"34", {\ar"44"},
"24", {\ar"54"},
"54", {\ar"64"},
"14", {\ar"05"},
"24", {\ar"15"},
"34", {\ar"25"},
"44", {\ar"35"},
"54", {\ar"25"},
"64", {\ar"55"},
"05", {\ar"15"},
"15", {\ar"25"},
"25", {\ar"35"},
"35", {\ar"45"},
"25", {\ar"55"},
"55", {\ar"65"},
"15", {\ar"06"},
"25", {\ar"16"},
"35", {\ar"26"},
"45", {\ar"36"},
"55", {\ar"26"},
"65", {\ar"56"},
"06", {\ar"16"},
"16", {\ar"26"},
"26", {\ar"36"},
"36", {\ar"46"},
"26", {\ar"56"},
"56", {\ar"66"},
"16", {\ar"07"},
"26", {\ar"17"},
"36", {\ar"27"},
"46", {\ar"37"},
"56", {\ar"27"},
"66", {\ar"57"},
"07", {\ar"17"},
"17", {\ar"27"},
"27", {\ar"37"},
"37", {\ar"47"},
"27", {\ar"57"},
"57", {\ar"67"},
"17", {\ar"08"},
"27", {\ar"18"},
"37", {\ar"28"},
"47", {\ar"38"},
"57", {\ar"28"},
"67", {\ar"58"},
"08", {\ar"18"},
"18", {\ar"28"},
"28", {\ar"38"},
"38", {\ar"48"},
"28", {\ar"58"},
"58", {\ar"68"},
"18", {\ar"09"},
"28", {\ar"19"},
"38", {\ar"29"},
"48", {\ar"39"},
"58", {\ar"29"},
"68", {\ar"59"},
"09", {\ar"19"},
"19", {\ar"29"},
"29", {\ar"39"},
"39", {\ar"49"},
"29", {\ar"59"},
"59", {\ar"69"},
"19", {\ar"010"},
"29", {\ar"110"},
"39", {\ar"210"},
"010", {\ar"110"},
"110", {\ar"210"},
"110", {\ar"011"},
\end{xy}
}
\]
{\centering $\CM^{\L}R$ for $(p_1,\ldots,p_{n-2},p_{n-1},p_{n}) = (2,\ldots,2,3,3)$

}

\[
\resizebox{\textwidth}{!}{
\begin{xy} 0;<18pt,0pt>:<0pt,18pt>::
(6,8) *+{\circ} ="9n",
(6,6) *+{\cdot} ="60",
(7,7) *+{\cdot} ="70",
(6,5) *+{\cdot} ="80",
(8,8) *+{\bullet} ="90",
(6,4) *+{\cdot} ="41",
(7,5) *+{\cdot} ="51",
(8,6) *+{\cdot} ="61",
(9,7) *+{\cdot} ="71",
(8,5) *+{\cdot} ="81",
(10,8) *+{\circ} ="91",
(6,2) *+{\bullet} ="22",
(7,3) *+{\cdot} ="32",
(8,4) *+{\cdot} ="42",
(9,5) *+{\cdot} ="52",
(10,6) *+{\cdot} ="62",
(11,7) *+{\cdot} ="72",
(10,5) *+{\cdot} ="82",
(12,8) *+{\bullet} ="92",
(8,2) *+{\circ} ="23",
(9,3) *+{\cdot} ="33",
(10,4) *+{\cdot} ="43",
(11,5) *+{\cdot} ="53",
(12,6) *+{\cdot} ="63",
(13,7) *+{\cdot} ="73",
(12,5) *+{\cdot} ="83",
(14,8) *+{\circ} ="93",
(10,2) *+{\bullet} ="24",
(11,3) *+{\cdot} ="34",
(12,4) *+{\cdot} ="44",
(13,5) *+{\cdot} ="54",
(14,6) *+{\cdot} ="64",
(15,7) *+{\cdot} ="74",
(14,5) *+{\cdot} ="84",
(16,8) *+{\bullet} ="94",
(12,2) *+{\circ} ="25",
(13,3) *+{\cdot} ="35",
(14,4) *+{\cdot} ="45",
(15,5) *+{\cdot} ="55",
(16,6) *+{\cdot} ="65",
(17,7) *+{\cdot} ="75",
(16,5) *+{\cdot} ="85",
(18,8) *+{\circ} ="95",
(14,2) *+{(0)} ="26",
(15,3) *+{23} ="36",
(16,4) *+{\cdot} ="46",
(17,5) *+{\cdot} ="56",
(18,6) *+{\cdot} ="66",
(19,7) *+{13} ="76",
(18,5) *+{22} ="86",
(20,8) *+{(\x_{n})} ="96",
(16,2) *+{\circ} ="27",
(17,3) *+{\cdot} ="37",
(18,4) *+{\cdot} ="47",
(19,5) *+{\cdot} ="57",
(20,6) *+{\cdot} ="67",
(21,7) *+{\cdot} ="77",
(20,5) *+{\cdot} ="87",
(22,8) *+{\circ} ="97",
(18,2) *+{\bullet} ="28",
(19,3) *+{\cdot} ="38",
(20,4) *+{\cdot} ="48",
(21,5) *+{\cdot} ="58",
(22,6) *+{\cdot} ="68",
(23,7) *+{\cdot} ="78",
(22,5) *+{12} ="88",
(24,8) *+{\bullet} ="98",
(20,2) *+{\circ} ="29",
(21,3) *+{21} ="39",
(22,4) *+{\cdot} ="49",
(23,5) *+{\cdot} ="59",
(24,6) *+{\cdot} ="69",
(25,7) *+{11} ="79",
(24,5) *+{\cdot} ="89",
(26,8) *+{(\x_1)} ="99",
(22,2) *+{(\x_{n-1})} ="210",
(23,3) *+{\cdot} ="310",
(24,4) *+{\cdot} ="410",
(25,5) *+{\cdot} ="510",
(26,6) *+{\cdot} ="610",
(27,7) *+{\cdot} ="710",
(26,5) *+{\cdot} ="810",
(28,8) *+{\bullet} ="910",
(24,2) *+{\circ} ="211",
(25,3) *+{\cdot} ="311",
(26,4) *+{\cdot} ="411",
(27,5) *+{\cdot} ="511",
(28,6) *+{\cdot} ="611",
(29,7) *+{\cdot} ="711",
(28,5) *+{\cdot} ="811",
(30,8) *+{\circ} ="911",
(26,2) *+{\bullet} ="212",
(27,3) *+{\cdot} ="312",
(28,4) *+{\cdot} ="412",
(29,5) *+{\cdot} ="512",
(30,6) *+{\cdot} ="612",
(31,7) *+{\cdot} ="712",
(30,5) *+{\cdot} ="812",
(32,8) *+{\bullet} ="912",
(28,2) *+{\circ} ="213",
(29,3) *+{\cdot} ="313",
(30,4) *+{\cdot} ="413",
(31,5) *+{\cdot} ="513",
(32,6) *+{\cdot} ="613",
(33,7) *+{\cdot} ="713",
(32,5) *+{\cdot} ="813",
(34,8) *+{\circ} ="913",
(30,2) *+{\bullet} ="214",
(31,3) *+{\cdot} ="314",
(32,4) *+{\cdot} ="414",
(33,5) *+{\cdot} ="514",
(34,6) *+{\cdot} ="614",
(35,7) *+{\cdot} ="714",
(34,5) *+{\cdot} ="814",
(36,8) *+{\bullet} ="914",
(32,2) *+{\circ} ="215",
(33,3) *+{\cdot} ="315",
(34,4) *+{\cdot} ="415",
(35,5) *+{\cdot} ="515",
(36,6) *+{\cdot} ="615",
(36,5) *+{\cdot} ="815",
(34,2) *+{\bullet} ="216",
(35,3) *+{\cdot} ="316",
(36,4) *+{\cdot} ="416",
(36,2) *+{\circ} ="217",
"9n", {\ar"70"},
"60", {\ar"70"},
"70", {\ar"90"},
"60", {\ar"51"},
"70", {\ar"61"},
"80", {\ar"51"},
"90", {\ar"71"},
"41", {\ar"51"},
"51", {\ar"61"},
"61", {\ar"71"},
"71", {\ar"91"},
"51", {\ar"81"},
"41", {\ar"32"},
"51", {\ar"42"},
"61", {\ar"52"},
"71", {\ar"62"},
"81", {\ar"52"},
"91", {\ar"72"},
"22", {\ar"32"},
"32", {\ar"42"},
"42", {\ar"52"},
"52", {\ar"62"},
"62", {\ar"72"},
"72", {\ar"92"},
"52", {\ar"82"},
"32", {\ar"23"},
"42", {\ar"33"},
"52", {\ar"43"},
"62", {\ar"53"},
"72", {\ar"63"},
"82", {\ar"53"},
"92", {\ar"73"},
"23", {\ar"33"},
"33", {\ar"43"},
"43", {\ar"53"},
"53", {\ar"63"},
"63", {\ar"73"},
"73", {\ar"93"},
"53", {\ar"83"},
"33", {\ar"24"},
"43", {\ar"34"},
"53", {\ar"44"},
"63", {\ar"54"},
"73", {\ar"64"},
"83", {\ar"54"},
"93", {\ar"74"},
"24", {\ar"34"},
"34", {\ar"44"},
"44", {\ar"54"},
"54", {\ar"64"},
"64", {\ar"74"},
"74", {\ar"94"},
"54", {\ar"84"},
"34", {\ar"25"},
"44", {\ar"35"},
"54", {\ar"45"},
"64", {\ar"55"},
"74", {\ar"65"},
"84", {\ar"55"},
"94", {\ar"75"},
"25", {\ar"35"},
"35", {\ar"45"},
"45", {\ar"55"},
"55", {\ar"65"},
"65", {\ar"75"},
"75", {\ar"95"},
"55", {\ar"85"},
"35", {\ar"26"},
"45", {\ar"36"},
"55", {\ar"46"},
"65", {\ar"56"},
"75", {\ar"66"},
"85", {\ar"56"},
"95", {\ar"76"},
"26", {\ar"36"},
"36", {\ar"46"},
"46", {\ar"56"},
"56", {\ar"66"},
"66", {\ar"76"},
"76", {\ar"96"},
"56", {\ar"86"},
"36", {\ar"27"},
"46", {\ar"37"},
"56", {\ar"47"},
"66", {\ar"57"},
"76", {\ar"67"},
"86", {\ar"57"},
"96", {\ar"77"},
"27", {\ar"37"},
"37", {\ar"47"},
"47", {\ar"57"},
"57", {\ar"67"},
"67", {\ar"77"},
"77", {\ar"97"},
"57", {\ar"87"},
"37", {\ar"28"},
"47", {\ar"38"},
"57", {\ar"48"},
"67", {\ar"58"},
"77", {\ar"68"},
"87", {\ar"58"},
"97", {\ar"78"},
"28", {\ar"38"},
"38", {\ar"48"},
"48", {\ar"58"},
"58", {\ar"68"},
"68", {\ar"78"},
"78", {\ar"98"},
"58", {\ar"88"},
"38", {\ar"29"},
"48", {\ar"39"},
"58", {\ar"49"},
"68", {\ar"59"},
"78", {\ar"69"},
"88", {\ar"59"},
"98", {\ar"79"},
"29", {\ar"39"},
"39", {\ar"49"},
"49", {\ar"59"},
"59", {\ar"69"},
"69", {\ar"79"},
"79", {\ar"99"},
"59", {\ar"89"},
"39", {\ar"210"},
"49", {\ar"310"},
"59", {\ar"410"},
"69", {\ar"510"},
"79", {\ar"610"},
"89", {\ar"510"},
"99", {\ar"710"},
"210", {\ar"310"},
"310", {\ar"410"},
"410", {\ar"510"},
"510", {\ar"610"},
"610", {\ar"710"},
"710", {\ar"910"},
"510", {\ar"810"},
"310", {\ar"211"},
"410", {\ar"311"},
"510", {\ar"411"},
"610", {\ar"511"},
"710", {\ar"611"},
"810", {\ar"511"},
"910", {\ar"711"},
"211", {\ar"311"},
"311", {\ar"411"},
"411", {\ar"511"},
"511", {\ar"611"},
"611", {\ar"711"},
"711", {\ar"911"},
"511", {\ar"811"},
"311", {\ar"212"},
"411", {\ar"312"},
"511", {\ar"412"},
"611", {\ar"512"},
"711", {\ar"612"},
"811", {\ar"512"},
"911", {\ar"712"},
"212", {\ar"312"},
"312", {\ar"412"},
"412", {\ar"512"},
"512", {\ar"612"},
"612", {\ar"712"},
"712", {\ar"912"},
"512", {\ar"812"},
"312", {\ar"213"},
"412", {\ar"313"},
"512", {\ar"413"},
"612", {\ar"513"},
"712", {\ar"613"},
"812", {\ar"513"},
"912", {\ar"713"},
"213", {\ar"313"},
"313", {\ar"413"},
"413", {\ar"513"},
"513", {\ar"613"},
"613", {\ar"713"},
"713", {\ar"913"},
"513", {\ar"813"},
"313", {\ar"214"},
"413", {\ar"314"},
"513", {\ar"414"},
"613", {\ar"514"},
"713", {\ar"614"},
"813", {\ar"514"},
"913", {\ar"714"},
"214", {\ar"314"},
"314", {\ar"414"},
"414", {\ar"514"},
"514", {\ar"614"},
"614", {\ar"714"},
"714", {\ar"914"},
"514", {\ar"814"},
"314", {\ar"215"},
"414", {\ar"315"},
"514", {\ar"415"},
"614", {\ar"515"},
"714", {\ar"615"},
"814", {\ar"515"},
"215", {\ar"315"},
"315", {\ar"415"},
"415", {\ar"515"},
"515", {\ar"615"},
"515", {\ar"815"},
"315", {\ar"216"},
"415", {\ar"316"},
"515", {\ar"416"},
"216", {\ar"316"},
"316", {\ar"416"},
"316", {\ar"217"},
\end{xy}
}
\]
{\centering $\CM^{\L}R$ for $(p_1,\ldots,p_{n-2},p_{n-1},p_{n}) = (2,\ldots,2,3,4)$

}

\[
\resizebox{\textwidth}{!}{
\begin{xy} 0;<18pt,0pt>:<0pt,18pt>::
(6,6) *+{\cdot} ="60",
(7,7) *+{\cdot} ="70",
(6,5) *+{\cdot} ="80",
(6,4) *+{\cdot} ="41",
(7,5) *+{\cdot} ="51",
(8,6) *+{\cdot} ="61",
(9,7) *+{\cdot} ="71",
(8,5) *+{\cdot} ="81",
(6,2) *+{\cdot} ="22",
(7,3) *+{\cdot} ="32",
(8,4) *+{\cdot} ="42",
(9,5) *+{\cdot} ="52",
(10,6) *+{\cdot} ="62",
(11,7) *+{\cdot} ="72",
(10,5) *+{\cdot} ="82",
(6,0) *+{(0)} ="03",
(7,1) *+{{24}} ="13",
(8,2) *+{\cdot} ="23",
(9,3) *+{\cdot} ="33",
(10,4) *+{\cdot} ="43",
(11,5) *+{\cdot} ="53",
(12,6) *+{\cdot} ="63",
(13,7) *+{23} ="73",
(12,5) *+{\cdot} ="83",
(8,0) *+{\circ} ="04",
(9,1) *+{\cdot} ="14",
(10,2) *+{\cdot} ="24",
(11,3) *+{\cdot} ="34",
(12,4) *+{\cdot} ="44",
(13,5) *+{\cdot} ="54",
(14,6) *+{\cdot} ="64",
(15,7) *+{\cdot} ="74",
(14,5) *+{\cdot} ="84",
(10,0) *+{\bullet} ="05",
(11,1) *+{\cdot} ="15",
(12,2) *+{\cdot} ="25",
(13,3) *+{\cdot} ="35",
(14,4) *+{\cdot} ="45",
(15,5) *+{\cdot} ="55",
(16,6) *+{\cdot} ="65",
(17,7) *+{\cdot} ="75",
(16,5) *+{\cdot} ="85",
(12,0) *+{\circ} ="06",
(13,1) *+{\cdot} ="16",
(14,2) *+{\cdot} ="26",
(15,3) *+{\cdot} ="36",
(16,4) *+{\cdot} ="46",
(17,5) *+{\cdot} ="56",
(18,6) *+{\cdot} ="66",
(19,7) *+{22} ="76",
(18,5) *+{\cdot} ="86",
(14,0) *+{\bullet} ="07",
(15,1) *+{\cdot} ="17",
(16,2) *+{\cdot} ="27",
(17,3) *+{\cdot} ="37",
(18,4) *+{\cdot} ="47",
(19,5) *+{\cdot} ="57",
(20,6) *+{\cdot} ="67",
(21,7) *+{\cdot} ="77",
(20,5) *+{\cdot} ="87",
(16,0) *+{\circ} ="08",
(17,1) *+{14} ="18",
(18,2) *+{\cdot} ="28",
(19,3) *+{\cdot} ="38",
(20,4) *+{\cdot} ="48",
(21,5) *+{\cdot} ="58",
(22,6) *+{\cdot} ="68",
(23,7) *+{13} ="78",
(22,5) *+{\cdot} ="88",
(18,0) *+{(\x_{n})} ="09",
(19,1) *+{\cdot} ="19",
(20,2) *+{\cdot} ="29",
(21,3) *+{\cdot} ="39",
(22,4) *+{\cdot} ="49",
(23,5) *+{\cdot} ="59",
(24,6) *+{\cdot} ="69",
(25,7) *+{\cdot} ="79",
(24,5) *+{\cdot} ="89",
(20,0) *+{\circ} ="010",
(21,1) *+{\cdot} ="110",
(22,2) *+{\cdot} ="210",
(23,3) *+{\cdot} ="310",
(24,4) *+{\cdot} ="410",
(25,5) *+{\cdot} ="510",
(26,6) *+{\cdot} ="610",
(27,7) *+{\cdot} ="710",
(26,5) *+{\cdot} ="810",
(22,0) *+{\bullet} ="011",
(23,1) *+{\cdot} ="111",
(24,2) *+{\cdot} ="211",
(25,3) *+{\cdot} ="311",
(26,4) *+{\cdot} ="411",
(27,5) *+{\cdot} ="511",
(28,6) *+{\cdot} ="611",
(29,7) *+{12} ="711",
(28,5) *+{\cdot} ="811",
(24,0) *+{\circ} ="012",
(25,1) *+{21} ="112",
(26,2) *+{\cdot} ="212",
(27,3) *+{\cdot} ="312",
(28,4) *+{\cdot} ="412",
(29,5) *+{\cdot} ="512",
(30,6) *+{\cdot} ="612",
(31,7) *+{\cdot} ="712",
(30,5) *+{\cdot} ="812",
(26,0) *+{(\x_{n-1})} ="013",
(27,1) *+{\cdot} ="113",
(28,2) *+{\cdot} ="213",
(29,3) *+{\cdot} ="313",
(30,4) *+{\cdot} ="413",
(31,5) *+{\cdot} ="513",
(32,6) *+{\cdot} ="613",
(33,7) *+{\cdot} ="713",
(32,5) *+{\cdot} ="813",
(28,0) *+{\circ} ="014",
(29,1) *+{\cdot} ="114",
(30,2) *+{\cdot} ="214",
(31,3) *+{\cdot} ="314",
(32,4) *+{\cdot} ="414",
(33,5) *+{\cdot} ="514",
(34,6) *+{\cdot} ="614",
(35,7) *+{\cdot} ="714",
(34,5) *+{\cdot} ="814",
(30,0) *+{\bullet} ="015",
(31,1) *+{\cdot} ="115",
(32,2) *+{\cdot} ="215",
(33,3) *+{\cdot} ="315",
(34,4) *+{\cdot} ="415",
(35,5) *+{\cdot} ="515",
(36,6) *+{\cdot} ="615",
(36,5) *+{\cdot} ="815",
(32,0) *+{\circ} ="016",
(33,1) *+{\cdot} ="116",
(34,2) *+{\cdot} ="216",
(35,3) *+{\cdot} ="316",
(36,4) *+{\cdot} ="416",
(34,0) *+{\bullet} ="017",
(35,1) *+{11} ="117",
(36,2) *+{\cdot} ="217",
(36,0) *+{(\x_1)} ="018",
"60", {\ar"70"},
"60", {\ar"51"},
"70", {\ar"61"},
"80", {\ar"51"},
"41", {\ar"51"},
"51", {\ar"61"},
"61", {\ar"71"},
"51", {\ar"81"},
"41", {\ar"32"},
"51", {\ar"42"},
"61", {\ar"52"},
"71", {\ar"62"},
"81", {\ar"52"},
"22", {\ar"32"},
"32", {\ar"42"},
"42", {\ar"52"},
"52", {\ar"62"},
"62", {\ar"72"},
"52", {\ar"82"},
"22", {\ar"13"},
"32", {\ar"23"},
"42", {\ar"33"},
"52", {\ar"43"},
"62", {\ar"53"},
"72", {\ar"63"},
"82", {\ar"53"},
"03", {\ar"13"},
"13", {\ar"23"},
"23", {\ar"33"},
"33", {\ar"43"},
"43", {\ar"53"},
"53", {\ar"63"},
"63", {\ar"73"},
"53", {\ar"83"},
"13", {\ar"04"},
"23", {\ar"14"},
"33", {\ar"24"},
"43", {\ar"34"},
"53", {\ar"44"},
"63", {\ar"54"},
"73", {\ar"64"},
"83", {\ar"54"},
"04", {\ar"14"},
"14", {\ar"24"},
"24", {\ar"34"},
"34", {\ar"44"},
"44", {\ar"54"},
"54", {\ar"64"},
"64", {\ar"74"},
"54", {\ar"84"},
"14", {\ar"05"},
"24", {\ar"15"},
"34", {\ar"25"},
"44", {\ar"35"},
"54", {\ar"45"},
"64", {\ar"55"},
"74", {\ar"65"},
"84", {\ar"55"},
"05", {\ar"15"},
"15", {\ar"25"},
"25", {\ar"35"},
"35", {\ar"45"},
"45", {\ar"55"},
"55", {\ar"65"},
"65", {\ar"75"},
"55", {\ar"85"},
"15", {\ar"06"},
"25", {\ar"16"},
"35", {\ar"26"},
"45", {\ar"36"},
"55", {\ar"46"},
"65", {\ar"56"},
"75", {\ar"66"},
"85", {\ar"56"},
"06", {\ar"16"},
"16", {\ar"26"},
"26", {\ar"36"},
"36", {\ar"46"},
"46", {\ar"56"},
"56", {\ar"66"},
"66", {\ar"76"},
"56", {\ar"86"},
"16", {\ar"07"},
"26", {\ar"17"},
"36", {\ar"27"},
"46", {\ar"37"},
"56", {\ar"47"},
"66", {\ar"57"},
"76", {\ar"67"},
"86", {\ar"57"},
"07", {\ar"17"},
"17", {\ar"27"},
"27", {\ar"37"},
"37", {\ar"47"},
"47", {\ar"57"},
"57", {\ar"67"},
"67", {\ar"77"},
"57", {\ar"87"},
"17", {\ar"08"},
"27", {\ar"18"},
"37", {\ar"28"},
"47", {\ar"38"},
"57", {\ar"48"},
"67", {\ar"58"},
"77", {\ar"68"},
"87", {\ar"58"},
"08", {\ar"18"},
"18", {\ar"28"},
"28", {\ar"38"},
"38", {\ar"48"},
"48", {\ar"58"},
"58", {\ar"68"},
"68", {\ar"78"},
"58", {\ar"88"},
"18", {\ar"09"},
"28", {\ar"19"},
"38", {\ar"29"},
"48", {\ar"39"},
"58", {\ar"49"},
"68", {\ar"59"},
"78", {\ar"69"},
"88", {\ar"59"},
"09", {\ar"19"},
"19", {\ar"29"},
"29", {\ar"39"},
"39", {\ar"49"},
"49", {\ar"59"},
"59", {\ar"69"},
"69", {\ar"79"},
"59", {\ar"89"},
"19", {\ar"010"},
"29", {\ar"110"},
"39", {\ar"210"},
"49", {\ar"310"},
"59", {\ar"410"},
"69", {\ar"510"},
"79", {\ar"610"},
"89", {\ar"510"},
"010", {\ar"110"},
"110", {\ar"210"},
"210", {\ar"310"},
"310", {\ar"410"},
"410", {\ar"510"},
"510", {\ar"610"},
"610", {\ar"710"},
"510", {\ar"810"},
"110", {\ar"011"},
"210", {\ar"111"},
"310", {\ar"211"},
"410", {\ar"311"},
"510", {\ar"411"},
"610", {\ar"511"},
"710", {\ar"611"},
"810", {\ar"511"},
"011", {\ar"111"},
"111", {\ar"211"},
"211", {\ar"311"},
"311", {\ar"411"},
"411", {\ar"511"},
"511", {\ar"611"},
"611", {\ar"711"},
"511", {\ar"811"},
"111", {\ar"012"},
"211", {\ar"112"},
"311", {\ar"212"},
"411", {\ar"312"},
"511", {\ar"412"},
"611", {\ar"512"},
"711", {\ar"612"},
"811", {\ar"512"},
"012", {\ar"112"},
"112", {\ar"212"},
"212", {\ar"312"},
"312", {\ar"412"},
"412", {\ar"512"},
"512", {\ar"612"},
"612", {\ar"712"},
"512", {\ar"812"},
"112", {\ar"013"},
"212", {\ar"113"},
"312", {\ar"213"},
"412", {\ar"313"},
"512", {\ar"413"},
"612", {\ar"513"},
"712", {\ar"613"},
"812", {\ar"513"},
"013", {\ar"113"},
"113", {\ar"213"},
"213", {\ar"313"},
"313", {\ar"413"},
"413", {\ar"513"},
"513", {\ar"613"},
"613", {\ar"713"},
"513", {\ar"813"},
"113", {\ar"014"},
"213", {\ar"114"},
"313", {\ar"214"},
"413", {\ar"314"},
"513", {\ar"414"},
"613", {\ar"514"},
"713", {\ar"614"},
"813", {\ar"514"},
"014", {\ar"114"},
"114", {\ar"214"},
"214", {\ar"314"},
"314", {\ar"414"},
"414", {\ar"514"},
"514", {\ar"614"},
"614", {\ar"714"},
"514", {\ar"814"},
"114", {\ar"015"},
"214", {\ar"115"},
"314", {\ar"215"},
"414", {\ar"315"},
"514", {\ar"415"},
"614", {\ar"515"},
"714", {\ar"615"},
"814", {\ar"515"},
"015", {\ar"115"},
"115", {\ar"215"},
"215", {\ar"315"},
"315", {\ar"415"},
"415", {\ar"515"},
"515", {\ar"615"},
"515", {\ar"815"},
"115", {\ar"016"},
"215", {\ar"116"},
"315", {\ar"216"},
"415", {\ar"316"},
"515", {\ar"416"},
"016", {\ar"116"},
"116", {\ar"216"},
"216", {\ar"316"},
"316", {\ar"416"},
"116", {\ar"017"},
"216", {\ar"117"},
"316", {\ar"217"},
"017", {\ar"117"},
"117", {\ar"217"},
"117", {\ar"018"},
\end{xy}
}
\]
{
\centering $\CM^{\L}R$ for $(p_1,\ldots,p_{n-2},p_{n-1},p_{n}) = (2,\ldots,2,3,5)$

}
\caption{
The Auslander-Reiten quivers of $\CM^{\L}R$ in the CM finite cases with $n=d+2$ are depicted as follows. The position of $R(\x)$ is denoted $(\x)$, $\bullet$, or $\circ$ in case $\x$ is not specified. (For $d=0$, $(\x_1)$ and $\circ$ should be deleted.) For $d\ge1$, these positions are overlapping so that $R(\x)$ and $R(\y)$ occupy the same position if and only if $\x-\y$
belongs to the $2$-subgroup $\langle \x_i - \x_1 \mid p_i = 2 \rangle$ of $\L$.
The interpretation is that every arrow starting or ending at one of these vertices represents $2^{d-1}$ arrows ($2^d$ arrows in the case $(2,\ldots,2,2,p)$) in the actual Auslander-Reiten quiver (each one connected to one of the $R(\x)$). The labels $ab$ denote the positions of
$\rho(E^{\s+(a-1)\x_{n-1}+(b-1)\x_{n}})$ given in Section~\ref{subsect tilting via tensor},
and their direct sum gives a tilting object $U^{\rm CM}:=T^{\rm CM}(\w)[d]$.
}
\label{CMfiniteFig}
\end{figure}

We give an explicit Auslander-Reiten quiver in a simple case.

\begin{example}\label{semisimpleAR}
Let $d=2$, $n=4$ and $p_1=p_2=p_3=p_4=2$. Then $\mathfrak{A}(\CM^{\L}R)$ is the following,
where $X=\rho(k)$.
\[{\scriptsize\xymatrix@R1em@C1em{
&\x_1\ar[dddr]&&\x_1+\x_2+\x_3+\x_4-\c\ar[dddr]&&\x_2+\x_3+\x_4\ar[dddr]&\\
&\x_2\ar[ddr]&&\x_1+\x_2\ar[ddr]&&\x_1+\x_3+\x_4\ar[ddr]&\\
&\x_3\ar[dr]&&\x_1+\x_3\ar[dr]&&\x_1+\x_2+\x_4\ar[dr]&\\
\cdots\ \ 
X\ar[uuur]\ar[uur]\ar[ur]\ar[r]\ar[dr]\ar[ddr]\ar[dddr]\ar[ddddr]&\x_4\ar[r]&
X(\x_1)\ar[uuur]\ar[uur]\ar[ur]\ar[r]\ar[dr]\ar[ddr]\ar[dddr]\ar[ddddr]&\x_1+\x_4\ar[r]&
X(-\w)\ar[uuur]\ar[uur]\ar[ur]\ar[r]\ar[dr]\ar[ddr]\ar[dddr]\ar[ddddr]&\x_1+\x_2+\x_3\ar[r]&
X(\x_1-\w) \ \ \cdots\\
&\x_2+\x_3+\x_4-\c\ar[ur]&&\x_2+\x_3\ar[ur]&&\x_4+\c\ar[ur]&\\
&\x_1+\x_3+\x_4-\c\ar[uur]&&\x_2+\x_4\ar[uur]&&\x_3+\c\ar[uur]&\\
&\x_1+\x_2+\x_4-\c\ar[uuur]&&\x_3+\x_4\ar[uuur]&&\x_2+\c\ar[uuur]&\\
&\x_1+\x_2+\x_3-\c\ar[uuuur]&&\c\ar[uuuur]&&\x_1+\c\ar[uuuur]&
}}\]
\end{example}

We pose the following natural problem.

\begin{problem}
Classify GL complete intersections $(R,\L)$ that are CM tame.
\end{problem}

In contrast to the CM finite case, there is a CM tame GL complete intersection which is not a hypersurface, e.g.\ $d=1$ with 4 weights $(2,2,2,2)$.

\medskip
Theorem~\ref{characterize RF} tells us that there are only very few CM finite GL complete
intersections.
In higher dimensional Auslander-Reiten theory, we introduce the notion of
`$d$-CM finiteness' as a proper substitute of CM finiteness.

\medskip\noindent
{\bf $d$-Cohen-Macaulay finiteness.}
Recall from Section~\ref{section: preliminaries 2} that a full subcategory 
$\CC$ of $\CM^{\L}R$ is called \emph{$d$-cluster tilting}
if it is a functorially finite subcategory of $\CM^{\L}R$ such that
\begin{eqnarray*}
\CC&=&\{X\in\CM^{\L}R\mid\forall i\in\{1,2,\ldots,d-1\}\ \Ext_{\mod^{\L}R}^i(\CC,X)=0\}\ \mbox{ and}\\
\CC&=&\{X\in\CM^{\L}R\mid\forall i\in\{1,2,\ldots,d-1\}\ \Ext_{\mod^{\L}R}^i(X,\CC)=0\}.
\end{eqnarray*}
Note that one of the equalities above implies the other \cite[2.2.2]{I1}.
In this case $\CC$ generates and cogenerates $\CM^{\L}R$ since it contains $\proj^{\L}R$. 

\begin{definition}
We say that a GL complete intersection $(R,\L)$ is \emph{$d$-Cohen-Macaulay finite}
(=\emph{$d$-CM finite})
if there exists a $d$-cluster tilting subcategory $\CC$ of $\CM^{\L}R$ such that
there are only finitely many isomorphism classes of indecomposable objects in $\CC$
up to degree shift.
\end{definition}

In the case $d=1$,  $d$-CM finiteness coincides with classical
CM finiteness since $\CM^{\L}R$ is the unique $1$-cluster tilting
subcategory of $\CM^{\L}R$.

Now we give some basic properties of $d$-cluster tilting subcategories, which will be used later.

\begin{proposition}\label{basic properties of d-CT for CM}
For a $d$-cluster tilting subcategory $\CC$ of $\CM^{\L}R$, the following assertions hold.
\begin{itemize}
\item[(a)] We have $\CC(\w)=\CC$.
\item[(b)] For any $X\in\CM^{\L}R$, there exist exact sequences
\[0\to C_{d-1}\to\cdots\to C_0\to X\to0\ \mbox{ and }\ 
0\to X\to C^0\to\cdots\to C^{d-1}\to0\]
in $\mod^{\L}R$ with $C_i,C^i\in\CC$ for any $0\le i\le d-1$.
\item[(c)] For any indecomposable object $X\in\CC$, there exists an exact sequence
\begin{equation}\label{almost split fundamental}
0\to X(\w)\to C_{d-1}\to\cdots\to C_1\to C_0\xrightarrow{f}X
\end{equation}
such that the following sequences of functors on $\CC$ are exact:
\begin{eqnarray*}
&0\to\Hom_{\CC}(-,X(\w))\to\Hom_{\CC}(-,C_{d-1})\to\cdots\to\Hom_{\CC}(-,C_0)\to\rad_{\CC}(-,X)\to0,&\\
&0\to\Hom_{\CC}(X,-)\to\Hom_{\CC}(C_0,-)\to\cdots\to\Hom_{\CC}(C_{d-1},-)\to\rad_{\CC}(X(\w),-)\to0.&
\end{eqnarray*}
If $X$ is non-projective, then $\Cokernel f=0$ and \eqref{almost split fundamental}
is called a \emph{$d$-almost split sequence}. If $X$ is projective, then 
$\Cokernel f={\rm top} X$ and \eqref{almost split fundamental}
is called a \emph{$d$-fundamental sequence}.
\end{itemize}
\end{proposition}

\begin{proof}
(a) This is immediate from Auslander-Reiten-Serre duality given in Theorem~\ref{AR duality}(a).

(b) is shown in \cite[Theorem 3.3.1]{I1}, and (c) is shown in \cite[Theorem 3.4.4]{I1}.
\end{proof}

We prepare the following easy observations.

\begin{lemma}\label{graded and ungraded}
Let $\a\in\L$ be an element which is not a torsion.
For any $M\in\mod^{\L}R$, the subcategory $\CC:=\add\{M(\ell\a)\mid\ell\in\Z\}$ (respectively, $\CC^+:=\add\{M(\ell\a)\mid\ell\ge0\}$, $\CC^-:=\add\{M(\ell\a)\mid\ell\le0\}$) is functorially finite in $\mod^{\L}R$.
\end{lemma}

\begin{proof}
We only show that $\CC$ is covariantly finite in $\mod^{\L}R$ since other assertions can
be shown similarly. The $\a$-Veronese subalgebra $R^{(\a)}=\bigoplus_{\ell\in\Z}R_{\ell\a}$
of $R$ is Noetherian. For any $X\in\mod^{\L}R$, 
$\bigoplus_{\ell\in\Z}\Hom^{\L}_R(X,M(\ell\a))=\Hom^{\L/\Z\a}_R(X,M)$
holds by Lemma \ref{graded and ungraded0}.
This is a finitely generated $R^{(\a)}$-module, and we take homogeneous generators $f_1,\ldots,f_m$ with $f_i\in\Hom^{\L}_R(X,M(\ell_i\a))$.
It is easy to check that $f:=(f_1,\ldots,f_m):X\to\bigoplus_{i=1}^mM(\ell_i\a)$ is a left $\CC$-approximation.
\end{proof}

When $(R,\L)$ is not Calabi-Yau, we have the following equivalent conditions.

\begin{lemma}
Assume that $(R,\L)$ is not Calabi-Yau. Then the following conditions are equivalent.
\begin{itemize}
\item[(a)] $(R,\L)$ is $d$-CM finite.
\item[(b)] There exists $M\in\CM^{\L}R$ satisfying
\begin{eqnarray*}
\add_{\CM^{\L}R}\{M(\ell\w)\mid\ell\in\Z\}&=&\{X\in\CM^{\L}R\mid\forall i\in\{1,2,\ldots,d-1\}\ \Ext_{\mod^{\L/\Z\w}R}^i(M,X)=0\}.\\
(\mbox{or } \add_{\CM^{\L}R}\{M(\ell\w)\mid\ell\in\Z\}&=&\{X\in\CM^{\L}R\mid\forall i\in\{1,2,\ldots,d-1\}\ \Ext_{\mod^{\L/\Z\w}R}^i(X,M)=0\}).
\end{eqnarray*}
\end{itemize}
\end{lemma}

\begin{proof}
(b)$\Rightarrow$(a) If $M\in\CM^{\L}R$ satisfies the equality above, then $\CC:=\add\{M(\ell\w)\mid\ell\in\Z\}$ is a functorially finite subcategory of $\CM^{\L}R$ by Lemma~\ref{graded and ungraded}.
Therefore $\CC$ is a $d$-cluster tilting subcategory of $\CM^{\L}R$ by Lemma~\ref{graded and ungraded0}.
Since $\CC$ clearly has only finitely many isomorphism classes of indecomposable objects up to degree shift, $(R,\L)$ is $d$-CM finite.

(a)$\Rightarrow$(b) Let $\CC$ be a $d$-cluster tilting subcategory $\CC$ of 
$\CM^{\L}R$ with only finitely many isomorphism classes of indecomposable 
objects up to degree shift. Since $\CC=\CC(\Z\w)$, there exists $M\in\CM^{\L}R$ such that $\CC=\add\{M(\ell\w)\mid\ell\in\Z\}$.
It follows from Lemma~\ref{graded and ungraded0} that $M$ satisfies the desired equality.
\end{proof}

The following main result in this section gives a sufficient condition for $d$-CM finiteness in terms of tilting theory in $\underline{\CM}^{\L}R$.

\begin{theorem}\label{construct dCT}
Let $(R,\L)$ be a GL complete intersection.
If $\underline{\CM}^{\L}R$ has a $d$-tilting object $U$,
then $(R,\L)$ is $d$-CM finite and $\CM^{\L}R$
has the $d$-cluster tilting subcategory
\[\UU:=\add\{U(\ell\w),\ R(\x)\mid\ell\in\Z,\ \x\in\L\}.\]
\end{theorem}

\begin{proof}
Let $\Lambda:=\underline{\End}_R^{\L}(U)$.
Since $\Lambda$ is $\nu_d$-finite by Theorem~\ref{d-tilting imply Fano},
we have a $d$-cluster tilting subcategory
\[\UU_\Lambda:=\add\{\nu_d^i(\Lambda)\mid i\in\Z\}\]
of $\DDD^{\bo}(\mod \Lambda)$ by Theorem~\ref{tau_d finite has d-CT}.
By Proposition~\ref{property of stable tilting}, we have that
\[F(\UU_\Lambda)=\add\{U(i\w)\mid i\in\Z\}\]
is a $d$-cluster tilting subcategory of $\underline{\CM}^{\L}R$.
Therefore $\add\{U(i\w),\ R(\x)\mid i\in\Z,\ \x\in\L\}$ is a $d$-cluster tilting
subcategory of $\CM^{\L}R$, and $(R,\L)$ is $d$-CM finite.
\end{proof}

\begin{example}\label{semisimpleAR2}
We continue to discuss Example~\ref{semisimpleAR}.
Since $\underline{\CM}^{\L} R$ is semisimple, it has the two 2-cluster tilting subcategories
\[ \{ X(j \w) \mid j \in \mathbb{Z} \} \text{ and } \{ X(\x_1 + j \w) \mid j \in \mathbb{Z} \}. \]
It follows that also $\CM^{\L}R$ has two 2-cluster tilting subcategories
\[ \UU_i = \add \{ X(i \x_1 + j \w),\ R(\x) \mid j \in \mathbb{Z}, \ \x \in \L \} \]
for $i=1,2$. The quiver of $\UU_1$ is the following:
\[{\scriptsize\xymatrix@R1em@C1em{
&\x_1\ar[dddr]&&\x_1+\x_2+\x_3+\x_4-\c\ar[rr]\ar[drr]\ar[ddrr]\ar[dddrr]&&\x_2+\x_3+\x_4\ar[dddr]&\\
&\x_2\ar[ddr]&&\x_1+\x_2\ar[drr]\ar[ddrr]\ar[dddddrr]\ar[ddddddrr]&&\x_1+\x_3+\x_4\ar[ddr]&\\
&\x_3\ar[dr]&&\x_1+\x_3\ar[urr]\ar[drr]\ar[dddrr]\ar[dddddrr]&&\x_1+\x_2+\x_4\ar[dr]&\\
\cdots&\x_4\ar[r]&X(\x_1)\ar[uuur]\ar[uur]\ar[ur]\ar[r]\ar[dr]\ar[ddr]\ar[dddr]\ar[ddddr]
&\x_1+\x_4\ar[uurr]\ar[urr]\ar[drr]\ar[ddddrr]&&\x_1+\x_2+\x_3\ar[r]&X(\x_1-\w)\ \ \cdots\\
&\x_2+\x_3+\x_4-\c\ar[ur]&&\x_2+\x_3\ar[uuuurr]\ar[urr]\ar[drr]\ar[ddrr]&&\x_4+\c\ar[ur]&\\
&\x_1+\x_3+\x_4-\c\ar[uur]&&\x_2+\x_4\ar[uuuuurr]\ar[uuurr]\ar[urr]\ar[drr]&&\x_3+\c\ar[uur]&\\
&\x_1+\x_2+\x_4-\c\ar[uuur]&&\x_3+\x_4\ar[uuuuuurr]\ar[uuuuurr]\ar[uurr]\ar[urr]&&\x_2+\c\ar[uuur]&\\
&\x_1+\x_2+\x_3-\c\ar[uuuur]&&\c\ar[uuurr]\ar[uurr]\ar[urr]\ar[rr]&&\x_1+\c\ar[uuuur]&
}}\]
\end{example}

\begin{example}\label{A2*A3 CT subcategory}
When $(R,\L)$ is Cohen-Macaulay finite, it is possible to describe the $d$-cluster tilting subcategory
\[ \UU= \add \{ U^{\rm CM}(j \w),\ R(\x) \mid j \in \mathbb{Z}, \ \x \in \L \} \]
of $\CM^{\L} R$. Indeed, Figure~\ref{CMfiniteFig} contains the Auslander-Reiten quiver of $\CM^{\L} R$ with the summands of $U^{\rm CM}$ marked. From this one can compute the Auslander-Reiten quiver of $\UU$. For example, in case $d = 2$ and $(p_1,p_2,p_3,p_4)=(2,2,3,4)$ we obtain\[
\resizebox{\textwidth}{!}
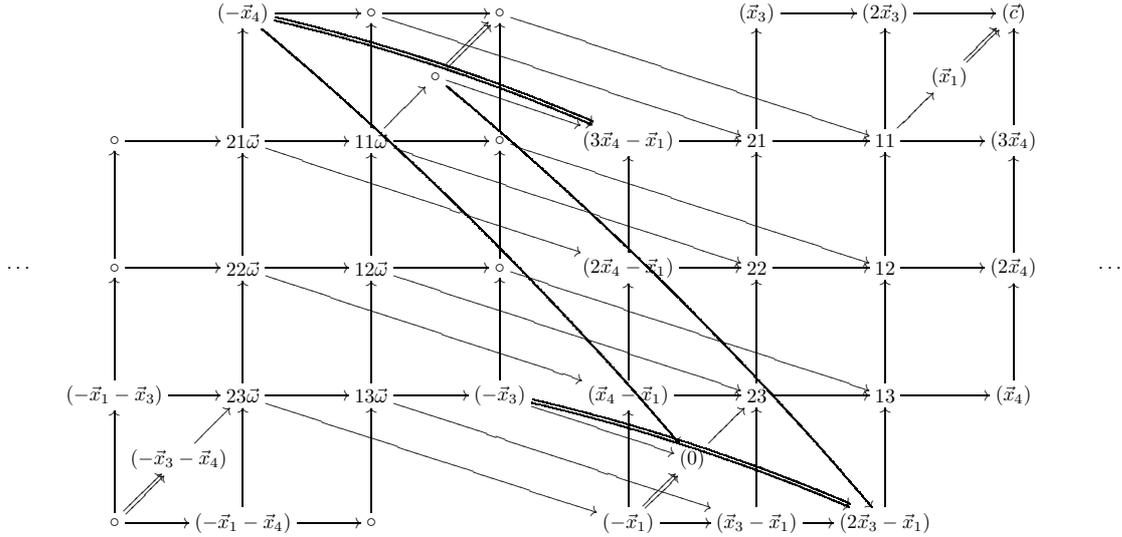
{
\begin{xy} 0;<0pt,32pt>:<32pt,0pt>:: 
(4,-1.5) *+{\cdots} ="x",
(4,15.5) *+{\cdots} ="x",
(0,0) *+{\circ} ="000",
(2,0) *+{(-\x_1-\x_3)} ="100",
(4,0) *+{\circ} ="200",
(6,0) *+{\circ} ="300",
(1,1) *+{(-\x_3-\x_4)} ="u00",
(0,2) *+{(-\x_1-\x_4)} ="010",
(2,2) *+{23\w} ="110",
(4,2) *+{22\w} ="210",
(6,2) *+{21\w} ="310",
(8,2) *+{(-\x_4)} ="410",
(0,4) *+{\circ} ="020",
(2,4) *+{13\w} ="120",
(4,4) *+{12\w} ="220",
(6,4) *+{11\w} ="320",
(8,4) *+{\circ} ="420",
(7,5) *+{\circ} ="l30",
(2,6) *+{(-\x_3)} ="130",
(4,6) *+{\circ} ="230",
(6,6) *+{\circ} ="330",
(8,6) *+{\circ} ="430",
(0,8) *+{(-\x_1)} ="001",
(2,8) *+{(\x_4-\x_1)} ="101",
(4,8) *+{(2\x_4-\x_1)} ="201",
(6,8) *+{(3\x_4-\x_1)} ="301",
(1,9) *+{(0)} ="u01",
(0,10) *+{(\x_3-\x_1)} ="011",
(2,10) *+{23} ="111",
(4,10) *+{22} ="211",
(6,10) *+{21} ="311",
(8,10) *+{(\x_3)} ="411",
(0,12) *+{(2\x_3-\x_1)} ="021",
(2,12) *+{13} ="121",
(4,12) *+{12} ="221",
(6,12) *+{11} ="321",
(8,12) *+{(2\x_3)} ="421",
(7,13) *+{(\x_1)} ="l31",
(2,14) *+{(\x_4)} ="131",
(4,14) *+{(2\x_4)} ="231",
(6,14) *+{(3\x_4)} ="331",
(8,14) *+{(\c)} ="431",
"000", {\ar"100"},
"100", {\ar"200"},
"200", {\ar"300"},
"010", {\ar"110"},
"110", {\ar"210"},
"210", {\ar"310"},
"310", {\ar"410"},
"020", {\ar"120"},
"120", {\ar"220"},
"220", {\ar"320"},
"320", {\ar"420"},
"130", {\ar"230"},
"230", {\ar"330"},
"330", {\ar"430"},
"000", {\ar"010"},
"100", {\ar"110"},
"200", {\ar"210"},
"300", {\ar"310"},
"010", {\ar"020"},
"110", {\ar"120"},
"210", {\ar"220"},
"310", {\ar"320"},
"410", {\ar"420"},
"120", {\ar"130"},
"220", {\ar"230"},
"320", {\ar"330"},
"420", {\ar"430"},
"000", {\ar@<-.2ex>"u00"},
"000", {\ar@<.2ex>"u00"},
"u00", {\ar"110"},
"320", {\ar"l30"},
"l30", {\ar@<-.2ex>"430"},
"l30", {\ar@<.2ex>"430"},
"110", {\ar"001"},
"210", {\ar"101"},
"310", {\ar"201"},
"410", {\ar@/^0.45pc/@<.4ex>"301"},
"410", {\ar@/^0.45pc/"301"},
"120", {\ar"011"},
"220", {\ar"111"},
"320", {\ar"211"},
"420", {\ar"311"},
"130", {\ar@/^0.45pc/@<.4ex>"021"},
"130", {\ar@/^0.45pc/"021"},
"230", {\ar"121"},
"330", {\ar"221"},
"430", {\ar"321"},
"410", {\ar@/^0.6pc/"u01"},
"l30", {\ar"301"},
"l30", {\ar@/^0.6pc/"021"},
"130", {\ar"u01"},
"001", {\ar"101"},
"101", {\ar"201"},
"201", {\ar"301"},
"011", {\ar"111"},
"111", {\ar"211"},
"211", {\ar"311"},
"311", {\ar"411"},
"021", {\ar"121"},
"121", {\ar"221"},
"221", {\ar"321"},
"321", {\ar"421"},
"131", {\ar"231"},
"231", {\ar"331"},
"331", {\ar"431"},
"001", {\ar"011"},
"101", {\ar"111"},
"201", {\ar"211"},
"301", {\ar"311"},
"011", {\ar"021"},
"111", {\ar"121"},
"211", {\ar"221"},
"311", {\ar"321"},
"411", {\ar"421"},
"121", {\ar"131"},
"221", {\ar"231"},
"321", {\ar"331"},
"421", {\ar"431"},
"001", {\ar@<-.2ex>"u01"},
"001", {\ar@<.2ex>"u01"},
"u01", {\ar"111"},
"321", {\ar"l31"},
"l31", {\ar@<-.2ex>"431"},
"l31", {\ar@<.2ex>"431"},
\end{xy}
}
\]
with the following interpretation. Vertices labeled $ab$ and $ab\w$ denote the positions of $E^{ab} := \rho(E^{\s+(a-1)\x_3+(b-1)\x_4})$ and $E^{ab}(\w)$ respectively. A vertex labelled $(\x)$ corresponds to the positions of both $R(\x)$ and $R(\x+\ttt)$, where $\ttt = \x_1-\x_2$. In case $\x$ is not specified we write $\circ$ instead of $(\x)$. With regards to arrows the following situations occur for $E = E^{ab}$ and $\x,\y \in \L$:
\[
\begin{xy} 0;<50pt,0pt>:<0pt,-20pt>:: 
(0.5,-1) *+{R(\x)} ="01",
(0.5,1) *+{R(\x+\ttt)} ="02",
(1.5,0) *+{E} ="1",
(2,0) *+{E} ="2",
(3,-1) *+{R(\x)} ="31",
(3,1) *+{R(\x+\ttt)} ="32",
(4,-1) *+{R(\x)} ="41",
(4,1) *+{R(\x+\ttt)} ="42",
(5,-1) *+{R(\y)} ="51",
(5,1) *+{R(\y+\ttt)} ="52",
(6,-1) *+{R(\x)} ="61",
(6,1) *+{R(\x+\ttt)} ="62",
(7,-1) *+{R(\y)} ="71",
(7,1) *+{R(\y+\ttt).} ="72",
"01", {\ar"1"},
"02", {\ar"1"},
"2", {\ar"31"},
"2", {\ar"32"},
"41", {\ar"51"},
"42", {\ar"52"},
"61", {\ar"71"},
"62", {\ar"72"},
"61", {\ar"72"},
"62", {\ar"71"},
\end{xy}
\]
We denote them simply by
\[
\begin{xy} 0;<50pt,0pt>:<0pt,-20pt>:: 
(0.5,0) *+{(\x)} ="0",
(1.5,0) *+{ab} ="1",
(2,0) *+{ab} ="2",
(3,0) *+{(\x)} ="3",
(4,0) *+{(\x)} ="4",
(5,0) *+{(\y)} ="5",
(6,0) *+{(\x)} ="6",
(7,0) *+{(\y).} ="7",
"0", {\ar"1"},
"2", {\ar"3"},
"4", {\ar"5"},
"6", {\ar@<-.2ex>"7"},
"6", {\ar@<.2ex>"7"},
\end{xy}
\]
and use similar notation for $E = E^{ab}(\w)$.
\end{example}

In the rest of this section, we assume that $(R,\L)$ is not Calabi-Yau and $d\ge1$. We consider the Veronese subring of $R$:
\[R^{(\w)}=\bigoplus_{\x\in\Z\w}R_{\x}.\]
We give basic properties of $R^{(\w)}$.

\begin{proposition}
$R^{(\w)}$ is a Noetherian normal domain with $\dim R^{(\w)}=d+1$.
\end{proposition}

\begin{proof}
By Theorem~\ref{L-factorial L-domain}, $R$ is an $\L$-factorial $\L$-domain. This implies the assertion, see \cite[Corollary 2.3]{IW2} for the case $d=1$.
\end{proof}

We regard $R^{(\w)}$ a $\Z\w$-graded algebra. We consider the functor
\[(-)^{(\w)}\colon\mod^{\L}R\to\mod^{\Z\w}R^{(\w)}\ \mbox{ given by }\ X\mapsto X^{(\w)}=\bigoplus_{\x\in\Z\w}X_{\x}.\]
We denote by $\CM R^{(\w)}$ the category of maximal Cohen-Macaulay $R^{(\w)}$-modules, and by $\CM^{\Z\w}R^{(\w)}$ its graded version.
Then the functor $(-)^{(\w)}$ restricts to $(-)^{(\w)}\colon\CM^{\L}R\to\CM^{\Z\w}R^{(\w)}$ by \cite[1.2.26]{BH}.

\begin{proposition}
$R^{(\w)}$ is a $\Z\w$-graded Gorenstein ring. Moreover, its $a$-invariant is $\w$, that is, $\Ext^{d+1}_{R^{(\w)}}(k,R^{(\w)})\simeq k(-\w)$ as $\Z\w$-graded $R^{(\w)}$-modules.
\end{proposition}

\begin{proof}
This is an $\L$-graded version of Goto-Watanabe's result \cite[I.3.6.21]{BH} (see also \cite[3.1.5]{GW}).
We include a proof for the convenience of the reader. It suffices to show that the $*$canonical module $C$ of $R^{(\w)}$ is given by $R^{(\w)}(\w)$. We calculate $C$ by using the local cohomology: For $X\in\mod^{\L}R^{(\w)}$, let
\[H^i_{\nn}(X)_{\x}:=\varinjlim_j\Ext_{\mod^{\L}R^{(\w)}}^i(R^{(\w)}/\nn^j,X(\x))\ \mbox{ and }\ H^i_{\nn}(X):=\bigoplus_{\x\in\L}H^i_{\nn}(X)_{\x},\]
where $\nn=R^{(\w)}\cap R_+$. By local duality \cite[3.6.19]{BH} (cf.\ Proposition~\ref{local duality}), we have $C\simeq D(H^{d+1}_\nn(R^{(\w)}))$ in $\mod^{\L}R^{(\w)}$, where $D$ is the graded dual.
Since $R$ is a Gorenstein ring with $a$-invariant $\w$, we have $R(\w)\simeq D(H^{d+1}_\nn(R))$ in $\mod^{\L}R$ again by local duality and Independence Theorem \cite[4.2.1]{BS}.
Thus we have
\[
C=D(H^{d+1}_\nn(R^{(\w)}))=D(H^{d+1}_\nn(R))^{(\w)}=(R(\w))^{(\w)}=R^{(\w)}(\w)
\]
as desired.
\end{proof}

We apply our results to non-commutative crepant resolutions, which is known to be closely related to $d$-cluster tilting objects \cite{I1,IW1}.
We will show that $d$-CM finiteness of $(R,\L)$ implies the existence of NCCRs of $R^{(w)}$.

\begin{definition}\cite{V}\label{define NCCR}
We say that $M\in\CM R^{(\w)}$ gives a \emph{non-commutative crepant resolution} (=\emph{NCCR})
of $R^{(\w)}$ if $E=\End_{R^{(\w)}}(M)\in\CM R^{(\w)}$ and $\gl E=d+1$.
Note that the condition $\gl E=d+1$ can be replaced by $\gl(R^{(\w)}_{\pp}\otimes_{R^{(\w)}}E)=\dim R^{(\w)}_{\pp}$ for all $\pp\in\Spec R^{(\w)}$ by \cite[Proposition 2.17(2)$\Rightarrow$(1)]{IW1}.
\end{definition}

We will prove the following analogue of \cite[Theorem 5.2.1]{I1}\cite[Theorems 4.3, 5.4]{IW1}.

\begin{theorem}\label{construct NCCR}
Let $(R,\L)$ be a GL complete intersection with $d\ge1$ which is not Calabi-Yau. For $V\in\CM^{\L}R$, assume that
$\UU:=\add\{V(\ell\w)\mid\ell\in\Z\}$ contains $\proj^{\L}R$.
Then the following conditions are equivalent.
\begin{itemize}
\item[(a)] $\UU$ is a $d$-cluster tilting subcategory of $\CM^{\L}R$.
\item[(b)] $V^{(\w)}$ gives an NCCR of $R^{(\w)}$.
\end{itemize}
Therefore if $(R,\L)$ is $d$-CM finite, then $R^{(\w)}$ has an NCCR.
\end{theorem}

For simplicity, we write $G=\L/\Z\w$. We regard $R$ as a $G$-graded ring naturally. Then its degree $0$ part is $R^{(\w)}$. By abuse of notations, we consider the functor
\[(-)^{(\w)}\colon\mod^GR\to\mod R^{(\w)}\ \mbox{ given by }\ X=\bigoplus_{g\in G}X_g\mapsto X_0,\]
which induces a functor $(-)^{(\w)}\colon\CM^GR\to\CM R^{(\w)}$.
Notice that, for $X,Y\in\mod^GR$, $\Hom_R^G(X,Y)=\Hom_R(X,Y)^{(\w)}$ is a direct summand
of $\Hom_R(X,Y)$ as an $R^{(\w)}$-module.

For $\x\in\L$, we define an ideal of $R^{(\w)}$ by
\[I^{\x}:=R(\x)^{(\w)}\cdot R(-\x)^{(\w)}\subset R^{(\w)}.\]
Then $I^{\x}$ depends only on the class of $\x$ in $G$.

We start with the following observation (cf.\ \cite[Theorem 4.3]{IW2}).

\begin{lemma}\label{Veronese functor}
The following assertions hold.
\begin{itemize}
\item[(a)] For all $\x\in\L$, we have $\dim(R^{(\w)}/I^{\x})\le d-1$.
\item[(b)] We have a functorial isomorphism $\Hom_R^G(X,Y)\simeq\Hom_{R^{(\w)}}(X^{(\w)},Y^{(\w)})$ for all $X\in\mod^GR$ and $Y\in\CM^GR$.
\end{itemize}
\end{lemma}

\begin{proof}
(a) For any $\x,\y\in\L$, we have $I^{\x}I^{\y}\subset I^{\x+\y}$ and hence
\[\dim(R^{(\w)}/I^{\x+\y})\le\max\{\dim(R^{(\w)}/I^{\x}),\dim(R^{(\w)}/I^{\y})\}.\]
Therefore it suffices to show that $\dim(R^{(\w)}/I^{\x_i})\le d-1$ holds for any $i$.

By the argument in the proof of \cite[Theorem 4.7(4)$\Rightarrow$(3)]{IW2}, $I^{\x_i}$
contains a power $f_1$ of $X_i$, and also a monomial $f_2$ in the $X_j$'s with $j\neq i$.
Since $(X_i,X_j)$ is an $R$-regular sequence by Proposition~\ref{prop.regularseq1}(c) for each $j\neq i$, $(f_1,f_2)$ is also an $R^{(\w)}$-regular sequence. Thus we have $\dim(R^{(\w)}/I^{\x_i})\le d-1$
as desired.

(b) We regard $R$ as a $G$-graded ring, and consider the ring $A=R^{[0]}$ 
associated to the subgroup $0$ of $G$ given in Definition~\ref{define covering}.
We denote by $e\in A$ the idempotent of $A$ whose $(0,0)$-entry is $1$ and all other entries are $0$. We identify the subring $eAe$ of $A$ with $R^{(\w)}$.
Since the diagonal subring of $A/(e)$ is $\prod_{\x\in G}R^{(\w)}/I^{\x}$ and $A/(e)$ is a finitely generated module over the diagonal subring, (a) implies
\begin{equation}\label{dim A/e=d-1}
\dim_{R^{(\w)}}(A/(e))\le d-1.
\end{equation}

Recall that we have an equivalence $F:\mod^{G}R\simeq \mod A$ in \eqref{general graded morita}.
For the functor $E:=e(-)\colon\mod A\to\mod eAe$, we have the following commutative diagram of functors.
\[\xymatrix@R=1.5em{
\mod^GR\ar[r]^F\ar@{=}[d]&\mod A\ar[r]^E&\mod eAe\ar@{=}[d]\\
\mod^GR\ar[rr]^{(-)^{(\w)}}&&\mod R^{(\w)}
}\]
Since we have functorial isomorphisms
\[\Hom_R^G(X,Y)\simeq\Hom_A(FX,FY)\ \mbox{ and }\ \Hom_{R^{(\w)}}(X^{(\w)},Y^{(\w)})\simeq\Hom_{eAe}(EFX,EFY),\]
it suffices to show that
$E_{FX,FY}:\Hom_A(FX,FY)\simeq\Hom_{eAe}(EFX,EFY)$ is an isomorphism.

Consider the canonical morphism $\epsilon\colon Ae\otimes_{eAe}E(-)\to{\rm id}$ of functors $\mod A\to\mod A$. For $X\in\mod^GR$, consider an exact sequence
\begin{equation}\label{AeFX to FX}
0\to C_1\to Ae\otimes_{eAe}EFX\xrightarrow{\epsilon_{FX}}FX\to C_0\to0.
\end{equation}
Since $E(\epsilon_{FX})$ is an isomorphism, we have $E(C_i)=0$ for $i=0,1$.
Thus $C_i$ is a finitely generated $A/(e)$-module, and hence
$\dim_RF^{-1}C_i=\dim_{R^{(\w)}}F^{-1}C_i=\dim_{R^{(\w)}}C_i\le d-1$ by \eqref{dim A/e=d-1}.
Since $Y\in\CM^GR$, we have
\[\Ext^j_A(C_i,FY)=\Ext^j_{\mod^GR}(F^{-1}C_i,Y)=0\ 
\mbox{ for } j=0,1\]
by \cite[Theorem 17.1]{Ma}. Applying $\Hom_A(-,FY)$ to the exact sequence \eqref{AeFX to FX},
we have a functorial isomorphsim
\[\Hom_A(FX,FY)\simeq\Hom_A(Ae\otimes_{eAe}EFX,FY).\]
Since $(Ae\otimes_{eAe}-,E)$ is an adjoint pair, we obtain an isomorphism $\Hom_A(FX,FY)\simeq\Hom_{eAe}(EFX,EFY)$ as desired.
\end{proof}

Next we prepare the following graded analog of \cite[Proposition 2.5.1]{I1}.

\begin{lemma}\label{d-rigid iff CM}
For $X,Y\in\CM^GR$, the following conditions are equivalent.
\begin{itemize}
\item[(a)] $\Ext^i_{\mod^GR}(X,Y)=0$ for all $1\le i\le d-1$.
\item[(b)] $\Hom_R^G(X,Y)\in\CM R^{(\w)}$.
\end{itemize}
\end{lemma}

\begin{proof}
We take a projective resolution $P_d\to\cdots\to P_0\to X\to0$ of $X$ in $\mod^GR$.
Applying $\Hom_R^G(-,Y)$, we have a complex
\begin{equation}\label{Hom(P,Y)}
0\to\Hom_R^G(X,Y)\to\Hom_R^G(P_0,Y)\to\cdots\to\Hom_R^G(P_d,Y)
\end{equation}
of $R^{(\w)}$-modules whose cohomologies are $\Ext^i_{\mod^GR}(X,Y)$.
Since $\Hom_R(P_i,Y)\in\CM^{\L}R$, we have $\Hom_R^G(P_i,Y)\in\CM R^{(\w)}$.

(a)$\Rightarrow$(b) By our assumption, the sequence \eqref{Hom(P,Y)} is exact.
Since $\Hom_R^G(P_i,Y)\in\CM R^{(\w)}$, we have $\Hom_R^G(X,Y)\in\CM R^{(\w)}$ by counting depth.

(b)$\Rightarrow$(a) 
Assume that the sequence \eqref{Hom(P,Y)} is not exact, and take the minimal $i$ with
$1\le i\le d-1$ such that $\Ext^i_{\mod^GR}(X,Y)\neq0$.
By Proposition~\ref{CM is locally free}, $\Ext^i_{\mod^GR}(X,Y)$
is a finite dimensional $k$-vector space, and has depth zero.
Now we consider an exact sequence
\begin{align*}
0\to\Hom_R^G(X,Y)\to\Hom_R^G(P_0,Y)\to\cdots\to\Hom_R^G(P_{i-1},Y)\to\Hom_R^G(\Omega^iX,Y)&\\
\to\Ext^i_{\mod^GR}(X,Y)\to&0,
\end{align*}
where $\Hom_R^G(X,Y)$ and $\Hom_R^G(P_i,Y)$ have depth $d+1$
and $\Hom_R^G(\Omega^iX,Y)$ has depth at least two.
By counting depth, $\Ext^i_{\mod^GR}(X,Y)$ has depth at least one, a contradiction.
\end{proof}

Finally we show the following.

\begin{lemma}\label{d-CT iff gldim=d+1}
For $V\in\CM^{\L}R$, let $\UU=\add\{V(\ell\w)\mid\ell\in\Z\}\subset\CM^{\L}R$.
If $\Ext^i_{\mod^GR}(V,V)=0$ for all $1\le i\le d-1$, then the following conditions are equivalent.
\begin{itemize}
\item[(a)] $\UU$ is a $d$-cluster tilting subcategory of $\CM^{\L}R$.
\item[(b)] $\gl\End_R^G(V)=d+1$.
\end{itemize}
\end{lemma}

\begin{proof}
Let $E=\End_R^G(V)$. This is a $\Z$-graded ring.

(a)$\Rightarrow$(b)
By \cite[Corollary 7.8]{NV1}, we have the equality $\gl E=\gl(\mod^{\Z}E)$.
It suffices to show that any $\Z$-graded simple $E$-module $S$ has projective dimension $d+1$
(e.g.\ \cite[Proposition 2.2]{IR}). For each simple $E$-module $S$, there exists
an indecomposable object $X\in\UU$ such that $S$ is the top of $\Hom_R^G(V,X)$.
By Proposition~\ref{basic properties of d-CT for CM}(c), there exists a $d$-almost split sequence
$0\to X(\w)\to C_{d-1}\to\cdots\to C_0\to X\to0$ in $\UU$.
Applying $\Hom_R^G(V,-)$, we have an exact sequence
\[0\to\Hom_R^G(V,X(\w))\to\Hom_R^G(V,C_{d-1})\to\cdots\to\Hom_R^G(V,C_0)\to\Hom_R^G(V,X)\to S\to0.\] 
This gives a projective resolution of the $E$-module $S$ by Lemma~\ref{Veronese functor},
and hence $S$ has projective dimension $d+1$.

(b)$\Rightarrow$(a)
It suffices to show that, if $X\in\CM^{\L}R$ satisfies $\Ext^i_{\mod^{\L}R}(\UU,X)=0$ for all $1\le i\le d-1$, then $X\in\UU$.
Let $0\to X\to I^0\to\cdots\to I^{d}$ be an injective resolution of $X$ in $\CM^{\L}R$.
Since $\Ext^i_{\mod^GR}(V,X)=0$ for all $1\le i\le d-1$, by applying $\Hom_R^G(V,-)$, we have an exact sequence
\[0\to\Hom_R^G(V,X)\to\Hom_R^G(V,I^0)\to\cdots\to\Hom_R^G(V,I^{d}).\] 
Since $I^i$ belongs to $\proj^{\L}R\subset\UU$, the $E$-module $\Hom_R^G(V,I^0)$ is projective.
Since $\gl E=d+1$, the $E$-module $\Hom_R^G(V,X)$ is projective. 
Thus $X\in\UU$.
\end{proof}

Now we are ready to prove Theorem~\ref{construct NCCR}.

\begin{proof}[Proof of Theorem~\ref{construct NCCR}]
Since $\End_{R^{(\w)}}(V^{(\w)})\simeq\End_R^G(V)$ holds by Lemma~\ref{Veronese functor},
the assertion follows from Lemmas~\ref{d-rigid iff CM} and \ref{d-CT iff gldim=d+1}.
\end{proof}

We end this section by posing the following problem.

\begin{problem}\label{Fano and d-RF}
How are the following conditions related to each other?
\begin{itemize}
\item[(a)] $(R,\L)$ is Fano.
\item[(b)] $(R,\L)$ is $d$-CM finite.
\item[(c)] $R^{(\w)}$ has an NCCR.
\item[(d)] $\underline{\CM}^{\L}R$ has a $d$-tilting object (or equivalently, $A^{\rm CM}$ is derived equivalent to an algebra of global dimension at most $d$).
\end{itemize}
\end{problem}

The statements (d)$\Rightarrow$(a) and (d)$\Rightarrow$(b)$\Rightarrow$(c) were shown in Theorems~\ref{d-tilting imply Fano}, \ref{construct dCT} and \ref{construct NCCR}.
On the other hand, (a)$\Rightarrow$(d) does not hold by Example~\ref{Fano no d-tilting}.

We give the following partial answer to the statement (a)$\Rightarrow$(d).

\begin{theorem}\label{Main1}
If one of the following conditions are satisfied, then $\underline{\CM}^{\L}R$ has a $d$-tilting object, $(R,\L)$ is $d$-CM finite and $R^{(\w)}$ has an NCCR.
\begin{itemize}
\item $n\le d+1$.
\item $n=d+2\ge2$ and $(p_1,p_2)=(2,2)$.
\item $n=d+2\ge3$ and $(p_1,p_2,p_3)=(2,3,3)$, $(2,3,4)$ or $(2,3,5)$.
\item $n=d+2\ge4$ and $(p_1,p_2,p_3,p_4)=(3,3,p_3,p_4)$ with $p_3,p_4\in\{3,4,5\}$.
\item $\#\{i\mid p_i=2\}\ge3(n-d)-4$.
\end{itemize}
\end{theorem}

\begin{proof}
The first three cases are immediate from Proposition~\ref{d-tilting object example2}.
The last case follows from Proposition~\ref{d-tilting object example1}.
\end{proof}

\section{Matrix factorizations and their tensor products}\label{subsection: Tensor}

It is well-known that for the hypersuface case, Cohen-Macaulay representations are described in terms of matrix factorizations.
Throughout this section, let $(R,\L)$ be a Geigle-Lenzing complete intersection with $n=d+2$. By Observation~\ref{Normalization} it is given by
\[R=S/(f)\mbox{ for $S:=k[X_1,\ldots,X_{d+2}]$ and $f:=\sum_{i=1}^{d+2}\lambda_iX_i^{p_i}$.}\]
The aim of this section is to prepare results on tensor products of matrix factorizations, which will be used in the following sections and are interesting by themselves.

\medskip\noindent{\bf Matrix factorization} 
An \emph{$\L$-graded matrix factorization} of $f$ is a pair
\[(\phi,\psi)=(\phi:Q\to P,\ \psi:P(-\c)\to Q)\]
of morphisms in $\proj^{\L}S$ satisfying
\[\phi(-\c)\psi=f:Q(-\c)\to Q\ \mbox{ and }\ \psi\phi=f:P(-\c)\to P.\]
For example, any $P\in\proj^{\L}R$ gives $\L$-graded matrix factorizations
\[(1,f)_P:=({\rm id}:P\to P,\ f:P(-\c)\to P)\ \mbox{ and }\ (f,1)_P:=(f:P\to P(\c),\ {\rm id}:P\to P).\]
The category $\MF^{\L}(S,f)$ of matrix factorizations is defined as follows 
\begin{itemize}
\item The objects are the $\L$-graded matrix factorizations of $f$.
\item For two $\L$-graded matrix factorizations $(\phi,\psi)=(\phi:Q\to P,\ \psi:P(-\c)\to Q)$ and $(\phi',\psi')=(\phi':Q'\to P',\ \psi':P'(-\c)\to Q')$ of $f$, a morphism from $(\phi,\psi)$ to $(\phi',\psi')$ is a pair $(\alpha,\beta)\in\Hom^{\L}_R(P,P')\times\Hom^{\L}_R(Q,Q')$ making the following diagram commutative.
\[\xymatrix@R2em{
P(-\c)\ar[rr]^\psi\ar[d]^{\alpha(-\c)}&&Q\ar[rr]^\phi\ar[d]^\beta&&P\ar[d]^\alpha\\
P'(-\c)\ar[rr]^{\psi'}&&Q'\ar[rr]^{\phi'}&&P'\\
}\]
The composition of morphisms is defined in an obvious way.
\end{itemize}

\medskip
It is elementary \cite{E} that there exists a full dense functor
\begin{equation}\label{MF CM}
\Cokernel\colon\MF^{\L}(S,f)\to\CM^{\L}R\ \mbox{ given by }\ (\phi,\psi)\mapsto\Cokernel \phi.
\end{equation}
More precisely, we consider full subcategories
\[\PP:=\add\{(1,f)_P,\ (f,1)_P\mid P\in\proj^{\L}R\}\supset\PP_+:=\add\{(1,f)_P\mid P\in\proj^{\L}R\}\]
of $\MF^{\L}(S,f)$. Let $[\PP]$ (respectively, $[\PP_+]$) be the ideal of the category $\MF^{\L}(S,f)$ consisting of all morphisms which factor through objects in $\PP$ (respectively, $\PP_+$),
and by $\underline{\MF}^{\L}(S,f):=\MF^{\L}(S,f)/[\PP]$ (respectively, $\MF^{\L}(S,f)/[\PP_+]$) the corresponding factor category.
Then the following result is well-known \cite{Y}.

\begin{proposition}\label{MF CM2}
The functor \eqref{MF CM} gives equivalences
\begin{equation*}
\Cokernel\colon\MF^{\L}(S,f)/[\PP_+]\simeq\CM^{\L}R\ \mbox{ and }\ \Cokernel\colon\underline{\MF}^{\L}(S,f)\simeq\underline{\CM}^{\L}R\simeq\DDD_{\rm sg}^{\L}(R),
\end{equation*}
where the last equivalence is Theorem~\ref{Buchweitz Eisenbud}(c).
\end{proposition}

Therefore one can calculate objects and morphisms in $\CM^{\L}R$ and $\underline{\CM}^{\L}R$ in terms of $\L$-graded matrix factorizations. In particular, for any $(\phi:Q\to P,\ \psi:P(-\c)\to Q)\in\MF^{\L}(S,f)$ and the corresponding object $X=\Cokernel \phi$, we have a projective resolution
\begin{equation}
\cdots\xrightarrow{\psi}R\otimes_SQ(-c)\xrightarrow{\phi}R\otimes_SP(-\c)\xrightarrow{\psi}R\otimes_SQ\xrightarrow{\phi}R\otimes_SP\to X\to0
\end{equation}
which is 2-periodic up to degree shift by $\c$.

\medskip\noindent
{\bf Tensor products.}
In the rest of this section, we fix two GL hypersurfaces $(R^j,\L_j)$ with weights $p_{j,1},\ldots,p_{j,n_j}$ for $j=1,2$. Thus
\begin{eqnarray*}
&R^j=S^j/(f_j)\ \mbox{ with }\ S^j=k[X_{j,1},\ldots,X_{j,n_j}]\ \mbox{ and }\ f_j=\sum_{i=1}^{n_j}\lambda_{j,i}X_{j,i}^{p_{j,i}},&\\
&\L_j=\langle\c_j,\x_{j,1},\ldots,\x_{j,n_j}\rangle/\langle p_{j,i}\x_{j,i}-\c_j\mid1\le i\le n_j\rangle.&
\end{eqnarray*}
We define a new GL hypersurface $(R,\L)$ with weights $p_{1,1},\ldots,p_{1,n_1},p_{2,1},\ldots,p_{2,n_2}$ given by
\begin{eqnarray*}
&R=S/(f)\ \mbox{ with }\ S=S^1\otimes_kS^2=k[X_{1,1},\ldots,X_{1,n_1},X_{2,1},\ldots,X_{2,n_2}]\ \mbox{ and }\ f=f_1+f_2,&\\
&\L=\langle\c,\x_{1,1},\ldots,\x_{1,n_1}\x_{2,1},\ldots,\x_{2,n_2}\rangle/\langle p_{j,i}\x_{j,i}-\c\mid j=1,2,\ 1\le i\le n_j\rangle.&
\end{eqnarray*}
The following observation is immediate.

\begin{observation}\label{L and R}
\begin{itemize}
\item[(a)] We have $\L=(\L_1\times\L_2)/\langle\c_1-\c_2\rangle$.
\item[(b)] We have $R^1\otimes_kR^2=R/(f_1)=R/(f_2)$.
\end{itemize}
\end{observation}

By Observation~\ref{L and R}(a), there is a functor $\mod^{\L_1\times\L_2}(R^1\otimes_kR^2)\to\mod^{\L}(R^1\otimes_kR^2)$ making the grading more coarse.
By Observation~\ref{L and R}(b), there is a functor $\mod^{\L}(R^1\otimes_kR^2)\to\mod^{\L}R$ given by restriction along the natural surjective morphism $R\to R^1\otimes_kR^2$. 
Composing them, we have an exact functor
\[\mod^{\L_1\times\L_2}(R^1\otimes_kR^2)\to\mod^{\L}R,\]
which induces a triangle functor $\DDD^{\bo}(\mod^{\L_1\times\L_2}(R^1\otimes_kR^2))\to\DDD^{\bo}(\mod^{\L}R)$.
Composing with the tensor functor $-\otimes_k-:\DDD^{\bo}(\mod^{\L_1}R^1)\times\DDD^{\bo}(\mod^{\L_2}R^2)\to\DDD^{\bo}(\mod^{\L_1\times\L_2}(R^1\otimes_kR^2))$, we have a bifunctor
\[-\otimes_k-:\DDD^{\bo}(\mod^{\L_1}R^1)\times\DDD^{\bo}(\mod^{\L_2}R^2)\to\DDD^{\bo}(\mod^{\L}R).\]
The following observation tells us that this bifunctor descends to singularity categories.

\begin{lemma}\label{tensor for sg}
\begin{itemize}
\item[(a)] Let $X^i\in\mod^{\L_i}R^i$ for $i=1,2$. If $X^i$ has finite projective dimension in $\mod^{\L_i}R^i$ for at least one of $i=1,2$, then $X^1\otimes_kX^2$ has finite projective dimension in $\mod^{\L}R$.
\item[(b)] There exists a bifunctor $-\otimes_k-:\DDD_{\sg}^{\L_1}(R^1)\times\DDD_{\sg}^{\L_2}(R^2)\to\DDD_{\sg}^{\L}(R)$ making the following diagram commutative.
\[\xymatrix@R=1em{
\DDD^{\bo}(\mod^{\L_1}R^1)\times\DDD^{\bo}(\mod^{\L_2}R^2)\ar[r]\ar[d]^{-\otimes_k-}&\DDD_{\sg}^{\L_1}(R^1)\times\DDD_{\sg}^{\L_2}(R^2)\ar[d]^{-\otimes_k-}\\
\DDD^{\bo}(\mod^{\L}R)\ar[r]&\DDD_{\sg}^{\L}(R).
}\]
\end{itemize}
\end{lemma}

\begin{proof}
(a) Without loss of generality, we can assume $i=1$ and $X^1=R^1$.

Since $f^2X^2=0$, $f$ acts on $S^1 \otimes_k X^2$ as $f_1\otimes1$. Thus $f$ is $(S^1 \otimes_k X^2)$-regular.
Since $S$ is regular, $S^1 \otimes_k X^2$ has finite projective dimension in $\mod^{\L}S$.
Applying $R\otimes_S-$, $R\otimes_S(S^1 \otimes_k X^2)$ has finite projective
dimension in $\mod^{\L}R$. Since
\[R\otimes_S(S^1 \otimes_k X^2)=(S^1 \otimes_k X^2) / f(S^1 \otimes_k X^2)=
(S^1 \otimes_k X^2) / f_1(S^1 \otimes_k X^2) = R^1 \otimes_k X^2,\]
the assertion follows.

(b) It suffices to show that, for $X^i\in\DDD^{\bo}(\mod^{\L_i}R^i)$ for $i=1,2$, $X^1\otimes_kX^2$ belongs to $\KKK^{\bo}(\proj^{\L}R)$ if $X^i$ belongs to $\KKK^{\bo}(\proj^{\L_i}R^i)$ for at least one of $i=1,2$.
Without loss of generality, assume that $X^1\in\KKK^{\bo}(\proj^{\L_1}R^1)$.
Since $\KKK^{\bo}(\proj^{\L_1}R^1)=\thick(\proj^{\L_1}R^1)$ and
$\DDD^{\bo}(\mod^{\L_2}R^2)=\thick(\mod^{\L_2}R^2)$, it suffices to consider the case
$X^1=R^1(\x)$ for some $\x\in\L_1$ and $X^2\in\mod^{\L_2}R^2$. Then
$X^1\otimes_kX^2$ has finite projective dimension in $\mod^{\L}R$ by (a),
and hence belongs to $\KKK^{\bo}(\proj^{\L}R)$.
\end{proof}

Now we recall tensor products of matrix factorizations introduced by Yoshino \cite{Y} (see also \cite{BFK,D}).

\medskip\noindent
{\bf Tensor product of matrix factorizations.}
For $i=1,2$, let $(\phi^i:Q^i\to P^i,\ \psi^i:P^i(-\c_i)\to Q^i)$ be a matrix factorization of $f_i$.
Define a matrix factorization $(\phi,\psi):=(\phi^1,\psi^1)\otimes_{\rm MF}(\phi^2,\psi^2)$ of $f=f_1+f_2$ by
\begin{eqnarray*}
&\phi:=\left(\begin{smallmatrix}\psi^1\otimes 1&-1\otimes \phi^2\\ 1\otimes \psi^2&\phi^1\otimes 1\end{smallmatrix}\right):(P^1\otimes_kQ^2)\oplus(Q^1\otimes_kP^2)\to(Q^1\otimes_kQ^2)(\c)\oplus(P^1\otimes_kP^2),&\\
&\psi:=\left(\begin{smallmatrix}\phi^1\otimes 1&1\otimes \phi^2\\ -1\otimes \psi^2&\psi^1\otimes 1\end{smallmatrix}\right):(Q^1\otimes_kQ^2)\oplus(P^1\otimes_kP^2)(-\c)\to(P^1\otimes_kQ^2)\oplus(Q^1\otimes_kP^2).&
\end{eqnarray*}
This gives a bifunctor
\[-\otimes_{\rm MF}-:\MF^{\L_1}(S^1,f^1)\times\MF^{\L_2}(S^2,f^2)\to\MF^{\L}(S,f),\]
that is defined on morphisms as follows.
For $i=1,2$, let $(\eta^i:G^i\to F^i,\ \theta^i:F^i(-\c_i)\to G^i)$ be another matrix factorization of $f_i$, and $(\alpha^i,\beta^i):(\phi^i,\psi^i)\to(\eta^i,\theta^i)$ a morphism of matrix factorizations.
Then we have a morphism $(\alpha^1,\beta^1)\otimes_{\rm MF}(\alpha^2,\beta^2):(\phi^1,\psi^1)\otimes_{\rm MF}(\phi^2,\psi^2)\to(\eta^1,\theta^1)\otimes_{\rm MF}(\eta^2,\theta^2)$ of matrix factorizations given by
\[{\small\xymatrix@C5.5em{
{\begin{array}{l}(Q^1\otimes_kQ^2)\\ \oplus(P^1\otimes_kP^2)(-\c)\end{array}}\ar[r]^{\left(\begin{smallmatrix}\phi^1\otimes 1&1\otimes \phi^2\\ -1\otimes \psi^2&\psi^1\otimes 1\end{smallmatrix}\right)}\ar[d]^{\begin{smallmatrix}(\beta^1\otimes\beta^2)\\ \oplus(\alpha^1\otimes\alpha^2)(-\c)\end{smallmatrix}}
&{\begin{array}{l}(P^1\otimes_kQ^2)\\ \oplus(Q^1\otimes_kP^2)\end{array}}\ar[r]^{\left(\begin{smallmatrix}\psi^1\otimes 1&-1\otimes \phi^2\\ 1\otimes \psi^2&\phi^1\otimes 1\end{smallmatrix}\right)}\ar[d]^{\begin{smallmatrix}(\alpha^1\otimes\beta^2)\\ \oplus(\beta^1\otimes\alpha^2)\end{smallmatrix}}&
{\begin{array}{l}(Q^1\otimes_kQ^2)(\c)\\ \oplus(P^1\otimes_kP^2)\end{array}}\ar[d]^{\begin{smallmatrix}(\beta^1\otimes\beta^2)(\c)\\ \oplus(\alpha^1\otimes\alpha^2)\end{smallmatrix}}\\
{\begin{array}{l}(G^1\otimes_kG^2)\\ \oplus(F^1\otimes_kF^2)(-\c)\end{array}}\ar[r]_{\left(\begin{smallmatrix}\eta^1\otimes 1&1\otimes \eta^2\\ -1\otimes \theta^2&\theta^1\otimes 1\end{smallmatrix}\right)}
&{\begin{array}{l}(F^1\otimes_kG^2)\\ \oplus(G^1\otimes_kF^2)\end{array}}\ar[r]_{\left(\begin{smallmatrix}\theta^1\otimes 1&-1\otimes \eta^2\\ 1\otimes \theta^2&\eta^1\otimes 1\end{smallmatrix}\right)}&
{\begin{array}{l}(G^1\otimes_kG^2)(\c)\\ \oplus(F^1\otimes_kF^2).\end{array}}
}}\]
For $X^i:=\Cokernel \phi^i,\ Y^i:=\Cokernel \eta^i\in\CM^{\L_i}R^i$, we also have a morphism $(\alpha^1,\beta^1)\otimes_{\rm MF}(\alpha^2,\beta^2):X^1\otimes_{\rm MF}X^2\to Y^1\otimes_{\rm MF}Y^2$ in $\CM^{\L}R$.

\medskip
We give the following basic property of tensor products $\otimes_{\rm MF}$.

\begin{proposition}\label{tensor product is unique}
The bifunctor $-\otimes_{\rm MF}-:\MF^{\L_1}(S^1,f^1)\times\MF^{\L_2}(S^2,f^2)\to\MF^{\L}(S,f)$ induces a bifunctor
\[-\otimes_{\rm MF}-:\underline{\MF}^{\L_1}(S^1,f^1)\times\underline{\MF}^{\L_2}(S^2,f^2)\to\underline{\MF}^{\L}(S,f).\]
\end{proposition}

\begin{proof}
We have
\[(f^1,1)\otimes_{\rm MF}(\phi^2,\psi^2)=\left(\left(\begin{smallmatrix}1\otimes 1&-1\otimes \phi^2\\ 1\otimes \psi^2&f^1\otimes1\end{smallmatrix}\right),\left(\begin{smallmatrix}f^1\otimes 1&1\otimes \phi^2\\ -1\otimes \psi^2&1\otimes 1\end{smallmatrix}\right)\right)\simeq
\left(\left(\begin{smallmatrix}1&0\\ 0&f\end{smallmatrix}\right),\left(\begin{smallmatrix}f&0\\ 0&1\end{smallmatrix}\right)\right)\]
and $(1,f^1)\otimes_{\rm MF}(\phi^2,\psi^2)\simeq\left(\left(\begin{smallmatrix}f&0\\ 0&1\end{smallmatrix}\right),\left(\begin{smallmatrix}1&0\\ 0&f\end{smallmatrix}\right)\right)$ similarly. Thus $-\otimes_{\rm MF}-$ gives the desired bifunctor.
\end{proof}

\begin{definition}
Using equivalences given in Propositions~\ref{MF CM2} and \ref{tensor product is unique}, we define the bifunctor
\[-\otimes_{\rm MF}-:\underline{\CM}^{\L_1}R^1\times\underline{\CM}^{\L_2}R^2\to\underline{\CM}^{\L}R.\]
\end{definition}

It is immediate from the definition that, for $\x\in\L_1$ and $\y\in\L_2$, we have isomorphisms
\begin{eqnarray}\notag
&(-\otimes_{\rm MF}(-[1]))\simeq(-\otimes_{\rm MF}-)[1]\simeq((-[1])\otimes_{\rm MF}-)\\ 
&((-(\x))\otimes_{\rm MF}(-(\y))\simeq(-\otimes_{\rm MF}-)(\x+\y)\label{[1] and (x+y)}
\end{eqnarray}
of bifunctors $\underline{\MF}^{\L_1}(S^1,f^1)\times\underline{\MF}^{\L_2}(S^2,f^2)\to\underline{\MF}^{\L}(S,f)$ and $\underline{\CM}^{\L_1}R^1\times\underline{\CM}^{\L_2}R^2\to\underline{\CM}^{\L}R$.

The following result shows that the tensor product $\otimes_{\rm MF}$ of matrix factorizations is a shadow of the ordinary tensor product $\otimes_k$ in the derived categories.

\begin{theorem}\label{tensor of MF is tensor}
We have the following commutative diagram up to natural isomorphism.
\[\xymatrix@R=1em{
\DDD^{\bo}(\mod^{\L_1}R^1)\times\DDD^{\bo}(\mod^{\L_2}R^2)\ar[r]\ar[d]^{-\otimes_k-}&\DDD_{\sg}^{\L_1}(R^1)\times\DDD_{\sg}^{\L_2}(R^2)\ar@{-}[r]^\sim\ar[d]^{-\otimes_k-}&\underline{\CM}^{\L_1}R^1\times\underline{\CM}^{\L_2}R^2\ar[d]^{-\otimes_{\rm MF}-}\\
\DDD^{\bo}(\mod^{\L}R)\ar[r]&\DDD_{\sg}^{\L}(R)\ar@{-}[r]^\sim&\underline{\CM}^{\L}R
}\]
\end{theorem}

To prove this, we start with the following result, which shows that the tensor product $\otimes_{\rm MF}$ is nothing but the Cohen-Macaulay approximation of the ordinary tensor product $\otimes_k$.

\begin{proposition}\label{tensor product is CM approximation}
\begin{itemize}
\item[(a)] For $i=1,2$, let $(\phi^i:Q^i\to P^i,\ \psi^i:P^i(-\c_i)\to Q^i)$ be a matrix factorization of $f_i$, and $X^i:=\Cokernel \phi^i\in\CM^{\L_i}R^i$.
Let $(\phi,\psi)=(\phi^1,\psi^1)\otimes_{\rm MF}(\phi^2,\psi^2)$ and $X=\Cokernel\phi$.
Then there exists an exact sequence
\[0\to Q\to X\xrightarrow{\gamma_{(\phi^1,\psi^1),(\phi^2,\psi^2)}} X^1\otimes_kX^2\to0\]
in $\mod^{\L}R$ with $Q:=R\otimes_{S}(Q^1\otimes_kQ^2)(\c)\in\proj^{\L}R$.
In particular, $X\in\CM^{\L}R$ is a Cohen-Macaulay approximation of $X^1\otimes_kX^2\in\mod^{\L}R$.
\item[(b)] We have the following commutative diagram up to natural isomorphism.
\[\xymatrix@R=1em@C=4.5em{
\MF^{\L_1}(S^1,f^1)\times\MF^{\L_2}(S^2,f^2)\ar[d]^{-\otimes_{\rm MF}-}\ar[r]^(.55){\Cokernel\times\Cokernel}&
\DDD_{\sg}^{\L_1}(R^1)\times\DDD_{\sg}^{\L_2}(R^2)\ar[d]^{-\otimes_k-}\\
\MF^{\L}(S,f)\ar[r]^{\Cokernel}&
\DDD_{\sg}^{\L}(R).
}\]
\end{itemize}
\end{proposition}

\begin{proof}
(a) Using the exact sequences $0\to Q^i\xrightarrow{\phi^i}P^i\to X^i\to0$ in $\mod^{\L_i}S^i$, we can construct the following commutative diagram of exact sequences in $\mod^{\L}R$.
{\small\[\xymatrix@R=1.5em@C=2em{
&&&0\ar[d]&&0\ar[d]\\
&&&0\ar[rr]\ar[d]&&(Q^1\otimes_kQ^2)(\c)\ar@{=}[r]\ar[d]^{{\left(\begin{smallmatrix}1&0\end{smallmatrix}\right)}}&(Q^1\otimes_kQ^2)(\c)\ar[r]&0\\
&0\ar[rr]&&{\begin{array}{l}(P^1\otimes_kQ^2)\\ \oplus(Q^1\otimes_kP^2)\end{array}}\ar@{=}[d]\ar[rr]^{\left(\begin{smallmatrix}\psi^1\otimes 1&-1\otimes \phi^2\\ 1\otimes \psi^2&\phi^1\otimes 1\end{smallmatrix}\right)}&&
{\begin{array}{l}(Q^1\otimes_kQ^2)(\c)\\ \oplus(P^1\otimes_kP^2)\end{array}}\ar[d]^(.6){{\left(\begin{smallmatrix}0\\ 1\end{smallmatrix}\right)}}\ar[r]&X\ar[r]&0\\
0\ar[r]&Q^1\otimes_kQ^2\ar[rr]^(.48){{\left(\begin{smallmatrix}\phi^1\otimes 1&1\otimes \phi^2\end{smallmatrix}\right)}}&&{\begin{array}{l}(P^1\otimes_kQ^2)\\ \oplus(Q^1\otimes_kP^2)\end{array}}\ar[rr]^{{\left(\begin{smallmatrix}-1\otimes \phi^2\\ \phi^1\otimes1\end{smallmatrix}\right)}}\ar[d]&&P^1\otimes_kP^2\ar[r]\ar[d]&X^1\otimes_kX^2\ar[r]&0\\
&&&0&&0
}\]}
Using the Snake Lemma, we have an exact sequence
\[0\to Q^1\otimes_kQ^2\xrightarrow{a} (Q^1\otimes_kQ^2)(\c)\to X\xrightarrow{\gamma_{X^1,X^2}}X^1\otimes_kX^2\to0.\]
Chasing the diagram, it is easy to check that the morphism $a$ is given by the multiplication by $f=f_1+f_2$.
Thus we have the desired exact sequence, and hence $X\in\CM^{\L}R$ is a Cohen-Macaulay approximation of $X^1\otimes_kX^2\in\mod^{\L}R$.

(b) For $i=1,2$, let $(\alpha^i,\beta^i):(\phi^i,\psi^i)\to(\eta^i,\theta^i)$ be a morphism of matrix factorizations. Then the following diagram is commutaive.
{\small\[\xymatrix@R=1.5em@C=1.5em{
{\begin{array}{c}(P^1\otimes_kQ^2)\\ \oplus(Q^1\otimes_kP^2)\end{array}}\ar@{=}[d]\ar[rr]|{{\left(\begin{smallmatrix}\psi^1\otimes 1&-1\otimes \phi^2\\ 1\otimes \psi^2&\phi^\otimes 1\end{smallmatrix}\right)}}\ar[ddr]|(.67){(\alpha^1\otimes\beta^2)\oplus(\beta^1\otimes\alpha^2)}&&{\begin{array}{c}(Q^1\otimes_kQ^2)(\c)\\\oplus(P^1\otimes_kP^2)\end{array}}\ar[r]\ar[d]_{\left(\begin{smallmatrix}0\\ 1\end{smallmatrix}\right)}\ar[ddr]|(.63){(\beta^1\otimes\beta^2)\oplus(\alpha^1\otimes\alpha^2)}&X^1\otimes_{\rm MF}X^2\ar[r]\ar[d]|{\gamma_{(\phi^1,\psi^1),(\phi^2,\psi^2)}}\ar[ddr]|(.63){\ \ \ (\alpha^1,\beta^1)\otimes_{\rm MF}(\alpha^2,\beta^2)}&0\\
{\begin{array}{c}(P^1\otimes_kQ^2)\\ \oplus(Q^1\otimes_kP^2)\end{array}}\ar[rr]|{{\left(\begin{smallmatrix}-1\otimes \phi^2\\ \phi^1\otimes1\end{smallmatrix}\right)}}\ar[ddr]
|(.3){(\alpha^1\otimes\beta^2)\oplus(\beta^1\otimes\alpha^2)}&&P^1\otimes_kP^2\ar[r]\ar[ddr]|(.25){\alpha^1\otimes\alpha^2}&X^1\otimes_kX^2\ar[r]\ar[ddr]|(.25){(\alpha^1,\beta^1)\otimes_k(\alpha^2,\beta^2)}&0\\
&{\begin{array}{c}(F^1\otimes_kG^2)\\ \oplus(G^1\otimes_kF^2)\end{array}}\ar@{=}[d]\ar[rr]|{{\left(\begin{smallmatrix}\psi^1\otimes 1&-1\otimes \phi^2\\ 1\otimes \psi^2&\phi^1\otimes 1\end{smallmatrix}\right)}}&&{\begin{array}{c}(G^1\otimes_kG^2)(\c)\\ \oplus(F^1\otimes_kF^2)\end{array}}\ar[r]\ar[d]^{\left(\begin{smallmatrix}0\\ 1\end{smallmatrix}\right)}&Y^1\otimes_{\rm MF}Y^2\ar[r]\ar[d]|{\gamma_{(\eta^1,\theta^1),(\eta^2,\theta^2)}}&0\\
&{\begin{array}{c}(F^1\otimes_kG^2)\\ \oplus(G^1\otimes_kF^2)\end{array}}\ar[rr]|{{\left(\begin{smallmatrix}-1\otimes \phi^2\\ \phi^1\otimes1\end{smallmatrix}\right)}}&&F^1\otimes_kF^2\ar[r]&Y^1\otimes_kY^2\ar[r]&0
}\]}
Thus we have a natural transformation $\gamma:\Cokernel(-\otimes_{\rm MF}-)\to(\Cokernel-\otimes_k\Cokernel-)$ of functors $\MF^{\L_1}(S^1,f^1)\times\MF^{\L_2}(S^2,f^2)\to\mod^{\L}R$.
After composing with the natural functor $\mod^{\L}R\to\DDD^{\bo}(\mod^{\L}R)\to\DDD_{\sg}^{\L}(R)$, the natural transformation $\gamma$ gives the desired isomorphism since $Q$ in the exact sequence in (a) belongs to $\proj^{\L}R$. Thus the assertion follows.
\end{proof}

Now we are ready to prove Theorem~\ref{tensor of MF is tensor}.

\begin{proof}[Proof of Theorem~\ref{tensor of MF is tensor}]
The left diagram is commutative by Lemma~\ref{tensor for sg}(b).
The right diagram is commutative up to natural isomorphism by Propositions~\ref{tensor product is CM approximation}(b) and \ref{MF CM2}.
\end{proof}

Now we denote by $A^{{\rm CM},i}$ the CM-canonical algebra of $(R^i,\L_i)$.
By Corollary~\ref{presentation of CM canonical}(b), we have $A^{\rm CM}=A^{{\rm CM},1}\otimes_kA^{{\rm CM},2}$.
Thus we have a bifunctor $-\otimes_k-:\DDD^{\bo}(\mod A^{{\rm CM},1})\times\DDD^{\bo}(\mod A^{{\rm CM},2})\to\DDD^{\bo}(\mod A^{\rm CM})$.

\begin{corollary}\label{tensor of CM canonical}
We have the following commutative diagram up to natural isomorphism.
\[\xymatrix@R=1em{
\DDD^{\bo}(\mod A^{{\rm CM},1})\times\DDD^{\bo}(\mod A^{{\rm CM},2})\ar[r]^(.62)\sim\ar[d]^{-\otimes_k-}&\underline{\CM}^{\L_1}R^1\times\underline{\CM}^{\L_2}R^2\ar[d]^{-\otimes_{\rm MF}-}\\
\DDD^{\bo}(\mod A^{\rm CM})\ar[r]^\sim&\underline{\CM}^{\L}R
}\]
\end{corollary}

\begin{proof}
Let $\de_i$ be the dominant element of $\L_i$ for $i=1,2$.
Then the natural map $\L_1\times\L_2\to\L$ induces a bijection $[0,\de_1]\times[0,\de_2]\simeq[0,\de]$. By Theorem~\ref{tilting in S^I}, we have the following commutative diagram up to natural isomorphism.
\[\xymatrix@R=1em{
\DDD^{\bo}(\mod A^{{\rm CM},1})\times\DDD^{\bo}(\mod A^{{\rm CM},2})\ar[r]\ar[d]^{-\otimes_k-}&\DDD^{\bo}(\mod^{\L_1}R^1)\times\DDD^{\bo}(\mod^{\L_2}R^2)\ar[d]^{-\otimes_k-}\\
\DDD^{\bo}(\mod A^{\rm CM})\ar[r]&\DDD^{\bo}(\mod^{\L}R)
}\]
Combining this with the diagram in Theorem~\ref{tensor of MF is tensor}, the assertion follows.
\end{proof}

\begin{proposition}\label{tensor of tilting is tilting}
Let $V^i\in\CM^{\L_i}R^i$ for $i=1,2$ and $V:=V^1\otimes_{\rm MF}V^2\in\CM^{\L}R$.
\begin{itemize}
\item[(a)] We have $\underline{\End}^{\L}_R(V)\simeq\underline{\End}^{\L_1}_{R^1}(V^1)\otimes_k\underline{\End}^{\L_2}_{R^2}(V^2)$.
\item[(b)] If $V^i$ gives a tilting object in $\underline{\CM}^{\L_i}R^i$ for $i=1,2$, then $V$ gives a tilting object in $\underline{\CM}^{\L}R$.
\end{itemize}
\end{proposition}

\begin{proof}
By Corollary~\ref{tensor of CM canonical}, we have
\begin{eqnarray*}
\underline{\End}^{\L}_R(V)&\simeq&\End_{\DDD^{\bo}(\mod A^{\rm CM})}(V^1\otimes_kV^2)=\End_{\DDD^{\bo}(\mod A^{{\rm CM},1})}(V^1)\otimes_k\End_{\DDD^{\bo}(\mod A^{{\rm CM},2})}(V^2)\\
&\simeq&\underline{\End}^{\L_1}_{R^1}(V^1)\otimes_k\underline{\End}^{\L_2}_{R^2}(V^2).
\end{eqnarray*}
To prove (b), we regard $V^i$ as a tilting object in $\DDD^{\bo}(\mod A^{{\rm CM},i})$. Then $V^1\otimes_kV^2$ is a tilting object in $\DDD^{\bo}(\mod A^{\rm CM})$, and the assertion follows.
\end{proof}

Now we give an explicit description of Kn\"orrer periodicity (Corollary~\ref{presentation of CM canonical}(d)) in terms of tensor products of matrix factorizations.

\begin{proposition}\label{explicit Knoerrer}
Let $(R^1,\L_1)$ and $(R^2=k[Y]/(Y^2),\L_2=\langle\y\rangle)$ be
GL hypersurfaces with weights $p_1,\ldots,p_{n}$ and $2$ respectively.
Let $(R,\L)$ be the corresponding GL hypersurface with weights $p_1,\ldots,p_{n},2$. 
\begin{itemize}
\item[(a)] The functor $-\otimes_{\rm MF}k:\underline{\CM}^{\L_1}R^1\to\underline{\CM}^{\L}R$ is a triangle equivalence.
\item[(b)] There exists an isomorphism $(\y)\simeq[1]$ of functors $\underline{\CM}^{\L}R\to\underline{\CM}^{\L}R$.
\end{itemize}
\end{proposition}

\begin{proof}
(a) By Corollary~\ref{tensor of CM canonical}, we have a commutative diagram
\[\xymatrix@R=1em{
\DDD^{\bo}(\mod A^{{\rm CM},1})\ar[r]^(.55)\sim\ar[d]^{-\otimes_kk}&\underline{\CM}^{\L_1}R^1\ar[d]^{-\otimes_{\rm MF}k}\\
\DDD^{\bo}(\mod A^{\rm CM})\ar[r]^(.55)\sim&\underline{\CM}^{\L}R
}\]
up to natural isomorphism. Since $A^{{\rm CM},1}\simeq A^{\rm CM}$, the left functors is a triangle equivalence. Thus so is the right functor.

(b) We have $k[1]=k(\y)$ in $\underline{\CM}^{\L_2}R^2$.
By \eqref{[1] and (x+y)}, we have isomorphisms
\[(-\otimes_{\rm MF}k)[1]\simeq(-\otimes_{\rm MF}(k[1]))\simeq(-\otimes_{\rm MF}(k(\y))\simeq(-\otimes_{\rm MF}k)(\y)\]
of functors $\underline{\CM}^{\L_1}R^1\to\underline{\CM}^{\L}R$.
Since $-\otimes_{\rm MF}k:\underline{\CM}^{\L_1}R^1\to\underline{\CM}^{\L}R$ is an equivalence, we have the desired isomorphism.
\end{proof}

\section{The tilting object via tensor products}\label{subsect tilting via tensor}

In the rest of this section, for an arbitrary Geigle-Lenzing complete intersection $(R,\L)$ with $n=d+2$, we give $\L$-graded matrix factorizations of the indecomposable direct summands of
the tilting object $T^{\rm CM}$ in $\underline{\CM}^{\L}R$ given in Theorem~\ref{CM tilting}(b). 
We also give another tilting object $U^{\rm CM}$ in $\underline{\CM}^{\L}R$, which has a simpler matrix factorization than $T^{\rm CM}$.
These observations give another proof of our result that $T^{\rm CM}$ (and also $U^{\rm CM}$) is a tilting object in $\underline{\CM}^{\L}R$.

Let us start with considering iterated tensor products. It is easy to check that the tensor product of matrix factorizations is associative, that is, there is a functorial isomorphism
\[((\phi^1,\psi^1)\otimes_{\rm MF}(\phi^2,\psi^2))\otimes_{\rm MF}(\phi^3,\psi^3)\simeq
(\phi^1,\psi^1)\otimes_{\rm MF}((\phi^2,\psi^2)\otimes_{\rm MF}(\phi^3,\psi^3)).\]
We give the following explicit description of the iterated tensor product of matrix factorizations.

\begin{proposition}\label{describe n tensor product}
For $1\le i\le n$, let $(R^i=S^i/(f_i),\L_i)$ be a GL complete intersection and $(\phi^i:Q^i\to P^i,\ \psi^i:P^i(-\c_i)\to Q^i)\in\MF^{\L_i}(S^i,f_i)$. Then the tensor product
\[(\phi:Q\to P,\ \psi:P(-\c)\to Q):=(\phi^1,\psi^1)\otimes_{\rm MF}\cdots\otimes_{\rm MF}(\phi^n,\psi^n)\]
is given by
\begin{eqnarray*}
&{\displaystyle Q = \bigoplus_{\substack{I \subseteq \{1, \ldots, n\} \\ |I|\ {\rm odd} }} \underbrace{\left(\bigotimes_{i\in I}Q^i\otimes\bigotimes_{i\notin I}P^i\right)\left(\frac{|I| - 1}{2} \c\right)}_{=: Q_I},\ P = \bigoplus_{\substack{I \subseteq \{1, \ldots, n\} \\ |I|\ {\rm even} }} \underbrace{\left(\bigotimes_{i\in I}Q^i\otimes\bigotimes_{i\notin I}P^i\right)\left(\frac{|I|}{2} \c\right)}_{=: P_I}}&\\
&\phi_{I,J} \colon Q_I \to P_J \colon
\begin{cases}
(-1)^{\#\{j \notin I \mid j < i \}} \phi^i & I = J \sqcup \{i\} \\
(-1)^{\#\{j \notin I \mid j < i \}} \psi^i & J = I \sqcup \{i\} \\
0 & \text{otherwise} \end{cases},&
\end{eqnarray*}
and $\psi_{I,J} \colon P_I(-\c) \to Q_J$ is defined by the identical formulas.
\end{proposition}

\begin{proof}
The assertion follows easily by indution on $n$.
\end{proof}

Recall that the dominant element in the case $n = d+2$ is given by $\de = \sum_{i=1}^n(p_i - 2) \x_i$ by \eqref{another delta}.
Let $\s:=\sum_{i=1}^n\x_i$, and we consider the interval $[\s,\s+\de]$.

The following is the first main result in this section, where $\rho$ denotes the natural functor $\DDD^{\bo}(\mod^{\L}R) \to \DDD_{\sg}^{\L}(R)\simeq\underline{\CM}^{\L}R$.

\begin{theorem}\label{thm.mf_for_E}
Let $(R=S/(f),\L)$ be a GL hypersurface of dimension $d+1$ with $n=d+2$ weights $p_1,\ldots,p_n$. For each element $\l=\sum_{i=1}^n\ell_i\x_i$ in the interval $[\s,\s+\de]$, let
\[E^{\l}:=R/(X_i^{\ell_i}\mid 1\le i\le n)\in\mod^{\L}R\ \mbox{ and }\ E := \bigoplus_{\l \in [\s,\s+\de]}E^{\l}.\]
Then the following assertions hold.
\begin{itemize}
\item[(a)] $U^{\rm CM}:=\rho(E)$ is a tilting object in $\underline{\CM}^{\L}R$. 
\item[(b)] An $\L$-graded matrix factorization corresponding to $\rho(E^{\l})$ is given by $(\phi=(\phi_{I,J})_{I,J}:Q\to P,\ \psi=(\psi_{I,J})_{I,J}:P(-\c)\to Q)$, where
\begin{eqnarray*}
&{\displaystyle Q = \bigoplus_{\substack{I \subseteq \{1, \ldots, n\} \\ |I|\ \text{\rm odd} }} \underbrace{R\left(\frac{|I| - 1}{2} \c - \sum_{i \in I} \ell_i \x_i\right)}_{=: Q_I},\ P = \bigoplus_{\substack{I \subseteq \{1, \ldots, n\} \\ |I|\ \text{\rm even} }} \underbrace{R\left(\frac{|I|}{2} \c - \sum_{i \in I} \ell_i \x_i\right)}_{=: P_I},}&\\
&\phi_{I,J} \colon Q_I \to P_J \colon
\begin{cases}
(-1)^{\#\{j \notin I \mid j < i \}} X_i^{\ell_i} & I = J \sqcup \{i\} \\
(-1)^{\#\{j \notin I \mid j < i \}} \lambda_i X_i^{p_i - \ell_i} & J = I \sqcup \{ i \} \\
0 & \text{otherwise}
\end{cases},&
\end{eqnarray*}
and $\psi_{I, J} \colon P_I(-\c) \to Q_J$ is given by the identical formulas.
\item[(c)] The equivalence $\mod^{[0,\de]}R\simeq\mod A^{\rm CM}$ in Proposition~\ref{equivalences for I-canonical} sends $E$ to $D(A^{\rm CM}$).
\end{itemize}
\end{theorem}

\begin{proof}
For $1\le i\le n$, let $S^i=k[X_i]$, $(R^i,\L_i)=(S^i / (\lambda_i X_i ^ {p_i}),\Z\x_i)$ and $\c_i = p_i \x_i$.
(Note that the value of $\lambda_i$ does not matter for the ring $R^i$.) Let
\[(\phi^i,\psi^i):=\left(X_i^{\ell_i}:S^i(-\ell_i\x_i) \to S^i,\ \lambda_i X_i^{p_i - \ell_i}:S^i(-\c_i) \to S^i(-\ell_i\x_i)\right)\in\MF^{\L_i}(S^i,\lambda_iX_i^{p_i}).\]
Since
$E^{\l}=(S^1/(\lambda_1X_1^{\ell_1}))\otimes_k\cdots\otimes_k(S^n/(\lambda_nX_n^{\ell_n}))$
holds, Theorem~\ref{tensor of MF is tensor} shows that
\[(\phi^1,\psi^1)\otimes_{\rm MF}\cdots\otimes_{\rm MF}(\phi^n,\psi^n)\in\MF^{\L}(S,f)\]
gives an $\L$-graded matrix factorization of $\rho(E^{\l})$. Thus the assertion (b) follows from Proposition~\ref{describe n tensor product}. Moreover, $\underline{\CM}^{\L_i}R^i$ has a tilting object $\bigoplus_{\ell=1}^{p_i-1}S^i / (X_i ^\ell)$ by Example~\ref{CM-canonical for d=-1}(a).
Thus the assertion (a) follows from Proposition~\ref{tensor of tilting is tilting}(b). The assertion (c) is clear from the shape of $E^{\l}$.
\end{proof}

We illustrate this explicit description of the matrices in two (relatively) small examples.

\begin{example}
\begin{itemize}
\item[(a)] Let $d=1$ and $n=3$. Then the $\L$-graded matrix factorization of $\rho(E^{\l})$ is given by 
\begin{align*}
P & = \underbrace{R}_{P_{\varnothing}} \oplus \underbrace{R(\c - \ell_1\x_1 - \ell_2\x_2)}_{P_{\{1,2\}}} \oplus \underbrace{R(\c - \ell_1\x_1 - \ell_3\x_3)}_{P_{\{1,3\}}} \oplus \underbrace{R(\c - \ell_2\x_2 - \ell_3\x_3)}_{P_{\{2,3\}}},\\
Q & = \underbrace{R(- \ell_1\x_1)}_{Q_{\{1\}}} \oplus \underbrace{R(- \ell_2\x_2)}_{Q_{\{2\}}} \oplus \underbrace{R(-\ell_3\x_3)}_{Q_{\{3\}}} \oplus \underbrace{R(\c - \ell_1\x_1 - \ell_2\x_2 - \ell_3\x_3)}_{Q_{\{1,2,3\}}},
\end{align*}
together with the matrices
\begin{align*}
\phi & = \left[\begin{array}{cccc}
X_1^{\ell_1} & \lambda_2 X_2^{p_2-\ell_2} & -\lambda_3 X_3^{p_3-\ell_3} & 0 \\
-X_2^{\ell_2} & \lambda_1X_1^{p_1-\ell_1} & 0 & -\lambda_3 X_3^{p_3-\ell_3} \\
X_3^{\ell_3} & 0 & \lambda_1X_1^{p_1-\ell_1} & -\lambda_2 X_2^{p_2-\ell_2} \\
0 & X_3^{\ell_3} & X_2^{\ell_2} & X_1^{\ell_1}
\end{array}\right],\\
\psi & = \left[\begin{array}{cccc}
\lambda_1 X_1^{p_1-\ell_1} & -\lambda_2X_2^{p_2-\ell_2} & \lambda_3 X_3^{p_3-\ell_3} & 0 \\
X_2^{\ell_2} & X_1^{\ell_1} & 0 & \lambda_3X_3^{p_3-\ell_3} \\
-X_3^{\ell_3} & 0 & X_1^{\ell_1} & \lambda_2X_2^{p_2-\ell_2} \\
0 & -X_3^{\ell_3} & -X_2^{\ell_2} & \lambda_1X_1^{p_1-\ell_1}
\end{array}\right].
\end{align*}
\item[(b)] Let $d=2$ and $n=4$. Then the $\L$-graded matrix factorization of $\rho(E^{\l})$ is given by 
\begin{align*}
P &= \underbrace{R}_{P_{\varnothing}} \oplus \underbrace{R(\c - \ell_1\x_1 - \ell_2\x_2)}_{P_{\{1,2\}}} \oplus \underbrace{R(\c - \ell_1\x_1 - \ell_3\x_3)}_{P_{\{1,3\}}} \oplus \underbrace{R(\c - \ell_1\x_1 - \ell_4 \x_4)}_{P_{\{1,4\}}} \oplus \underbrace{R(\c -\ell_2\x_2 - \ell_3\x_3)}_{P_{\{2,3\}}} \\
& \quad \oplus \underbrace{R(\c - \ell_2\x_2 - \ell_4\x_4)}_{P_{\{2, 4\}}} \oplus \underbrace{R(\c - \ell_3\x_3 - \ell_4\x_4)}_{P_{\{3,4\}}} \oplus \underbrace{R(2 \c - \ell_1\x_1 - \ell_2\x_2 - \ell_3\x_3 - \ell_4\x_4)}_{P_{\{1,2,3,4\}}}\\
Q & = \underbrace{R( - \ell_1\x_1)}_{Q_{\{1\}}} \oplus \underbrace{R( - \ell_2\x_2)}_{Q_{\{2\}}} \oplus \underbrace{R( - \ell_3\x_3)}_{Q_{\{3\}}} \oplus \underbrace{R( - \ell_4\x_4)}_{Q_{\{4\}}} \oplus \underbrace{R(\c - \ell_1\x_1 - \ell_2\x_2 - \ell_3\x_3)}_{Q_{\{1,2,3\}}} \\
& \quad \oplus \underbrace{R(\c - \ell_1\x_1 - \ell_2\x_2 - \ell_4\x_4)}_{Q_{\{1,2,4\}}}\oplus \underbrace{R(\c - \ell_1\x_1 - \ell_3\x_3 - \ell_4\x_4)}_{Q_{\{1,3,4\}}} \oplus \underbrace{R(\c - \ell_2\x_2 - \ell_3\x_3 - \ell_4\x_4)}_{Q_{\{2,3,4\}}}
\end{align*}
together with the matrices \\
\scalebox{.7}{\parbox{\linewidth}{
\begin{align*}
\phi & = \left[\begin{array}{cccccccc}
X_1^{\ell_1} & \lambda_2 X_2^{p_2-\ell_2} & - \lambda_3 X_3^{p_3-\ell_3} & \lambda_4 X_4^{p_4-\ell_4} & 0 & 0 & 0 & 0 \\
- X_2^{\ell_2} & \lambda_1 X_1^{p_1-\ell_1} & 0 & 0 & - \lambda_3 X_3^{p_3-\ell_3} & \lambda_4 X_4^{p_4-\ell_4} & 0 & 0 \\
X_3^{\ell_3} & 0 & \lambda_1 X_1^{p_1-\ell_1} & 0 & -\lambda_2 X_2^{p_2-\ell_2} & 0 & \lambda_4 X_4^{p_4-\ell_4} & 0 \\ 
- X_4^{\ell_4} & 0 & 0 & \lambda_1 X_1^{p_1-\ell_1} & 0 & -\lambda_2 X_2^{p_2-\ell_2} & \lambda_3 X_3^{p_3-\ell_3} & 0 \\
0 & X_3^{\ell_3} & X_2^{\ell_2} & 0 & X_1^{\ell_1} & 0 & 0 & \lambda_4 X_4^{p_4-\ell_4} \\
0 & - X_4^{\ell_4} & 0 & X_2^{\ell_2} & 0 & X_1^{\ell_1} & 0 & \lambda_3 X_3^{p_3-\ell_3} \\
0 & 0 & - \lambda_4 X_4^{p_4-\ell_4} & - \lambda_3 X_3^{p_3-\ell_3} & 0 & 0 & \lambda_1 X_1^{p_1-\ell_1} & X_2^{\ell_2} \\
0 & 0 & 0 & 0 & - \lambda_4 X_4^{p_4-\ell_4} & - \lambda_3 X_3^{p_3-\ell_3} & - \lambda_2 X_2^{p_2-\ell_2} & X_1^{\ell_1} \\
\end{array}\right], \\
\psi & = \left[\begin{array}{cccccccc}
\lambda_1 X_1^{p_1-\ell_1} & - \lambda_2 X_2^{p_2-\ell_2} & \lambda_3 X_3^{p_3-\ell_3} & - \lambda_4 X_4^{p_4-\ell_4} & 0 & 0 & 0 & 0 \\
X_2^{\ell_2} & X_1^{\ell_1} & 0 & 0 & \lambda_3 X_3^{p_3-\ell_3} & - \lambda_4 X_4^{p_4-\ell_4} & 0 & 0 \\
- X_3^{\ell_3} & 0 & X_1^{\ell_1} & 0 & \lambda_2 X_2^{p_2-\ell_2} & 0 & - \lambda_4 X_4^{p_4-\ell_4} & 0 \\
X_4^{\ell_4} & 0 & 0 & X_1^{\ell_1} & 0 & \lambda_2 X_2^{p_2-\ell_2} & - \lambda_3 X_3^{p_3-\ell_3} & 0 \\
0 & - X_3^{\ell_3} & - X_2^{\ell_2} & 0 & \lambda_1 X_1^{p_1-\ell_1} & 0 & 0 & - \lambda_4 X_4^{p_4-\ell_4} \\
0 & X_4^{\ell_4} & 0 & - X_2^{\ell_2} & 0 & \lambda_1 X_1^{p_1-\ell_1} & 0 & - \lambda_3 X_3^{p_3-\ell_3} \\
0 & 0 & X_4^{\ell_4} & X_3^{\ell_3} & 0 & 0 & \lambda_1 X_1^{p_1-\ell_1} & - \lambda_2 X_2^{p_2-\ell_2} \\
0 & 0 & 0 & 0 & X_4^{\ell_4} & X_3^{\ell_3} & X_2^{\ell_2} & X_1^{\ell_1} \\
\end{array}\right].
\end{align*}
}}
\end{itemize}
\end{example}

Later we will need the information below on the Cohen-Macaulay module $\rho(E^{\l})$, where we refer to Section~\ref{section: GL CI 2} for the definition of $\rank$. For $\l \in [\s,\s+\de]$, let
\[U^{\l}=\rho(E^{\l})\in\CM^{\L}R.\]
This has an $\L$-graded matrix factorization $(\phi\colon Q\to P,\ \psi\colon P(-\c)\to Q)$ given in Theorem~\ref{thm.mf_for_E}. Thus $U^{\l}$ is the image of $\psi(\c)\colon P\to Q(\c)$. For $m$ with $1\le m\le n$, we define decompositions $P=P^{\ni m}\oplus P^{\not\ni m}$ and $Q=Q^{\ni m}\oplus Q^{\not\ni m}$ by
\[P^{\ni m} = \bigoplus_{\substack{I \subseteq \{1, \ldots, n\} \\ |I| \text{ even},\ m \in I}} P_I \qquad \text{and} \qquad P^{\not\ni m} =  \bigoplus_{\substack{I \subseteq \{1, \ldots, n\} \\ |I| \text{ even},\ m\notin I }} P_I \]
and similar equalities for $Q$. 

\begin{proposition}\label{rank 2^d}
\begin{itemize}
\item[(a)] $\rank U^{\l}=2^d$.
\item[(b)] The compositions $P^{\ni m}\to P\to U^{\l}$ and $P^{\not\ni m}\to P\to U^{\l}$ are injective and their cokernels have rank 0.
\item[(c)] The compositions $U^{\l}\to Q(\c)\to Q^{\ni m}(\c)$ and $U^{\l}\to Q(\c)\to Q^{\not\ni m}(\c)$ are injective and their cokernels have rank 0.
\end{itemize}
\end{proposition}

\begin{proof}
(a) The component of $\psi(\c)$ mapping $P^{\ni m}$ to $Q^{\not\ni m}(\c)$ (respectively, $P^{\not\ni m}$ to $Q^{\ni m}(\c)$) is given by multiplication by $\pm X_m^{\ell_m}$ (respectively, $\pm\lambda_mX_m^{p_m-\ell_m}$), and hence injective. 
Thus $\rank U^{\l} \geq \rank P^{\ni m}  = 2^d$.
The same type of argument shows that $\rank(\Image\phi) \geq 2^d$. Since $\rank U^{\l}+ \rank(\Image\phi)= \rank P = 2^{d+1}$,
so both inequalities above are in fact equalities.

(b)(c) By the above argument, the kernel and the cokernel of these morphisms have rank $0$.
Then the kernel has to be $0$ by Proposition~\ref{non-zero is injective}(a).
\end{proof}

Using matrix factorizations, we can describe how the non-free Cohen-Macaulay modules are connected to the free ones in the Auslander-Reiten quiver $\mathfrak{A}(\CM^{\L}R)$.

\begin{lemma}\label{action of xi}
The Auslander-Reiten quiver $\mathfrak{A}(\CM^{\L}R)$ has the following arrows.
\begin{itemize}
\item For each $\x\in\L$, a unique arrow $\rho(E^{\s})(\x)[-1]\to R(\x)$ ending at $R(\x)$, and a unique arrow $R(\x)\to\rho(E^{\s+\de})(\x)$ starting at $R(\x)$.
\end{itemize}
In particular, for each $1\le i\le n$, there is a unique arrow $\rho(E^{\s+(p_i-2)\x_i})\to R(\x_i)$ ending at $R(\x_i)$. Thus the degree shift functor $(\x_i)$ sends $\tau(\rho(E^{\s+\de}))$ to $\rho(E^{\s+(p_i-2)\x_i})$.
\end{lemma}

\begin{proof}
Thanks to Theorem~\ref{AR quiver in general}, it suffices to show $\rho(k)(-\w)[-d]=\rho(E^{\s+\de})$. Recall that $A^{\rm CM}=\bigotimes_{i=1}^nk\A_{p_i-1}$. The equivalence $\mod^{[0,\de]}R\simeq\mod A^{\rm CM}$ sends $k$ to the unique simple injective $A^{\rm CM}$-module, and $E^{\s+\de}$ to the corresponding projective $A^{\rm CM}$-module. Since $(\w)[d]$ is the Serre functor of $\underline{\CM}^{\L}R$, we have $\rho(k)(-w)[-d]=\rho(E^{\s+\de})$ as desired.

We have an exact sequence $0\to k[X_i]/(X_i^{p_i-1})\to k[X_i]/(X_i^{p_i})(\x_i)\to k(\x_i)\to0$ in $\mod^{\L}R$, and the middle term has finite projective dimension by Lemma~\ref{perfect}(a). Thus $\rho(k)(\x_i)[-1]\simeq\rho(k[X_i]/(X_i^{p_i-1}))=\rho(E^{\s+(p_i-2)\x_i})$. The assertions follow from the first part.
\end{proof}

As an application of Theorem~\ref{thm.mf_for_E}, we prove Theorem~\ref{AR quiver for CM finite}.

\begin{proof}[Proof of Theorem~\ref{AR quiver for CM finite}]
By Theorem~\ref{CM tilting} and Lemma~\ref{tensor product}, we have triangle equivalences
\[\underline{\CM}^{\L}R\simeq\DDD^{\bo}(\mod A^{\rm CM})\simeq\DDD^{\bo}(\mod kQ)\]
for $Q=\A_{p-1}$ if the weights are $(2,\ldots,2,2,p)$, $Q=\D_4$ if the weights are $(2,\ldots,2,3,3)$, 
$Q=\E_6$ if the weights are $(2,\ldots,2,3,4)$, and $Q=\E_8$ if the weights are $(2,\ldots,2,3,5)$.
Thus the stable Auslander-Reiten quiver of $\underline{\CM}^{\L}R$ has the form $\Z Q$.
By Theorem~\ref{thm.mf_for_E}(c), the positions of the direct summands $\rho(E^{\s+(a-1)\x_{n-1}+(b-1)\x_n})$
of $U^{\rm CM}$ are the positions of the indecomposable injective $A^{\rm CM}$-modules. 
They are given as in Figure~\ref{CMfiniteFig} by an explicit calculation of the triangle equivalence
$\DDD^{\bo}(\mod A^{\rm CM})\simeq\DDD^{\bo}(\mod kQ)$ (e.g.\ \cite[Fig.\,3,4]{KLM}). By Lemma~\ref{action of xi}, for each $1\le i\le n$, the position of $R(\x_i)$ is given as in Figure~\ref{CMfiniteFig}.
\end{proof}

We end this section with giving an explicit description of $\L$-graded matrix factorization of $T^{\rm CM}$, which is very close to Theorem~\ref{thm.mf_for_E}.
Again let $s=\sum_{i=1}^n\x_i$.

\begin{corollary} \label{thm.mf_for_F}
For $\l = \sum_{i=1}^n \ell_i \x_i \in [\s, \s+\de]$, let
\[ F^{\l} := (R/(X_i^{\ell_i} \mid 1\le i \le n))(\l-\s-\de) \in \mod^{\L}R\ \mbox{ and }\ F=\bigoplus_{\l\in [\s,\s+\de]}F^{\l}. \]
\begin{itemize}
\item[(a)] We have $T^{\rm CM} = \rho(F)$ and $U^{\rm CM}=T^{\rm CM}(\w)[d]$. They are tilting objects in $\underline{\CM}^{\L}R$.
\item[(b)] An $\L$-graded matrix factorization corresponding to $\rho(F^{\l})$ is given by $(\phi=(\phi_{I,J})_{I,J}:Q\to P,\ \psi=(\psi_{I,J})_{I,J}:P(-\c)\to Q)$, where
\begin{eqnarray*}
&{\displaystyle Q = \bigoplus_{\substack{I \subseteq \{1, \ldots, n\} \\ |I|\ {\rm odd} }} \underbrace{R\left(\frac{|I| - 1}{2} \c + \sum_{i \not \in I} \ell_i \x_i - \s-\de\right)}_{=: Q_I},\ \ \ 
P = \bigoplus_{\substack{I \subseteq \{1, \ldots, n\} \\ |I|\ {\rm even} }} \underbrace{R\left(\frac{|I|}{2} \c + \sum_{i \not \in I} \ell_i \x_i - \s-\de\right)}_{=: P_I},}&\\
&\phi_{I,J} \colon Q_I \to P_J \colon 
\begin{cases}
(-1)^{\#\{j \notin I \mid j < i \}} X_i^{\ell_i+1} & I = J \sqcup \{i\} \\
(-1)^{\#\{j \notin I \mid j < i \}} \lambda_i X_i^{p_i - \ell_i-1} & J = I \sqcup \{ i \} \\
(0 & \text{otherwise}
\end{cases},&
\end{eqnarray*}
and $\psi_{I,J} \colon P_I(-\c) \to Q_J$ is defined by the identical formulas.
\item[(c)] The equivalence $\mod^{[0,\de]}R\simeq\mod A^{\rm CM}$ in Proposition~\ref{equivalences for I-canonical} sends $F$ to $A^{\rm CM}$.
\end{itemize}
\end{corollary}

\begin{proof}
(a)(c) Notice that $F^{\l}=R(\l-\s-\de)_{[0,\de]}$ and $F=U^{[0,\de]}$ hold, where $(-)_{[0,\de]}:\mod^{\L}R\to\mod^{[0,\de]}R$ is the restriction functor. Thus $T^{\rm CM}=\rho(U^{[0,\de]})=\rho(F)$ and the assertion (c) hold.

By (c) and Theorem~\ref{thm.mf_for_E}(c), the equivalence $\mod^{[0,\de]}R\simeq\mod A^{\rm CM}$ sends $F$ and $E$ to $A^{\rm CM}$ and $D(A^{\rm CM})=\nu(A^{\rm CM})$ respectively. By the commutative diagram in Proposition~\ref{property of stable tilting}, we have $U^{\rm CM}=\rho(E)=\rho(F)(\w)[d]=T^{\rm CM}(\w)[d]$. By Theorem~\ref{thm.mf_for_E}(a), both $U^{\rm CM}$ and $T^{\rm CM}$ are tilting objects in $\underline{\CM}^{\L}R$.

(b) Applying $(\l-\s+\de)$ to the matrix factorization of $\rho(E^{\l})$ given in Theorem~\ref{thm.mf_for_E}, we obtain the assertion.
\end{proof}

%
%
%

\part{Geigle-Lenzing projective spaces}

\chapter{Geigle-Lenzing projective spaces}\label{section: GL space}

Throughout this chapter, we assume $d\ge1$.
Let $R$ be a Geigle-Lenzing complete intersection of dimension $d+1$ over a field $k$
associated with linear forms $\ell_1, \ldots, \ell_n$ and
weights $p_1, \ldots, p_n$.
Let $\mod^{\L}R$ be the category of $\L$-graded finitely generated
$R$-modules, and let $\mod_0^{\L}R$ be the full subcategory of 
$\mod^{\L}R$ consisting of finite dimensional modules.

\medskip\noindent
{\bf Geigle-Lenzing projective space.}
In the setup above, the category of \emph{coherent sheaves} on
\emph{Geigle-Lenzing} (\emph{GL}) \emph{projective space}
$\X$ of dimension $d$ is defined as the quotient category
\[ \coh \X = \qgr^{\L}R:=\mod^{\L}R/\mod_0^{\L}R\]
of $\mod^{\L}R$ by its Serre subcategory $\mod^{\L}_0R$.

We denote by $\pi:\mod^{\L}R\to\coh\X$ the natural functor. The object
\[\OO:=\pi(R)\]
is called the \emph{structure sheaf} of $\X$.
We have a triangle equivalence (e.g.\ \cite[3.2]{Miyachi})
\[\DDD^{\bo}(\mod^{\L} R) / \DDD^{\bo}_{\mod^{\L}_{0}R}(\mod^{\L} R)
\simeq\DDD^{\bo}(\coh \X),\]
and we denote by $\pi:\DDD^{\bo}(\mod^{\L}R)\to\DDD^{\bo}(\coh\X)$
the natural functor.

Our category $\coh\X$ has a geometric interpretation as the
the category of coherent sheaves $\coh\X$ on the quotient stack $\X=[X/G]$,
where $X$ is the punctured spectrum $X=\Spec R\setminus\{R_+\}$ and $G$
is the group scheme $G=\Spec k[\L]$ acting on $X$ (cf. \cite[Proposition 2.17]{O}). 
In fact, the category $\coh^G(\Spec R)$ of $G$-equivalent coherent sheaves on $X$
is equivalent to $\mod^{\L}R$, and the full subcategory consisting of sheaves
supported at $\{R_+\}$ is $\mod^{\L}_0R$.
Therefore the category $\coh^GX$ of $G$-equivariant coherent sheaves on $X$
is given by $\mod^{\L}R/\mod_0^{\L}R$, and this is nothing but the category of
coherent sheaves on the stack $\X$.

\begin{remark}\label{GL order}
In \cite{IL,LO}, Lerner and the second and the fourth authors introduce an order $A$ on the
projective space $\P^d$ called a \emph{Geigle-Lenzing order}
associated with linear forms $\ell_1,\ldots,\ell_n$ and weights $p_1,\ldots,p_n$.
They prove that there exists an equivalence
\[\coh\X\simeq\mod A.\]
This gives another approach to GL projective spaces, which will not be used in this paper.
\end{remark}

\section{Basic properties}
In this section, we give some basic properties of the categories $\coh\X$ and
$\DDD^{\bo}(\coh\X)$. Let us start with recalling the notion of \emph{local cohomology} \cite{BH}
which relates the categories $\mod^{\L}R$ and $\coh\X$.

\begin{definition}[Local cohomology]
For all $X\in\mod^{\L}R$, $i\ge0$ and $\x\in\L$, let
\begin{eqnarray*}
H_{\mm}^i(X)_{\x}&:=&\varinjlim_j\Ext_{\mod^{\L}R}^i(R/\mm^j,X(\x)),
\end{eqnarray*}
where $\mm:=R_+=\bigoplus_{\x>0}R_{\x}$.
Let $H_{\mm}^i(X):=\bigoplus_{\x\in\L}H^i_{\mm}(X)_{\x}$.
\end{definition}

The following result is fundamental, where $D=\Hom_k(-,k)$ is the $k$-dual.

\begin{proposition}\cite[3.6.19]{BH}\label{local duality} \emph{(local duality)}
For all $X\in\mod^{\L}R$ and $i\ge0$, we have an isomorphism in $\mod^{\L}R$:
\[\Ext^{d+1-i}_R(X,R(\w))\simeq
\bigoplus_{\x\in\L}D(H_{\mm}^i(X)_{-\x}).\]
\end{proposition}

Evaluating for $X=k$, we get the equality in Definition~\ref{define w}.
Evaluating for $X=R$, we have
\begin{equation}\label{local cohomology}
H_{\mm}^i(R)_{\x}\simeq\left\{
\begin{array}{ll}
D(R_{\w-\x})&\mbox{if $i=d+1$,}\\
0&\mbox{otherwise.}
\end{array}\right.
\end{equation}
Another immediate consequence is the following description of $\CM^{\L}_iR$:
\begin{equation}\label{describe CM by LC}
\CM^{\L}_iR=\{X\in\mod^{\L}R\mid\forall j\neq i,\ H^j_{\mm}(X)=0\}.
\end{equation}

The following exact sequence is basic.

\begin{proposition}\label{X and H} \cite[4.1.5]{BV}
For all $X\in\mod^{\L}R$, we have an exact sequence
\[0\to H^0_{\mm}(X)_{\x}\to X_{\x}\to\Hom_{\X}(\OO,X(\x))\to H^1_{\mm}(X)_{\x}\to0\]
and an isomorphism $\Ext_{\X}^i(\OO,X(\x))\to H^{i+1}_{\mm}(X)_{\x}$ for all $i\ge1$.
\end{proposition}

We have the following useful description of extension spaces between line bundles.

\begin{proposition}\label{vanishing}
For all $\x,\y\in\L$ and $i\in\Z$, we have
\[\Ext_{\X}^i(\OO(\x),\OO(\y))=\left\{
\begin{array}{ll}
R_{\y-\x}&\mbox{if $i=0$,}\\
D(R_{\x-\y+\w})&\mbox{if $i=d$,}\\
0&\mbox{otherwise.}
\end{array}\right.\]
\end{proposition}

\begin{proof}
For $i>0$, we have $\Ext_{\X}^i(\OO(\x),\OO(\y))=H^{i+1}_{\mm}(R)_{\y-\x}$
by Proposition~\ref{X and H}. Thus the assertion follows from \eqref{local cohomology}.

For $i=0$, we have $\Hom_{\X}(\OO(\x),\OO(\y))=R_{\y-\x}$
by Proposition~\ref{X and H} and \eqref{local cohomology}.
\end{proof}

Now we give a list of fundamental properties of our category $\coh\X$.

\begin{theorem}\label{Serre}
\begin{itemize}
\item[(a)] $\coh\X$ is a Noetherian abelian category.
\item[(b)] $\coh\X$ has global dimension $d$.
\item[(c)] $\Ext^i_{\X}(X,Y)$ is a finite dimensional $k$-vector space for all $i\ge0$
and $X,Y\in\coh\X$. 
\item[(d)] $\Hom_{\DDD^{\bo}(\coh\X)}(X,Y)$ is
a finite dimensional $k$-vector space for all $X,Y\in\DDD^{\bo}(\coh\X)$.
\item[(e)] We have $\DDD^{\bo}(\coh\X)=\thick\{\OO(\x)\mid\x\in\L\}$.
\item[(f)] \emph{(Auslander-Reiten-Serre duality)}
We have a functorial isomorphism for $X, Y\in\DDD^{\bo}(\coh\X)$:
\[\Hom_{\DDD^{\bo}(\coh\X)}(X,Y)\simeq D\Hom_{\DDD^{\bo}(\coh\X)}(Y,X(\w)[d]).\]
In other words, $\DDD^{\bo}(\coh\X)$ has a Serre functor $(\w)[d]$.
\end{itemize}
\end{theorem}


\begin{proof}
(a) See e.g.\ \cite[5.8.3]{Popescu}.

(e) This is an immediate consequence of Proposition~\ref{generating derived category}.

(c)(d) It is enough to prove (d). By (e), it is enough to show that
$\Ext^i_{\X}(\OO(\x),\OO(y))$ is finite dimensional for all
$\x,\y\in\L$ and $i\ge0$. This was shown in Proposition~\ref{vanishing}.

(f) We will give a complete proof in Section~\ref{section: Serre duality}.

(b) For all $i >d$ and $X,Y \in \coh \X$ we have, by (f), that
\[\Ext_{\X}^i(X,Y) \simeq D\Ext_{\X}^{d-i}(Y,X(\w)) = 0. \qedhere\]
\end{proof}

Recall that we say that a full subcategory $\CC$ of an abelian category $\AA$ 
\emph{generates} $\AA$ if any object in $\AA$ is a factor object of some object in $\CC$.

We have the following basic results.

\begin{proposition}\label{Serre vanishing}
\begin{itemize}
\item[(a)] $\add\{\OO(-\x)\mid \x\in\L_+\}$ generates $\coh\X$.
\item[(b)] \emph{(Serre vanishing)} For any $X\in\coh\X$, there exists 
$\a\in\L$ such that $\Ext^i_{\X}(\OO,X(\x))=0$ holds for any $i>0$ and any 
$\x\in\L$ satisfying $\x\ge\a$.
\end{itemize}
\end{proposition}

\begin{proof}
(a) For $X\in\mod^{\L}R$, let $X_{\L_+}:=\bigoplus_{\x\in\L_+}X_{\x}$ be
a subobject of $X$ in $\mod^{\L}R$. Since $X_{\L_+}$ is finitely generated,
there exists a surjection $f:P\to X_{\L_+}$ in $\mod^{\L}R$ with $P\in\proj^{\L_+}R$.
Then $\pi(f):\pi(P)\to\pi(X)$ is an epimorphism in $\coh\X$ since
$\Cokernel f=X/X_{\L_+}$ belongs to $\mod^{\L}_0R$.

(b) If $X=\OO(\y)$ for some $\y\in\L$, then the assertion follows from Proposition~\ref{vanishing}.

For general $X\in\coh\X$, applying (a) repeatedly, we have an exact 
sequence
\begin{equation}\label{resolution}
\cdots\xrightarrow{f_2}L_1\xrightarrow{f_1}L_0\to X\to0
\end{equation}
in $\coh\X$, where each $L_i$ is a finite direct sum of the degree shifts of $\OO$.
Recall that $\coh\X$ has global dimension $d$ by Theorem~\ref{Serre}(b). We take $\a\in\L$ such that $\Ext_{\X}^i(\OO,\bigoplus_{j=0}^{d-1}L_j(\x))=0$
for any $i>0$ and any $\x\in\L$ satisfying $\x\ge\a$.
Applying $\Hom_{\X}(\OO,-(\x))$ with $\x\ge\a$ to \eqref{resolution}, we have
\[\Ext^i_{\X}(\OO,X(\x))\simeq \Ext^{i+1}_{\X}(\OO,\Image f_1(\x))\simeq\cdots\simeq\Ext^{i+d}_{\X}(\OO,\Image f_d(\x))=0.\qedhere\]
\end{proof}

We note the following easy property, which will be used later.
Recall that $C$ is the polynomial algebra $k[T_0,\ldots,T_d]$ in $d+1$ variables.

\begin{lemma}\label{surjection from X to X(c)}
For $X\in\mod^{\L}R$ and $\ell\ge0$, let
$f_\ell=(t)_t:\bigoplus_{t}X\to X(\ell\c)$ be the morphism in $\mod^{\L}R$, 
where $t$ runs over all monomials on $T_0,\ldots,T_d$ of degree $\ell$.
\begin{itemize}
\item[(a)] The cokernel of $f_\ell:\bigoplus_{t}X\to X(\ell\c)$ belongs to 
$\mod^{\L}_0R$.
\item[(b)] $\pi(f_\ell):\bigoplus_{t}\pi(X)\to\pi(X)(\ell\c)$ is an epimorphism in 
$\coh\X$.
\end{itemize}
\end{lemma}

\begin{proof}
(a) Since the cokernel is annihilated by all monomials on
$T_0,\ldots,T_d$ in degree $\ell$, it is a finitely generated module over 
the finite dimensional $k$-algebra $C/(T_0,\ldots,T_d)^\ell$. Thus the assertion follows.

(b) Immediate from (a).
\end{proof}

The full subcategory
\[(\mod^{\L}_0R)^{\perp_{0,1}}:=\{X\in\mod^{\L}R\mid\Ext^i_R(Y,X)=0\ \mbox{ for any }\ Y\in\mod^{\L}_0R\ \mbox{ and }\ i=0,1\}\]
of $\mod^{\L}R$ is called the \emph{perpendicular category} \cite{GL2} of $\mod^{\L}_0R$. Clearly $(\mod^{\L}_0R)^{\perp_{0,1}}$ is closed under kernels, and contains $\proj^{\L}R$. Thus $\Omega^2(\mod^{\L}R)\subset(\mod^{\L}_0R)^{\perp_{0,1}}$.

The following observation will be used later.

\begin{lemma}\label{depth 2}\cite[2.1]{GL2}
For any $X\in\mod^{\L}R$ and $Y\in(\mod^{\L}_0R)^{\perp_{0,1}}$, the map $\Hom^{\L}_R(X,Y)\to\Hom_{\X}(X,Y)$ is bijective.
In particular, the functor $(\mod^{\L}_0R)^{\perp_{0,1}}\to\coh\X$ is fully faithful.
\end{lemma}

In the rest of this section, we study the following trichotomy of GL projective spaces.

\begin{definition}[Trichotomy]
We say that $\X$ is \emph{Fano} (respectively,
\emph{Calabi-Yau}, \emph{anti-Fano}) if so is $(R,\L)$,
that is, $\delta(\w)<0$  (respectively, $\delta(\w)=0$,
$\delta(\w)>0$) holds, where $\delta(\w)$ was given in \eqref{degree of w}.
\end{definition}

We will characterize these three types using the following categorical ampleness due to Artin-Zhang \cite{AZ}.

\begin{definition}\label{define ample}
Let $\AA$ be an abelian category. We say that an automorphism
$\alpha$ of $\AA$
\emph{ample} if there exists an object $V\in\AA$ satisfying the following conditions.
\begin{itemize}
\item $\add\{\alpha^{-\ell}(V)\mid \ell\ge0\}$ generates $\AA$.
\item For any epimorphism $f:X\to Y$ in $\AA$, there exists an
integer $\ell_0$ such that for every $\ell\geq\ell_0$ the map
$f:\Hom_{\AA}(\alpha^{-\ell}(V),X)\to\Hom_{\AA}(\alpha^{-\ell}(V),Y)$ 
is surjective.
\end{itemize}
\end{definition}

We have the following characterization of ampleness of the degree shift automorphisms.

\begin{theorem}\label{AZ ampleness0}
Let $\coh\X$ be a GL projective space, and let $\a\in\L$.
Then the automorphism $(\a)$ of $\coh\X$ is ample if and only if $\delta(\a)>0$.
\end{theorem}

\begin{proof}
Assume $\delta(\a)>0$. Then there exists a finite subset $S$ of $\L$
such that $S+\Z_{\ge0}\a\supset\L_+$.
We show that $V:=\bigoplus_{\x\in S}\OO(-\x)$ satisfies the two 
conditions in Definition~\ref{define ample}. Since 
\[\add\{V(-\ell\a)\mid \ell\ge0\}\supset\{\OO(-\x)\mid\x\in\L_+\},\]
the first condition is satisfied by Proposition~\ref{Serre vanishing}(a).
For an epimorphism $f:X\to Y$ in $\coh\X$, we have 
$\Ext^1_{\X}(V,(\Kernel f)(\ell\a))=0$ for $\ell\gg0$ by Proposition~\ref{Serre vanishing}(b).
Thus the second condition follows.

On the other hand, assume that $(\a)$ is ample, but $\delta(\a)\le0$.
Let $V\in\coh\X$ be an object satisfying the conditions in Definition
\ref{define ample}. By Proposition~\ref{Serre vanishing}(a), there exists a 
finite subset $S$ of $\L$ and an epimorphism
$L\to V$ in $\coh\X$ with $L\in\add\{\OO(-\x)\mid\x\in S\}$.
Then any object in $\coh\X$ is a quotient of an object in
$\CC:=\add\{\OO(-\x-\ell\a)\mid\x\in S,\ \ell\ge0\}$.
On the other hand, since $\delta(\a)\le0$, there exists an element 
$\b\in\L$ which is smaller than all elements in $-S-\Z_{\ge0}\a$.
Then any morphism from an object in $\CC$ to $\OO(\b)$ is zero,
a contradiction. Therefore $\delta(\a)>0$.
\end{proof}

We have the following interpretation of our trichotomy in terms of ampleness.

\begin{corollary}\label{AZ ampleness}
Let $\coh\X$ be a GL projective space.
\begin{itemize}
\item[(a)] $\X$ is Fano if and only if the automorphism $(-\w)$ of $\coh\X$ is ample.
\item[(b)] $\X$ is anti-Fano if and only if the automorphism $(\w)$ of $\coh\X$
is ample.
\item[(c)] $\X$ is Calabi-Yau if and only if 
$\DDD^{\bo}(\coh\X)$ is a fractionally Calabi-Yau triangulated category.
\end{itemize}
\end{corollary}

\begin{proof}
(a)(b) Immediate from Theorem~\ref{AZ ampleness0}.

(c) Since  $\DDD^{\bo}(\coh\X)$ has the Serre functor $(\w)[d]$
by Theorem~\ref{Serre}(f), it is fractionally Calabi-Yau  if and only if
$\w$ is a torsion element in $\L$. This means that $\X$ is Calabi-Yau.
\end{proof}

\section{Vector bundles}\label{section: vector bundles}

Recall that the canonical module $\omega_R$ of $R$ is defined as $\omega_R:=R(\w)$.
Since $R$ is Gorenstein, we have a duality
\[(-)^{\vee}:=\RHom_R(-,\omega_R):\DDD^{\bo}(\mod^{\L}R)\to\DDD^{\bo}(\mod^{\L}R)\]
which induces dualities $(-)^{\vee}:\DDD^{\bo}(\mod^{\L}_0R)\to\DDD^{\bo}(\mod^{\L}_0R)$ and
\[(-)^{\vee}:\DDD^{\bo}(\coh\X)\to\DDD^{\bo}(\coh\X).\]

\begin{definition}[Cohen-Macaulay sheaves]
For each $i$ with $0\le i\le d$, we define the category of \emph{Cohen-Macaulay sheaves of dimension $i$} by
\[\CM_i\X:=\coh\X\cap(\coh\X[i-d])^{\vee}=\{X\in\coh\X\mid H^j(X^\vee)=0\ \mbox{ for all $j\neq d-i$}\}.\]
Cohen-Macaulay sheaves of dimension $d$ are called \emph{vector bundles}:
\[\vect\X:=\CM_d\X.\]
Clearly $\OO(\x)\in\vect\X$ for any $\x\in\L$. Let
\[\lb\X:=\add\{\OO(\x)\mid\x\in\L\}\subset \vect \X.\]
\end{definition}

Immediately, for any $0\le i\le d$, we have a duality
\begin{equation}\label{duality of CM sheaf}
(-)^{\vee}[d-i]:\CM_i\X\to\CM_i\X\ \mbox{ and hence }\  (-)^{\vee}:\vect\X\to\vect\X.
\end{equation}
The following observation characterizes objects in the category $\CM_i\X$ 
in terms of $\L$-graded $R$-modules.
Let $\pi:\mod^{\L}R\to\coh\X$ be the natural functor.

\begin{proposition}\label{FLC modules}
Let $0\le i\le d$. 
\begin{itemize}
\item[(a)] We have
\[\pi^{-1}(\CM_i\X)=\{X\in\mod^{\L}R\mid\forall j\neq d-i\ 
\Ext^j_R(X,R)\in\mod^{\L}_0R\}.\]
\item[(b)] $\pi:\mod^{\L}R\to\coh\X$ restricts to a functor 
$\CM^{\L}_{i+1}R\to\CM_i\X$.
\item[(c)] We have
\begin{eqnarray*}
\pi^{-1}(\vect\X)=\{X\in\mod^{\L}R\mid
\mbox{$X$ is locally free on the punctured spectrum (Definition~\ref{locally free at punctured spectrum})}\}.
\end{eqnarray*}
\end{itemize}
\end{proposition}


\begin{proof}
(a) Let $X\in\mod^{\L}R$.
Then $\pi(X^\vee)$ belongs to $(\coh\X)[i-d]$ if and only if
$\pi(H^j(X^\vee))=0$ for all $j\neq d-i$ if and only if
$H^j(X^\vee)=\Ext^j_R(X,R)$ belongs to $\mod^{\L}_0R$ for any $j\neq d-i$.

(b) The assertion follows from (a) and $\CM_{i+1}^{\L}R=\{X\in\mod^{\L}R\mid\forall j\neq d-i\ 
\Ext^j_R(X,R)=0\}$.

(c) The assertion is immediate from (a) and
Definition-Proposition~\ref{locally free at punctured spectrum}.
\end{proof}

\begin{definition}[Vector bundle finiteness]
We say that a GL projective space $\X$ is \emph{vector bundle finite} (=\emph{VB finite}) if there are only finitely many isomorphism classes of indecomposable objects in $\vect\X$ up to degree shift.
\end{definition}

The following result gives a classification of VB finite GL projective spaces.

\begin{theorem}\label{VB finite}
A GL projective space $\X$ is VB finite if and only if $d=1$ and $\X$ is Fano (or equivalently, domestic).
\end{theorem}

\begin{proof}
For the case $d=1$, it is classical that $\X$ is VB finite if and only if $\X$ is domestic \cite{GL}.

We show that, if $d\ge2$, then $\X$ is never VB finite.
For any $X\in\mod^{\L}_0R$, we consider an exact sequence
\[0\to\Omega^2(X)\xrightarrow{g} P_1\xrightarrow{f} P_0\to X\to0\]
of $\L$-graded $R$-modules with $P_0,P_1\in\proj^{\L}R$.
Since $R$ is a Gorenstein ring of dimension $d+1\ge3$, we have that $\Ext^i_R(X,R)=0$ for any $i\le 2$. Therefore $g$ above is a left $(\proj^{\L}R)$-approximation of $\Omega^2(X)$, and $f$ above gives a left $(\proj^{\L}R)$-approximation of $\Omega(X)$. Hence the correspondence $X\mapsto\Omega^2(X)$ preserves indecomposability and respects isomorphism classes.

On the other hand, $\Omega^2(X)$ is locally free on the punctured spectrum, and hence $\pi(\Omega^2(X))$ belongs to $\vect\X$ by Proposition~\ref{FLC modules}(c). Since $\Omega^2(X)\in(\mod^{\L}_0R)^{\perp_{0,1}}$, the functor
\[\mod^{\L}_0R\to\vect\X,\ \ X\mapsto\pi(\Omega^2(X))\]
preserves indecomposability and respects isomorphism classes by Lemma~\ref{depth 2}.
Since there are infinitely many indecomposable objects in $\mod^{\L}_0R$ even up to degree shift, we have the assertion.
\end{proof}

We have the following easy property.

\begin{lemma}\label{second syzygy}
Any object in $\vect\X$ is isomorphic to $\pi(X)$ for some
$X\in\mod^{\L}R$ such that there exists an exact sequence
$0\to X\to P^0\to P^1$ in $\mod^{\L}R$ with $P^0,P^1\in\proj^{\L}R$.
In particular, $\vect\X\subset\pi(\Omega^2(\mod^{\L}R))\subset\pi((\mod^{\L}_0R)^{\perp_{0,1}})$.
\end{lemma}

\begin{proof}
Let $V\in\vect\X$. Take $Y\in\mod^{\L}R$ such that $V^\vee\simeq\pi(Y)$
and a projective resolution $P_1\to P_0\to Y\to0$ in $\mod^{\L}R$.
Applying $(-)^\vee$, we have the desired exact sequence with $X:=Y^\vee$.
\end{proof}

As a special case of Proposition~\ref{FLC modules}(b), we have
a functor
\[\CM^{\L}R\to\vect\X.\]
The statement (a) below shows that this is fully faithful. Therefore
$\CM^{\L}R$ has two exact structures, one is the 
restriction of the exact structure on $\mod^{\L}R$, and the other
is the restriction of that on $\coh\X$. These are different, for example,
$R$ is projective in $\mod^{\L}R$, but not in $\coh\X$.
But the statement (c) below shows that they are still very close.

\begin{proposition}\label{arith CM}
\begin{itemize}
\item[(a)] $\pi:\mod^{\L}R\to\coh\X$ restricts to a fully faithful 
functor $\CM^{\L}R\to\vect\X$
and an equivalence $\proj^{\L}R\to\lb\X$.
\item[(b)] We have
\begin{eqnarray*}
\pi(\CM^{\L}R)&=&
\{X\in\vect\X\mid\forall i\in\{1,2,\ldots,d-1\}\ \Ext_{\X}^i(\lb\X,X)=0\}\\
&=&\{X\in\vect\X\mid\forall i\in\{1,2,\ldots,d-1\}\ \Ext_{\X}^i(X,\lb\X)=0\}.
\end{eqnarray*}
\item[(c)] For any $i$ with $0\le i\le d-1$, we have a functorial isomorphism \[\Ext_{\mod^{\L}R}^i(X,Y)\simeq\Ext_{\X}^i(X,Y)\]
for any $X\in\mod^{\L}R$ and $Y\in\CM^{\L}R$.
\end{itemize}
\end{proposition}

Note that, by (b) above, the equality $\pi(\CM^{\L}R)=\vect\X$ holds for $d=1$ \cite[5.1]{GL}\cite[8.3]{GL2}.
On the other hand, for $d\ge 2$, the category $\vect\X$ is much bigger than $\pi(\CM^{\L}R)$.
In the context of projective geometry (e.g.\ \cite{CH,CMP}), the objects in $\pi(\CM^{\L}R)$
are called \emph{arithmetically Cohen-Macaulay bundles}.


\begin{proof}
(a) Since $\CM^{\L}R\subset(\mod^{\L}_0R)^{\perp_{0,1}}$, the assertions follow from Lemma~\ref{depth 2}.

(b) By Lemma~\ref{second syzygy}, any object in $\vect\X$
can be written as $\pi(X)$ with $X\in(\mod^{\L}_0R)^{\perp_{0,1}}$. Since $H_{\mm}^i(X)=0$ holds for $i=0,1$, it follows from
\eqref{describe CM by LC} that $X$ belongs to $\CM^{\L}R$ if and only if $H_{\mm}^i(X)=0$ for any $i$ with $2\le i\le d$.
By Proposition~\ref{X and H}, this is equivalent to 
$\Ext_{\X}^i(\lb\X,X)=0$ for any $i$ with $1\le i\le d-1$.
Thus the first equality follows. The second one is a consequence of Auslander-Reiten-Serre duality.

(c) Let $\cdots\to P_1\to P_0\to X\to0$ be a projective resolution of $X$
in $\mod^{\L}R$. Applying $\Hom^{\L}_R(-,Y)$, we have a complex
\begin{equation}\label{in mod R}
0\to\Hom^{\L}_R(P_0,Y)\to\Hom^{\L}_R(P_1,Y)\to\Hom^{\L}_R(P_2,Y)\to\cdots
\end{equation}
whose homology at $\Hom^{\L}_R(P_i,Y)$ is $\Ext_{\mod^{\L}R}^i(X,Y)$.

On the other hand, applying $\Hom_{\X}(-,Y)$ to an exact sequence
$\cdots\to P_1\to P_0\to X\to0$ in $\coh\X$, we have a complex
\begin{equation}\label{in coh X}
0\to\Hom_{\X}(P_0,Y)\to\Hom_{\X}(P_1,Y)\to\Hom_{\X}(P_2,Y)\to\cdots.
\end{equation}
Since we have $\Ext_{\X}^j(P_i,Y)=0$ for all $i$ and $j$ with
$1\le j\le d-1$ by (b), it is easily checked that the homology of 
\eqref{in coh X} at $\Hom_{\X}(P_i,Y)$ is $\Ext_{\X}^i(X,Y)$ for any $i$
with $1\le i\le d-1$.

Since the complexes \eqref{in mod R} and \eqref{in coh X} are 
isomorphic by (a), we have $\Ext_{\mod^{\L}R}^i(X,Y)\simeq\Ext_{\X}^i(X,Y)$ for all $i$ with $0\le i\le d-1$.
\end{proof}

We have Serre vanishing for vector bundles, which is a generalization of
Proposition~\ref{Serre vanishing}.

\begin{theorem}\label{general Serre vanishing}
Let $V\in\vect\X$ be non-zero.
\begin{itemize}
\item[(a)] $\add\{V(-\x)\mid\x\in\L_+\}$ generates $\coh\X$ and $\add\{V(\x)\mid\x\in\L_+\}$ cogenerates $\vect\X$.
\item[(b)] \emph{(Serre vanishing)} For any $X\in\coh\X$, there exists 
$\a\in\L$ such that $\Ext^i_{\X}(V,X(\x))=0$ holds for any $i>0$ and any
$\x\in\L$ satisfying $\x\ge\a$.
\end{itemize}
\end{theorem}

\begin{proof}
By Proposition~\ref{FLC modules}(c) and Lemma~\ref{second syzygy}, we can take a lift $0\neq V\in\Omega^2(\mod^{\L}R)$ of $V\in\vect\X$ such that $V$ is locally free on the punctured spectrum.

(a) For $X\in\mod^{\L}R$, let $\Hom_R(V,X)_{\L_+}:=\bigoplus_{\x\in\L_+}\Hom_R^{\L}(V,X(\x))$.
The natural morphism
\[f:V\otimes_R\Hom_R(V,X)_{\L_+}\to X\]
has a cokernel in $\mod^{\L}_0R$ since for any
$\pp\in\Spec^{\L}R\setminus\{R_+\}$, we have 
$(\Hom_R(V,X)_{\L_+})_{(\pp)}=\Hom_{R_{(\pp)}}(V_{(\pp)},X_{(\pp)})$ and
\[f_{(\pp)}:V_{(\pp)}\otimes_{R_{(\pp)}}\Hom_{R_{(\pp)}}(V_{(\pp)},X_{(\pp)})\to X_{(\pp)}\]
is an epimorphism. In fact, $V_{(\pp)}$ is a non-zero free $R_{(\pp)}$-module since $V_{(0)}\neq0$ holds by Proposition~\ref{non-zero is injective}(a).

Let $g_i:V(-\x_i)\to X$ with $1\le i\le m$ and $\x_i\in\L_+$ be homogeneous generators
of the $R$-module $\Hom_R(V,X)_{\L_+}$. Then the morphism
\[g:=(g_1,\ldots,g_m)^t:V(-\x_1)\oplus\cdots\oplus V(-\x_m)\to X\]
in $\mod^{\L}R$ has a cokernel in $\mod^{\L}_0R$. Therefore 
$\pi(g)$ is an epimorphism in $\coh\X$.

The latter assertion is immediate from the former one and the duality $(-)^\vee:\vect\X\to\vect\X$.

(b) By the argument of the proof of Proposition~\ref{Serre vanishing}(b), we only have to
consider the case $X=\OO$. Since $\coh\X$ has global dimension $d$, we can assume $0<i\le d$.

First we consider the case $i\neq d$. By Proposition~\ref{arith CM}(c), 
we have an isomorphism
\[\Ext^i_{\X}(V,\OO(\x))\simeq\Ext_{\mod^{\L}R}^i(V,R(\x)).\]
By Proposition~\ref{locally free at punctured spectrum}, we have 
$\Ext^i_R(V,R)\in\mod^{\L}_0R$. Thus 
$\Ext_{\mod^{\L}R}^i(V,R(\x))=0$ holds for all but finitely
many $\x\in\L$, and the assertion follows.

Now we consider the case $i=d$. By Auslander-Reiten-Serre duality
and Proposition~\ref{arith CM}(c), we have isomorphisms
\[\Ext^d_{\X}(V,\OO(\x))\simeq D\Hom_{\X}(\OO,V(\w-\x))
\simeq D\Hom^{\L}_R(R,V(\w-\x))=D(V_{\w-\x}).\]
This is zero for sufficiently large $\x$.
\end{proof}

Recall from Section~\ref{section: preliminaries 2} that a full subcategory 
$\CC$ of $\vect\X$ is called \emph{$d$-cluster tilting}
if $\CC$ is a generating and cogenerating functorially finite subcategory of $\vect\X$ such that
\begin{eqnarray*}
\CC&=&\{X\in\vect\X\mid\forall i\in\{1,2,\ldots,d-1\}\ \Ext_{\X}^i(\CC,X)=0\}\ \mbox{ and}\\
\CC&=&\{X\in\vect\X\mid\forall i\in\{1,2,\ldots,d-1\}\ \Ext_{\X}^i(X,\CC)=0\}.
\end{eqnarray*}
Note that one of the equalities above implies the other as in the case of $d$-cluster tilting subcategories of $\CM^{\L}R$ \cite[2.2.2]{I1}. The generating and cogenerating condition is automatic if $\CC$ contains $\lb\X$ (Theorem~\ref{general Serre vanishing}(a)).

Now we give some basic properties of $d$-cluster tilting subcategories,
which will be used later, where we need our assumption that $\CC$ generates and cogenerates $\vect\X$.

\begin{proposition}\label{basic properties of d-CT}
For a $d$-cluster tilting subcategory $\CC$ of $\vect\X$, the following assertions hold.
\begin{itemize}
\item[(a)] We have $\CC(\w)=\CC$.
\item[(b)] For any $X\in\vect\X$, there exist exact sequences
\[0\to C_{d-1}\to\cdots\to C_0\to X\to0\ \mbox{ and }\ 
0\to X\to C^0\to\cdots\to C^{d-1}\to0\]
in $\coh\X$ with $C_i,C^i\in\CC$ for any $0\le i\le d-1$.
\item[(c)] For any indecomposable object $X\in\CC$, there exists an exact sequence (called a \emph{$d$-almost split sequence})
\[0\to X(\w)\to C_{d-1}\to\cdots\to C_1\to C_0\to X\to0\]
such that the following sequences of functors on $\CC$ are exact:
\begin{eqnarray*}
&0\to\Hom_{\CC}(-,X(\w))\to\Hom_{\CC}(-,C_{d-1})\to\cdots\to\Hom_{\CC}(-,C_0)\to\rad_{\CC}(-,X)\to0,&\\
&0\to\Hom_{\CC}(X,-)\to\Hom_{\UU}(C_0,-)\to\cdots\to\Hom_{\UU}(C_{d-1},-)\to\rad_{\UU}(X(\w),-)\to0.&
\end{eqnarray*}
\end{itemize}
\end{proposition}

\begin{proof}
The proof is parallel to that of Proposition~\ref{basic properties of d-CT for CM}.
\end{proof}

\begin{definition}[$d$-vector bundle finiteness]
We say that a GL projective space $\X$ is \emph{$d$-vector bundle finite}
(=\emph{$d$-VB finite})
if there exists a $d$-cluster tilting subcategory $\CC$ of $\vect\X$ 
(see Section~\ref{section: preliminaries 2}) such that there are only finitely many isomorphism 
classes of indecomposable objects in $\CC$ up to degree shift.
\end{definition}

Now it is easy to prove the following key result in this section.

\begin{theorem}\label{CT subcategory}
The correspondence $\CC\mapsto\pi(\CC)$ gives a bijection between
the following objects.
\begin{itemize}
\item $d$-cluster tilting subcategories of $\CM^{\L}R$.
\item $d$-cluster tilting subcategories of $\vect\X$ containing $\lb\X$.
\end{itemize}
In particular, if $(R,\L)$ is $d$-CM finite, then $\X$ is $d$-VB finite.
\end{theorem}

For example, $\X$ is $d$-VB finite in the cases given in Theorem~\ref{Main1}.

To prove Theorem~\ref{CT subcategory}, we need the following observation.

\begin{proposition}\label{CM approximation}
$\pi(\CM^{\L}R)$ is a functorially finite subcategory of $\vect\X$.
\end{proposition}

\begin{proof}
It follows from Theorem~\ref{AR duality}(d) that $\CM^{\L}R$ is
a functorially finite subcategory of $\mod^{\L}R$, and hence of $(\mod^{\L}_0R)^{\perp_{0,1}}$.
By Lemma~\ref{depth 2}, we have that $\pi(\CM^{\L}R)$ is a functorially finite subcategory
of $\pi((\mod^{\L}_0R)^{\perp_{0,1}})$. Since $\vect\X$ is contained in $\pi((\mod^{\L}_0R)^{\perp_{0,1}})$ by Lemma~\ref{second syzygy},
we have the assertion.
\end{proof}

Now we are ready to prove Theorem~\ref{CT subcategory}.

\begin{proof}[Proof of Theorem~\ref{CT subcategory}.]
For a full subcategory $\CC$ of $\CM^{\L}R$, it follows from Proposition~\ref{CM approximation} that $\CC$ is functorially finite in $\CM^{\L}R$ if and only if $\pi(\CC)$ is functorially finite in $\vect\X$.

Let $\CC$ be a $d$-cluster tilting subcategory of $\CM^{\L}R$.
Since $\CC$ contains $\proj^{\L}R$, it follows that $\pi(\CC)$ 
contains $\lb\X$. Then we have
\begin{eqnarray*}
\pi(\CC)&=&\pi\left(\{X\in\CM^{\L}R\mid\forall i\in\{1,2,\ldots,d-1\}\ \Ext^i_{\mod^{\L}R}(\CC,X)=0\}\right)\\
&=&\{Y\in\vect\X\mid\forall i\in\{1,2,\ldots,d-1\}\ \Ext^i_{\X}(\pi(\CC),Y)=0\},
\end{eqnarray*}
where the second equality follows from Proposition~\ref{arith CM}(b)(c). Dually we have
\[\pi(\CC)=\{Y\in\vect\X\mid\forall i\in\{1,2,\ldots,d-1\}\ \Ext^i_{\X}(Y,\pi(\CC))=0\}.\]
Therefore $\pi(\CC)$ is a $d$-cluster tilting subcategory of $\vect\X$.

Conversely, any $d$-cluster tilting subcategory of $\vect\X$ containing $\lb\X$ is contained in $\pi(\CM^{\L}R)$ by
Proposition~\ref{arith CM}(c), and hence can be written as
$\pi(\CC)$ for a subcategory $\CC$ of $\CM^{\L}R$.
By a similar argument as above, one can check that $\CC$ is a
$d$-cluster tilting subcategory of $\CM^{\L}R$.
\end{proof}

As an immediate consequence of Theorem~\ref{CT subcategory}, we have
the following result, where one of the implications generalizes Horrocks' splitting criterion
stating that any vector bundle $V$ on $\P^d$ satisfying $H^i(\P^d,V)=0$
for all $1\le i\le d-1$ is a direct sum of line bundles \cite[2.3.1]{OSS}.

\begin{corollary}\label{line is CT subcategory}
Assume that $p_i\ge2$ for all $i$. Then the following conditions are equivalent.
\begin{itemize}
\item $n\le d+1$ (or equivalently, $R$ is regular).
\item $\lb\X$ is a $d$-cluster tilting subcategory of $\vect\X$.
\end{itemize}
\end{corollary}

\begin{proof}
The first condition is equivalent to $\CM^{\L}R=\proj^{\L}R$
by Proposition~\ref{when stable category is zero}.
On the other hand, it is clear from the definition of $d$-cluster tilting subcategories that $\CM^{\L}R=\proj^{\L}R$ 
holds if and only if $\proj^{\L}R$ is a $d$-cluster tilting subcategory of 
$\CM^{\L}R$. This is equivalent to the second condition by
Theorem~\ref{CT subcategory}.
\end{proof}

In the rest of this section, we give a geometric characterization of Cohen-Macaulay sheaves on $\X$ in terms of the projective space $\P^d$. By abuse of notation, for each $i$ with $0\le i\le d$, let
\[\CM_i\P^d:=\{X\in\coh\P^d\mid\forall\ \mbox{closed point}\ x\in\P^d,\ X_{x}\in\CM_i(\OO_{\P^d,x})\}\]
be the category of \emph{Cohen-Macaulay sheaves} of dimension $i$ on $\P^d$.
In particular
\[\vect\P^d:=\CM_d\P^d\]
is the category of \emph{vector bundles}.
We identify $\coh\P^d$ with $\mod^{\Z}C/\mod^{\Z}_0C$ for the $(\Z\c)$-Veronese
subalgebra $C=k[T_0,\ldots,T_d]$ of $R$. We have an exact functor
\[f_*:\mod^{\L}R\xrightarrow{\sim}\mod^{\Z}R^{[\Z\c]}\to\mod^{\Z}C,\]
where the first functor is given in Proposition~\ref{graded morita for R} and
the second one is the restriction with respect to the inclusion $C\to R^{[\Z\c]}$.
Since $f_*(\mod^{\L}_0R)\subset\mod^{\Z}_0C$, we have an induced exact functor
\[f_*:\coh\X=\mod^{\L}R/\mod^{\L}_0R\to\mod^{\Z}C/\mod^{\Z}_0C=\coh\P^d.\]
The following description of $\CM_i\X$ in terms of $\CM_i\P^d$ also shows that two definitions of $\CM_i\P^d$ are the same.

\begin{theorem}\label{CM in terms of P^d}
For $0\le i\le d$, we have
\begin{eqnarray*}
\CM_i\X&=&\{X\in\coh\X\mid f_*X\in\CM_i\P^d\},\\
\vect\X&=&\{X\in\coh\X\mid f_*X\in\vect\P^d\}.
\end{eqnarray*}
\end{theorem}

In the rest of this section, we prove Theorem~\ref{CM in terms of P^d}.
First, we give a description of $\CM_i\P^d$ in terms of the duality
given by right derived functor of the sheaf Hom functor
\[(-)^\vee:=\RshHom_{\P^d}(-,\omega_{\P^d}):
\DDD^{\bo}(\coh\P^d)\to\DDD^{\bo}(\coh\P^d).\]

\begin{proposition}\label{describe CM on P^d}
For $0\le i\le d$, we have
$\CM_i\P^d=(\coh\P^d)\cap(\coh\P^d[i-d])^{\vee}$.
\end{proposition}

\begin{proof}
For any closed point $x\in\P^d$, we have a commutative diagram
\[\xymatrix{
\DDD^{\bo}(\coh\P^d)\ar[rrr]^{(-)^{\vee}=\RshHom_{\P^d}(-,\omega_{\P^d})}\ar[d]^{(-)_x}&&&\DDD^{\bo}(\coh\P^d)\ar[d]^{(-)_x}\\
\DDD^{\bo}(\mod\OO_{\P^d,x})\ar[rrr]^{\RHom_{\OO_{\P^d,x}}(-,\OO_{\P^d,x})}&&&\DDD^{\bo}(\mod\OO_{\P^d,x})}\]
up to natural isomorphism by \cite[Chap.2, Prop.5.8]{Har}.
Fix $X\in\coh\P^d$. Then $X$ belongs to $(\coh\P^d[i-d])^{\vee}$
if and only if $H^j(X^{\vee})=0$ for any $j\neq d-i$ if and only if
$H^j(X^{\vee})_x=0$ for any $j\neq d-i$ and any closed point $x\in\P^d$.
By using the commutative diagram above, this is equivalent to 
$\Ext_{\OO_{\P^d,x}}^j(X_x,\OO_{\P^d,x})=0$ for any $j\neq d-i$ and any closed point $x\in\P^d$. 
This means $X_x\in\CM_i(\OO_{\P^d,x})$ for any closed point $x\in\P^d$, that is, $X\in\CM_i\P^d$.
\end{proof}

Next we describe the duality
$(-)^{\vee}:\DDD^{\bo}(\coh\P^d)\to\DDD^{\bo}(\coh\P^d)$
in terms of $C$.

\begin{lemma}
We have the following commutative diagram up to natural isomorphism.
\[\xymatrix{
\DDD^{\bo}(\mod^{\Z}C)\ar[rrr]^{\RHom_C(-,\omega_C)}\ar[d]^{\pi}
&&&\DDD^{\bo}(\mod^{\Z}C)\ar[d]^{\pi}\\
\DDD^{\bo}(\coh\P^d)\ar[rrr]^{(-)^\vee=\RshHom_{\P^d}(-,\omega_{\P^d})}
&&&\DDD^{\bo}(\coh\P^d).}\]
\end{lemma}

\begin{proof}
By \cite[2.5.13]{G}, the diagram 
\[\xymatrix{ 
\mod^{\Z}C\ar[rrr]^{F:=\Hom_C(-,\omega_C)}\ar[d]^{\pi}&&& 
\mod^{\Z}C\ar[d]^{\pi}\\ 
\coh\P^d\ar[rrr]^{G:=\shHom_{\P^d}(-,\omega_{\P^d})}&&& 
\coh\P^d
}\]
commutes up to natural isomorphism. Thus we have an isomorphism ${\bf R}(\pi\circ F) \simeq {\bf 
R}(G\circ\pi)$ of functors $\DDD^{\bo}(\mod^{\Z}C)\to\DDD^{\bo}(\coh\P^d)$. 
 
Since $\pi:\mod^{\Z}C\to\coh\P^d$ is an exact functor, 
by \cite[I.5.4.(b)]{Har} we have ${\bf R}(\pi\circ F) \simeq  \pi\circ{\bf R}F$. 
Note that we can compute the derived functor ${\bf R}G$ by using locally free resolutions. 
Since the image of projective resolutions in $\DDD^{\bo}(\mod^{\Z}C)$ by $\pi$ 
give locally free resolutions in $\DDD^{\bo}(\coh\P^d)$, by \cite[I.5.4.(b)]{Har} 
we have ${\bf R}(G\circ\pi) \simeq {\bf R}G\circ\pi$. Combining all of this, 
we have  ${\bf R}G \circ \pi \simeq \pi \circ {\bf R}F$ as desired.
\end{proof}

Our exact functor $f_*:\coh\X\to\coh\P^d$ induces a triangle functor
\[f_*:\DDD^{\bo}(\coh\X)\to\DDD^{\bo}(\coh\P^d)\]
which makes the diagram
\begin{equation}\label{cohomology}
\xymatrix{
\DDD^{\bo}(\coh\X)\ar[rr]^{H^i}\ar[d]^{f_*}&&\coh\X\ar[d]^{f_*}\\
\DDD^{\bo}(\coh\P^d)\ar[rr]^{H^i}&&\coh\P^d
}\end{equation}
commutative for any $i\in\Z$.
The following observation is easy.

\begin{lemma}\label{R and C}
We have the following commutative diagram up to natural isomorphism.
\[
\xymatrix{
\DDD^{\bo}(\coh\X)\ar[rrr]^{(-)^{\vee}=\RHom_R(-,\omega_R)}\ar[d]^{f_*}
&&&\DDD^{\bo}(\coh\X)\ar[d]^{f_*}\\
\DDD^{\bo}(\coh\P^d)\ar[rrr]^{(-)^{\vee}=\RHom_C(-,\omega_C)}
&&&\DDD^{\bo}(\coh\P^d).
}\]
\end{lemma}

\begin{proof}
We have $\omega_R=\RHom_C(R,\omega_C)$ (e.g.\ \cite[3.3.7(b)]{BH}). Thus
\[\RHom_R(-,\omega_R)=\RHom_R(-,\RHom_C(R,\omega_C))=\RHom_C(-,\omega_C)\]
holds, and we have the assertion.
\end{proof}

Now we are ready to prove Theorem~\ref{CM in terms of P^d}.

\begin{proof}[Proof of Theorem~\ref{CM in terms of P^d}]
Let $X\in\coh\X$.
By Proposition~\ref{describe CM on P^d}, $f_*X\in\CM_i\P^d$ if and 
only if $f_*(X^{\vee})=(f_*X)^{\vee}\in(\coh\P^d)[i-d]$, where the 
equality holds by Lemma~\ref{R and C}.
This is equivalent to $X^{\vee}\in(\coh\X)[i-d]$ by the commutative
diagram \eqref{cohomology}. This means $X\in\CM_i\X$ by definition.
\end{proof}

\section{Appendix: Proof of Auslander-Reiten-Serre duality}\label{section: Serre duality}

In the rest of this section, we give a complete proof of Theorem~\ref{Serre}(f)
following idea of proof of \cite[Theorem A.4]{DV}.
We refer to \cite[Chapter 4]{Popescu} for background on quotients of
abelian categories.

Let $\Mod^{\L} R$ be the category of $\L$-graded $R$-modules, 
and $\Mod^{\L}_{0} R$ be the localizing  subcategory of $\L$-graded modules 
obtained as a colimit of finite dimensional modules. 
We set 
\[
\Qcoh  \X := \Mod^{\L} R/\Mod^{\L}_{0} R.  
\]
Then the quotient functor  $\pi: \Mod^{\L} R \to \Qcoh \X$ has 
the section functor $\varpi: \Qcoh \X \to \Mod^{\L} R$, 
that is, the right adjoint of $\pi$ such that $\pi\circ \varpi \simeq \textup{id}_{\Qcoh \X}$. 
We set $Q := \varpi \circ \pi$ to be the localization functor. 
Then it follows $Q^2=Q$. 
The torsion functor  $\Gamma_{\mm}$ associates
an $\L$-graded $R$-module $M$ with its largest torsion submodule $\Gamma_{\mm}(M)$, 
which is the kernel of the unit morphism  $u_{M}: M \to Q(M)$. 
Note that $\Gamma_{\mm}(M)$ is the $0$-th local cohomology group $H_{\mm}^{0}(M)$, and its $i$-th derived functor $\R^{i}\Gamma_{\mm}(M)$ is the $i$-th local cohomology group $H^{i}_{\mm}(M)$.

The following is well-known. 

\begin{lemma}\label{GMQ triangle}
We have an exact triangle 
\[
\R \Gamma_{\mm}(M) \to M \to \R Q(M) \to\R\Gamma_{\mm}(M)[1] 
\]
for $M \in \DDD(\Mod^{\L}R)$. 
\end{lemma}

\begin{proof}
By \cite[Proposition 7.1 (5)]{AZ} and \cite[Lemma 5.1]{Popescu} 
every injective object $I$ is a direct sum $I_{\textup{t}} \oplus I_{\textup{f}}$ 
where $I_{\textup{t}}$ is an injective object  such that  
$\Gamma_{\mm}(I_{\textup{t}}) = I_{\textup{t}},Q(I_{\textup{t}}) =0$ 
and $I_{\textup{f}}$ is an injective object such that 
$\Gamma_{\mm}(I_{\textup{f}}) =0 , Q(I_{\textup{f}}) = I_{\textup{f}}$. 
Therefore 
the exact sequence $ 0\to \Gamma_{\mm} (I) \to I \to Q(I)$ 
is isomorphic to 
the split exact sequence 
$0 \to I_{\textup{t}} \to I \to I_{\textup{f}} \to 0$. 

Since $\Gamma_{\mm}$ and $Q$ are left exact, 
we can compute  the derived functors by using K-injective resolution. 
Since each term of a K-injective complex is injective, 
we have the assertion. 
\end{proof}

We denote by $\Ext_R^{i} (M,N)_{\bullet}$ the graded extension group. 
Namely 
\[
\Ext^{i}_R(M,N)_{\bullet}
:= \bigoplus_{\x \in \L} \Ext^{i}_{\Mod^{\L}R}(M,N(\x)). 
\]
Also let $D$ denote the graded $k$-dual.

\begin{lemma}\label{prepare to prove Serre}
\begin{itemize}
\item[(a)] We have an isomorphism $\Ext_R^{d+1}(DR,R(\w))_{\bullet} \simeq R$ of $\L$-graded $R$-modules. 
In particular $\Ext_{\Mod^{\L}R}^{d+1}(DR,R(\w))\simeq k$.
\item[(b)] $\R Q(R)={\rm Cone}(f\colon(DR)(-\w)[-d-1]\to R)$ for any non-zero morphism $f$.
\item[(c)] We have an isomorphism  $D\R Q(R) \simeq \R Q(R)(\w)[d]$.
\end{itemize}
\end{lemma}

\begin{proof}
(a) By local duality and Proposition~\ref{a-invariant} 
for any finite dimensional $\L$-graded $R$-module $M$, we have 
$\Ext_R^{d+1}(M,R(\w))_{\bullet} \simeq DM$ and 
$\Ext_R^{i}(M,R(\w))_{\bullet}  =0 $ for $i\neq d+1$. 

Let $\phi_{m}: R_{\leq m\c} \to R_{\leq (m-1)\c}$ be the canonical projection. 
Since $DR$ is the colimit of the following linear diagram 
\[
D(R_{0}) \xrightarrow{D(\phi_{1})} 
D(R_{\leq \c})  \xrightarrow{D(\phi_{2})}  
D(R_{\leq 2\c}) \xrightarrow{D(\phi_{3})} \cdots,  
\]
it fits into the   exact sequence in $\Mod^{\L} R$ 
\[
0\to \bigoplus_{m \geq 0} D(R_{\leq m\c})  \xrightarrow{\Psi}
\bigoplus_{m \geq 0} D(R_{\leq m\c}) \to DR \to 0
\]
where the restriction of $\Psi$ to $D(R_{\leq m\c} )$ is 
$(\textup{id}, -D(\phi_{m+1})): D(R_{\leq m\c} ) \to D(R_{\leq m\c}) \oplus D(R_{\leq (m+1)\c})$.  
The contravariant functor $\Ext^{i}_R(-,R(\w))_{\bullet}$ sends coproducts to products,
and satisfies $\Ext^{d+1}_R(D(R_{\leq m\c}),R(\w))_{\bullet}\simeq R_{\leq m\c}$.
Therefore, by considering the Ext long exact sequence of the above exact sequence,  
we obtain the exact sequence 
\[
0 \to \Ext^{d+1}_R(DR,R(\w))_{\bullet}
\to \prod_{ m \geq 0} R_{\leq m\c} \xrightarrow{\Phi} \prod_{ m \geq 0} R_{\leq m\c} \to0, 
\]
such that the composition of $\Phi$ and the projection to $R_{\leq m\c}$ is ${-\phi_{m+1}\choose\textup{id}}: R_{\leq (m+1)\c} \oplus R_{\leq m\c}\to R_{\leq m\c}$.
Since $R$ is the limit of the following linear diagram in $\Mod^{\L}R$ 
\[
\cdots \xrightarrow{\phi_{3}} R_{\leq 2\c} 
\xrightarrow{\phi_{2}} R_{\leq \c} 
\xrightarrow{\phi_{1}} R_{0},  
\]
we obtain the desired result. 

(b) By Lemma~\ref{GMQ triangle}, we have a triangle
$\R \Gamma_{\mm}(R) \xrightarrow{f} R \to \R Q(R) \to\R\Gamma_{\mm}(R)[1]$
By \eqref{local cohomology}, $H^i(\R\Gamma_{\mm}(R))=H_{\mm}^i(R)$ is zero for all $i\neq d+1$,
and equals $(DR)(-\w)$ for $i=d+1$.
Therefore $\R\Gamma_{\mm}(R)=(DR)(-\w)[-d-1]$ holds.

By (a), $\Hom^{\L}_R((DR)(-\w)[-d-1],R)$ is a one-dimensional $k$-vector space,
and the morphism $f\colon\R \Gamma_{\mm}(R) \to R$ is non-zero since
$\R\Gamma_{\mm}(R)[1]$ can not be a direct summand of $\R Q(R)$.
Thus the assertion follows.

(c) By (b), we have triangles $(D R)(-\w)[-d-1] \to R\to \R Q(R) \to (D R)(-w)[-d]$ and 
\[D R \to R(\w)[d+1] \to \R Q(R)(\w)[d+1] \to (DR)[1].\]
Applying $D$ to the first triangle, we obtian a triangle 
\[D R \to R(\w)[d+1] \to D(\R Q(R))[1] \to (DR)[1].\]
Comparing these triangles and using (a), we obtain the desired isomorphism.
\end{proof}

In fact, $\R Q(R)$ determines the functor $\R Q$ on $\KKK^{\bo}(\proj^{\L}R)$
by the following observation.

\begin{lemma}\label{RQ(M)=MRQ(R)}
If $M\in\KKK^{\bo}(\proj^{\L}R)$, then $\R Q(M)=M\Lotimes_R\R Q(R)$.
\end{lemma}

\begin{proof}
Let $R \xrightarrow{\sim} I$ be an injective resolution of $R$. If $M\in\KKK^{\bo}(\proj^{\L}R)$,
then the complex $ M \otimes_R I $ is a left bounded complex of injective $R$-modules 
which is quasi-isomorphic to $M$. Hence we have 
\[
\R Q(M) \simeq Q(M\otimes_R I) \simeq M \otimes_R Q(I) \simeq M \Lotimes_R \R Q(R). \qedhere
\]
\end{proof}

Now we are ready to prove Theorem~\ref{Serre}(f).

\begin{proof}[Proof of Theorem~\ref{Serre}(f).]
Let $\PP:= \KKK^{\bo}(\proj^\L R)$.
By Proposition~\ref{generating derived category}, $\DDD^{\bo}(\coh\X)=\add\pi(\PP)$ holds.
Thus it is enough to show that there exists a functorial isomorphism
$\Hom_{\Qcoh \X}(\pi N, \pi M)\simeq D\Hom_{\Qcoh \X}(\pi M, \pi N(\w)[d])$
for any $M,N\in\PP$.

Our strategy is to use the following sequence of isomorphisms:
\begin{align*}
\Hom_{\Qcoh \X}(\pi M, \pi N(\w)[d])
& \simeq \Hom_{\Mod^{\L}R}(M,\R Q(N)(\w)[d]) \\
& \simeq \Hom_{\Mod^{\L}R}(M,N \Lotimes_R \R Q(R)(\w)[d]) \\
& \simeq \Hom_{\Mod^{\L}R}(M, N \otimes_R D\R Q(R))\\
& \simeq D\Hom_{\Mod^{\L}R}(N,M \Lotimes_R \R Q(R)) \\
& \simeq D\Hom_{\Mod^{\L}R}(N,\R Q(M)) \\
& \simeq D\Hom_{\Qcoh \X}(\pi N ,\pi M).
\end{align*}
Here the first and last isomorphisms are just adjunction, the second and second last are Lemma~\ref{RQ(M)=MRQ(R)}, and the third one is Lemma~\ref{prepare to prove Serre}(c). Thus it only remains to prove
\[ \RHom_{\Mod^{\L}R}(M, N \otimes_R D\R Q(R)) \simeq D\RHom_{\Mod^{\L}R}(N,M \Lotimes_R \R QR). \]

Below we omit the subscripts $R$ for simplicity.
For arbitrary complexes $M$, $N$, and $S$ of $\L$-graded $R$-modules,
we have the following morphisms
\[
\Hom(M,N \otimes DS) \to \Hom(M,D\Hom(N,S)) \to D(M \otimes \Hom (N,S) )\leftarrow D\Hom(N , M \otimes S)
\]
which are natural in $M$, $N$, and $S$, from the natural morphisms
\[
N \otimes DS \rightarrow D\Hom(N,S),\ \Hom(M,DT) \to D(M \otimes T),\ M \otimes \Hom(N, S)
\to \Hom(N,M \otimes S)
\]
where we put $T : = \Hom(N,S)$.
If the complexes $M$ and $N$ are K-projective, then the above diagram gives the diagram in the derived category
\begin{equation*}
\RHom(M,N \Lotimes DS) \to \RHom(M,D\RHom(N,S)) \to D(M \Lotimes \RHom (N,S) )
\leftarrow D\RHom(N , M \Lotimes S).
\end{equation*}
Now we observe that if $M$ and $N$ belong to $\PP$,
then the all the above morphisms are isomorphisms. This completes the proof.
\end{proof}

%
%
%


\chapter{$d$-canonical algebras}\label{section: canonical}

Throughout this chapter, we assume $d\ge1$.
Let $\X$ be a Geigle-Lenzing projective space of dimension $d$ over a field $k$ associated with
linear forms $\ell_1,\ldots,\ell_n$ and weights $p_1,\ldots,p_n$.
In this chapter we show that $\X$ has a tilting bundle, and in
particular the category of coherent sheaves is derived equivalent to a certain finite dimensional
algebra $A^{\ca}$ which we call a $d$-canonical algebra.
Then we study basic properties of $d$-canonical algebras.
In particular we show that the global dimension of $A^{\ca}$ is $d$ if $n\le d+1$
and $2d$ otherwise. In the former case, 
$A^{\ca}$ belongs to a special class of algebras called `$d$-representation 
infinite algebras of type $\widetilde{\mathbb A}$' studied in \cite{HIO}.

\section{Basic properties}\label{section: basic properties of canonical}

Our $d$-canonical algebras are a special class of $I$-canonical algebras
introduced in Definition~\ref{define I-canonical}.

\begin{definition}[$d$-canonical algebra]
The \emph{$d$-canonical algebra} of $\X$ (or $(R,\L)$) is defined as
\[A^{\ca}:=A^{[0,d\c]}=(R_{\x-\y})_{\x,\y\in [0,d\c]}.\]
The multiplication of $A^{\ca}$ is given by
\[(r_{\x,\y})_{\x,\y\in[0,d\c]}\cdot(r'_{\x,\y})_{\x,\y\in[0,d\c]}=
(\sum_{\z\in[0,d\c]}r_{\x,\z}\cdot r'_{\z,\y})_{\x,\y\in[0,d\c]}.\]
\end{definition}

Our main result in this section is the following.

\begin{theorem}\label{tilting}
The object
\[ T^{\ca}:=\bigoplus_{\x\in[0,d\c]}\OO(\x)\]
is tilting in $\DDD^{\bo}(\coh\X)$ such that $\End_{\X}(T^{\ca})\simeq A^{\ca}$. Thus $K_0(\coh\X)$ is a free abelian group of rank $\#[0,d\c]$.
\end{theorem}

The following special cases are known.
\begin{itemize}
\item Let $d=1$. Then $T^{\ca}$ is a tilting bundle on a weighted projective line due to Geigle-Lenzing \cite{GL}, and $A^{\ca}$ is
the canonical algebra due to Ringel \cite{Rin} (see Example~\ref{canonical example}).
\item The case $n=0$ is due to Beilinson \cite{Be} (see Example~\ref{Beilinson example}), the case $n\le d+1$ is due to Baer \cite{Ba} (see Theorem~\ref{Atilde theorem}),
and the case $n=d+2$ is due to Ishii-Ueda \cite{IU}.
\item In terms of the Geigle-Lenzing order (see Remark~\ref{GL order}),
Theorem~\ref{tilting} is independently given by Lerner and
the second author \cite{IL}.
\end{itemize}

We give two different proofs of Theorem~\ref{tilting}: One is to show directly that $T^{\ca}$
is rigid in Proposition~\ref{vanishing statement}, and that $T^{\ca}$
generates the derived category of $\coh \X$ in Proposition~\ref{generate}.
The other proof using semiorthogonal decompositions of $\DDD^{\bo}(\mod^{\L}R)$
will be given in the next section. It is parallel to the proof of Theorem~\ref{CM tilting}.

\begin{proposition}\label{vanishing statement}
We have $\Ext_{\X}^i(T^{\ca},T^{\ca})=0$ for all $i>0$.
\end{proposition}

\begin{proof}
By Proposition~\ref{vanishing}, we have $\Ext_{\X}^i(T^{\ca},T^{\ca})=0$ for all $i$ with $i\neq0,d$.
Let $\x,\y\in[0,d\c]$. By Proposition~\ref{vanishing}, we have $\Ext_{\X}^d(\OO(\x),\OO(\y))=D(R_{\x-\y+\w})$.
By Lemma~\ref{order omega}, we have $\x+\w \not\ge 0$.
Therefore $\x-\y+\w \not\ge 0$ holds, and so $R_{\x-\y+\w}=0$ by
Observation~\ref{basic results on L}(c). Thus we have $\Ext_{\X}^d(T^{\ca},T^{\ca})=0$.
\end{proof}

In the rest, we show that $T^{\ca}$ generates the derived category of $\coh \X$.

\begin{proposition}\label{generate}
We have $\thick T^{\ca}=\DDD^{\bo}(\coh\X)$.
\end{proposition}

\begin{proof}
Let $\L':=\{\x\in\L\mid\OO(\x)\in\thick T^{\ca}\}$. Since 
$\DDD^{\bo}(\coh\X)=\thick(\lb\X)$ holds by Theorem~\ref{Serre}(e),
it is enough to prove that $\L' = \L$. Clearly $[0,d\c]\subset\L'$ holds.

The key observation is the following.

\begin{lemma}\label{induction}
Let $\x\in\L$. If there exists a subset $I$ of $\{1,\ldots,n\}$ satisfying the two conditions below, then $\x\in\L'$.
\begin{itemize}
\item $|I|=d+1$.
\item For any non-empty subset $I'$ of $I$, we have 
$\x-\sum_{i\in I'}\x_i\in\L'$.
\end{itemize}
\end{lemma}

\begin{proof}
By Lemma~\ref{prop.regularseq1}(c), $(X_i)_{i\in I}$ is an $R$-regular sequence.
Hence the corresponding Koszul complex
\begin{equation}\label{Koszul}
0\to R(\x-\sum_{i\in I}\x_i)\to\cdots\to\bigoplus_{i,j\in I,\ i\neq j}R(\x-\x_i-\x_j)
\to\bigoplus_{i\in I}R(\x-\x_i)\to R(\x)\to0.
\end{equation}
of $R$ is exact except in the rightmost position whose homology is
$(R/(X_i \mid i \in I))(\x)$. Since this belongs to $\mod^{\L}_0R$,
the image of \eqref{Koszul} in $\coh\X$ is exact.
Since all the terms except $R(\x)$ belongs to $\thick T^{\ca}$
by our assumption, we have $R(\x)\in\thick T^{\ca}$.
\end{proof}

We continue our proof of Proposition~\ref{generate} by showing that $\L_+\subset\L'$.
We use induction with respect to the partial order on $\L_+$. Let $\x \in \L_+$
and assume that any $\y\in[0,\x]$ with $\y\neq\x$ belongs to $\L'$.
If $\x \in [0,d\c]$, then $\x \in \L'$ holds. Otherwise let
\[\x=\sum_{i=1}^na_i\x_i+a\c\]
be the normal form of $\x$ with $0\le a_i<p_i$. Since $\x\not\le d\c$,
we have $a+\#\{i\mid a_i>0\}\ge d+1$ by Lemma~\ref{order omega}.
Thus there exists a subset $I$ of $\{1,\ldots,n\}$ with $|I|=d+1$
such that $\x-\sum_{i\in I}\x_i\in\L_+$.
Thus for any non-empty subset $I'$ of $I$, we have
$\x-\sum_{i\in I'}\x_i\in\L'$ by the induction assumption,
and we have $\x\in\L'$ by Lemma~\ref{induction}.

Now, the fact that $\L' = \L$ is shown by the induction with respect to the opposite of the partial order on $\L$ by using $\L_+ \subset \L'$ and the dual of Lemma~\ref{induction}.
\end{proof}

Now we are ready to prove Theorem~\ref{tilting}.

\begin{proof}[Proof of Theorem~\ref{tilting}]
$T^{\ca}$ is a tilting object in $\DDD^{\bo}(\coh\X)$ by 
Propositions~\ref{vanishing statement} and \ref{generate}.

It remains to show $\End_{\X}(T^{\ca})\simeq A^{\ca}$.
Since $\Hom_{\X}(\OO(\x),\OO(\y))=R_{\y-\x}$ holds, we have
\[\End_{\X}(T^{\ca})=(R_{\y-\x})_{\x,\y\in[0,d\c]}=(A^{\ca})^{\op},\]
which is isomorphic to $A^{\ca}$ by
Proposition~\ref{I-canonical has finite global dimension}(c).
\end{proof}

We give a list of basic properties of the $d$-canonical algebras, where
$\nu_d:=(D\Lambda)[-d]\Lotimes_\Lambda-:\DDD^{\bo}(\mod\Lambda)
\to\DDD^{\bo}(\mod\Lambda)$ is the $d$-shifted Nakayama functor.

\begin{proposition}\label{basic of canonical algebra}
\begin{itemize}
\item[(a)] $A^{\ca}$ has finite global dimension.
\item[(b)] $A^{\ca}$ is isomorphic to its opposite algebra $(A^{\ca})^{\op}$.
\item[(c)] We have a triangle equivalence $T^{\ca}\Lotimes_{A^{\ca}}-:\DDD^{\bo}(\mod A^{\ca})\to\DDD^{\bo}(\coh\X)$ which makes the following diagram commutative up to natural isomorphism:
\[\xymatrix{
\DDD^{\bo}(\mod A^{\ca})\ar[rr]^{T^{\ca}\Lotimes_{A^{\ca}}-}
\ar[d]^{\nu_d}&&\DDD^{\bo}(\coh\X)\ar[d]^{(\w)}\\
\DDD^{\bo}(\mod A^{\ca})\ar[rr]^{T^{\ca}\Lotimes_{A^{\ca}}-}
&&\DDD^{\bo}(\coh\X).
}\]
\end{itemize}
\end{proposition}

\begin{proof}
(a)(b) These are shown in Proposition~\ref{I-canonical has finite global dimension}.

(c) Since the triangle functor $T^{\ca}\Lotimes_{A^{\ca}}-:\DDD^{\bo}(\mod A^{\ca})\to\DDD^{\bo}(\coh\X)$
sends a tilting object $A^{\ca}$ to a tilting object $T^{\ca}$, it is a 
triangle equivalence (cf.\ Proposition~\ref{tilting theorem}).
The diagram commutes by uniqueness of Serre functors.
\end{proof}

The above commutative diagram is useful to study further properties of
$d$-canonical algebras.
In particular, we apply it to study $d$-canonical algebras in the context 
of (almost) $d$-representation infinite algebras (see Definitions~\ref{define d-RI} and \ref{define almost d-RI}).

Our $d$-canonical algebras have the following property, where the part (c) is a version of
a result due to Buchweitz-Hille \cite{BuH}.

\begin{theorem}\label{global dimension}
Without loss of generality, we assume that $p_i\ge 2$ for all $i$.
\begin{itemize}
\item[(a)] We have
\[\gl A^{\ca}=\left\{\begin{array}{cc}
d&n \le d+1,\\
2d&n>d+1.
\end{array}\right.\]
\item[(b)] $A^{\ca}$ is an almost $d$-representation infinite algebra.
\item[(c)] $n \le d+1$ if and only if $A^{\ca}$ is a $d$-representation infinite algebra.
\end{itemize}
\end{theorem}

Although assertion (a) follows from Theorem~\ref{gl.dim of A^I},
we give a more conceptual proof using the properties of the category $\coh\X$.

\begin{proof}
By Proposition~\ref{basic of canonical algebra}(a)(c), the $d$-canonical algebra
$A^{\ca}$ has finite global dimension, and
\[\Hom_{\DDD^{\bo}(\mod A^{\ca})}(A^{\ca},\nu_d^j(A^{\ca})[i])
=\Ext^i_{\X}(T^{\ca},T^{\ca}(j\w)),\]
which is zero except $i=0$ or $i=d$ by Proposition~\ref{vanishing}.
Thus $A^{\ca}$ is almost $d$-representation infinite, and hence it has 
global dimension $d$ or $2d$ by Proposition~\ref{d or 2d}.
By Observation~\ref{gl dim}, $\gl A^{\ca}=d$ if and only if $\Ext^{2d}_{A^{\ca}}(DA^{\ca},A^{\ca})=0$.
Again by Proposition~\ref{vanishing}, we have
\begin{eqnarray*}
\Ext^{2d}_{A^{\ca}}(DA^{\ca},A^{\ca})
&=&\Hom_{\DDD^{\bo}(\mod A^{\ca})}(A^{\ca},\nu_d^{-1}(A^{\ca})[d])
=\Ext^d_{\X}(T^{\ca},T^{\ca}(-\w))\\
&=&\bigoplus_{\x,\y\in[0,d\c]}D(R_{\x-\y+2\w}).
\end{eqnarray*}
By Observation~\ref{basic results on L}(c), it remains to show that $\x-\y+2\w\not\ge0$ for all $\x,\y\in[0,d\c]$ if and only if $n \le d+1$.

Since $\x-\y+2\w$ takes the maximum value when $\x=d\c$ and $\y=0$, the former condition is equivalent to $d\c+2\w\not\ge0$.
This is equivalent to $n\le d+1$ since the normal form of $\de=d\c+2w$ is given by \eqref{another delta}.
\end{proof}

We give an explicit description of the $(d+1)$-preprojective 
algebra (see Section~\ref{section: preliminaries 2}) of the $d$-canonical algebra under the assumption that $\X$ is not $d$-Calabi-Yau.

\begin{proposition}\label{Pi and covering}
Let $\X$ be a GL projective space which is not Calabi-Yau.
Let $R^{[\Z\w]}$ be the covering of $R$ (Definition~\ref{define covering}), and $e$ the idempotent of $R^{[\Z\w]}$
corresponding to the image of the map $[0,d\c]\to\L/\Z\w$.
\begin{itemize}
\item[(a)] The $(d+1)$-preprojective algebra $\Pi$ of $A^{\ca}$ is 
Morita equivalent to $eR^{[\Z\w]}e$.
\item[(b)] Let $S\subset[-d\c,0]$ be a complete set of 
representatives of $([-d\c,0]+\Z\w)/\Z\w$. For $e_S:=\sum_{\x\in S}e_{-\x}\in A^{\ca}$, there is an isomorphism $e_S\Pi e_S\simeq eR^{[\Z\w]}e$ of $\Z$-graded $k$-algebras.
\end{itemize}
\end{proposition}

\begin{proof}
(b) For any $\ell\in\Z$, the degree $\ell$ part of $e_S\Pi e_S$ is given by
\begin{eqnarray*}
&&(e_S\Pi e_S)_{\ell}
=\Hom_{\DDD^{\bo}(\mod A^{\ca})}(A^{\ca}e_S,\nu_d^{\ell}(A^{\ca}e_S))
\simeq\Hom_{\X}(T^{\ca}e_S,T^{\ca}e_S(\ell\w))\\
&=&(\Hom_{\X}(\OO(-\x),\OO(-\y+\ell\w)))_{\x,\y\in S}
=(R_{\x-\y+\ell\w})_{\x,\y\in S}=(eR^{[\Z\w]}e)_{\ell}.
\end{eqnarray*}
Thus the assertion follows.

(a) It suffices to show that $e_S\Pi e_S$ is Morita equivalent to $\Pi$.

For any $\x\in[-d\c,0]$, there exists $\y\in S$ such that $\x-\y\in\Z\w$.
We have an isomorphism
\[\Pi e_{-\x}\simeq\bigoplus_{\ell\in\Z}\Hom_{\X}(T^{\ca},\OO(-\x+\ell\w))\simeq
\bigoplus_{\ell\in\Z}\Hom_{\X}(T^{\ca},\OO(-\y+\ell\w))\simeq\Pi e_{-\y}\]
of projective $\Pi$-modules. Thus $\Pi e_S$ is a progenerator of $\Pi$, and
we have the assertion.
\end{proof}

As an immediate consequence, we have the following result.

\begin{theorem}
Let $\X$ be a GL projective space which is not Calabi-Yau, $A^{\ca}$ 
the $d$-canonical algebra and $\Pi$ the $(d+1)$-preprojective algebra of $A^{\ca}$.
\begin{itemize}
\item[(a)] The center of $\Pi$ is the Veronese subring $R^{(\w)}$ of $R$,
and $\Pi$ is a finitely generated $R^{(\w)}$-module. In particular $\Pi$ is a Noetherian algebra.
\item[(b)] If $n\le d+1$, then we have an isomorphism
$\Pi\simeq R^{[\Z\w]}$ of $\Z$-graded $k$-algebras.
In particular, $A^{\ca}$ is a $d$-representation tame algebra.
\item[(c)] If $\X$ is Fano, then $\Pi$ is Morita equivalent to $R^{[\Z\w]}$.
\end{itemize}
\end{theorem}

\begin{proof}
(a) By Proposition~\ref{Pi and covering}(a), the center of
$\Pi$ is isomorphic to that of $eR^{[\Z\w]}e$, which is
clearly the diagonal $R^{(\w)}$ of $eR^{[\Z\w]}e$.
Since $R$ is a finitely generated $R^{(\w)}$-module,
the remaining assertion follows immediately.

(b) By Proposition~\ref{compare [0,dc] with L/w}(b), we have
that $S:=[-d\c,0]$ itself gives a complete set of representatives and
hence $e_S=1$ and $e=1$.
Thus the assertion follows from Proposition~\ref{Pi and covering}(b).

(c) By Proposition~\ref{compare [0,dc] with L/w}(c), we have $e=1$.
Thus the assertion follows from Proposition~\ref{Pi and covering}(a).
\end{proof}

In the rest of this section, we give examples of $d$-canonical algebras using 
the quiver presentations given in Theorem~\ref{endomorphism convex}. Hence we 
work under Setting~\ref{minimal presentation}. 

\begin{example}\label{canonical example}
For $d=1$ we obtain the classical canonical algebras \cite{Rin,GL}.
More explicitly, the $1$-canonical algebra of type $(p_1, \ldots , p_n)$ has the quiver
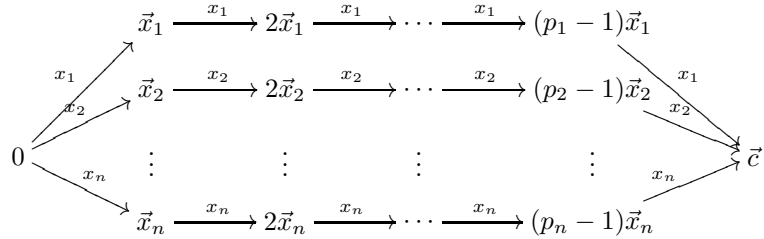
\[
\begin{xy} 0;<5pt,0pt>:<0pt,5pt>:: 
(0,0) *+{0} ="0",
(10,10) *+{\x_1} ="1",
(10,5) *+{\x_2} ="2",
(10,0) *+{\vdots} ="e",
(10,-5) *+{\x_n} ="n",
(20,10) *+{2\x_1} ="11",
(20,5) *+{2\x_2} ="22",
(20,0) *+{\vdots} ="ee",
(20,-5) *+{2\x_n} ="nn",
(30,10) *+{\cdots} ="1e1",
(30,5) *+{\cdots} ="2e2",
(30,0) *+{\vdots} ="eee",
(30,-5) *+{\cdots} ="nen",
(43,10) *+{(p_1-1)\x_1} ="p1",
(43,5) *+{(p_2-1)\x_2} ="p2",
(43,0) *+{\vdots} ="pe",
(43,-5) *+{(p_n-1)\x_n} ="pn",
(55,0) *+{\c} ="c",
"0", {\ar^{x_1}"1"},
"0", {\ar^{x_2}"2"},
"0", {\ar^{x_n}"n"},
"1", {\ar^{x_1}"11"},
"2", {\ar^{x_2}"22"},
"n", {\ar^{x_n}"nn"},
"11", {\ar^{x_1}"1e1"},
"22", {\ar^{x_2}"2e2"},
"nn", {\ar^{x_n}"nen"},
"1e1", {\ar^{x_1\rule{15pt}{0pt}}"p1"},
"2e2", {\ar^{x_2\rule{15pt}{0pt}}"p2"},
"nen", {\ar^{x_n\rule{15pt}{0pt}}"pn"},
"p1", {\ar^{x_1}"c"},
"p2", {\ar^{x_2}"c"},
"pn", {\ar^{x_n}"c"},
\end{xy}
\]
with relations $x_i^{p_i}=\lambda_{i0}x_1^{p_1}+\lambda_{i1}x_2^{p_2}$ for any $i$ with $3\le i\le n$.
\end{example}

\begin{example}\label{Beilinson example}
The $d$-canonical algebra of type $(1,\ldots, 1)$ (where $n = d+1$) is isomorphic to the $d$-Beilinson algebra \cite{Be} and has the quiver
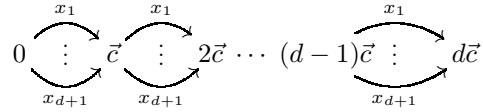
\[\xymatrix{
0 \ar@/^1pc/[r]^{x_1}_{\scalebox{0.85}{$\vdots$}}\ar@/_1pc/[r]_{x_{d+1}}
& \c \ar@/^1pc/[r]^{x_1}_{\scalebox{0.85}{$\vdots$}}\ar@/_1pc/[r]_{x_{d+1}}
& 2\c 
}\cdots\xymatrix{
(d-1)\c \ar@/^1pc/[r]^{x_1}_{\scalebox{0.85}{$\vdots$}}\ar@/_1pc/[r]_{x_{d+1}}
& d\c
}\]
with relations $x_ix_j=x_jx_i$ for any $i$ and $j$ with $1\le i<j\le d+1$.
\end{example}

The $d$-canonical algebras with $d+1$ weights (or less) will be treated in detail in the end of this section. The smallest example with more than $d+1$ weights is considered in the next example.

\begin{example}\label{2222-2} 
The $2$-canonical algebra of type $(2,2,2,2)$ has the quiver
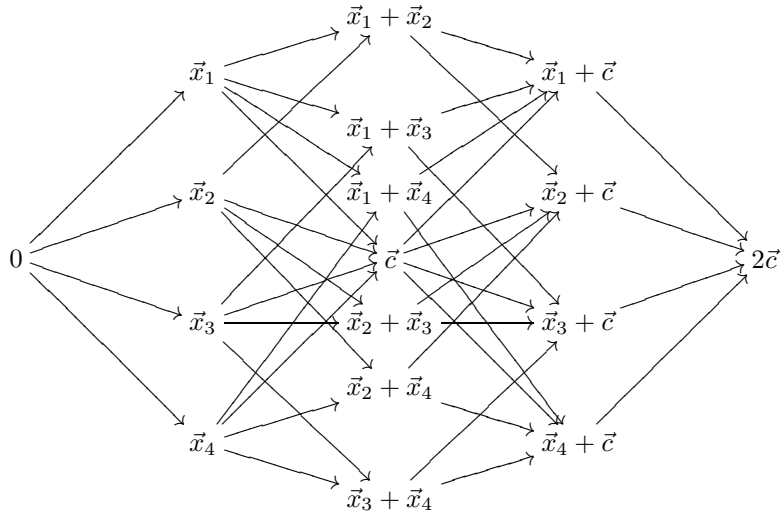
\[
\begin{xy} 0;<3.5pt,0pt>:<0pt,3.5pt>:: 
(-40,0) *+{0} ="0",
(-20,20) *+{\x_1} ="1",
(-20,7) *+{\x_2} ="2",
(-20,-7) *+{\x_3} ="3",
(-20,-20) *+{\x_4} ="4",
(0,26) *+{\x_1+\x_2} ="12",
(0,14) *+{\x_1+\x_3} ="13",
(0,7) *+{\x_1+\x_4} ="14",
(0,-7) *+{\x_2+\x_3} ="23",
(0,-14) *+{\x_2+\x_4} ="24",
(0,-26) *+{\x_3+\x_4} ="34",
(0,0) *+{\c} ="c",
(20,20) *+{\x_1+\c} ="c1",
(20,7) *+{\x_2+\c} ="c2",
(20,-7) *+{\x_3+\c} ="c3",
(20,-20) *+{\x_4+\c} ="c4",
(40,0) *+{2\c} ="2c",
"0", {\ar"1"},
"0", {\ar"2"},
"0", {\ar"3"},
"0", {\ar"4"},
"1", {\ar"c"},
"1", {\ar"12"},
"1", {\ar"13"},
"1", {\ar"14"},
"2", {\ar"12"},
"2", {\ar"c"},
"2", {\ar"23"},
"2", {\ar"24"},
"3", {\ar"13"},
"3", {\ar"23"},
"3", {\ar"c"},
"3", {\ar"34"},
"4", {\ar"14"},
"4", {\ar"24"},
"4", {\ar"34"},
"4", {\ar"c"},
"12", {\ar"c2"},
"12", {\ar"c1"},
"13", {\ar"c3"},
"13", {\ar"c1"},
"14", {\ar"c4"},
"14", {\ar"c1"},
"23", {\ar"c3"},
"23", {\ar"c2"},
"24", {\ar"c4"},
"24", {\ar"c2"},
"34", {\ar"c4"},
"34", {\ar"c3"},
"c", {\ar"c1"},
"c", {\ar"c2"},
"c", {\ar"c3"},
"c", {\ar"c4"},
"c1", {\ar"2c"},
"c2", {\ar"2c"},
"c3", {\ar"2c"},
"c4", {\ar"2c"},
\end{xy}
\]
with relations $\sum_{i=1}^4\lambda_ix_i^4=0$ and $x_ix_j=x_jx_i$ for any $i$ and $j$ with $1\le i<j\le 4$.
\end{example}

For any two element subset $\{i,j\} \subset \{1,2,3,4\}$ we get a full subquiver of the quiver in Example~\ref{2222-2} by identifying the two vertices labeled $\c$ in the following quiver
\[
\begin{xy} 0;<4pt,0pt>:<0pt,4pt>:: 
(0,0) *+{0} ="0",
(10,0) *+{\x_i} ="1",
(20,0) *+{\c} ="c",
(0,7) *+{\x_j} ="2",
(10,7) *+{\x_i+\x_j} ="12",
(20,7) *+{\x_j+\c} ="c2",
(0,14) *+{\c} ="ca",
(10,14) *+{\x_i+\c} ="c1",
(20,14) *+{2\c} ="2c",
"0", {\ar"1"},
"0", {\ar"2"},
"1", {\ar"c"},
"1", {\ar"12"},
"2", {\ar"ca"},
"2", {\ar"12"},
"ca", {\ar"c1"},
"12", {\ar"c1"},
"12", {\ar"c2"},
"c", {\ar"c2"},
"c1", {\ar"2c"},
"c2", {\ar"2c"},
\end{xy}
\]
In fact, the whole quiver is the union of these subquivers.

\begin{observation}
\begin{itemize}
\item[(a)] The quiver of any $2$-canonical algebra of type
$(p_1,\ldots, p_n)$ is the union of the full subquivers parametrized
by two element subsets $\{i,j\} \subset \{1,\ldots, n\}$, that are obtained from
\[
\begin{xy} 0;<4pt,0pt>:<0pt,4pt>:: 
(0,0) *+{0} ="0",
(10,0) *+{\x_i} ="1",
(20,0) *+{\cdots} ="1e",
(30,0) *+{\c} ="c",
(0,7) *+{\x_j} ="2",
(10,7) *+{\x_i+\x_j} ="12",
(20,7) *+{\cdots} ="12e",
(30,7) *+{\x_j+\c} ="c2",
(0,14) *+{\vdots} ="2d",
(10,14) *+{\vdots} ="12d",
(20,14) *+{\vdots} ="dd",
(30,14) *+{\vdots} ="c2d",
(0,21) *+{\c} ="ca",
(10,21) *+{\x_i+\c} ="c1",
(20,21) *+{\cdots} ="c1e",
(30,21) *+{2\c} ="2c",
"0", {\ar"1"},
"0", {\ar"2"},
"1", {\ar"1e"},
"1", {\ar"12"},
"2", {\ar"2d"},
"2", {\ar"12"},
"ca", {\ar"c1"},
"12", {\ar"12e"},
"12", {\ar"12d"},
"c", {\ar"c2"},
"c1", {\ar"c1e"},
"c2", {\ar"c2d"},
"1e", {\ar"c"},
"12e", {\ar"c2"},
"12d", {\ar"c1"},
"2d", {\ar"ca"},
"c1e", {\ar"2c"},
"c2d", {\ar"2c"},
\end{xy}
\]
by identifying the two vertices labeled $\c$.

\item[(b)] For $d > 2$ a similar construction can be carried out replacing
the above rectangles by $d$ dimensional parallelepipeds.
\end{itemize}
\end{observation}

In the rest of this section we treat the case $n\le d+1$ in detail,
which is precisely the case studied by Baer \cite{Ba}. 
In this case $A^{\ca}$ is $d$-representation infinite by
Theorem~\ref{global dimension}.
A class of $d$-representation infinite algebras called
\emph{type $\widetilde{\A}$} was introduced in \cite{HIO}.
We will show that $d$-canonical algebras for the case $n\le d+1$
are $d$-representation infinite algebras of type $\widetilde{\A}$.

\medskip\noindent
{\bf $d$-representation infinite algebras of type $\widetilde{\A}$.}
Let $L$ be the root lattice of the root system
\[\{e_i-e_j \mid 1 \le i \neq j \le d+1\}\]
of type $\A_d$ in $\{v \in \R^{d+1} \mid \sum_{i=1}^{d+1} v_i  = 0\}$.
The abelian group $L$ is freely generated by the simple roots
$\alpha_i = e_i - e_{i-1}$, where $2 \le i \le d+1$.
We further define $\alpha_1  = e_1 - e_{d+1}$ and obtain the relation $\sum_{i=1}^{d+1} \alpha_i = 0$.
Let $\widehat{Q}$ be the quiver defined by
\begin{itemize}
\item $\widehat{Q}_0:=L$,
\item $\widehat{Q}_1:=\{a_i : v \to v+\alpha_i\mid v \in L,\ 1\le i \le d+1\}$. 
\end{itemize}
The group $L$ acts on $\widehat{Q}_0$ by translations, which induces an $L$-action on $\widehat{Q}$ uniquely.

Let $B$ be a subgroup of $L$ such that $L/B$ is finite.
Denote by $\widehat{Q}/B$ the $B$-orbit quiver of $\widehat{Q}$.
We denote the $B$-orbit of a vertex or arrow $x$ by $\overline{x}$. Let
\[\Lambda_B:=k(\widehat{Q}/B)/(\overline{a}_i\overline{a}_j - \overline{a}_j\overline{a}_i:
\overline{v}\to\overline{v} + \overline{\alpha}_i +\overline{\alpha}_j\mid
\overline{v}\in(\widehat{Q}/B)_0,\ 1\le i<j\le d+1).\]
A set of arrows $C$ in $\widehat{Q}$ is called a \emph{cut} if it contains
exactly one arrow from each cycle of length $d+1$ in $\widehat{Q}$.
A cut $C$ is called \emph{$B$-acyclic} if $BC=C$ and the quiver
$\widehat{Q}\setminus C$ is acyclic.
In this case, we call the factor algebra
\[\Lambda_{B,C}:=\Lambda_B/(C/B)\]
a \emph{$d$-representation infinite algebra of type $\widetilde{\A}$} associated
to $(B,C)$. In fact it is $d$-representation infinite by \cite[5.6]{HIO}.
Observe that $\Lambda_{B,C}$ is presented by the quiver 
$(\widehat{Q}\setminus C)/B$ and the images of all commutativity relations $\overline{a}_i\overline{a}_j -\overline{a}_j\overline{a}_i$ in $k((\widehat{Q}\setminus C)/B)$.

\medskip
We will prove the following observation for $d\le n+1$, where we assume without loss of generality that $n= d+1$ by adding hyperplanes with weights $1$ (Observation~\ref{Weights 1}).

\begin{theorem}\label{Atilde theorem}
Let $\X$ be a GL projective space with weights $(p_1,\ldots,p_{d+1})$. Then the $d$-canonical algebra $A^{\ca}$ is
isomorphic to a $d$-representation infinite algebra $\Lambda_{B,C}$
of type $\widetilde{\A}$ for
$B := \langle p_i \alpha_i - p_j \alpha_j \mid 1 \le i,j \le d+1\rangle$
and some $B$-acyclic cut $C$.
\end{theorem}

\begin{proof}
In Theorem~\ref{endomorphism convex}, the $d$-canonical algebra
$A^{\ca}=A^{[0,d\c]}$ was presented by the quiver $Q^{[0,d\c]}$
with commutativity relations $x_ix_j=x_jx_i$.

By our choice of $B$, we have an epimorphism of abelian groups
\[\phi : \L \to L/B\ \mbox{ given by }\ \phi(\x_i) = \alpha_i + B\]
such that the kernel is generated by $\w=-\sum_{i=1}^{d+1}\x_i$.
Since the map $[0,d\c]\to\L/\Z\w$ is bijective
by Proposition~\ref{compare [0,dc] with L/w}(b), the induced map
\[\phi:[0,d\c]\to L/B\]
is bijective. Thus $\phi$ is a bijection between the sets of vertices of the quivers 
$Q^{[0,dc]}$ and $\widehat{Q}/B$. 
In fact $\phi$ extends to a morphism 
\[\phi:Q^{[0,d\c]}\to \widehat{Q}/B\]
of quivers sending
each arrow $x_i$ to the corresponding arrow $\overline{a}_i$. Let $C$ be the union of all $B$-orbits of arrows in $\widehat{Q}$ not in the image of $\phi$.
Then $\phi$ induces an isomorphism
\[\phi:Q^{[0,d\c]}\xrightarrow{\sim}(\widehat{Q}\setminus C)/B\]
of quivers.

We will show that $C$ is a cut.
Any cycle of length $d+1$ in $\widehat{Q}$ is of the form
\begin{equation}\label{d+1 cycle}
v_0 \xrightarrow{a_{\sigma(1)}}v_1 \xrightarrow{a_{\sigma(2)}}\cdots
\xrightarrow{a_{\sigma(d+1)}} v_{d+1} =v_0
\end{equation}
for some $v_0 \in L$ and permutation $\sigma$ of $\{1,\ldots,d+1\}$, where
$v_i = v_0 + \sum_{j=1}^{i} \alpha_{\sigma(j)}$. We need to show that all arrows
in \eqref{d+1 cycle} except one belongs to $\phi(Q^{[0,d\c]}_1)$.
Now let $\y_0$ be the unique element in $[0,d\c]$ such that $\phi(\y_0) = \overline{v}_0$.
Setting $\y_i := \y_{i-1} + \x_{\sigma(i)}$ for each $1 \le i \le d+1$,
we have a sequence
\[0 \le \y_0 < \y_1 < \cdots < \y_{d+1} = \y_0 - \w\]
in $\L$ satisfying $\phi(\y_i) = \overline{v}_i$.
Since $\y_{d+1} = \y_0 - \w \ge -\w$, we have $\y_{d+1}\not\le d\c$
by Lemma~\ref{order omega}.
Thus there is a unique $1\le t \le d+1$ such that $ \y_t \not\le d\c$
and $\y_{t-1} \le d\c$. Hence we have a sequence
\[
0 \le \y_t + \w < \y_{t+1} + \w < \cdots < \y_{d+1} + \w = \y_0 < \y_1 < \cdots < \y_{t-1} \le d\c,
\]
in $[0,d\c]$, and so there is a path
\[
\y_t + \w \xrightarrow{x_{\sigma(t+1)}} \y_{t+1} + \w \xrightarrow{x_{\sigma(t+2)}} \cdots \xrightarrow{x_{\sigma(d+1)}} \y_{d+1} + \w = \y_0 \xrightarrow{x_{\sigma(1)}} \y_1 \xrightarrow{x_{\sigma(2)}} \cdots \xrightarrow{x_{\sigma(t-1)}} \y_{t-1}
\]
in the quiver $Q^{[0,d\c]}$. Thus all arrows in \eqref{d+1 cycle} except $a_{\sigma(t)}: v_{t-1} \to v_t$ are contained in $\phi(Q^{[0,d\c]}_1)$.
Moreover, there is no arrow labeled $x_{\sigma(t)}$ starting at $\y_{t-1}$, since $\y_{t-1} + \x_{\sigma(t)} = \y_t \not \in [0,d\c]$.
Hence $a_{\sigma(t)}: v_{t-1} \to v_t$ is the unique arrow in  \eqref{d+1 cycle}
that is not contained in $\phi(Q^{[0,d\c]}_1)$.
We conclude that $C$ is a cut.

The cut $C$ is $B$-acyclic since $BC=C$ holds clearly and
$(\widehat{Q}\setminus C)/B\simeq Q^{[0,d\c]}$ is acyclic.
Finally, the quiver isomorphism $Q^{[0,d\c]}\simeq(\widehat{Q}\setminus C)/B$
induces an isomorphism
\[A^{[0,d\c]} \to \Lambda_{B,C}\]
of $k$-algebras since both algebras are defined by all the commutativity relations
in their quivers.
\end{proof}

We illustrate Theorem~\ref{Atilde theorem} for $d =2$ in the following example.

\begin{example}
Here is the quiver $Q^{[0,d\c]}$ of the $2$-canonical algebra $A^{\ca}$ of type $(2,3,4)$,
where for each $i=1,2,3$, the full subquivers with vertices $\{a\x_i +\c \mid 0\le a \le p_i\}$
that appear twice should be identified.
{\tiny
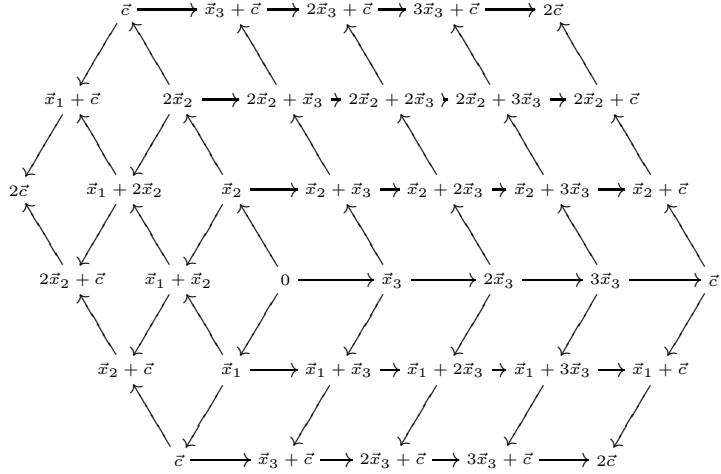
\[
\begin{xy}
0;<40pt,0cm>:<-20pt,34pt>:: 
(0,0) *+{0} ="000",
(1,0) *+{\x_3} ="100",
(2,0) *+{2\x_3} ="200",
(3,0) *+{3\x_3} ="300",
(4,0) *+{\c} ="400",
(0,1) *+{\x_2} ="010",
(0,2) *+{2\x_2} ="020",
(0,3) *+{\c} ="030",
(-1,-1) *+{\x_1} ="001",
(-2,-2) *+{\c} ="002",
(1,1) *+{\x_2+\x_3} ="110",
(2,1) *+{\x_2+2\x_3} ="210",
(3,1) *+{\x_2+3\x_3} ="310",
(4,1) *+{\x_2 +\c} ="410",
(1,2) *+{2\x_2+\x_3} ="120",
(2,2) *+{2\x_2+2\x_3} ="220",
(3,2) *+{2\x_2+3\x_3} ="320",
(4,2) *+{2\x_2 +\c} ="420",
(1,3) *+{\x_3+\c} ="130",
(2,3) *+{2\x_3+\c} ="230",
(3,3) *+{3\x_3+\c} ="330",
(4,3) *+{2\c} ="430",
(0,-1) *+{\x_1+\x_3} ="101",
(1,-1) *+{\x_1+2\x_3} ="201",
(2,-1) *+{\x_1+3\x_3} ="301",
(3,-1) *+{\x_1 +\c} ="401",
(-1,-2) *+{\x_3+\c} ="102",
(0,-2) *+{2\x_3+\c} ="202",
(1,-2) *+{3\x_3+\c} ="302",
(2,-2) *+{2\c} ="402",
(-1,0) *+{\x_1+\x_2} ="011",
(-1,1) *+{\x_1+2\x_2} ="021",
(-1,2) *+{\x_1+\c} ="031",
(-2,-1) *+{\x_2+\c} ="012",
(-2,0) *+{2\x_2+\c} ="022",
(-2,1) *+{2\c} ="032",
"000", {\ar"100"},
"100", {\ar"200"},
"200", {\ar"300"},
"300", {\ar"400"},
"010", {\ar"110"},
"110", {\ar"210"},
"210", {\ar"310"},
"310", {\ar"410"},
"020", {\ar"120"},
"120", {\ar"220"},
"220", {\ar"320"},
"320", {\ar"420"},
"030", {\ar"130"},
"130", {\ar"230"},
"230", {\ar"330"},
"330", {\ar"430"},
"000", {\ar"010"},
"010", {\ar"020"},
"020", {\ar"030"},
"100", {\ar"110"},
"110", {\ar"120"},
"120", {\ar"130"},
"200", {\ar"210"},
"210", {\ar"220"},
"220", {\ar"230"},
"300", {\ar"310"},
"310", {\ar"320"},
"320", {\ar"330"},
"400", {\ar"410"},
"410", {\ar"420"},
"420", {\ar"430"},
"001", {\ar"101"},
"101", {\ar"201"},
"201", {\ar"301"},
"301", {\ar"401"},
"002", {\ar"102"},
"102", {\ar"202"},
"202", {\ar"302"},
"302", {\ar"402"},
"000", {\ar"001"},
"001", {\ar"002"},
"100", {\ar"101"},
"101", {\ar"102"},
"200", {\ar"201"},
"201", {\ar"202"},
"300", {\ar"301"},
"301", {\ar"302"},
"400", {\ar"401"},
"401", {\ar"402"},
"001", {\ar"011"},
"011", {\ar"021"},
"021", {\ar"031"},
"002", {\ar"012"},
"012", {\ar"022"},
"022", {\ar"032"},
"010", {\ar"011"},
"011", {\ar"012"},
"020", {\ar"021"},
"021", {\ar"022"},
"030", {\ar"031"},
"031", {\ar"032"},
\end{xy}
\]}

For comparison we give the corresponding quiver $\widehat{Q}$ below, with cut $C$ indicated by dotted lines. The vertices in the sublattice $B$ are labeled $\ci$. The vertices whose $B$-orbits  are in the image of $\c$ and $2\c$ under $\phi$ are also labelled accordingly.
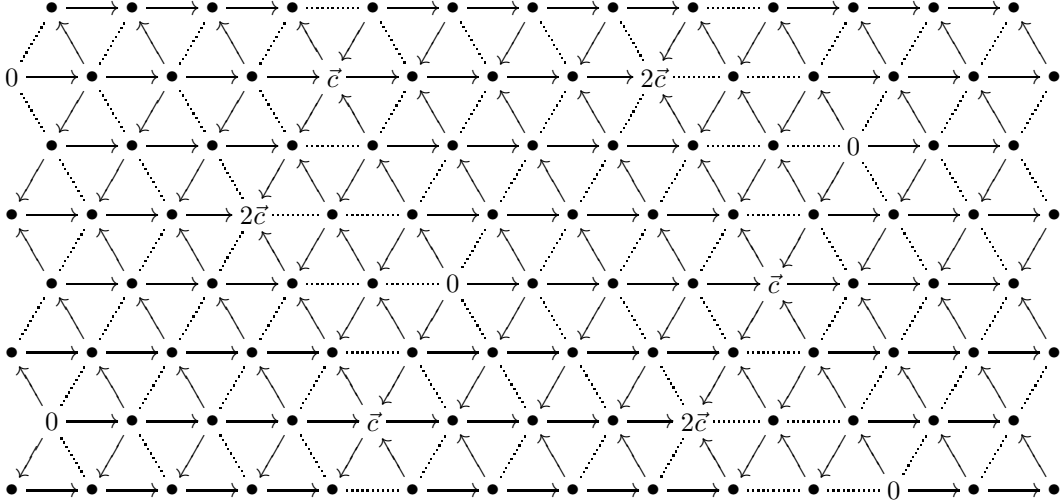
\[
\begin{xy}
0;<30pt,0cm>:<-15pt,25.98pt>:: 
(0,0) *+{\ci} ="000",
(1,0) *+{\bu} ="100",
(2,0) *+{\bu} ="200",
(3,0) *+{\bu} ="300",
(4,0) *+{\c} ="400",
(5,0) *+{\bu} ="500",
(6,0) *+{\bu} ="600",
(7,0) *+{\bu} ="700",
(0,1) *+{\bu} ="010",
(1,1) *+{\bu} ="110",
(2,1) *+{\bu} ="210",
(3,1) *+{\bu} ="310",
(4,1) *+{\bu} ="410",
(5,1) *+{\bu} ="510",
(6,1) *+{\bu} ="610",
(7,1) *+{\bu} ="710",
(8,1) *+{\bu} ="810",
(0,2) *+{\bu} ="020",
(1,2) *+{\bu} ="120",
(2,2) *+{\bu} ="220",
(3,2) *+{\bu} ="320",
(4,2) *+{\bu} ="420",
(5,2) *+{\bu} ="520",
(6,2) *+{\ci} ="620",
(7,2) *+{\bu} ="720",
(8,2) *+{\bu} ="820",
(0,3) *+{\c} ="030",
(1,3) *+{\bu} ="130",
(2,3) *+{\bu} ="230",
(3,3) *+{\bu} ="330",
(4,3) *+{2\c} ="430",
(5,3) *+{\bu} ="530",
(6,3) *+{\bu} ="630",
(7,3) *+{\bu} ="730",
(8,3) *+{\bu} ="830",
(9,3) *+{\bu} ="930",
(0,4) *+{\bu} ="040",
(1,4) *+{\bu} ="140",
(2,4) *+{\bu} ="240",
(3,4) *+{\bu} ="340",
(4,4) *+{\bu} ="440",
(5,4) *+{\bu} ="540",
(6,4) *+{\bu} ="640",
(7,4) *+{\bu} ="740",
(8,4) *+{\bu} ="840",
(9,4) *+{\bu} ="940",
(-1,-1) *+{\bu} ="001",
(0,-1) *+{\bu} ="101",
(1,-1) *+{\bu} ="201",
(2,-1) *+{\bu} ="301",
(3,-1) *+{\bu} ="401",
(4,-1) *+{\bu} ="501",
(5,-1) *+{\bu} ="601",
(6,-1) *+{\bu} ="701",
(7,-1) *+{\bu} ="801",
(-2,-2) *+{\c} ="002",
(-1,-2) *+{\bu} ="102",
(0,-2) *+{\bu} ="202",
(1,-2) *+{\bu} ="302",
(2,-2) *+{2\c} ="402",
(3,-2) *+{\bu} ="502",
(4,-2) *+{\bu} ="602",
(5,-2) *+{\bu} ="702",
(6,-2) *+{\bu} ="802",
(-3,-3) *+{\bu} ="003",
(-2,-3) *+{\bu} ="103",
(-1,-3) *+{\bu} ="203",
(0,-3) *+{\bu} ="303",
(1,-3) *+{\bu} ="403",
(2,-3) *+{\bu} ="503",
(3,-3) *+{\bu} ="603",
(4,-3) *+{\ci} ="703",
(5,-3) *+{\bu} ="803",
(6,-3) *+{\bu} ="903",
(-1,0) *+{\bu} ="011",
(-1,1) *+{\bu} ="021",
(-1,2) *+{\bu} ="031",
(-1,3) *+{\bu} ="041",
(-1,4) *+{\bu} ="051",
(-2,-1) *+{\bu} ="012",
(-2,0) *+{\bu} ="022",
(-2,1) *+{2\c} ="032",
(-2,2) *+{\bu} ="042",
(-2,3) *+{\bu} ="052",
(-2,4) *+{\bu} ="062",
(-3,-2) *+{\bu} ="013",
(-3,-1) *+{\bu} ="023",
(-3,0) *+{\bu} ="033",
(-3,1) *+{\bu} ="043",
(-3,2) *+{\bu} ="053",
(-3,3) *+{\bu} ="063",
(-3,4) *+{\bu} ="073",
(-4,-3) *+{\bu} ="014",
(-4,-2) *+{\bu} ="024",
(-4,-1) *+{\bu} ="034",
(-4,0) *+{\bu} ="044",
(-4,1) *+{\bu} ="054",
(-4,2) *+{\bu} ="064",
(-4,3) *+{\ci} ="074",
(-5,-3) *+{\bu} ="025",
(-5,-2) *+{\bu} ="035",
(-5,-1) *+{\bu} ="045",
(-5,0) *+{\bu} ="055",
(-5,1) *+{\bu} ="065",
(-6,-3) *+{\bu} ="036",
(-6,-2) *+{\ci} ="046",
(-6,-1) *+{\bu} ="056",
(-7,-3) *+{\bu} ="047",
"000", {\ar"100"},
"100", {\ar"200"},
"200", {\ar"300"},
"300", {\ar"400"},
"400", {\ar"500"},
"500", {\ar"600"},
"600", {\ar"700"},
"010", {\ar"110"},
"110", {\ar"210"},
"210", {\ar"310"},
"310", {\ar"410"},
"410", {\cut"510"},
"510", {\ar"610"},
"610", {\ar"710"},
"710", {\ar"810"},
"020", {\ar"120"},
"120", {\ar"220"},
"220", {\ar"320"},
"320", {\ar"420"},
"420", {\cut"520"},
"520", {\cut"620"},
"620", {\ar"720"},
"720", {\ar"820"},
"030", {\ar"130"},
"130", {\ar"230"},
"230", {\ar"330"},
"330", {\ar"430"},
"430", {\cut"530"},
"530", {\cut"630"},
"630", {\ar"730"},
"730", {\ar"830"},
"830", {\ar"930"},
"040", {\cut"140"},
"140", {\ar"240"},
"240", {\ar"340"},
"340", {\ar"440"},
"440", {\ar"540"},
"540", {\cut"640"},
"640", {\ar"740"},
"740", {\ar"840"},
"840", {\ar"940"},
"001", {\ar"101"},
"101", {\ar"201"},
"201", {\ar"301"},
"301", {\ar"401"},
"401", {\cut"501"},
"501", {\ar"601"},
"601", {\ar"701"},
"701", {\ar"801"},
"002", {\ar"102"},
"102", {\ar"202"},
"202", {\ar"302"},
"302", {\ar"402"},
"402", {\cut"502"},
"502", {\cut"602"},
"602", {\ar"702"},
"702", {\ar"802"},
"003", {\cut"103"},
"103", {\ar"203"},
"203", {\ar"303"},
"303", {\ar"403"},
"403", {\ar"503"},
"503", {\cut"603"},
"603", {\cut"703"},
"703", {\ar"803"},
"803", {\ar"903"},
"011", {\cut"000"},
"022", {\cut"011"},
"033", {\ar"022"},
"044", {\ar"033"},
"055", {\ar"044"},
"021", {\cut"010"},
"032", {\cut"021"},
"043", {\ar"032"},
"054", {\ar"043"},
"065", {\ar"054"},
"031", {\cut"020"},
"042", {\ar"031"},
"053", {\ar"042"},
"064", {\ar"053"},
"041", {\ar"030"},
"052", {\ar"041"},
"063", {\ar"052"},
"074", {\ar"063"},
"051", {\ar"040"},
"062", {\ar"051"},
"073", {\ar"062"},
"012", {\cut"001"},
"023", {\ar"012"},
"034", {\ar"023"},
"045", {\ar"034"},
"056", {\ar"045"},
"013", {\ar"002"},
"024", {\ar"013"},
"035", {\ar"024"},
"046", {\ar"035"},
"014", {\ar"003"},
"025", {\ar"014"},
"036", {\ar"025"},
"047", {\ar"036"},
%
"000", {\ar"010"},
"010", {\ar"020"},
"020", {\ar"030"},
"030", {\ar"040"},
"001", {\ar"011"},
"011", {\ar"021"},
"021", {\ar"031"},
"031", {\cut"041"},
"041", {\ar"051"},
"002", {\ar"012"},
"012", {\ar"022"},
"022", {\ar"032"},
"032", {\cut"042"},
"042", {\cut"052"},
"052", {\ar"062"},
"003", {\cut"013"},
"013", {\ar"023"},
"023", {\ar"033"},
"033", {\ar"043"},
"043", {\cut"053"},
"053", {\cut"063"},
"063", {\ar"073"},
"014", {\cut"024"},
"024", {\ar"034"},
"034", {\ar"044"},
"044", {\ar"054"},
"054", {\cut"064"},
"064", {\cut"074"},
"025", {\cut"035"},
"035", {\ar"045"},
"045", {\ar"055"},
"055", {\ar"065"},
"036", {\cut"046"},
"046", {\ar"056"},
"100", {\ar"110"},
"110", {\ar"120"},
"120", {\ar"130"},
"130", {\cut"140"},
"200", {\ar"210"},
"210", {\ar"220"},
"220", {\ar"230"},
"230", {\cut"240"},
"300", {\ar"310"},
"310", {\ar"320"},
"320", {\ar"330"},
"330", {\cut"340"},
"400", {\ar"410"},
"410", {\ar"420"},
"420", {\ar"430"},
"430", {\cut"440"},
"500", {\cut"510"},
"510", {\ar"520"},
"520", {\ar"530"},
"530", {\ar"540"},
"600", {\cut"610"},
"610", {\cut"620"},
"620", {\ar"630"},
"630", {\ar"640"},
"700", {\cut"710"},
"710", {\cut"720"},
"720", {\ar"730"},
"730", {\ar"740"},
"810", {\cut"820"},
"820", {\ar"830"},
"830", {\ar"840"},
"930", {\ar"940"},
"101", {\cut"000"},
"202", {\cut"101"},
"303", {\ar"202"},
"201", {\cut"100"},
"302", {\cut"201"},
"403", {\ar"302"},
"301", {\cut"200"},
"402", {\cut"301"},
"503", {\ar"402"},
"401", {\cut"300"},
"502", {\ar"401"},
"603", {\ar"502"},
"501", {\ar"400"},
"602", {\ar"501"},
"703", {\ar"602"},
"601", {\ar"500"},
"702", {\ar"601"},
"803", {\ar"702"},
"701", {\ar"600"},
"802", {\ar"701"},
"903", {\ar"802"},
"801", {\ar"700"},
"102", {\cut"001"},
"203", {\ar"102"},
"103", {\ar"002"},
%
"000", {\ar"001"},
"001", {\ar"002"},
"002", {\ar"003"},
"100", {\ar"101"},
"101", {\ar"102"},
"102", {\cut"103"},
"200", {\ar"201"},
"201", {\ar"202"},
"202", {\cut"203"},
"300", {\ar"301"},
"301", {\ar"302"},
"302", {\cut"303"},
"400", {\ar"401"},
"401", {\ar"402"},
"402", {\cut"403"},
"500", {\cut"501"},
"501", {\ar"502"},
"502", {\ar"503"},
"600", {\cut"601"},
"601", {\cut"602"},
"602", {\ar"603"},
"700", {\cut"701"},
"701", {\cut"702"},
"702", {\cut"703"},
"801", {\cut"802"},
"802", {\cut"803"},
"010", {\ar"011"},
"011", {\ar"012"},
"012", {\cut"013"},
"013", {\ar"014"},
"020", {\ar"021"},
"021", {\ar"022"},
"022", {\cut"023"},
"023", {\cut"024"},
"024", {\ar"025"},
"030", {\ar"031"},
"031", {\ar"032"},
"032", {\cut"033"},
"033", {\cut"034"},
"034", {\cut"035"},
"035", {\ar"036"},
"040", {\cut"041"},
"041", {\ar"042"},
"042", {\ar"043"},
"043", {\cut"044"},
"044", {\cut"045"},
"045", {\cut"046"},
"046", {\ar"047"},
"051", {\cut"052"},
"052", {\ar"053"},
"053", {\ar"054"},
"054", {\cut"055"},
"055", {\cut"056"},
"062", {\cut"063"},
"063", {\ar"064"},
"064", {\ar"065"},
"073", {\cut"074"},
"110", {\cut"000"},
"220", {\cut"110"},
"330", {\cut"220"},
"440", {\ar"330"},
"210", {\cut"100"},
"320", {\cut"210"},
"430", {\cut"320"},
"540", {\ar"430"},
"310", {\cut"200"},
"420", {\cut"310"},
"530", {\ar"420"},
"640", {\ar"530"},
"410", {\cut"300"},
"520", {\ar"410"},
"630", {\ar"520"},
"740", {\cut"630"},
"510", {\ar"400"},
"620", {\ar"510"},
"730", {\cut"620"},
"840", {\cut"730"},
"610", {\ar"500"},
"720", {\ar"610"},
"830", {\cut"720"},
"940", {\cut"830"},
"710", {\ar"600"},
"820", {\ar"710"},
"930", {\cut"820"},
"810", {\ar"700"},
"120", {\cut"010"},
"230", {\cut"120"},
"340", {\ar"230"},
"130", {\cut"020"},
"240", {\ar"130"},
"140", {\ar"030"},
\end{xy}
\]
\end{example}

\section{Orlov-type semiorthogonal decompositions}

In this section, we give an alternative proof of Theorem~\ref{tilting}
by constructing an embedding of $\DDD^{\bo}(\coh\X)$ to
$\DDD^{\bo}(\mod^{\L}R)$, which is parallel to the proof of Theorem~\ref{CM tilting}.
Then we will give Orlov-type semiorthogonal decompositions
for $\DDD^{\bo}(\coh\X)$ and $\underline{\CM}^{\L}R$.
We use the notation from Section~\ref{subsection: GL CI 2}.

We start with the following analog of Theorem~\ref{basic embedding}.

\begin{theorem}\label{basic embedding 2}
For any non-trivial upset $I$ in $\L$, we have
\[ \DDD^{\bo}(\mod^{\L}R)=\DDD^{\bo}(\mod_0^I R) \perp( \DDD^{\bo}(\mod^{I} R) \cap ( \DDD^{\bo}(\mod^{\w-I^c} R))^\star)\perp \DDD^{\bo}(\mod_0^{I^c} R).\]
Thus the composition
\[ \DDD^{\bo}(\mod^{I} R) \cap ( \DDD^{\bo}(\mod^{\w-I^c} R))^\star\subset\DDD^{\bo}(\mod^{\L}R)
\stackrel{\pi}{\longrightarrow}\DDD^{\bo}(\coh\X)\]
is a triangle equivalence.
\end{theorem}

We need the following elementary observation, which is an $\L$-graded version
of \cite[2.3]{O}.

\begin{lemma}\label{decompositions for S}
Let $I$ be a non-empty upset in $\L$.
\begin{itemize}
\item[(a)] We have
\[ \DDD^{\bo}(\mod_0^\L R) = \DDD^{\bo}(\mod_0^I R) \perp \DDD^{\bo}(\mod_0^{I^c} R) \text{ and } \DDD^{\bo}(\mod^\L R) = \DDD^{\bo}(\mod^I R) \perp \DDD^{\bo}(\mod_0^{I^c} R). \]
More generally, for two upsets $I \subseteq J \subseteq \L$, we have
\[ \DDD^{\bo}(\mod_0^J R)=\DDD^{\bo}(\mod_0^I R) \perp \DDD^{\bo}(\mod_0^{J\setminus I} R) \text{ and } \DDD^{\bo}(\mod^J R) = \DDD^{\bo}(\mod^I R) \perp \DDD^{\bo}(\mod_0^{J\setminus I} R). \]
\item[(b)] We have a triangle equivalence $\DDD^{\bo}(\mod^I R) / \DDD^{\bo}(\mod_0^I R) \simeq\DDD^{\bo}(\coh\X)$.
\end{itemize}
\end{lemma}

\begin{proof}
(a) Clearly $\Hom_{\DDD^{\bo}(\mod^\L R)}(\DDD^{\bo}(\mod^I R),\DDD^{\bo}(\mod_0^{I^c} R))=0$ holds.
We have the exact functors $(-)^I:\mod^{\L}R\to\mod^IR$,
$(-)^{I^c}:\mod^{\L}R\to\mod^{I^c}_0R$ and the sequence
\[0\to(-)^I\to {\rm id}\to(-)^{I^c}\to0\]
of natural transformations which is objectwise exact.
Therefore we have induced triangle functors
$(-)^I\colon  \DDD^{\bo}(\mod^\L R) \to \DDD^{\bo}(\mod^I R)$, $(-)^{I^c}\colon \DDD^{\bo}(\mod^\L R) \to \DDD^{\bo}(\mod_0^{I^c} R)$ and a functorial triangle
$X^I\to X\to X^{I^c}\to X^I[1]$ for $X\in \DDD^{\bo}(\mod^\L R)$.
Thus we have the first two equalities.

The remaining equalities are shown similarly.

(b) By (a), we have triangle equivalences
\[ \DDD^{\bo}(\mod_0^I R) \simeq \frac{\DDD^{\bo}(\mod_0^\L)}{\DDD^{\bo}(\mod_0^{I^c} R)} \text{ and } \DDD^{\bo}(\mod^I R) \simeq \frac{\DDD^{\bo}(\mod^\L R)}{\DDD^{\bo}(\mod_0^{I^c} R)}. \]
Therefore we have
\[ \DDD^{\bo}(\coh\X) \simeq \frac{\DDD^{\bo}(\mod^\L R)}{\DDD^{\bo}(\mod_0^\L R)} \simeq \frac{\DDD^{\bo}(\mod^\L R) / \DDD^{\bo}(\mod_0^{I^c} R)}{\DDD^{\bo}(\mod_0^\L R) / \DDD^{\bo}(\mod_0^{I^c} R)} \simeq \frac{\DDD^{\bo}(\mod^I R)}{\DDD^{\bo}(\mod_0^I R)}. \qedhere \]
\end{proof}

\begin{proof}[Proof of Theorem~\ref{basic embedding 2}]
Applying Lemma~\ref{decompositions for S}(a) to the upset $\w-I^c$, we have
\[ \DDD^{\bo}(\mod^\L R) = \DDD^{\bo}(\mod^{\w-I^c} R) \perp \DDD^{\bo}(\mod_0^{\w-I} R). \] Applying $(-)^\star$ and using \eqref{* for SP}, we obtain
\[ \DDD^{\bo}(\mod^\L R) = \DDD^{\bo}(\mod^\L R)^\star= \DDD^{\bo}(\mod_0^{I} R) \perp(\DDD^{\bo}(\mod^{\w-I^c} R))^\star.\]
Taking the intersections of $\DDD^{\bo}(\mod^I R)$
with both sides and applying Observation~\ref{distributive}, we have
\[ \DDD^{\bo}(\mod^{I} R)= \DDD^{\bo}(\mod_0^{I} R) \perp(\DDD^{\bo}(\mod^{I} R) \cap (\DDD^{\bo}(\mod^{\w-I^c} R))^\star). \]
Thus the first claim holds by Lemma~\ref{decompositions for S}(a), and the second one by
\[ \DDD^{\bo}(\mod^{I} R) \cap ( \DDD^{\bo}(\mod^{\w-I^c} R))^\star \simeq \frac{\DDD^{\bo}(\mod^{I} R)}{\DDD^{\bo}(\mod_0^{I} R)}\]
and Lemma~\ref{decompositions for S}(b).
\end{proof}

Using Theorem~\ref{basic embedding 2}, we are able to give an alternative proof of
Theorem~\ref{tilting}.
The following analog of Theorem~\ref{CM tilting} is the main result in this section.

\begin{theorem}\label{coh X tilting}
Let $\X$ be a Geigle-Lenzing projective space.
\begin{itemize}
\item[(a)] The following composition is a triangle equivalence:
\[\KKK^{\bo}(\proj^{[0,d\c]} R) \subset \DDD^{\bo}(\mod^{\L}R)\stackrel{\pi}{\longrightarrow}\DDD^{\bo}(\coh\X).\]
\item[(b)] We have triangle equivalences
\[\DDD^{\bo}(\mod A^{\ca})\simeq \KKK^{\bo}(\proj^{[0,d\c]} R) \simeq\DDD^{\bo}(\coh\X)\ \mbox{ such that }\ A^{\rm ca}\mapsto T^{[0,d\c]}\mapsto \pi(T^{[0,d\c]})=T^{\ca}(-d\c).\]
In particular $\DDD^{\bo}(\coh\X)$ has a tilting object $T^{\ca}$.
\item[(c)] We have
\begin{eqnarray*}
\KKK^{\bo}(\proj^{[0,d\c]} R) &=& \DDD^{\bo}(\mod^{\L_+} R) \cap (\DDD^{\bo}(\mod^{\w-\L_+^c} R))^\star,\\
\DDD^{\bo}(\mod^{\L}R) &=& \DDD^{\bo}(\mod_0^{\L_+} R) \perp \KKK^{\bo}(\proj^{[0,d\c]} R) \perp \DDD^{\bo}(\mod_0^{\L_+^c} R).
\end{eqnarray*}
\end{itemize}
\end{theorem}

\begin{proof}
We only have to prove (c). In fact, (a) follows from (c) and Theorem~\ref{basic embedding 2},
and (b) follows from (a) and Theorem~\ref{tilting in S^I}(a).

In the rest, we prove the statement (c). By Lemma~\ref{order omega}, $-\w\not\le\x$
if and only if $\x\le d\c$. Thus we have
\begin{equation}\label{dc}
\L_+\cap(\L_+^c-\w)=[0,d\c].
\end{equation}
As a consequence, we have
\begin{align*}
& \DDD^{\bo}(\mod^{\L_+} R) \cap (\DDD^{\bo}(\mod^{\w-\L_+^c} R))^\star \supset \KKK^{\bo}(\proj^{\L_+} R) \cap (\KKK^{\bo}(\proj^{\w-\L_+^c} R))^\star \\
& \quad \stackrel{\eqref{* for SP}}{=} \KKK^{\bo}(\proj^{\L_+} R) \cap \KKK^{\bo}(\proj^{\L_+^c-\w} R) 
= \KKK^{\bo}(\proj^{\L_+\cap(\L_+^c-\w)} R) \stackrel{\eqref{dc}}{=}\KKK^{\bo}(\proj^{[0,d\c]} R).
\end{align*}
To show the converse, it is enough to show that the composition
$\KKK^{\bo}(\proj^{[0,d\c]} R) \subset \DDD^{\bo}(\mod^\L R) \stackrel{\pi}{\to} \DDD^{\bo}(\coh\X)$ is dense.
This follows from Proposition~\ref{generate}.
\end{proof}

In the rest of this section, we show that there is a close connection between
$\underline{\CM}^{\L}R$ and $\DDD^{\bo}(\coh\X)$ given in terms of
the following Orlov-type semiorthogonal decompositions \cite[2.5]{O}\cite[C.4]{KLM}.

\begin{theorem}\label{Orlov-type}
There exist embeddings $\DDD^{\bo}(\coh\X)\to\DDD^{\bo}(\mod^{\L}R)$
and $\underline{\CM}^{\L}R\to\DDD^{\bo}(\mod^{\L}R)$ such that we have
semiorthogonal decompositions
\[\begin{array}{ll}
\DDD^{\bo}(\coh\X)=\underline{\CM}^{\L}R\perp\KKK^{\bo}(\proj^{\delta^{-1}[0,-\delta(\w))} R)
&\mbox{if $\X$ is Fano},\\
\DDD^{\bo}(\coh\X)=\underline{\CM}^{\L}R
&\mbox{if $\X$ is Calabi-Yau},\\
\DDD^{\bo}(\coh\X)\perp\DDD^{\bo}(\mod^{\delta^{-1}[0,\delta(\w))} R)=
\underline{\CM}^{\L}R&\mbox{if $\X$ is anti-Fano}.
\end{array}\]
In particular, we have the equality:
\[{\rm rank} K_0(\coh\X)-{\rm rank} K_0(\underline{\CM}^{\L}R)=\left\{
\begin{array}{ll}
\#(\L/\Z\w)&\mbox{$\X$ is Fano},\\
0&\mbox{$\X$ is Calabi-Yau},\\
-\#(\L/\Z\w)&\mbox{$\X$ is anti-Fano}.
\end{array}\right.\]
\end{theorem}

\begin{proof}
Using the map $\delta:\L\to\Q$ defined by $\delta(\x_i)=1/p_i$, we obtain the non-trivial upset $I:=\{\x\in\L\mid\delta(\x)\ge0\}$.
Applying Theorems~\ref{basic embedding 2} and \ref{basic embedding} to $I$,
we have identifications
\begin{equation} \label{SODs}
\begin{aligned}
\DDD^{\bo}(\coh\X) & =\DDD^{\bo}(\mod^I R) \cap (\DDD^{\bo}(\mod^{\w-I^c} R))^\star\ \text{ and } \\
\underline{\CM}^{\L}R & = \DDD^{\bo}(\mod^I R) \cap (\DDD^{\bo}(\mod^{-I^c} R))^\star.
\end{aligned}
\end{equation}
Note that we have $\w-I^c=\{\x\in\L\mid\delta(\x)>\delta(\w)\}$ and
$-I^c=\{\x\in\L\mid\delta(\x)>0\}$.

Assume first that $\X$ is Calabi-Yau.
Then $\w-I^c=-I^c$ holds, and we have the assertion immediately from \eqref{SODs}.

Assume next that $\X$ is Fano. Then $\w-I^c\supset-I^c$ holds.
We have
\[ \DDD^{\bo}(\mod^{\w-I^c} R) = \KKK^{\bo}(\proj^{(\w-I^c)\setminus(-I^c)} R) \perp \DDD^{\bo}(\mod^{-I^c} R) \]
by Lemma~\ref{decompositions for P}(a).
Applying $(-)^\star$, we obtain
\[ (\DDD^{\bo}(\mod^{\w-I^c} R))^\star=(\DDD^{\bo}(\mod^{-I^c} R))^\star\perp\KKK^{\bo}(\proj^{(-\w+I^c)\setminus I^c} R). \]
Taking the intersections of $\DDD^{\bo}(\mod^I R)$ with both sides and applying
Observation~\ref{distributive}, we have
\[ \DDD^{\bo}(\mod^I R) \cap (\DDD^{\bo}(\mod^{\w-I^c} R))^\star=(\DDD^{\bo}(\mod^I R) \cap (\DDD^{\bo}(\mod^{-I^c} R))^\star) \perp \KKK^{\bo}(\proj^{(-\w+I^c)\setminus I^c} R).\]
Since $(-\w+I^c)\setminus I^c=\delta^{-1}([0,-\delta(\w)])$, we have the desired assertion
by from \eqref{SODs}.

Assume finally that $\X$ is anti-Fano. Then $\w-I^c\subset-I^c$.
We have
\[ \DDD^{\bo}(\mod^{\w-I^c} R) \perp \DDD^{\bo}(\mod_0^{(-I^c) \setminus(\w-I^c)} R)=\DDD^{\bo}(\mod^{-I^c} R) \]
by Lemma~\ref{decompositions for S}(a).
Applying $(-)^\star$, we obtain
\[ \DDD^{\bo}(\mod_0^{(\w+I^c)\setminus I^c} R) \perp(\DDD^{\bo}(\mod^{\w-I^c} R))^\star=(\DDD^{\bo}(\mod^{-I^c} R))^\star. \]
Taking the intersection of $\DDD^{\bo}(\mod^I R)$ with both sides and applying
Observation~\ref{distributive}, we have
\[ \DDD^{\bo}(\mod_0^{(\w+I^c)\setminus I^c} R) \perp(\DDD^{\bo}(\mod^I R) \cap (\DDD^{\bo}(\mod^{\w-I^c} R))^\star)=\DDD^{\bo}(\mod^I R)\cap (\DDD^{\bo}(\mod^{-I^c} R))^\star.\]
Since $(\w+I^c)\setminus I^c=\delta^{-1}([0,\delta(\w)])$,
we have the desired assertion from \eqref{SODs}.
\end{proof}

As a consequence of Theorem~\ref{Orlov-type}, we have semiorthogonal
decompositions between the derived categories of the $d$-canonical algebra
$A^{\ca}$ and the CM-canonical algebra $A^{\rm CM}$.

We often have a more direct connection between $A^{\ca}$ and $A^{\rm CM}$. The following is such an example.

\begin{example}
Assume that $n = 2d+2$ holds. Then we have $\de = d\c + \sum_{i=1}^n (p_i-2) \x_i$,
and in particular, $[0, d\c] \subset [0, \de]$ holds.
Therefore there exists an idempotent $e$ of $A^{\rm CM}$ such that
\[A^{\ca}=eA^{\rm CM}e.\]
If moreover all $p_i =2$, then $[0, d\c] = [0, \de]$ holds and we have
$A^{\ca}=A^{\rm CM}$. In this case $\X$ is Calabi-Yau, and a derived
equivalence between $A^{\ca}$ and $A^{\rm CM}$ can be obtained
from Theorem~\ref{Orlov-type}.
\end{example}

\section{Coxeter polynomials}\label{section:Coxeter} 

Let $\X$ be a Geigle-Lenzing projective space of dimension $d+1$ with weights $p_1, \ldots, p_n$.
For $X \in \DDD^{\bo}(\coh \X)$, we denote by $[X]$ the 
corresponding element in the Grothendieck group $K_0(\DDD^{\bo}(\coh \X))=K_0(\coh\X)$. By abuse of notation we denote by $(\x)$ the action of the shift by $\x\in\L$ on $K_0(\coh\X)$, that is $[X](\x) = [X(\x)]$. 

By Theorem~\ref{Serre}(f), the $d$-th suspension of the Serre functor of $\DDD^{\bo}(\coh\X)$ is given by $(\w)$.
The aim of this section is to determine the \emph{Coxeter polynomial} \( \Phi_{\X} \) of $\X$, which is the characteristic polynomial of the action of $(\w)$ on the Grothendieck group $K_0(\coh\X)$. We start with giving an explicit basis of $K_0(\coh\X)$.

\begin{definition} 
Let $0 \leq e \leq d$, and let $I \subseteq \{1, \ldots, n\}$ have 
cardinality at most $d - e$. 
Choose $d-e-|I|$ linear forms $f_j\in C=k[T_0,\ldots,T_d]$ such that $(\ell_i,f_j\mid i \in I,\ 1\le j\le d-e-|I|)$ are linearly independent. Define the element of the Grothendieck group by
\[G^e_I=[R^e_I]\in K_0(\coh\X)\ \mbox { for }\ R^e_I=\frac{R}{(X_i, f_j \mid i \in I,\ 1\le j\le
d-e-|I|)}. \]
Note that $G^e_I$ is independent of our choice of the $f_j$ by Lemma~\ref{lem.dimvectors_objects} below though $R^e_I$ does depends on the choice.
\end{definition} 

Vaguely the interpretation of this module is that it corresponds to the 
structure sheaf on the intersection of the $e$ dimensional subspace 
formed by the intersection of the $d-e-|I|$ ``generic'' hyperplanes 
$f_j$, and the $|I|$ special hyperplanes $\ell_i$. 

\begin{lemma} \label{lem.dimvectors_objects} 
We have 
\[ G^e_I = \sum_{J \subseteq I} \sum_{a=0}^{d-e-|I|} (-1)^{a + 
|J|} {d-e-|I| \choose a} [\mathcal{O}](-a \c - \sum_{j \in J} \x_j). \] 
In particular $G^e_I$ is independent of the choice of hyperplanes 
$f_j$. 
\end{lemma} 

\begin{proof} 
Since $(X_i,f_j \mid i \in I,\ 1\le j\le d-e-|I|)$ form a regular sequence by 
Lemma~\ref{prop.regularseq1}(c), 
we may use the associated Koszul complex to compute $G^e_I$ in terms of the classes of shifts of $\mathcal{O}$. The formula follows. 
\end{proof} 

We collect the following immediate consequences of the definition, 
which will allow us to compute Coxeter polynomials.

\begin{proposition} \label{prop.calculating_basis_vector_relation} 
For $I$ and $e$ as above we have the following 
\begin{itemize} 
\item[(a)] If $j \not\in I$ then $G^e_I - G^e_I(- \x_j) = 
\begin{cases} 0 & \text{if } e=0 \\ G^{e-1}_{I \cup \{j\}} & 
\text{otherwise.} \end{cases}$; 
\item[(b)] $G^e_I - G^e_I(- \c) = \begin{cases} 0 & \text{if } 
e=0 \\ G^{e-1}_I & \text{otherwise.} \end{cases}$; 
\item[(c)] For $i \in I$ we have $\sum_{a=0}^{p_i - 1} G^e_I(- a \x_i) 
= G^e_{I \setminus \{i\}}$. 
\end{itemize} 
\end{proposition} 

\begin{proof} 
(a) Choosing the $f_k$ such that $(\ell_i,f_k \mid i \in I\cup\{j\},\ 1\le k\le d-e-|I|)$ are linearly 
independent, we define $R^{e-1}_{I \cup \{j\}}$ and $R^e_I$. Since the sequence 
\[ 0 \to R^e_I(- \x_j) \xrightarrow{X_j} R^e_I \to R^{e-1}_{I \cup \{j\}} \to 0, \] 
is exact by Lemma~\ref{prop.regularseq1}(c), the claim follows.

(b) This can be seen similarly from the exact sequence $0 \to R^e_I(-\c) \xrightarrow{f_1} R^e_I \to R^{e-1}_I \to 0$.

(c) Choosing the $f_k$ such that $(\ell_j,f_k \mid j \in I,\ 1\le k\le d-e-|I|)$ are linearly independent, we define $R^e_I$. Setting $f_{d-e-|I|+1}=\ell_i$, we define $R^e_{I\setminus\{i\}}$. Since $R^e_{I\setminus\{i\}}$ is filtered by $R^e_I(-a\x_i)$ with $0\le a\le p_i-1$, the claim follows.
\end{proof} 

For $I,I'\subset\{1,\ldots,n\}$ and $0\le e\le d-|I|$, $0\le e'\le d-|I'|$, we define
\begin{equation}\label{(e,I) order}
(e',I')\le(e,I)\Longleftrightarrow\mbox{$e'<e$ or ($e'=e$ and $I' \subseteq I$).}
\end{equation}
This gives a partial order on such pairs.
We define subgroups of $K_0(\coh\X)$ by
\[ K_0^{\le (e,I)} =\langle G^{e'}_{I'}(\x)\mid \x\in\L,\ (e',I')\le (e,I)\rangle\supset K_0^{< (e,l)} =\langle G^{e'}_{I'}(\x)\mid \x\in\L,\  (e',I')<(e,I)\rangle. \]
We denote by \( K_0^{(e,I)} = K_0^{\le (e,I)} / K_0^{< (e,I)} \) the factor groups.

Note that by Proposition~\ref{prop.calculating_basis_vector_relation}, the subgroups are closed under the action of \( \L \), and in particular under the action of \( \w \). It follows that \( \w \) induces an endomorphism of the subfactors \( K_0^{(e,I)} \), whose characteristic polynomial we will denote by \( \Phi^{(e, I)} \).

\begin{lemma}\label{KL}
The subfactor group \( K_0^{(e,I)} \) is generated by the images of  the \( G^e_I(\sum_{i \in I} a_i\x_i) \) with \( 0 < a_i < p_i \).
\end{lemma}

\begin{proof}
Let \( \L' \) be the set of all \( \x \in \L \) such that the image of \( G^e_I(\x) \) lies in the subgroup generated by the \( G^e_I(\sum_{i \in I} a_i\x_i) \) with \( 0 < a_i < p_i \). Thus we need to show that \( \L' = \L \).

Obviously $\sum_{i\in I}a_i\x_i\in\L'$ for all $0<a_i<p_i$.
By Proposition~\ref{prop.calculating_basis_vector_relation}(c), $\sum_{i\in I}a_i\x_i\in\L'$ for all $0\le a_i<p_i$. By Proposition~\ref{prop.calculating_basis_vector_relation}(a), $\sum_{i=1}^na_i\x_i\in\L'$ for all $0\le a_i<p_i$.
Taking the normal form of each element of $\L$ and applying Proposition~\ref{prop.calculating_basis_vector_relation}(b), we have $\L'=\L$ as desired.
\end{proof}

Now we are ready to prove the following result.

\begin{theorem} \label{prop.basis_Grothendieck} 
Let $\X$ be a Geigle-Lenzing projective space of dimension $d+1$ with weights $p_1, \ldots, p_n$. Then $K_0(\coh\X)$ has the basis 
\[ G=\{ G^e_I(\x) \mid I \subseteq \{1, \ldots, n\},\ 0 \leq e \leq d 
- |I|,\ \x = \sum_{i \in I} a_i \x_i \text{ for some } 0 < a_i < p_i \}. \]
\end{theorem} 

\begin{proof}
Applying Lemma~\ref{KL} we see inductively that
\[  \{ G^{e'}_{I'}(\x) \mid (e',I') \leq (e,I),  \x = \sum_{i \in I'} a_i \x_i \text{ for some } 0 < a_i < p_i \} \]
generates \( K_0^{\leq (e, I)} \).
In particular, \( G \) generates \( K_0^{\leq (d, \emptyset)} \). On the other hand, since $[\OO(\x)] =G^d_\emptyset(\x)\in K_0^{\leq (d, \emptyset)} $ for all $\x\in\L$, we have \( K_0^{\leq (d, \emptyset)} = K_0(\coh\X) \) by Theorem~\ref{tilting}.

The map $G^e_I(\x) \mapsto e \c + \x$ defines a 
bijection $G\to[0, d\c]$. Since $K_0(\coh\X)$ is a free abelian group of rank $\#[0,d\c]$ by Theorem~\ref{tilting}, $G$ is a basis of $K_0(\coh\X)$.
\end{proof}

\begin{example} 
Consider the usual projective line $\mathbb{P}^1$. Then the basis of the 
Grothendieck group given in Theorem~\ref{prop.basis_Grothendieck} 
consists of $G_{\emptyset}^1 = [ \mathcal{O} ]$ and 
$G_{\emptyset}^0 = [\mathcal{S}]$ for any simple sheaf $\mathcal{S}$. 
Notice that this basis does not arise as the classes of the indecomposable direct summands of a tilting object in $\DDD^{\bo}(\coh\P^1)$.
\end{example} 

As a direct consequence of Theorem~\ref{prop.basis_Grothendieck} we have the following:

\begin{corollary} \label{cor.prod_for_coxeter_polynonial}
The subfactor groups \( K_0^{(e, I)} \) have bases \( \{ G^e_I(\sum_{i \in I} a_i \x_i) \mid 0 < a_i < p_i \} \).
For the characteristic polynomial \( \Phi^{(e, I)} \) of the induced action of \( (\w) \) on \( K_0^{(e, I)} \), we have
\[ \Phi_{\X} = \prod_{\substack{ I \subseteq \{ 1, \ldots, n \} \\ 0 \le e \le d - |I| }}  \Phi^{(e, I)}. \]
\end{corollary}

\begin{proof}
We have seen in Lemma~\ref{KL} that these sets generate the \( K_0^{(e, I)} \). These set must be linear independent by Theorem~\ref{prop.basis_Grothendieck} since $K_0(\coh\X)$ has a filtration with subfactors $K_0^{(e, I)}$.

For the second statement we observe that the action of \( \w \) on \( K_0(\coh\X) \) is block triangular with respect to the preordering on \( G \), whence its characteristic polynomial is the product of the characteristic polynomials of the blocks.
\end{proof}

Thus, for the calculation of the Coxeter polynomials, we need to find explicit formulas for the factors \(  \Phi^{(e, I)} \).
Now, for \( I \subseteq \{ 1, \ldots, n \} \), let
\[C_I = \Z[t_i \mid i \in I] / ( 1+t_i+t_i^2+\cdots+t_i^{p_i-1} \mid i \in I ).\]

\begin{lemma} \label{lem.block_indep_dim}
There is an isomorphism $(K_0^{(e, I)}, (\w)) \cong (C_I, \prod_{i \in I} t_i )$
of abelian groups with endomorphisms.
Thus \( \Phi^{(e,I)} \) is the characteristic polynomial of the action of \( \prod_{i \in I} t_i \) on \( C_I \), and in particular is independent of \( e \).
\end{lemma}

\begin{proof}
Consider the homomorphism given by sending \( G^e_I( \sum_{i \in I} a_i \x_i) \) to \( \prod_{i \in I} t_i^{p_i - a_i - 1} \) for \( 0 < a_i < p_i \).
This is an isomorphism since it sends a basis to a basis.
The claim that the two endomorphisms translate to each other follows from the explicit description of the action of \( \w \) as given by Proposition~\ref{prop.calculating_basis_vector_relation}. (Note that all right hand sides in that proposition vanish in the factor group \( K_0^{(e, I)} \).)
\end{proof}

Since \( \Phi^{(e,I)} \) is independent of \( e \), we will denote it simply by \( \Phi^I \) in the sequel.

\medskip
To calculate the characteristic polynomial \( \Phi^I \) of \( \prod_{i \in I} t_i \) on \( C_I \),
we introduce the following simpler version: Let
\[ \widetilde{C}_I = \Z[t_i \mid i \in I] / (t_i^{p_i} - 1 \mid i \in I), \]
and denote by \( \widetilde{\Phi}^I \) the characteristic polynomial of the action of \( \prod_{i \in I} t_i \) on \( \widetilde{C}_I \). This characteristic polynomial can easily be described explicitly.

\begin{lemma}
We have
\[ \widetilde{\Phi}^I(T) = ( 1 - T^{{\rm lcm}(p_i \mid i \in I)})^{\frac{\prod_{i \in I} p_i}{{\rm lcm}(p_i \mid i \in I)}}. \]
\end{lemma}

\begin{proof}
The action of \( \prod_{i \in I} t_i \) on \( \widetilde{C}_I \) induces a permutation of the basis \( \{ \prod_{i \in I} t_i^{a_i} \mid 0 \leq a_i < p_i \} \) of $\widetilde{C}_I$. All of the orbits of this permutation have length \( {\rm lcm}(p_i \mid i \in I) \). It follows that there are \( \frac{\prod_{i \in I} p_i}{{\rm lcm}(p_i \mid i \in I)} \) orbits, and we obtain the formula of the lemma.
\end{proof}

Next we will show that we can calculate \( \Phi^I \) from the \( \widetilde{\Phi}^J \).

\begin{lemma} \label{lem.Phi_via_hat-Phi}
We have
\[ \Phi^I = \prod_{J \subseteq I} (\widetilde{\Phi}^J)^{(-1)^{|I| + |J|}}. \]
\end{lemma}

\begin{proof}
Observe that we have the tensor product decomposition \( (C_I, \prod_{i \in I} t_i) = \bigotimes_{i \in I} (C_{i}, t_i) \) as abelian groups with endomorphisms, and similarly for \( \widetilde{C}_I \). Also observe that we have the short exact sequences 
\[ 0 \to (\Z, 1) \xrightarrow{f_i} (\widetilde{C}_{ \{ i \} }, t_i) \to (C_{ \{ i \} }, t_i) \to 0 \]
of free abelian groups with automorphisms, where the first map sends \( 1 \) to \( \sum_{j = 0}^{p_i - 1} t_i^j \), and the second map is the natural projection.

Tensoring the morphisms $f_i$, we obtain an exact sequence
\begin{align*} 0 \to & (\Z, 1) \to \bigoplus_{i \in I} (\widetilde{C}_{ \{ i \} }, t_i) \to \bigoplus_{\substack{i,j \in I \\ i \neq j}} (\widetilde{C}_{ \{ i, j \} }, t_i t_j) \to \bigoplus_{\substack{J \subseteq I \\ |J| = 3}} (\widetilde{C}_J, \prod_{j \in J} t_j) \to \cdots \\
& \to \bigoplus_{\substack{J \subseteq I \\ |J| = |I|-1}} (\widetilde{C}_J, \prod_{j \in J} t_j) \to ( \widetilde{C}_I, \prod_{i \in I} t_i) \to ( C_I, \prod_{i \in I} t_i) \to 0
\end{align*}
It follows that the characteristic polynomial of the right end term is the alternating product of the previous terms, as claimed in the lemma.
\end{proof}

For example, we have
\[\Phi^{\varnothing}(T) = 1 - T,\ \Phi^{\{i\}}(T) = 
\frac{1-T^{p_i}}{1-T},\ \Phi^{ \{i,j\}}(T) = \frac{(1 - T^{{\rm lcm}(p_i,p_j)})^{{\rm 
gcd}(p_i,p_j)} (1-T)}{(1-T^{p_i})(1-T^{p_j})},\dots.\]

We are now ready to compute the Coxeter polynomial of $\X$.

\begin{theorem} \label{thm.coxeter_polynomial}
Let \( \X \) be a Geigle-Lenzing projective space of dimension $d+1$ with weights $p_1, \ldots, p_n$.
\begin{itemize}
\item If \( n \leq d+1 \), then the Coxeter polynomial of $\X$ is
\[ \Phi_{\X}(T) = (1 - T^{{\rm lcm} (p_i)})^{(d+1-n) \frac{\prod_{i} p_i}{{\rm lcm} (p_i)}}  \prod_{i = 1}^n (1 - T^{{\rm lcm} (p_j \mid i \neq j)})^{\frac{\prod_{j \neq i} p_j}{{\rm lcm} (p_j \mid i \neq j)}}. \]
\item If \( n > d+1 \), then the Coxeter polynomial of \( \X \) is
\[ \Phi_{\X}(T) = \prod_{\substack{I \subseteq \{1, \ldots, n \} \\ |I| \leq d}} ( 1 - T^{{\rm lcm}(p_i \mid i \in I)})^{(-1)^{d+|I|} {n - 2 - |I| \choose d - |I|} \frac{\prod_{i \in I} p_i}{{\rm lcm}(p_i \mid i \in I)}}. \]
\end {itemize}
\end{theorem}

\begin{proof}
We observed in Corollary~\ref{cor.prod_for_coxeter_polynonial} and  Lemma~\ref{lem.block_indep_dim} that
\begin{align*} \Phi_{\X} \stackrel{{\rm Cor.} \ref{cor.prod_for_coxeter_polynonial}}{=} \prod_{\substack{I \subseteq \{ 1, \ldots, n \} \\ 0 \leq e \leq d - |I|}} \Phi^{(e, I)} \stackrel{{\rm Lem.}\ref{lem.block_indep_dim}}{=}  \prod_{\substack{I \subseteq \{ 1, \ldots, n \} \\ |I| \leq d}} (\Phi^{I})^{d+1-|I|}
\stackrel{{\rm Lem.}\ref{lem.Phi_via_hat-Phi}}{=}
\prod_{\substack{I \subseteq \{ 1, \ldots, n \} \\ |I| \leq d}} (\prod_{J \subseteq I} \widetilde{\Phi}_J^{(-1)^{|I|+|J|}})^{d+1-|I|}.
\end{align*}
We collect all factors \( \widetilde{\Phi}_J \) in the above product, and obtain a product of the form
\[ \Phi_{\X} = \prod_{\substack{J \subseteq \{ 1, \ldots, n \} \\ |I| \leq d}} (\widetilde{\Phi}^J)^{e(J)} \ \mbox{ with }\ 
e(J) = \sum_{\substack{I \supseteq J \\ |I| \leq d}} (-1)^{|I| + |J|} (d+1-|I|). \]
Since the terms on the right hand side only depend on the cardinality of \( I \), we can instead sum over the possible cardinalities, adding the number of such subsets as a factor. That is
\[ e(J) = \sum_{c = |J|}^{\min(n, d)} (-1)^{c + |J|} (d + 1 - c) {{n - |J|} \choose {c - |J|}}. \]
Note that the summand for \( c = d+1 \) is zero, whence it does not matter if the sum runs to \( d \) or to \( d+1 \). Additionally we replace \( c \) by \( i + |J| \), and obtain
\begin{align*}
e(J) & = \sum_{i = 0}^{\min(n,d+1) - |J|} (-1)^{i} (d + 1 - i - |J|) {n - |J| \choose i} \\
& = - \sum_{i = 0}^{\min(n,d+1) - |J|} (-1)^{i} {n - |J| \choose i} i + (d + 1 - |J|) \sum_{i = 0}^{\min(n,d+1) - |J|} (-1)^{i} {n - |J| \choose i} 
\end{align*}
In order to evaluate this formula further, we will employ the following combinatorial facts:

\begin{lemma}
For \( m \geq n \geq 0 \) we have
\begin{align*}
\sum_{i = 0}^n (-1)^i {m \choose i} & = \begin{cases} (-1)^n {m-1 \choose n} & m > n \\ 0 & m = n > 0 \\ 1 & m = n = 0. \end{cases} \text{, and } \\
\sum_{i = 0}^n (-1)^i {m \choose i} i & = \begin{cases} (-1)^n  {m-2 \choose n-1}m & m > n \\ 0 & m = n \neq 1 \\ -1 & m = n = 1. \end{cases}
\end{align*}
\end{lemma}

\begin{proof}
The first claim is proven by induction on \( n \).

For the second claim, observe that for \( i > 0 \) we have \( {m \choose i}i = {m-1 \choose i-1}m \). With this observation, the second formula may be reduced to the first one.
\end{proof}

Now we continue the calculation of \( e(J) \) using these combinatorial formulas. Assume first that \( n \leq d + 1 \). Then
\begin{align*}
e(J) & = - \sum_{i = 0}^{n - |J|} (-1)^{i} {n - |J| \choose i} i + (d + 1 - |J|) \sum_{i = 0}^{n - |J|} (-1)^{i} {n - |J| \choose i}  \\
& = \begin{cases} d+1-n & |J| = n \\ 1 & |J| = n - 1 \\ 0 & \text{else.} \end{cases}.
\end{align*}
For \( n > d+1 \) we obtain
\begin{align*}
e(J) & = (-1)^{d+1-|J|} {n - |J| - 2 \choose d - |J| - 1} (n - |J|) + (-1)^{d - |J|} {n - |J| - 1 \choose d - |J| } (d + 1 - |J|)\\
& = (-1)^{d - |J|} \frac{(n - |J| - 2)!}{(d - |J|)! (n - 1 - d)!} \left( (d+1-|J|) (n - |J| - 1) - (n - |J|) (d - |J|) \right) \\
& = (-1)^{d - |J|} \frac{(n - |J| - 2)!}{(d - |J|)! (n - 1 - d)!} (n - d - 1) \\
& = (-1)^{d - |J|} {n - |J| - 2 \choose d - |J|} \qedhere
\end{align*}
\end{proof}

\begin{example}
For the usual projective space \( \P^d \), we recover the Coxeter polynomial
\[ \Phi_{\P^d}(T) = (1-T)^{d+1}. \]
\end{example}

While the explicit formulas of Theorem~\ref{thm.coxeter_polynomial} may seem quite involved, one may observe that they simplify considerably in the case that the weights are coprime.

\begin{corollary}
Let \( \X \) be a Geigle-Lenzing projective space of dimension $d+1$ with weights $p_1, \ldots, p_n$, which are pairwise coprime. (Equivalently, \( \L \) is torsion free.)
\begin{itemize}
\item If \( n \leq d+1 \), then the Coxeter polynomial of $\X$ is
\[ \Phi_{\X}(T) = (1 - T^{\prod_i p_i})^{d+1-n}  \prod_{i = 1}^n (1 - T^{\prod_{j \neq i} p_j}). \]
\item If \( n > d+1 \), then the Coxeter polynomial of \( \X \) is
\[ \Phi_{\X}(T) = \prod_{\substack{I \subseteq \{1, \ldots, n \} \\ |I| \leq d}} ( 1 - T^{\prod_{i \in I} p_i })^{(-1)^{d+|I|} {n - 2 - |I| \choose d - |I|} }. \]
\end {itemize}
\end{corollary}

\begin{example} 
Let $\X$ be a Geigle-Lenzing projective space of dimension $2$ with weights $2,3$. Then the Coxeter polynomial of $\X$ is given by 
\[ \Phi_{\X}(T) = (1-T^6) (1 - T^3) (1 - T^2). \]
\end{example}

Finally we note that in the hypersurface case \( n = d+2 \) all binomial coefficients in the formula for the Coxeter polynomial evaluate to 1, so we get the following simple formula.

\begin{corollary}
Let \( \X \) be Geigle-Lenzing projective space of dimension $d+1$ with weights $p_1, \ldots, p_{d+2}$. Then the its Coxeter polynomial is given by
\[ \Phi_{\X}(T) = \prod_{\substack{I \subseteq \{1, \ldots, n \} \\ |I| \leq d}} ( 1 - T^{{\rm lcm}(p_i \mid i \in I)})^{(-1)^{d+|I|} \frac{\prod_{i \in I} p_i}{{\rm lcm}(p_i \mid i \in I)}}. \]
If moreover the weights \( p_i \) are coprime, then this simplifies to
\[ \Phi_{\X}(T) =  \prod_{\substack{I \subseteq \{1, \ldots, n \} \\ |I| \leq d}} ( 1 - T^{\prod_{i \in I} p_i })^{(-1)^{d+|I|}}. \]
\end{corollary}

\begin{example}
Let \( \X \) be a  Geigle-Lenzing projective space of dimension $2$ with weights $2,3,5,7$. Then the Coxeter polynomial of $\X$ is given by 
\[ \Phi_{\X}(T) = \frac{(1-T)(1-T^6)(1-T^{10})(1-T^{14})(1-T^{15})(1-T^{21})(1-T^{35})}{(1-T^2) (1 - T^3) (1 - T^5) (1-T^7)}. \]
\end{example}

We end this section by posing the following question.

\begin{problem}
What is the Coxeter polynomial of $\underline{\CM}^{\L}R$?
\end{problem}

For the hypersurface case $n=d+2$, an answer was given by Hille-M\"uller \cite{HM}.

%
%
%


\chapter{Tilting theory on Geigle-Lenzing projective spaces}

Let $\X$ be a Geigle-Lenzing projective space over a field $k$
associated with linear forms $\ell_1, \ldots, \ell_n$ and
weights $p_1, \ldots, p_n$.
In this chapter, we study tilting objects $V$ in $\DDD^{\bo}(\coh\X)$ 
that belong to $\vect\X$ (respectiely, $\coh\X$). 
We call such a $V$ a \emph{tilting bundle} (respectively, \emph{tilting sheaf}) on $\X$. 
We study the endomorphism algebras $\End_{\X}(V)$ of tilting bundles $V$. 
A typical example of a tilting bundle is $T^{\ca}$ given in Theorem~\ref{tilting}.
In this case, the endomorphism algebra is the $d$-canonical algebra $A^{\ca}$
studied in the previous chapter.

\section{Basic properties of tilting sheaves}
Throughout this section, let $V$ be a tilting bundle on $\X$ with \[\Lambda:=\End_{\X}(V).\]
Then we have triangle equivalences
\[V\Lotimes_\Lambda-:\DDD^{\bo}(\mod\Lambda)\to\DDD^{\bo}(\coh\X)
\ \mbox{ and }\ \RHom_{\X}(V,-):\DDD^{\bo}(\coh\X)\to\DDD^{\bo}(\mod\Lambda)\]
which are mutually quasi-inverse and make the following diagram commutative:
\begin{equation}\label{derived equivalence}
\xymatrix{
\DDD^{\bo}(\mod\Lambda)\ar@{-}[rr]^{\sim}\ar[d]^{\nu_d}&&\DDD^{\bo}(\coh\X)\ar[d]^{(\w)}\\
\DDD^{\bo}(\mod\Lambda)\ar@{-}[rr]^{\sim}&&\DDD^{\bo}(\coh\X).
}\end{equation}
In the rest, we identify $\DDD^{\bo}(\mod\Lambda)$ and $\DDD^{\bo}(\coh\X)$
by these triangle equivalences.

Let $0\le j\le d$. Recall that $R_{T_j}$ is the localization of $R$ with respect to the
multiplicative set $\{T_j^{\ell}\mid\ell\in\Z\}$.
Since $(\mod^{\L}_0R)_{T_j}=0$ holds, the natural functor $(-)_{T_j}:\DDD^{\bo}(\mod^{\L}R)\to\DDD^{\bo}(\mod^{\L}R_{T_j})$ factors as
\[\DDD^{\bo}(\mod^{\L}R)\xrightarrow{\pi}\DDD^{\bo}(\coh\X)\to\DDD^{\bo}(\mod^{\L}R_{T_j})\]
by universality.
The following observation shows that tilting bundles on $\X$
give progenerators in $\mod^{\L}R_{T_j}$.

\begin{lemma}\label{tilting bundle is a progenerator}
Assume that $V\in\mod^{\L}R$ gives a tilting bundle on $\X$.
Then for any $j$ with $0\le j\le d$, 
we have $\proj^{\L}R_{T_j}=\add V_{T_j}$.
\end{lemma}

\begin{proof}
It follow from Proposition~\ref{FLC modules}(c) that $V_{T_j}\in\proj^{\L}R_{T_j}$.
Since the functor $\DDD^{\bo}(\coh\X)=\thick V\to
\DDD^{\bo}(\mod^{\L}R_{T_j})$ is dense, we have $\DDD^{\bo}(\mod^{\L}R_{T_j})=\thick V_{T_j}$.
In particular, $V_{T_j}$ has to be a progenerator in $\mod^{\L}R_{T_j}$.
\end{proof}

The following useful result strengthens Theorem~\ref{general Serre vanishing}(a)
for tilting bundles.

\begin{theorem}\label{Serre surjection}
Let $V$ be a tilting bundle on $\X$.
Then for any $X\in\coh\X$, there exists $\a\in\L$ such that for any $\x\in\L$
satisfying $\x\ge\a$, there exists an epimorphism $V'\to X(\x)$ in $\coh\X$ with
$V'\in\add V$.
\end{theorem}

\begin{proof}
It suffices to show that for any $W\in\mod^{\L}R$, there exists $\a\in\L$
such that for any $\x\in\L$ satisfying $\x\ge\a$, there exists a morphism
$V'\to W(\x)$ in $\mod^{\L}R$ with $V'\in\add V$ which has a cokernel in
$\mod^{\L}_0R$.

(i) Fix $X\in\mod^{\L}R$. We show that for $\ell\gg0$, there exists a
morphism $V'\to X(\ell\c)$ in $\mod^{\L}R$ with $V'\in\add V$ which has
a cokernel in $\mod^{\L}_0R$.

Fix $j=0,\ldots,d$.
Since $\proj^{\L}R_{T_j}=\add V_{T_j}$ holds by
Lemma~\ref{tilting bundle is a progenerator}, there exists an epimorphism
$f_j:V^j_{T_j}\to X_{T_j}$ in $\mod^{\L}R$ with $V^j\in\add V$. Since
\[\Hom_{R_{T_j}}^{\L}(V^j_{T_j},X_{T_j})=(\Hom_R(V^j,X)_{T_j})_0
=\sum_{a\ge0}\Hom_R^{\L}(V^j,X(a\c))T_j^{-a},\]
we can write $f_j=g_jT_j^{-a_j}$ with $g_j\in\Hom_R^{\L}(V^j,X(a_j\c))$.
Then $\Cokernel g_j\in\mod^{\L}R$ satisfies $(\Cokernel g_j)_{T_j}=0$.

For $a:=\max\{a_0,\ldots,a_d\}$, let $e_j:=g_jT_j^{a-a_j}\in\Hom_R^{\L}(V^j,X(a\c))$ and
\[e:=(e_0,\ldots,e_d)^t:V^0\oplus\cdots\oplus V^d\to X(a\c).\]
Then there exists an epimorpism $\Cokernel e_j\to\Cokernel e$.
Since $(\Cokernel e_j)_{T_j}=(\Cokernel g_j)_{T_j}=0$ holds, we have
$(\Cokernel e)_{T_j}=0$ for any $j$ with $0\le j\le d$.
Thus $\Cokernel e$ belongs to $\mod^{\L}_0R$.

Using Lemma~\ref{surjection from X to X(c)}(a), we have the assertion.

(ii) Let $I:=\{\sum_{i=1}^na_i\x_i\mid 0\le a_i<p_i\}$ be the complete
set of representatives in $\L/\Z\c$. Applying (i) to $X:=W(\x)$
for each $\x\in I$, we have the assertion.
\end{proof}

We have the following description of $\coh\X$ in terms of $\Lambda$.

\begin{theorem}\label{describe coh X by A}
Let $V$ be a tilting bundle on $\X$, and $\Lambda:=\End_{\X}(V)$.
\begin{itemize}
\item[(a)] For any $\a\in\L$ satisfying
$\delta(\a)>0$, we have
\[\coh\X=\{X\in\DDD^{\bo}(\coh\X)\mid\forall \ell\gg0\ X(\ell \a)\in\mod\Lambda\}.\]
\item[(b)] If $\X$ is Fano, then $\coh\X=\{X\in\DDD^{\bo}(\coh\X)\mid\forall \ell\gg0\ X(-\ell \w)\in\mod\Lambda\}$.
\item[(c)] If $\X$ is anti-Fano, then $\coh\X=\{X\in\DDD^{\bo}(\coh\X)\mid\forall \ell\gg0\ X(\ell \w)\in\mod\Lambda\}$.
\end{itemize}
\end{theorem}

This theorem follows immediately from the following result, describing 
the standard t-structure of $\DDD^{\bo}(\coh\X)$ in terms of $\Lambda$:

\begin{proposition}\label{describe t-structure}
Let $V$ be a tilting bundle on $\X$. Then for any $\a\in\L$ satisfying
$\delta(\a)>0$, we have equalities
\begin{eqnarray*}
\DDD^{\le0}(\coh\X)&=&\{X\in\DDD^{\bo}(\coh\X)\mid\forall \ell\gg0\ X(\ell\a)\in\DDD^{\le0}(\mod\Lambda)\},\\
\DDD^{\ge0}(\coh\X)&=&\{X\in\DDD^{\bo}(\coh\X)\mid\forall \ell\gg0\ X(\ell\a)\in\DDD^{\ge0}(\mod\Lambda)\}.
\end{eqnarray*}
\end{proposition}

\begin{proof}
We only show the first equality, the second one can be shown similarly.

Fix $X\in\DDD^{\bo}(\coh\X)$.
Then $X$ belongs to $\DDD^{\le0}(\coh\X)$ if and only if the following condition holds:
\begin{itemize}
\item[(i)] $H^i(X)=0$ for any $i>0$.
\end{itemize}
By Theorem~\ref{Serre surjection}, this is equivalent to the following condition (since $H^i(X)=0$ for almost all $i$):
\begin{itemize}
\item[(ii)] For $\ell\gg0$, we have $\Hom_{\X}(V, H^i(X)(\ell\a))=0$ for any $i>0$.
\end{itemize}
We have $\Hom_{\DDD^{\bo}(\coh\X)}(V,  \sigma^{\geq i+1} X(\ell\a)[j]) =0$ for any $j\leq i$ and $\ell\in\Z$ by t-structure, and $\Hom_{\DDD^{\bo}(\coh\X)}(V,  \sigma^{\leq i-1} X(\ell\a)[j]) =0$ for any $j\geq i$ and $\ell\gg0$ by Serre vanishing Theorem~\ref{general Serre vanishing}(b).
Applying the functor $\Hom_{\DDD^{\bo}(\coh\X)}(V,(-)(\ell\a)[i])$ to the triangles
\[\sigma^{\geq i+1} X[-1] \to H^{i}(X)[-i]\to
\sigma^{\geq i} X \to \sigma^{\geq i+1} X\ \mbox{ and }
\sigma^{\leq i-1} X \to X\to \sigma^{\geq i} X \to \sigma^{\leq i-1} X[1],\]
we obtain
\[\Hom_{\DDD^{\bo}(\coh\X)}(V,  H^{i}(X)(\ell\a)) = \Hom_{\DDD^{\bo}(\coh\X)}(V,  \sigma^{\geq i} X(\ell\a)[i]) = \Hom_{\DDD^{\bo}(\coh\X)}(V,  X(\ell\a)[i])\]
for $\ell\gg0$. Therefore (ii) is equivalent to the following condition:
\begin{itemize}
\item[(iii)] For $\ell\gg0$, we have $\Hom_{\DDD^{\bo}(\coh\X)}(V, X(\ell\a)[i])=0$ for any $i>0$.
\end{itemize}
This is equivalent to the following condition since $V$ corresponds to $\Lambda$ under the identification
$\DDD^{\bo}(\coh\X)=\DDD^{\bo}(\mod\Lambda)$:
\begin{itemize}
\item[(iv)] For $\ell\gg0$, we have $X(\ell\a)\in\DDD^{\le0}(\mod\Lambda)$.
\end{itemize}
Thus the first equality follows.
\end{proof}

Now we describe the duality $(-)^{\star}=\RHom_R(-,R):\DDD^{\bo}(\coh\X)\to\DDD^{\bo}(\coh\X)$ in terms of $\Lambda$.

\begin{proposition}\label{compare dualities}
\begin{itemize}
\item[(a)] $V^{\star}$ is also a tilting bundle on $\X$.
\item[(b)] The following diagram is commutative.
\[\xymatrix{\DDD^{\bo}(\coh\X)\ar[d]^{(-)^\star[d]}\ar[rrr]^{\RHom_{\X}(V,-)}&&&\DDD^{\bo}(\mod\Lambda)\ar[d]^{D=\Hom_k(-,k)}\\
\DDD^{\bo}(\coh\X)\ar[rrr]^{\RHom_{\X}(V(\w)^\star,-)}&&&\DDD^{\bo}(\mod\Lambda^{\op}).}\]
\item[(c)] For any $\a\in\L$ satisfying
$\delta(\a)>0$, we have
\[(\coh\X)^\star=\{X\in\DDD^{\bo}(\coh\X)\mid \forall\ell\gg0\ X(-\ell\a)\in(\mod\Lambda)[-d]\}.\]
\end{itemize}
\end{proposition}

\begin{proof}
(a) This is clear since $(-)^\star:\DDD^{\bo}(\coh\X)\to\DDD^{\bo}(\coh\X)$
is a duality.

(b) Using Auslander-Reiten-Serre duality, we have
$\RHom_{\X}(V,V(\w))=D\Lambda[-d]$. Thus we have isomorphisms of functors:
\begin{eqnarray*}
(D\RHom_{\X}(V,-))[-d]&=&\RHom_{\Lambda}(\RHom_{\X}(V,-),D\Lambda[-d])\\
&=&\RHom_{\Lambda}(\RHom_{\X}(V,-),\RHom_{\X}(V,V(\w)))\\
&=&\RHom_{\Lambda}(V\Lotimes_\Lambda\RHom_{\X}(V,-),V(\w))\\
&=&\RHom_{\X}(-,V(\w))\\
&=&\RHom_{\X}(V(\w)^\star,(-)^\star).
\end{eqnarray*}
Thus the assertion follows.

(c) Let $X\in\DDD^{\bo}(\coh\X)$.
Applying Theorem~\ref{describe coh X by A} to $V(\w)^\star$,
we have that $X^\star\in\coh\X$ if and only if
$X(-\ell\a)^{\star}=X^\star(\ell\a)\in\mod\Lambda^{\op}$ for $\ell\gg0$.
Using the commutative diagram in (b), this is equivalent to
$X(-\ell\a)\in(\mod\Lambda)[-d]$ for $\ell\gg0$.
Thus the assertion follows.
\end{proof}

Identifying $\DDD^{\bo}(\coh\X)$ with $\DDD^{\bo}(\mod\Lambda)$, we have the following description of $\CM_i\X$.

\begin{proposition}\label{describe CM X in terms of A}
Let $0\le i\le d$.
\begin{itemize}
\item[(a)] For any $\a\in\L$ satisfying
$\delta(\a)>0$, we have
\begin{eqnarray*}
&\CM_i\X=\{X\in\coh\X\mid \forall \ell\gg0\ X(\ell\a)\in\mod\Lambda,\ X(-\ell\a)\in(\mod\Lambda)[-i]\},&\\
&\vect\X=\{X\in\coh\X\mid \forall \ell\gg0\ X(\ell\a)\in\mod\Lambda,\ X(-\ell\a)\in(\mod\Lambda)[-d]\}.&
\end{eqnarray*}
\item[(b)] If $X$ is Fano, then
\[\CM_i\X=\{X\in\coh\X\mid \forall\ell\gg0\ X(-\ell\w)\in\mod\Lambda,\ X(\ell\w)\in(\mod\Lambda)[-i]\}.\]
\item[(c)] If $X$ is anti-Fano, then
\[\CM_i\X=\{X\in\coh\X\mid \forall\ell\ll0\ X(\ell\w)\in\mod\Lambda,\ X(-\ell\w)\in(\mod\Lambda)[-i]\}.\]
\end{itemize}
\end{proposition}

\begin{proof}
(a) By Theorem~\ref{describe coh X by A}, an object $X\in\DDD^{\bo}(\coh\X)$ 
belongs to $\coh\X$ if and only if $X(\ell\a)\in\mod\Lambda$ holds for 
$\ell\gg0$. 
Now we fix $X\in\coh\X$. By definition, $X\in\CM_i\X$ if and only if
$X[i-d]^{\star}\in\coh\X$. By Proposition~\ref{compare dualities}(c), 
this is equivalent to that $X(-\ell\a)\in(\mod\Lambda)[-i]$ holds for $\ell\gg0$.
Thus the assertion follows.

(b)(c) Immediate from (a).
\end{proof}

Next we give some properties of $\Lambda=\End_{\X}(V)$.

In general, let $\Lambda$ be a finite dimensional $k$-algebra of finite global dimension.
Let $U$ be a two-sided tilting complex of $\Lambda$. For $\ell\ge0$, we simply write
\[U^{\ell}:=\overbrace{U\Lotimes_\Lambda\cdots\Lotimes_\Lambda U}^{\ell \mbox{ times}}.\]
We have an autoequivalence $U^\ell\Lotimes_\Lambda-$ of $\DDD^{\bo}(\mod \Lambda)$, which we simply denote by $U^\ell$.
The following notion was introduced by the third author \cite{M}.

\begin{definition}
\begin{itemize}
\item[(a)] We say that $U$ is \emph{quasi-ample} if $U^\ell\in\mod \Lambda$ for $\ell\gg0$.
\item[(b)] We say that $U$ is \emph{ample} if it is quasi-ample and $(\DDD^{U,\le0},\DDD^{U,\ge0})$ is a t-structure in $\DDD^{\bo}(\mod \Lambda)$, where
\begin{eqnarray*}
\DDD^{U,\le0}&:=&\{X\in\DDD^{\bo}(\mod \Lambda)\mid 
\forall\ell\gg0\ U^\ell(X)\in\DDD^{\le0}(\mod \Lambda)\},\\
\DDD^{U,\ge0}&:=&\{X\in\DDD^{\bo}(\mod \Lambda)\mid 
\forall\ell\gg0\ U^\ell(X)\in\DDD^{\ge0}(\mod \Lambda)\}.
\end{eqnarray*}
\end{itemize}
\end{definition}

We consider 2-sided tilting complexes $\omega_\Lambda:=D\Lambda[-d]$ and $\omega_\Lambda^{-1}:=\RHom_\Lambda(\omega_\Lambda,\Lambda)$ of $\Lambda$.

\begin{definition}\label{define Fano algebra}
We say that $\Lambda$ is \emph{quasi $d$-Fano}
(respectively, \emph{quasi $d$-anti-Fano}) if the two-sided tilting complex
$\omega_\Lambda^{-1}$ (respectively, $\omega_\Lambda$) is quasi-ample.
More strongly, we say that $\Lambda$ is \emph{$d$-Fano}
(respectively, \emph{$d$-anti-Fano}) if the two-sided tilting complex
$\omega_\Lambda^{-1}$ (respectively, $\omega_\Lambda$) is ample.
\end{definition}

We show the following trichotomy, which generalizes a result in \cite{M} for the case $d=1$.

\begin{theorem}\label{trichotomy}
Let $V$ be a tilting bundle on $\X$, and $\Lambda:=\End_{\X}(V)$.
\begin{itemize}
\item[(a)] $\X$ is Fano if and only if $\Lambda$ is a $d$-Fano algebra.
\item[(b)] $\X$ is anti-Fano if and only if $\Lambda$ is a $d$-anti-Fano algebra.
\item[(c)] $\X$ is Calabi-Yau if and only if $\DDD^{\bo}(\mod\Lambda)$
is a fractionally Calabi-Yau triangulated category.
\end{itemize}
\end{theorem}

\begin{proof}
Since $d$-Fano algebras, fractionally Calabi-Yau algebras and
$d$-anti-Fano algebras are disjoint classes, we only have to show the `only if' part of all statements.

(c) This is clear from the diagram \eqref{derived equivalence}.

(a) Assume that $\X$ is Fano. By \eqref{derived equivalence}, we have 
\begin{equation}\label{check d-Fano}
H^i(\omega_\Lambda^{-\ell})=
\Hom_{\DDD^{\bo}(\mod\Lambda)}(\Lambda,\nu_d^{-\ell}(\Lambda)[i])=
\Hom_{\DDD^{\bo}(\coh\X)}(V,V(-\ell\w)[i]).
\end{equation}
This is clearly zero for $i<0$. Assume $i>0$. Since $\X$ is Fano,
the element $-\ell\w$ is sufficiently large for $\ell\gg0$. Therefore
\eqref{check d-Fano} is zero for $\ell\gg0$ by Serre vanishing Theorem~\ref{general Serre vanishing}.
Therefore $\omega_\Lambda^{-1}$ is quasi-ample.

On the other hand, Proposition~\ref{describe t-structure} shows that
$\DDD^{\omega_\Lambda^{-1},\le0}=\DDD^{\le0}(\coh\X)$ and 
$\DDD^{\omega_\Lambda^{-1},\ge0}=\DDD^{\ge0}(\coh\X)$ hold.
In particular $(\DDD^{\omega_\Lambda^{-1},\le0},\DDD^{\omega_\Lambda^{-1},\ge0})$ is a t-structure in
$\DDD^{\bo}(\mod\Lambda)$. Thus $\Lambda$ is a $d$-Fano algebra.

(b) Assume that $\X$ is anti-Fano. 
One can show that $\Lambda$ is a $d$-anti-Fano algebra by a parallel 
argument as in (a) above.
\end{proof}

\begin{remark}\label{Fano is not weak 2-RI} 
Any $d$-representation infinite algebra is quasi $d$-Fano by definition.
But the converse is not true. For example, if $\X$ is Fano with $n\ge d+2$ and $p_i\ge2$ for all $i$,
then the $d$-canonical algebra $A^{\ca}$ is a $d$-Fano algebra which
is not $d$-representation infinite
by Theorems~\ref{trichotomy} and \ref{global dimension}.

Note that $d$-Fano algebras are not necessarily almost $d$-representation infinite.
For example let $\X_1=\P^1$ and $\X_2$ a domestic weighted projective line with $n\ge 3$ and $p_i\ge2$ for all $i$.
Then $\Lambda:=A^{\ca}_1\otimes_kA^{\ca}_2$ is a $2$-Fano algebra by Theorem~\ref{trichotomy}. But $\gl\Lambda=\gl A^{\ca}_1+\gl A^{\ca}_2=3$ holds by Theorem~\ref{global dimension}, and hence $\Lambda$ is not almost $2$-representation infinite by
Proposition~\ref{d or 2d}.
\end{remark}

\section{Tilting-cluster tilting correspondence}\label{section: tilting-cluster tilting}

The aim of this section is to study when a Geigle-Lenzing projective space $\X$ is derived equivalent to 
a $d$-representation infinite algebra. We show that this is closely related to $d$-VB finiteness of $\X$.
As in the case of $\underline{\CM}^{\L}R$, the following notion plays an important role.

\begin{definition}[$d$-tilting objects]
A tilting object $V$ in $\DDD^{\bo}(\coh\X)$ is called \emph{$d$-tilting} if $\End_{\DDD^{\bo}(\coh\X)}(V)$ has global dimension at most $d$. By Proposition~\ref{d-tilting imply Fano2} below, this is equivalent to the global dimension being precisely $d$.
\end{definition}

We give some basic properties of the endomorphism algebras of tilting 
objects. Result (a) below shows that $d$ gives a lower bound for $\gl\Lambda$.

\begin{proposition}\label{preliminaries for end}
Let $V\in\DDD^{\bo}(\coh\X)$ be a tilting object, $\Lambda=\End_{\DDD^{\bo}(\coh\X)}(V)$
and $\Pi$ the $(d+1)$-preprojective algebra of $\Lambda$.
\begin{itemize}
\item[(a)] $\gl\Lambda\ge d$ holds.
\item[(b)] If $V\simeq\pi(U)$ for $U\in(\mod^{\L}R)^{\perp_{0,1}}$,
then $\Pi\simeq\End_R^{\L/\Z\w}(U)$ as $\Z$-graded $k$-algebras.
\end{itemize}
\end{proposition}

\begin{proof}
(a) Assume $\gl\Lambda<d$. Then for any $X,Y\in\DDD^{\bo}(\mod\Lambda)$, we have
$\Hom_{\DDD^{\bo}(\mod\Lambda)}(X,\nu_d^\ell(Y))=0$ for almost all 
$\ell\in\Z$ by Proposition~\ref{fractional CY and global dimension}(b). On the other hand, by \eqref{derived equivalence} and 
Lemma~\ref{surjection from X to X(c)}(b), we have
\[\Hom_{\DDD^{\bo}(\mod\Lambda)}(\Lambda,\nu_d^{\ell}(\Lambda))\simeq\Hom_{\DDD^{\bo}(\coh\X)}(V,V(\ell\w))\neq0\]
for infinitely many $\ell\in\Z$, a contradiction.

(b) By \eqref{derived equivalence} and Lemma~\ref{depth 2}, we have
\[\Pi_\ell=
\Hom_{\DDD^{\bo}(\mod\Lambda)}(\Lambda,\nu_d^{-\ell}(\Lambda))
\simeq\Hom_{\DDD^{\bo}(\coh\X)}(V,V(-\ell\w))
=\Hom_R^{\L}(U,U(-\ell\w))\]
for any $\ell\in\Z$. Therefore we have
$\Pi=\bigoplus_{\ell\in\Z}\Hom_R^{\L}(U,U(-\ell\w))=\End_R^{\L/\Z\w}(U)$.
\end{proof}

We prepare the following general observation.

\begin{lemma}\label{non-zero endo}
Let $p:={\rm l.c.m.}(p_1,\ldots,p_n)$.
If there exists a non-zero object $X\in\DDD^{\bo}(\coh\X)$ satisfying
$\Hom_{\DDD^{\bo}(\coh\X)}(X,X(p\w))=0$, then $\X$ is Fano.
\end{lemma}

\begin{proof}
Assume that $\X$ is not Fano. Then $p\w=q\c$ holds for some $q\ge0$.
Consider a morphism $f_q:=(t)_t:\bigoplus_{t}X\to X(q\c)=X(p\w)$, where
$t$ runs over all monomials on $T_0,\ldots,T_d$ of degree $q$.
For any $i\in\Z$, the morphism $H^i(f_q):\bigoplus_tH^i(X)\to H^i(X)(q\c)$
is an epimorphism in $\coh\X$ by Lemma~\ref{surjection from X to X(c)}(b).
In particular $f_q$ is non-zero in $\DDD^{\bo}(\coh\X)$, a contradiction.
Thus $\X$ is Fano.
\end{proof}

By result (a) below, the existence of $d$-tilting objects in $\DDD^{\bo}(\coh\X)$ implies that $\X$ is Fano, as in Theorem~\ref{d-tilting imply Fano} for $\underline{\CM}^{\L}R$. Moreover, result (b) below due to Buchweitz-Hille \cite{BuH} explains the importance of $d$-tilting sheaves.

\begin{proposition}\label{d-tilting imply Fano2}
Let $\X$ be a GL projective space, and $V$ a $d$-tilting object in $\DDD^{\bo}(\coh\X)$.
\begin{itemize}
\item[(a)] $\X$ is Fano. Moreover $\Hom_{\X}(V,V(\ell\w))=0$ for any $\ell>0$.
\item[(b)] \cite{BuH} If $V\in\coh\X$, then $\End_{\X}(V)$ is a $d$-representation infinite algebra.
\end{itemize} 
\end{proposition}

\begin{proof}
(a) Thanks to Lemma~\ref{non-zero endo}, we only have to prove the latter assertion.

Since $\Lambda$ has global dimension at most $d$, we have
$\nu_d^{\ell-1}(D\Lambda)\in\DDD^{\ge0}(\mod\Lambda)$ for any
$\ell>0$ by Proposition~\ref{fractional CY and global dimension}.
Hence the commutative diagram \eqref{derived equivalence} shows
\[\Hom_{\X}(V,V(\ell\w))\simeq
\Hom_{\DDD^{\bo}(\mod\Lambda)}(\Lambda,\nu_d^{\ell}(\Lambda))
=H^0(\nu_d^{\ell-1}(D\Lambda)[-d])=0.\]

(b) For the convenience of the reader, we include a proof. We only have to show that
$H^i(\nu_d^{-\ell}(\Lambda))=0$ holds for any $i\neq0$ and $\ell\ge0$.
This is clear for $i<0$ since $V\in\coh\X$ implies
\[H^i(\nu_d^{-\ell}(\Lambda))=
\Hom_{\DDD^{\bo}(\mod\Lambda)}(\Lambda,\nu_d^{-\ell}(\Lambda)[i])=
\Hom_{\DDD^{\bo}(\coh\X)}(V,V(-\ell\w)[i])=0\]
by \eqref{derived equivalence}.
On the other hand, for $\ell\ge0$, since
$\nu_d^{-\ell}(\Lambda)\in\DDD^{\le0}(\mod\Lambda)$ holds
by Proposition~\ref{fractional CY and global dimension},
we have $H^i(\nu_d^{-\ell}(\Lambda))=0$ for any $i>0$.
Thus the assertion follows.
\end{proof}

Now we show that the existence of $d$-tilting bundles implies $d$-VB finiteness,
as in Theorem~\ref{construct dCT} for $\underline{\CM}^{\L}R$.

\begin{theorem}\label{construct dCT2}
If a GL projective space $\X$ has a $d$-tilting bundle $V$, then $\X$ is $d$-VB finite and
$\vect\X$ has the $d$-cluster tilting subcategory
\[\UU:=\add\{V(\ell\w)\mid\ell\in\Z,\ \x\in\L\}.\]
\end{theorem}

\begin{proof}
Let $\Lambda:=\End_{\X}(V)$.
Then we have a derived equivalence $\DDD^{\bo}(\coh\X)\to\DDD^{\bo}(\mod\Lambda)$
which makes the diagram \eqref{derived equivalence} commutative.
Moreover $\Lambda$ is $d$-representation infinite
by Proposition~\ref{d-tilting imply Fano2}(b).
In particular, Theorem~\ref{d-RI has d-CT} shows that
\[\VV_\Lambda:=\{X\in\DDD^{\bo}(\mod\Lambda)\mid\forall\ell\gg0,\ \nu_d^{-\ell}(X)\in\mod\Lambda,\ \nu_d^\ell(X)\in(\mod\Lambda)[-d]\}.\]
has a $d$-cluster tilting subcategory
\[\UU_\Lambda:=\add\{\nu_d^\ell(\Lambda)\mid\ell\in\Z\}.\]
On the other hand, $\X$ is Fano by Proposition~\ref{d-tilting imply Fano2}(a).
By Proposition~\ref{describe CM X in terms of A}(b), the equivalence
$\DDD^{\bo}(\mod\Lambda)\to\DDD^{\bo}(\coh\X)$ restricts to
equivalences
\begin{eqnarray*}
\VV_\Lambda\to\vect\X\ \mbox{ and }\ \UU_\Lambda\to\UU:=\add\{V(\ell\w)\mid\ell\in\Z\}.
\end{eqnarray*}
Therefore $\UU$ is a $d$-cluster tilting subcategory of $\vect\X$,
and we have that $\X$ is $d$-VB finite.
\end{proof}

To give a more explicit version of Theorem~\ref{construct dCT2}, we need the following notion.

\begin{definition}\label{define slice}
Let $\UU$ be a $d$-cluster tilting subcategory of $\vect\X$
(respectively, $\CM^{\L}R$).
We call an object $V\in\UU$ \emph{slice} in $\UU$ if the following conditions are satisfied.
\begin{itemize}
\item[(a)] $\UU=\add\{ V(\ell\w)\mid\ell\in\Z\}$.
\item[(b)] $\Hom_{\UU}(V,V(\ell\w))=0$ for any $\ell>0$.
\end{itemize}
In this case, any $(\w)$-orbit of indecomposable objects in 
$\UU$ contains exactly one element in $\add V$.
\end{definition}

The following is the main result in this section.

\begin{theorem}[tilting-cluster tilting correspondence]\label{tilting-cluster tilting}
Let $\X$ be a GL projective space. Then $d$-tilting bundles on $\X$ are precisely
slices in $d$-cluster tilting subcategories of $\vect\X$.
\end{theorem}

The direction `$\Rightarrow$' follows directly from Theorem~\ref{construct dCT2} and
Proposition~\ref{d-tilting imply Fano2}(a).

In the rest, we prepare to prove the direction `$\Leftarrow$'.

Let $V$ be a slice in a $d$-cluster tilting subcategory $\UU$ of $\vect\X$
and $\Lambda:=\End_{\X}(V)$.
Since $\Hom_{\X}(V,V(\ell\w))=0$ holds for any $\ell>0$ by our assumption, $\X$ is Fano by Lemma~\ref{non-zero endo}.
Now we show that $V$ satisfies one of the conditions for being a tilting object.

\begin{lemma}\label{Ext^i(V,V)=0}
$\Ext^i_{\X}(V,V)=0$ holds for any $i\neq0$.
\end{lemma}

\begin{proof}
Since $V$ is an object in a $d$-cluster tilting subcategory $\UU$,
we have $\Ext^i_{\X}(V,V)=0$ for any $i$ with $1\le i\le d-1$.
On the other hand, by Auslander-Reiten-Serre duality,
we have $\Ext^d_{\X}(V,V)=D\Hom_{\X}(V,V(\w))$,
which is zero since $V$ is a slice. Thus the assertion follows.
\end{proof}

Next we show the following easy properties of a slice.

\begin{lemma}\label{ell function}
Let $X$ and $Y$ be indecomposable objects in $\UU$.
\begin{itemize}
\item[(a)] There exists a unique integer $\ell=\ell(X)$ satisfying $X\in\add V(\ell\w)$.
\item[(b)] If there exist a sequence $X\to\cdots\to Y$ of indecomposable objects in $\UU$ and
non-zero morphisms between them,
then $\ell(X)\ge\ell(Y)$.
\item[(c)] If $\ell(X)\ge0$, then the $\Lambda$-module $\Hom_{\X}(V,X)$
is projective.
\item[(d)] Let
\[0\to X(\w)\to C_{d-1}\to\cdots\to C_1\to C_0\to X\to0\]
be a $d$-almost split sequence from Proposition~\ref{basic properties of d-CT}(c).
If $Y$ is a direct summand of $C_i$ for 
some $0\le i\le d-1$. Then either $\ell(Y)=\ell(X)$ or $\ell(Y)=\ell(X)+1$ holds.
\end{itemize}
\end{lemma}

\begin{proof}
(a)(b) Both assertions are clear from our definition of slice.

(c) If $\ell(X)>0$, then $\Hom_{\X}(V,X)=0$. If $\ell(X)=0$, then the assertion
holds since $X\in\add V$.

(d) Since there exists a chain of non-zero morphisms from $Y$ to $X$,
we have $\ell(Y)\ge\ell(X)$ by (b). Since there exists a chain of non-zero
morphisms from $X(\w)$ to $Y$, we have $\ell(X)+1\ge\ell(Y)$ by (b).
Thus the assertion follows.
\end{proof}

Now we are ready to prove the following observation.

\begin{proposition}\label{global dimension at most d}
Let $V$ be a slice in a $d$-cluster tilting subcategory $\UU$ of $\vect\X$.
Then $\Lambda=\End_{\DDD^{\bo}(\coh\X)}(V)$ has global dimension $d$.
\end{proposition}

\begin{proof}
It suffices to show that any simple $\Lambda$-module $S$ has projective dimension
at most $d$. We naturally regard $S$ as a simple $\UU$-module. Then there exists
an indecomposable object $X\in\add V$ such that
$S=\Hom_{\UU}(-,X)/\rad_{\UU}(-,X)$. Let 
\[0\to X(\w)\to C_{d}\to\cdots\to C_1\to X\to0\]
be a $d$-almost split sequence in $\UU$ given in
Proposition~\ref{basic properties of d-CT}(c).
Since $V$ is a slice and $X\in\add V$, we have $\Hom_{\UU}(V,X(\w))=0$.
Hence we have an exact sequence
\begin{equation}\label{resolution of S}
0\to\Hom_{\UU}(V,C_{d})\to\cdots\to\Hom_{\UU}(V,C_1)\to\Hom_{\UU}(V,X)\to S\to0
\end{equation}
of $\Lambda$-modules.
On the other hand, let $Y$ be an indecomposable summand of $C_i$.
Then $\ell(Y)\ge0$ holds by Lemma~\ref{ell function}(d). Hence
$\Hom_{\UU}(V,Y)$ is a projective $\Lambda$-module by Lemma
\ref{ell function}(c).
Therefore the sequence \eqref{resolution of S} gives a projective resolution 
of the simple $\Lambda$-module $S$, and we have the assertion.
\end{proof}

We also need the following observation.

\begin{lemma}\label{generate V(w)}
There exist exact sequences
\[0\to U_d\to\cdots\to U_0\to V(-\w)\to0\ \mbox{ and }\ 0\to V(\w)\to U^0\to\cdots\to U^d\to0\]
in $\coh\X$ with $U_i,U^i\in\add V$ for any $0\le i\le d$.
\end{lemma}

\begin{proof}
We only construct the first sequence since the second one can be 
constructed in a similar way.
Let $\UU^+:=\add\{V(\ell\w)\mid\ell\ge0\}$. This is a functorially finite subcategory of $\vect\X$ by Lemmas~\ref{graded and ungraded} and \ref{depth 2}.
Let $f:U_0\to V(-\w)$ be a right $\UU^+$-approximation of $V(-\w)$.
Then $f$ is surjective by Lemma~\ref{surjection from X to X(c)}(b) since $\X$ is Fano.
Since $\UU$ is a $d$-cluster tilting subcategory of $\vect\X$,
there exists an exact sequence
\[0\to U_d\to\cdots\to U_1\to\Kernel f\to0\]
in $\vect\X$ with $U_i\in\UU$ by Proposition~\ref{basic properties of d-CT}(b).
It suffices to show that $U_i$ belongs to $\add V$ for any $i$ with $0\le i\le d$.
By Lemma~\ref{ell function}(b), we have
\begin{equation}\label{ell of Y}
U_i\in\UU^+
\end{equation}
for any $i$ with $0\le i\le d$.
Now we define an $\UU$-module $F$ by an exact sequence
\[\Hom_{\UU}(-,U_0)\stackrel{f}{\to}\Hom_{\UU}(-,V(-\w))\to F\to0.\]
Since $V$ is a slice, we have $F(V(\ell\w))=0$ for any $\ell<-1$.
Since $f$ is a right $\UU^+$-approximation, we have $F(\UU^+)=0$.
Hence the support of $F$ is contained in $\add V(-\w)$,
and therefore $F$ has a finite length as an $\UU$-module.
In partcular, $F$ has a finite filtration by simple $\UU$-modules
of the form $S_X:=\Hom_{\UU}(-,X)/\rad_{\UU}(-,X)$ for indecomposable direct summands $X$ of $V(-\w)$.

On the other hand, a minimal projetive resolution of $S_X$ is given by a
$d$-almost split sequence whose terms belong to
$\add(V(-\w)\oplus V)$ by Lemma~\ref{ell function}(d).
Applying Horseshoe Lemma repeatedly, we have that each $U_i$ belongs to
$\add(V(-\w)\oplus V)$. By \eqref{ell of Y}, we have that
$U_i\in\UU^+\cap\add(V(-\w)\oplus V)=\add V$ for any $i$ with $0\le i\le d$.
Thus the assertion follows.
\end{proof}

Now we show that $V$ satisfies the remaining condition for being a tilting object. 

\begin{lemma}\label{V generates D^b(cohX)}
We have $\thick V=\DDD^{\bo}(\coh\X)$.
\end{lemma}

\begin{proof}
Using Lemma~\ref{generate V(w)} repeatedly, we have $V(\ell\w)\in\thick V$
for any $\ell\in\Z$. Thus $\UU\subset\thick V$.
Since $\UU$ is a $d$-cluster tilting subcategory of $\vect\X$, we have
$\vect\X\subset\thick\UU\subset\thick V$ by Proposition~\ref{basic properties of d-CT}(b).
Therefore $\DDD^{\bo}(\coh\X)=\thick V$ holds by Theorem~\ref{Serre}(e).
\end{proof}

Now we are ready to complete the proof of Theorem~\ref{tilting-cluster tilting}.

\begin{proof}[Proof of Theorem~\ref{tilting-cluster tilting} `$\Leftarrow$']
Let $V$ be a slice in a $d$-cluster tilting subcategory of $\vect\X$.
By Lemmas~\ref{Ext^i(V,V)=0} and \ref{V generates D^b(cohX)}, $V$ is a tilting bundle on $\X$.
It is $d$-tilting by Lemma~\ref{global dimension at most d}.
\end{proof}

Now we consider the special class of $d$-tilting bundles which are contained in $\CM^{\L}R$.
In this case, Theorem~\ref{tilting-cluster tilting} gives the following result,
which improves Theorem~\ref{construct NCCR}.

\begin{theorem}\label{tilting-cluster tilting 2}
Let $(R,\L)$ be a GL complete intersection, and $\X$ the corresponding GL projective space.
For $V\in\CM^{\L}R$, the following conditions are equivalent.
\begin{itemize}
\item[(a)] $V$ is a $d$-tilting bundle on $\X$.
\item[(b)] $V$ is a slice in a $d$-cluster tilting subcategories of $\CM^{\L}R$.
\item[(c)] $V^{(\w)}$ gives an NCCR of $R^{(\w)}$ such that $\End_{R^{(\w)}}(V^{(\w)})_i=0$ for all $i<0$.
\end{itemize}
In this case, there are isomorphisms
\[\Pi\simeq\End_R^{\L/\Z\w}(V)\simeq\End_{R^{(\w)}}(V^{(\w)})\]
of $\Z$-graded algebras, where $\Pi$ is the $(d+1)$-preprojective algebra of $\End_{\X}(V)$.
\end{theorem}

\begin{proof}
(a)$\Leftrightarrow$(b) This is immediate from Theorems~\ref{CT subcategory} and
\ref{tilting-cluster tilting}.

(b)$\Leftrightarrow$(c) Let $\UU=\add\{V(\ell\w)\mid\ell\in\Z\}$.
By Theorem~\ref{construct NCCR}, $\UU$ is a $d$-cluster tilting subcategory of $\CM^{\L}R$
if and only if $V^{(\w)}$ gives an NCCR of $R^{(\w)}$.
Clearly $V$ is a slice in $\UU$ if and only if $\End_{R^{(\w)}}(V^{(\w)})_i=0$ holds for all $i<0$.
Thus the assertion follows.

The isomorphisms follow from Proposition~\ref{preliminaries for end}(b) and
Lemma~\ref{Veronese functor}.
\end{proof}

\begin{example}
Let $\X$ be a GL projective space with $d=1$ which is Fano.
Then there exists a tilting bundle $V$ on $\X$ such that $\End_{\X}(V)$
is isomorphic to the path algebra $kQ$ of an extended Dynkin quiver $Q$,
and $\Pi_2(kQ)$ is the corresponding classical preprojective algebra.
On the other hand, $(R,\L)$ is (1-)CM finite and we have $\CM^{\L}R=
\add\{V(\ell\w)\mid\ell\in\Z\}$. Moreover $R^{(\w)}$ is a simple surface singularity, and its (1-)Auslander algebra is isomorphic to the preprojective
algebra $\Pi_2(kQ)$.
\end{example}

Thanks to Theorems~\ref{tilting-cluster tilting} and \ref{tilting-cluster tilting 2}, we have the following diagram
which shows connections between the relevant notions.
\[\begin{array}{ccc}
\left\{\mbox{$d$-tilting bundles on $\X$}\right\}&=&
\left\{\begin{array}{c}\mbox{slices in $d$-cluster tilting}\\ \mbox{subcat. of $\vect\X$}\end{array}\right\}\\
\cup&&\cup\\
\left\{\begin{array}{c}\mbox{$d$-tilting bundles on $\X$}\\ \mbox{contained in $\CM^{\L}R$}\end{array}\right\}&=&
\left\{\begin{array}{c}\mbox{slices in $d$-cluster tilting}\\ \mbox{subcat. of $\CM^{\L}R$}\end{array}\right\}
\end{array}\]

Now we pose the following question.

\begin{problem}\label{Fano conjecture}
How are the following conditions related to each other?
\begin{itemize}
\item[(a)] $\X$ is Fano.
\item[(b)] $(R,\L)$ is $d$-CM finite.
\item[(c)] $R^{(\w)}$ has an NCCR.
\item[(d)] $\underline{\CM}^{\L}R$ has a $d$-tilting object.
\item[(e)] $\X$ is $d$-VB finite.
\item[(f)] $\DDD^{\bo}(\coh\X)$ has a $d$-tilting bundle.
\item[(g)] $\X$ is derived equivalent to a $d$-representation infinite algebra.
\end{itemize}
\end{problem}

The table after Theorem~\ref{Main2 in introduction} contains our results on Problem~\ref{Fano conjecture}.

\section{$d$-tilting bundles on $\X$}\label{Section: d-tilting bundle}

In this section, we provide a program to construct $d$-tilting bundles on some GL projective spaces which are Fano. Our strategy consists of the following two steps.
\begin{itemize}
\item[(Step 1)] Find a $d$-tilting object $U$ in $\underline{\CM}^{\L}R$. Then this gives a $d$-cluster tilting subcategory $\UU:=\add\{U(\ell\w),\ R(\x)\mid\ell\in\Z,\ \x\in\L\}$ of $\CM^{\L}R$ by Theorem~\ref{construct dCT}(a).
\item[(Step 2)] Extend $U$ to a slice $V$ in the $d$-cluster tilting subcategory $\UU$. Then this gives a $d$-tilting bundle by Theorem~\ref{tilting-cluster tilting}.
\end{itemize}
Both steps are non-trivial. However (Step 1) has already been addressed in Chapter~\ref{section: CM GL CI}: Theorem~\ref{Main1} gives explicitly hypersurface cases in which the category $\underline{\CM}^{\L}R$ has a $d$-tilting object $U$.


In this section, we will focus on (Step 2), and give a general strategy to construct a slice in the $d$-cluster tilting subcategory $\UU$.

Our starting point is the following general result using a certain upset of $\L$.

\begin{proposition}\label{extend1}
Let $(R,\L)$ be a GL complete intersection, and $U$ a $d$-tilting object in $\underline{\CM}^{\L}R$.
Assume that there exists a subset $I$ of $\L$ satisfying the following conditions.
\begin{itemize}
\item[(i)] $I$ is a non-trivial upset of $\L$ satisfying $I-\w\subset I$.
\item[(ii)] The injective hull of $U$ in $\CM^{\L}R$ belongs to $\add\{R(\x)\mid \x\in I\}$.
\item[(iii)] The projective cover of $U$ in $\CM^{\L}R$ belongs to $\add\{R(\x)\mid \x\in I^c\}$.
\end{itemize}
In this case, $S:=\{\x\in I\mid x+\w\in I^c\}$ satisfies the following conditions.
\begin{itemize}
\item[(a)] $S$ is a complete set of representatives of $\L/\Z\w$ in $\L$ such that $S\subset I$ and $S+\w\subset I^c$.
\item[(b)] $V:=\pi(U)\oplus(\bigoplus_{\x\in S}R(\x))$ is a slice in the $d$-cluster tilting subcategory $\UU:=\add\{U(\ell\w),\ R(\x)\mid\ell\in\Z,\ \x\in\L\}$ of $\vect\X$.
\item[(c)] $V$ gives a $d$-tilting bundle on $\X$.
\item[(d)] We have $\underline{\End}^{\L}_R(U)=\End^{\L}_R(U)$.
\end{itemize}
\end{proposition}

\begin{proof}
Since $\underline{\CM}^{\L}R$ has a $d$-tilting object, $(R,\L)$ is Fano by Proposition~\ref{d-tilting imply Fano2}(a).

(a) It is enough to show that $S$ is a complete set of representatives of $\L/\Z\w$ in $\L$.

Since $(R,\L)$ is Fano, $\x+i\w\in I$ and $\x-i\w\in I^c$ hold for $i\ll0$.
By the condition $I-\w\subset I$, there exists a unique $\ell\in\Z$ such that $\x+i\w\in I$
holds for any $i\le \ell$ and $\x+i\w\in I^c$ holds for any $i>\ell$.
Then $\x+\ell\w$ gives a unique element in $S\cap(\x+\Z\w)$.
Thus $S$ is a complete set of representatives of $\L/\Z\w$ in $\L$.

(b) Condition (a) in Definition~\ref{define slice} is satisfied by (a).
We need to check condition (b) there, that is, $\Hom_R^{\L}(V,V(\ell\w))=0$ for any $\ell>0$.

(1) Since $I$ is an upset of $\L$, we have $\Hom_R^{\L}(R(\x),R(\y))=0$ for any $\x\in I$ and $\y\in I^c$.
By (i), we have $I^c+\w\subset I^c$. By our definition of $S$, we have $S\subset I$ and $S+\ell\w\subset I^c$ for any $\ell>0$. Consequently $\Hom_R^{\L}(R(\x),R(\y+\ell\w))=0$ holds for any $\x,\y\in S$ and $\ell>0$.

(2) We show that $\Hom_R^{\L}(U,R(\y+\ell\w))=0$ holds for any $\y\in S$ and $\ell>0$.

Take any $f\in\Hom_R^{\L}(U,R(\y+\ell\w))$. By (ii), $f$ factors through $\add\{R(\x)\mid \x\in I\}$.
Since $\Hom_R^{\L}(R(\x),R(\y+\ell\w))=0$ holds for any $\x\in S$ by (1), we have $f=0$.

(3) We show that $\Hom_R^{\L}(R(\x),U(\ell\w))=0$ holds for any $\x\in S$ and $\ell\ge0$.

Take any $f\in\Hom_R^{\L}(R(\x),U(\ell\w))$. By (iii), $f$ factors through $\add\{R(\y+\ell\w)\mid \y\in I^c\}$. 
Since $\y+\ell\w\in I^c$ holds for any $\y\in I^c$, we have $\Hom_R^{\L}(R(\x),R(\y+\ell\w))=0$ by (1).
Thus $f=0$.

(4) For any $i\in\Z$ and $\ell\ge0$, any composition $U\xrightarrow{f}R(\x+i\w)\xrightarrow{g}U(\ell\w)$ is zero. In fact, if $i>0$, then $f=0$ holds by (2), and if $i\le0$, then $g=0$ holds by (3).

(5) We show that $\Hom_R^{\L}(U,U(\ell\w))=0$ holds for any $\ell>0$.

Since $U$ is a $d$-tilting object in $\underline{\CM}^{\L}R$, we have $\underline{\Hom}_R^{\L}(U,U(\ell\w))=0$ for any $\ell>0$, see for instance \cite[Proposition~2.3(b)]{HIO}.
Thus any morphism $U\to U(\ell\w)$ factors through $\proj^{\L}R$, and hence must be zero by (4).

(c) This follows from (b) and Theorem~\ref{tilting-cluster tilting}.

(d) This is immediate from (4) above.
\end{proof}

To construct upset $I$ satisfying the conditions in Proposition~\ref{extend1}, we prepare the following notions, which are rather easy to control.

\begin{definition}\label{define increasing}
\begin{enumerate}[\rm(a)]
\item A map $\gamma\colon\L\to\Q$ is \emph{increasing} if
\[\mbox{$\gamma(\x)\le\gamma(\y)$ holds for any $\x,\y\in\L$ satisfying $\x\le\y$.}\]
An increasing map $\gamma$ is \emph{equi-increasing} if
$\gamma(\x+\c)=\gamma(\x)+1$ holds for any $\x\in\L$.
An (equi-)increasing map $\gamma$ is \emph{$\w$-(equi-)increasing} if
$\gamma(\x+\w)\le\gamma(\x)$ holds for any $\x\in\L$.
\item Let $\gamma\colon\L\to\Q$ be an increasing map. For $X\in\CM^{\L}R$, let $I_X$ be the injective hull of $X$ in $\CM^{\L}R$, and $P_X$ the projective cover of $X$.
We say that $X\in\CM^{\L}R$ has a \emph{$\gamma$-presentation} if
\[I_X\in\add\{R(\x)\mid \gamma(\x)>0\}\ \mbox{ and }\ P_X\in\add\{R(\x)\mid \gamma(\x)\le0\}.\]
More strongly, we say that $X\in\CM^{\L}R$ has a \emph{strict $\gamma$-presentation} if 
\[I_X\in\add\{R(\x)\mid \gamma(\x)=1/2\}\ \mbox{ and }\ P_X\in\add\{R(\x)\mid \gamma(\x)=0\}.\]
\end{enumerate}
\end{definition}

\begin{observation}
\begin{enumerate}
\item If $\gamma \colon \L \to \Q$ is an increasing map, then
\[ I = \{ \x \in \L \mid \gamma(\x) > 0 \} \]
is an upset in $\L$. Moreover, if $\gamma$ is $\w$-increasing, then $I$ is closed under subtraction of $\w$.
\item Conversely, if $I$ is a nontrivial upset (meaning $I \not\in \{ \emptyset, \L \}$), then we can define an equi-increasing map $\gamma$ by
\[ \gamma(\x) = \min \{ i \in \Z \mid \x - i \c \notin I \}. \]
Moreover, if $I$ is closed under subtraction of $\w$, then $\gamma$ is $\w$-equi-increasing.
\item These two constructions give mutually inverse maps between the collection of non-trivial upsets and the collection of integer-valued equi-increasing functions.
\end{enumerate}
\end{observation}

We give the following basic examples.

\begin{observation}
Let $(R,\L)$ be a GL complete intersection.
\begin{itemize}
\item[(a)] The degree map $\delta\colon\L\to\Q$ (see Section~\ref{section: GL CI 1}) is equi-increasing.
\item[(b)] The following conditons are equivalent.
\begin{itemize}
\item[(i)] $(R,\L)$ is Fano.
\item[(ii)] The degree map $\delta\colon\L\to\Q$ is $\w$-equi-increasing.
\item[(iii)] There exists a $\w$-increasing map $\gamma\colon\L\to\Q$.
\end{itemize}
\end{itemize}
\end{observation}

\begin{proof}
(a) and (b)(i)$\Rightarrow$(ii)$\Rightarrow$(iii) are clear.

(b)(iii)$\Rightarrow$(i) Recall that there exist integers $p>0$ and $q$ such that $p\w=q\c$, and $(R,\L)$ is Fano if and only if $q<0$. Thus the assertion follows.
\end{proof}

We will use the following special equi-increasing map later.


\begin{example} \label{prop.d=-1}
Let $(R,\L)= (k[X_1] / (X_1^{p}), \left< \x_1 \right>)$ be a GL hypersurface with $d=-1$ associated with a weight $p$.
We have an equi-increasing map $\gamma\colon\L\to\Q$ defined by
\[\gamma(\ell\c +\ell_1\x_1)=\left\{\begin{array}{ll}\ell&\ell_1=0,\\ \ell + \frac{1}{2}&1\le \ell_1\le p-1.\end{array}\right.\]
Moreover $U=\bigoplus_{i=1}^{p-1}k[x_1]/(x_1^i)$ is a tilting object in $\underline{\CM}^{\L}R$ which has a strict $\gamma$-presentation such that $\underline{\End}^{\L}_R(U)\simeq k\A_{p-1}$.
\end{example}

\begin{proof}
It is easy to check that $\gamma$ is equi-increasing.
It is shown in Example~\ref{CM-canonical for d=-1}(a) that $U$ is a tilting object in $\underline{\CM}^{\L}R$. 
Moreover $U$ has a strict $\gamma$-presentation since $P_U=R^{\oplus p-1}$ and $I_U=\bigoplus_{i=1}^{p-1}R((p-i)\x_1)$.
\end{proof}


Our main result in this section is the following.

\begin{theorem}\label{existence of d-tilting bundle}
Let $(R,\L)$ be a GL complete intersection of dimension $d+1$ with weights $p_1,\ldots,p_n$.
Assume that $U\in\CM^{\L}R$ does not have non-zero free direct summands and satisfies the following conditions.
\begin{itemize}
\item $U$ is a $d$-tilting object in $\underline{\CM}^{\L}R$.
\item $U$ has a $\gamma$-presentation for some $\w$-increasing map $\gamma\colon\L\to\Q$.
\end{itemize}
Then the GL projective space $\X$ has a $d$-tilting bundle given by
\[\pi(U)\oplus(\bigoplus_{\x\in S_{\gamma}}\OO(\x))\ \mbox{ for }\ S_\gamma=\{\x\in\L\mid\gamma(\x+\w)\le0<\gamma(\x)\}.\]
\end{theorem}


\begin{proof}
Let $I:=\{\x\in\L\mid\gamma(\x)>0\}$.
Then the conditions (i)--(iii) in Proposition~\ref{extend1} are satisfied.
In fact, since $\gamma$ is $\w$-increasing, the condition (i) is satisfied. Since $U$ has a $\gamma$-presentation, the conditions (ii) and (iii) are satisfied.
Thus the assertion follows from Proposition~\ref{extend1}(c).
\end{proof}

As a direct application of Theorem~\ref{existence of d-tilting bundle}, we have the following.
The quiver of the endomorphism algebra will be given in Example~\ref{d-tilting bundle (2,2) 2}.

\begin{corollary}\label{d-tilting bundle (2,2)}
Let $\X$ be a GL projective space of dimension $d$ with weights $p_1,\ldots,p_n$ such that $n=d+2$ and $p_1=p_2=2$.
\begin{itemize}
\item[(a)] There exists an $\w$-equi-increasing map given by
\[\gamma(\x)=\ell+\frac{1}{2}\#\{i\in[1,n]\mid\ell_i\ge1\},\]
where $\x=\ell\c+\sum_{i=1}^{n}\ell_i\x_i$ is a normal form.
\item[(b)] The tilting object $U^{\rm CM}$ in $\underline{\CM}^{\L}R$ given in Section~\ref{subsect tilting via tensor} is $d$-tilting and has a strict $\gamma$-presentation.
\item[(c)] $\X$ has a $d$-tilting bundle
\[\pi(U^{\rm CM})\oplus(\bigoplus_{\x\in S_\gamma}\OO(\x)),\]
where $S_\gamma$ consists of elements $\x=\ell\c+\sum_{i=1}^{n}\ell_i\x_i$ satisfying
\[0 < 2 \ell + \# \{i\in[1,n]\mid\ell_i\ge1\} \le \#\{i\in[3,n]\mid\ell_i \in \{0,1\}\}.\]
\end{itemize}
\end{corollary}

\begin{proof}
(a) The map $\gamma$ is clearly an equi-increasing map. Since $\gamma(\x-\x_i)=\gamma(\x)-\frac{1}{2}$ holds for $i=1,2$, we have $\gamma(\x+\w)\le\gamma(\x)$.

(b) By Proposition~\ref{d-tilting object example1}, $T^{\rm CM}$ gives a $d$-tilting object in $\underline{\CM}^{\L}R$.
By Corollary~\ref{thm.mf_for_F}, $U^{\rm CM}=T^{\rm CM}(\w)[d]$ is also a $d$-tilting object.
The resolution of $U^{\rm CM}$ constructed in Theorem~\ref{thm.mf_for_E} shows that $U^{\rm CM}$ has a strict $\gamma$-presentation.


(c) By Theorem~\ref{existence of d-tilting bundle}, $\X$ has a $d$-tilting bundle. 
Since the middle term is $2\gamma(\x)$, the left inequality is equivalent to $0<\gamma(\x)$ and the right one is equivalent to $\gamma(\x+\w)\le0$ by
\begin{align*} \gamma(\x+\w) &=\ell+1+\frac{1}{2}\#\{i\in[1,n]\mid\ell_i\ge2\}-\frac{1}{2}\#\{i\in[1,n]\mid\ell_i=0\}\\
& = \gamma(\x) - \frac{1}{2} \# \{ i \in [3,n] \mid \ell_i \in \{0,1\}\} .\qedhere
\end{align*}
\end{proof}


Now we will extend Corollary~\ref{d-tilting bundle (2,2)}.

The stronger assumption of having strict $\gamma$-presentations lends itself nicely to iteration. In particular we obtain the following result, giving an infinite family of GL projective spaces of arbitrary high dimension having suitable tilting bundles.

\begin{theorem} \label{thm.more_weights_ok}
Let $(R,\L)$ be a GL hypersurface associated with
\[R=k[X_1,\ldots,X_{d+2}]/(\sum_{i=1}^{d+2}\lambda_iX_i^{p_i}).\]
Assume that there exists a $d$-tilting object $U$ in $\underline{\CM}^{\L}R$ which has a strict $\gamma$-presentation for some $\w$-equi-increasing map $\gamma\colon\L\to\Q$.
Then for any $d'\ge d$, each GL projective space associated with
\[R'=k[X_1,\ldots,X_{d'+2}]/(\sum_{i=1}^{d'+2}\lambda_iX_i^{p_i})\]
for some positive integers $p_{d+3},\ldots,p_{d'+2}$ and parameters $\lambda_{d+3},\ldots,\lambda_{d'+2}$ has a $d'$-tilting bundle.
\end{theorem}

An ingredient to prove Theorem~\ref{thm.more_weights_ok} is the tensor products of matrix factorizations given in Section~\ref{subsection: Tensor}.
Based on this tensor product, we give a method to combine two equi-increasing maps.

Let us recall the setting. For two GL hypersurfaces $(R^j,\L_j)$ of dimension $d_j+1$
\begin{eqnarray*}
&R^j=S^j/(f_j)\ \mbox{ with }\ S^j=k[X_{j,1},\ldots,X_{j,n_j}]\ \mbox{ and }\ f_j=\sum_{i=1}^{n_j}\lambda_{j,i}X_{j,i}^{p_{j,i}},&\\
&\L_j=\langle\c_j,\x_{j,1},\ldots,\x_{j,n_j}\rangle/\langle p_{j,i}\x_{j,i}-\c_j\mid1\le i\le n_j\rangle&
\end{eqnarray*}
with $n_j=d_j+2$ and weights $p_{j,1},\ldots,p_{j,n_j}$, we define a new GL hypersurface $(R,\L)$ of dimension $d_1+d_2+1$
\begin{eqnarray*}
&R=S/(f)\ \mbox{ with }\ S=S^1\otimes_kS^2=k[X_{1,1},\ldots,X_{1,n_1},X_{2,1},\ldots,X_{2,n_2}]\ \mbox{ and }\ f=f_1+f_2,&\\
&\L=\langle\c,\x_{1,1},\ldots,\x_{1,n_1}\x_{2,1},\ldots,\x_{2,n_2}\rangle/\langle p_{j,i}\x_{j,i}-\c\mid j=1,2,\ 1\le i\le n_j\rangle&
\end{eqnarray*}
with weights $p_{1,1},\ldots,p_{1,n_1},p_{2,1},\ldots,p_{2,n_2}$.

\begin{proposition} \label{prop.combining_normal}
Under the above setting, let $\gamma_i \colon \L_i \to \mathbb{Q}$ be an equi-increasing map for $i=1,2$.
\begin{itemize}
\item[(a)] An equi-increasing map $\gamma=\gamma_1+\gamma_2 \colon \L \to\Q$ is given by 
\[ \gamma(\x+\y):=\gamma_1(\x)+\gamma_2(\y) \quad \text{for any $\x\in\L_1$ and $\y\in\L_2$} . \]
\item[(b)] If $\gamma_i$ is $\w_i$-equi-increasing for at least one of $i=1,2$, then $\gamma$ is $\w$-equi-increasing.
\item[(c)] If $U_i \in\CM^{\L_i}R^i$ has a strict $\gamma_i$-presentation for $i=1,2$, then $U_1 \otimes_{\rm MF} U_2 \in\CM^{\L}R$ has a strict $\gamma$-presentation.
\item[(d)] If $U_i$ is a $r_i$-tilting object in $\underline{\CM}^{\L_i}R^i$ for $i=1,2$, then $U_1 \otimes_{\rm MF} U_2$ is a $(r_1+r_2)$-tilting object in $\underline{\CM}^{\L}R$.
\end{itemize}
\end{proposition}

\begin{proof}
(a) $\gamma$ is well-defined by Observation~\ref{L and R}(a). It is clearly an equi-increasing map.

(b) Let $\x\in\L_1$ and $\y\in\L_2$. Assume that $\gamma_1$ is $\w_1$-equi-increasing. Then
\[\gamma(\x+\y+\w)=\gamma_1(\x+\c-\sum_{i=1}^n\x_i)+\gamma_2(\y-\sum_{j=1}^m\y_j)\le\gamma_1(\x)+\gamma_2(\y)=\gamma(\x+\y).\]

(c) Let $(M^i\colon Q^i\to P^i,\ N^i\colon P^i(-\c_i)\to Q^i)$ be a matrix factorization of $U_i$ for $i=1,2$. Then
\[P^i=P_{U_i}\in\add\{R^i(\x)\mid\gamma_i(\x)=0\}\ \mbox{ and }\ Q^i=I_{U_i}(-\c_i)\in\add\{R^i(\x)\mid\gamma_i(\x)=-1/2\}.\]
By the definition of the tensor product of matrix factorizations, $U:=U_1\otimes_{\rm MF}U_2$ satisfies
\[P_U=(Q^1\otimes_kQ^2)(\c)\oplus(P^1\otimes_kP^2)\ \mbox{ and }\ I_U=(P^1\otimes_kQ^2)(\c)\oplus(Q^1\otimes_kP^2)(\c).\]
By the definition of $\gamma$, we have $P_U\in\add\{R(\x)\mid\gamma(\x)=0\}$ and $I_U\in\add\{R(\x)\mid\gamma(\x)=1/2\}$.

(d) Immediate from Proposition~\ref{tensor of tilting is tilting}.
\end{proof}

Now we are ready to prove Theorem~\ref{thm.more_weights_ok}.

\begin{proof}[Proof of Theorem~\ref{thm.more_weights_ok}]
Let $(R',\L')$ be a GL hypersurface asscoiated with $(p_1,\ldots,p_{d'+2})$.
By Proposition~\ref{prop.combining_normal} and Example~\ref{prop.d=-1}, we obtain a $\w'$-equi-increasing map $\gamma' \colon \L' \to \mathbb{Q}$ and a $d'$-tilting object $U'$ in $\underline{\CM}^{\L'}R'$ which has a strict $\gamma'$-presentation.
Thus the result follows from Theorem~\ref{existence of d-tilting bundle}.
\end{proof}

As an application, we have the following list of cases where we have $d$-tilting bundles.

\begin{corollary}\label{example of d-tilting bundle}
Let $\X$ be a GL projective space of dimension $d$ with weights $p_1,\ldots,p_{d+2}$.
If one of the following conditions are satisfied, then there exists a $d$-tilting bundle on $\X$. 
Therefore $\X$ is $d$-VB finite and derived equivalent to a $d$-representation infinite algebra.
\begin{itemize}
\item $d \geq 0$ and $(p_1, p_2) = (2,2)$.
\item $d\ge1$ and $(p_1,p_2,p_3)$ is one of $(2,3,3)$, $(2,3,4)$ or $(2,3,5)$.
\item $d\ge2$ and $(p_1,p_2,p_3,p_4)=(3, 3, p_3, p_4)$ with $p_3,p_4\in\{3,4,5\}$.
\end{itemize}
\end{corollary}

To prove Corollary~\ref{example of d-tilting bundle}, we start with an easy observation, where $R_+=\bigoplus_{\x>0}R_{\x}\in\mod^{\L}R$.

\begin{lemma}\label{IX PX lemma}
Let $(R,\L)$ be a GL complete intersection of dimension $1$. Let $X$ be an indecomposable non-projective object in $\CM^{\L}R$, and $\x\in\L$.
\begin{itemize}
\item[(a)] $R(\x)$ is a direct summand of the injective hull $I_X$ of $X$ if and only if $\underline{\Hom}^{\L}_R(X,R_+(\x))\neq0$.
\item[(b)] $R(\x)$ is a direct summand of the projective cover $P_X$ of $X$ if and only if $\underline{\Hom}^{\L}_R(R_+(\x),X(\w)[-1])\neq0$.
\end{itemize}
\end{lemma}

\begin{proof} 
(a) The inclusion $R_+(\x)\to R(\x)$ is right almost split in $\CM^{\L}R$. Let $0\to X\to I_X\to Y\to0$ be an exact sequence. We have a commutative diagram
\[\xymatrix@C.5em@R1.5em{
0\ar[r]&\Hom_R^{\L}(Y,R_+(\x))\ar[r]\ar[d]^\wr&\Hom_R^{\L}(I_X,R_+(\x))\ar[r]\ar[d]&\Hom_R^{\L}(X,R_+(\x))\ar[r]\ar[d]^\wr&\underline{\Hom}_R^{\L}(X,R_+(\x))\to0\\
0\ar[r]&\Hom_R^{\L}(Y,R(\x))\ar[r]&\Hom_R^{\L}(I_X,R(\x))\ar[r]&\Hom_R^{\L}(X,R(\x))\ar[r]&0
}\]
of exact sequences, where the left and the right vertical maps are bijective since $X$ and $Y$ are indecomposable non-projective.
Then $R(\x)$ is not a direct summand of $I_X$ if and only if the middle map is bijective if and only if $\underline{\Hom}^{\L}_R(X,R_+(\x))\neq0$.

(b) Since $R$ is non-regular, $R_+(\x)$ is indecomposable non-projective by Theorem~\ref{AR quiver in general}. By Theorem~\ref{AR duality}(b), we have an almost split sequence $0\to R_+(\x)\to E\to \Omega^-R_+(\x-\w)\to0$ in $\CM^{\L}R$.
By the dual argument to (a), $R(\x)$ is a direct summand of $P_X$ if and only if $\underline{\Hom}^{\L}_R(\Omega^-R_+(\x-\w),X)\neq0$.
This is equivalent to $\underline{\Hom}^{\L}_R(R_+(\x),X(\w)[-1])\neq0$.
\end{proof}

Generalizing Definition~\ref{define increasing}(a), we say that an equi-increasing map $\gamma\colon\L\to\Q$ is \emph{semi-$\w$-equi-increasing} if $\gamma(\x+\w)\le\gamma(\x) + \frac{1}{2}$ holds for any $\x\in\L$.

The following easy observation is an analog of Proposition~\ref{prop.combining_normal}(b).

\begin{lemma}\label{prop.combining_normal2}
In the setting of Proposition~\ref{prop.combining_normal}(a), if $\gamma_i\colon\L_i\to\Q$ is semi-$\w_i$-equi-increasing for $i=1,2$, then $\gamma_1+\gamma_2\colon\L\to\Q$ is $\w$-equi-increasing.
\end{lemma}

\begin{proof}
This is immediate from
\[\gamma(\x+\y+\w)=\gamma_1(\x+\c-\sum_{i=1}^n\x_i)+\gamma_2(\y+\c-\sum_{j=1}^m\y_j) - 1\le \gamma_1(\x) + \frac{1}{2} + \gamma_2(\y) + \frac{1}{2} - 1 =\gamma(\x+\y).\qedhere\]
\end{proof}

Now we prepare the following observations, which make Proposition~\ref{prop.base_cases} stronger.

\begin{proposition} \label{prop.base_cases2}
Let $(R,\L)$ be a GL hypersurface associated with weights $p_1,\ldots,p_{d+2}$.
\begin{itemize}
\item[(a)] For $d = -1$, the following conditions are equivalent.
\begin{itemize}
\item[(i)] $p_1 = 2$.
\item[(ii)] $\underline{\CM}^{\L}R$ has a $0$-tilting object.
\item[(iii)] There are a semi-$\w$-equi-increasing map $\gamma$ and a $0$-tilting object $T \in \underline{\CM}^{\L}R$ which has a strict $\gamma$-presentation.
\end{itemize}
\item[(b)] For $d = 0$, the following conditions are equivalent.
\begin{itemize}
\item[(i)] $(p_1, p_2)=(2,p_2)$, $(3,3)$, $(3,4)$, or $(3,5)$.
\item[(ii)] $\underline{\CM}^{\L}R$ has a $1$-tilting object.
\item[(iii)] There are a semi-$\w$-equi-increasing map $\gamma$ and a $1$-tilting object $T \in \underline{\CM}^{\L}R$ which has a strict $\gamma$-presentation. 
\end{itemize}
\end{itemize}
\end{proposition}

\begin{proof}
In both (a) and (b), the equivalence between (i) and (ii) was shown in Proposition~\ref{prop.base_cases}, and the implication from (iii) to (ii) is clear.

Thus it only remains to show that (ii) implies (iii). In both cases we have a triangle equivalence $G \colon \underline{\CM}^{\L}R \to \DDD^{\bo}(\mod \Lambda)$, where $\gl \Lambda \leq d+1$. Since $d \leq 0$, any object in  $\DDD^{\bo}(\mod \Lambda)$ is a sum of stalk complexes. For $\x \in \L$, we define $s(\x)\in\Z$ and $\gamma(\x)\in\Q$ by
\[G(R_+(\x)) \in (\mod \Lambda)[s(\x)]\ \mbox{ and }\ \gamma(\x) = \frac{s(\x)+1}{2}.\]
This gives an equi-increasing map $\L \to \mathbb{Q}$:
Firstly, using the equality $(\c) = [2]$ from Theorem~\ref{Buchweitz Eisenbud}(b), we have that $G(R_+(\x + \c)) = G(R_+(\x)[2]) = G(R_+(\x)) [2]$, so $\gamma(\x + \c) = \gamma(\x) + 1$.
Next, we show that $\x \leq \y$ implies $s(\x)\leq s(\y)$. Since this is easy for $d=-1$, we assume $d=0$. 
Since $\x \leq \y$, $\Hom_R^{\L}(R_+(\x), R_+(\y)) \neq 0$ holds. Thus there exists a path from $R_+(\x)$ to $R_+(\y)$ in the Auslander-Reiten quiver of $\CM^{\L}R$. 
Thus $G(R_+(\x) )$ cannot be in a higher suspension of the module category than $G(R_+(\y))$, and $s(\x)\le s(\y)$ holds.

Now we show that $\gamma$ is semi-$\w$-equi-increasing, that is, $s(\x+\w)\le s(\x)+1$ for any $\x\in\L$.
If $d=-1$, then $\Lambda=k$ and hence $(\w)=(\x_1)=[1]$. Therefore $s(\x+\w)-s(\x)=1$ holds for any $\x\in\L$.
If $d=0$, then $(\w)=(D\Lambda)\Lotimes_\Lambda-$ holds by Proposition~\ref{property of stable tilting}, and hence $s(\x+\w)-s(\x)\in\{0,1\}$ holds for any $\x\in\L$.

It remains to show that the tilting object $U := G^{-1}(\Lambda)$ has a strict $\gamma$-presentation. 
Since the case $d=-1$ is clear, we assume $d=0$.
Fix $\x\in\L$. Then $R(\x)$ is a direct summand of the injective hull $I_U$ if and only if $\underline{\Hom}_R^{\L}(U, R_+(\x)) \neq 0$ by Lemma~\ref{IX PX lemma}. 
Applying the equivalence $G$, this is equivalent to $\Hom_{\DDD^{\bo}(\mod \Lambda)}(\Lambda, G( R_+(\x))) \neq 0$. This holds if and only if $G(R_+(\x)) \in \mod \Lambda$ if and only if $\gamma( \x ) = \frac{1}{2}$.
Dually, $R(\x)$ is a direct summand of the projective cover $P_U$ of $U$ if and only if $\underline{\Hom}(R_+(\x), U(\w)[-1]) \neq 0$ by Lemma~\ref{IX PX lemma}.
Again applying the equivalence $G$, this is equivalent to $\Hom_{\DDD^{\bo}(\mod \Lambda)}( G(R_+(\x)), D \Lambda[-1])\neq0$.
This holds if and only if $G(R_+(\x)) \in (\mod \Lambda)[-1]$. By definition of $\gamma$, we thus have $\gamma( \x) = 0$.
\end{proof}

Now we are ready to prove Corollary~\ref{example of d-tilting bundle}.

\begin{proof}[Proof of Corollary~\ref{example of d-tilting bundle}]
Let $(R,\L)$ be a GL hypersurface of one of the following type.
\begin{itemize}
\item $d=0$ and $(p_1, p_2) = (2,2)$.
\item $d=1$ and $(p_1,p_2,p_3)$ is one of $(2,3,3)$, $(2,3,4)$ or $(2,3,5)$.
\item $d=2$ and $(p_1,p_2,p_3,p_4)=(3, 3, p_3, p_4)$ with $p_3,p_4\in\{3,4,5\}$.
\end{itemize}
Then it is obtained from two GL hypersurfaces $(R^1,\L_1)$ and $(R^2,\L_2)$ of one of the following type.
\begin{itemize}
\item $d_i=-1$ and $p_{i,1}=2$.
\item $d_i=0$ and $(p_{i,1},p_{i,2})$ is one of $(3,3)$, $(3,4)$ or $(3,5)$.
\end{itemize}
Then we have an equality $d=d_1+d_2+2$.

By Proposition~\ref{prop.base_cases2}, there exist a semi-$\w_i$-equi-increasing map $\gamma_i:\L_i\to\Q$ and a $(d_i+1)$-tilting object $T^i$ in $\underline{\CM}^{\L_i}R^i$ which has a strict $\gamma_i$-presentation.
By Lemma~\ref{prop.combining_normal2} and Proposition~\ref{prop.combining_normal}, there exists an $\w$-equi-increasing map $\gamma=\gamma_1+\gamma_2:\L\to\Q$ such that $T:=T^1\otimes_{\rm MF}T^2$ is a $(d_1+d_2+2)$-tilting object in $\underline{\CM}^{\L}R$ which has a strict $\gamma$-presentation.  Since $d_1+d_2+2=d$, we obtain the claim for the above cases.

Now the general claim follows from Theorem~\ref{thm.more_weights_ok}.
\end{proof}



The rest of this section is devoted to studying the converse of (Step 2) at the beginning of
the section. More precisely, we give a partial answer to the following problem.

\begin{problem}\label{d-tilting bundle is d-tilting object}
Let $V\in\CM^{\L}R$ be a $d$-tilting bundle on $\X$. Is $V$ a $d$-tilting object in $\underline{\CM}^{\L}R$?
\end{problem}

The following result gives a partial answer to Question~\ref{d-tilting bundle is d-tilting object}.

\begin{theorem}\label{d-tilting is silting}
Let $V\in\CM^{\L}R$ is a $d$-tilting bundle on a GL projective space $\X$. Take a decomposition $V=P\oplus W$, where $P\in\proj^{\L}R$ is a maximal projective direct summand of $V$.
\begin{itemize}
\item[(a)] $V$ satisfies $\underline{\Hom}^{\L}_R(V,V[i])=0$ and $\Hom^{\L}_R(V,V(i\w))=0$ for all $i>0$.
\item[(b)] If $\underline{\End}^{\L}_R(W)=\End^{\L}_R(W)$, then $V$ is a $d$-tilting object in $\underline{\CM}^{\L}R$.
\end{itemize}
\end{theorem}

\begin{proof}
(a) By Theorem~\ref{tilting-cluster tilting 2}, we know that $V$ is a slice in a $d$-cluster tilting subcategory $\UU:=\add\{V(\ell\w)\mid\ell\in\Z\}$ of $\CM^{\L}R$.
By Proposition~\ref{arith CM}(c), we have
\[\underline{\Hom}^{\L}_R(V,V[i])=\Ext^i_{\mod^{\L}R}(V,V)=\Ext^i_{\X}(V,V)=0\ \mbox{ for any $i$ with $1\le i\le d-1$.}\]
Let $P$ be a maximal direct summand of $V$ which belongs to $\proj^{\L}R$. Since $V$ is a slice in $\UU$, we have
\begin{eqnarray}\label{(V,V(lw))=0}
&\Hom^{\L}_R(V,V(\ell\w))=\Hom_{\X}(V,V(\ell\w))=0\ \mbox{ for any }\ \ell>0,&\\
\label{proj R=add P(lw)}
&\proj^{\L}R=\add\{P(\ell\w)\mid\ell\in\Z\}.&
\end{eqnarray}
By \eqref{(V,V(lw))=0}, the latter part of (a) follows, and we have $\underline{\Hom}^{\L}_R(V,V[d])=D\underline{\Hom}^{\L}_R(V,V(\w))=0$. 

It remains to prove $\underline{\Hom}^{\L}_R(V,V[i+d+1])=0$ for any $i\ge0$.
Take a minimal projective resolution
\[\xymatrix{\cdots\ar[r]^{}&P_2\ar[r]^{}&P_1\ar[r]^{}&P_0\ar[r]^{}&V\ar[r]&0}\]
of $V$ in $\mod^{\L}R$. 
Applying Lemma~\ref{(Q,X) non zero}(b) to $X:=V$, any indecomposable direct summand $Q$ of $P_i$ with $i\ge0$ satisfies $\Hom^{\L}_R(Q,V)\neq0$.
Thus the equalities \eqref{(V,V(lw))=0} and \eqref{proj R=add P(lw)} imply that $Q$ belongs to $\PP^+:=\add\{P(\ell\w)\mid\ell\ge0\}$. Thus $P_i\in\PP^+$ holds for any $i\ge0$. This implies
\[\Hom^{\L}_R(V,\Omega^{i+1}V(\w))\subset\Hom^{\L}_R(V,P_{i}(\w))\stackrel{\eqref{(V,V(lw))=0}}{=}0\]
for any $i\ge0$. Therefore $\underline{\Hom}^{\L}_R(V,V[i+d+1])=D\underline{\Hom}^{\L}_R(V,V(\w)[-i-1])=0$ holds as desired.


(b) It follows from (a) and Lemma~\ref{End and stable End}(b) that $\underline{\Hom}^{\L}_R(V,V[i])=0$ holds for all $i\neq0$.

Now we prove $\thick_{\underline{\CM}^{\L}R}V=\underline{\CM}^{\L}R$ by modifying the proof of Proposition~\ref{generate V(w)}. Let $\UU^+:=\add\{ W(\ell\w),\ \proj^{\L}R\mid \ell\ge0\}\subset\UU$.
By Lemma~\ref{graded and ungraded}, there exists a minimal right $\UU^+$-approximation $f:U_0\to W(-\w)$. Clearly we have an exact sequence
\begin{equation}\label{generate W(-w) 1}
0\to\Kernel f\to U_0\xrightarrow{f} W(-\w)\to0.
\end{equation}
By Proposition~\ref{basic properties of d-CT for CM}(b), there exists an exact sequence
\begin{equation}\label{generate W(-w) 2}
0\to U_d\to\cdots\to U_1\to\Kernel f\to0
\end{equation}
in $\mod^{\L}R$ with $U_i\in\UU$ such that all morphisms are right minimal. Take a decomposition $U_i=P_i\oplus W_i$, where $P_i\in\proj^{\L}R$ is a maximal projective direct summand of $U_i$.

We claim $P_i\in\add\{P(\ell\w)\mid\ell\ge-1\}$ and $W_i\in\add\{W(\ell\w)\mid\ell\ge0\}$ for all $i\ge0$. By (a), $\Hom^{\L}_R(V,V(\ell\w))=0$ holds for all $\ell>0$.
Thus the claim is clear for $i=0$. Assume that the claim is true for $i$. Then $U_{i+1}\in\add\{V(\ell\w)\mid\ell\ge-1\}$ holds. It suffices to prove that $W_{i+1}$ does not have an indecomposable direct summand $X\in\add W(-\w)$. Otherwise, there exists a non-zero morphism $X\to Q_i$ for an indecomposable direct summand $Q_i\in\add P(-\w)$ of $P_i$. Moreover, there exists a sequence $Q_i\to Q_{i-1}\to\cdots\to Q_0$ of non-zero morphisms for an indecomposable direct summand $Q_j\in\add P(-\w)$ of $P_j$. Then the composition $Q_i\to Q_0\subset U_0\xrightarrow{f} W(-\w)$ is injective by Proposition~\ref{non-zero is injective}(c), and hence the composition $X\to Q_i\to W(-\w)$ is non-zero, a contradiction to $\underline{\End}^{\L}_R(W(-\w))=\End^{\L}_R(W(-\w))$. Thus the claim holds also for $i+1$.

Consider the $\UU$-module $F$ given by an exact sequence
\[\Hom_{\UU}(-,U_0)\xrightarrow{f}\Hom_{\UU}(-,W(-\w))\to F\to0.\]
As in the proof of Proposition~\ref{generate V(w)}, we have $U_i\in\add(V(-\w)\oplus V)$ for all $i$. Consequently, we have $W_i\in\add V$ for all $i$.
In particular, the exact sequences \eqref{generate W(-w) 1} and \eqref{generate W(-w) 2} show $V(-\w)\in\thick_{\underline{\CM}^{\L}R}V$. As in the proof of Proposition~\ref{generate V(w)}, we have $\thick_{\underline{\CM}^{\L}R}V=\underline{\CM}^{\L}R$.

It remains to show that $\underline{\End}^{\L}_R(V)=\End^{\L}_R(W)$ has global dimension at most $d$. 
Let $I$ be the finite subset of $\L$ such that $P=\bigoplus_{\x\in I}R(\x)$. Let
\[I^+=\{\x\in I\mid\Hom^{\L}_R(R(\x),W)=0\},\ I^-=I\setminus I^+\ \mbox{ and }\ P^\pm=\bigoplus_{\x\in I^{\pm}}R(\x).\]
Clearly $\Hom^{\L}_R(P^+,W)=0$ holds.
For any $\x\in I^-$, the condition $\Hom^{\L}_R(R(\x),W)\neq0$ implies $\Hom^{\L}_R(W,R(\x))=0$ by Lemma~\ref{End and stable End}(a)(i)$\Rightarrow$(ii), and $\Hom^{\L}_R(P^+,R(\x))=0$ by Proposition~\ref{non-zero is injective}(c).
Thus $\Hom^{\L}_R(W\oplus P^+,P^-)=0$ holds, and $\End^{\L}_R(V)$ has a triangular form
\[\End^{\L}_R(V)=\End^{\L}_R(P^-\oplus W\oplus P^+)=\left[\begin{smallmatrix}
\End^{\L}_R(P^-)&\Hom^{\L}_R(P^-,W)&\Hom^{\L}_R(P^-,P^+)\\
0&\End^{\L}_R(W)&\Hom^{\L}_R(W,P^+)\\
0&0&\End^{\L}_R(P^+)
\end{smallmatrix}\right].\]
Therefore $\gl\End^{\L}_R(W)\le\gl\End^{\L}_R(V)\le d$ holds.
\end{proof}

\section{Examples}

In this section, we give an example of $d$-tilting bundles on GL hypersurfaces.

We start with giving the quiver of the $d$-tilting bundle constructed in Corollary~\ref{d-tilting bundle (2,2)}.

\begin{example}\label{d-tilting bundle (2,2) 2}
Let $\X$ be a GL projective space with $d=2$ and $n=4$ weights $(2,2,3,4)$.
Then the tilting bundle given in Corollary~\ref{d-tilting bundle (2,2)} has the following quiver, where $\ttt:=\x_1-\x_2$.
{\tiny
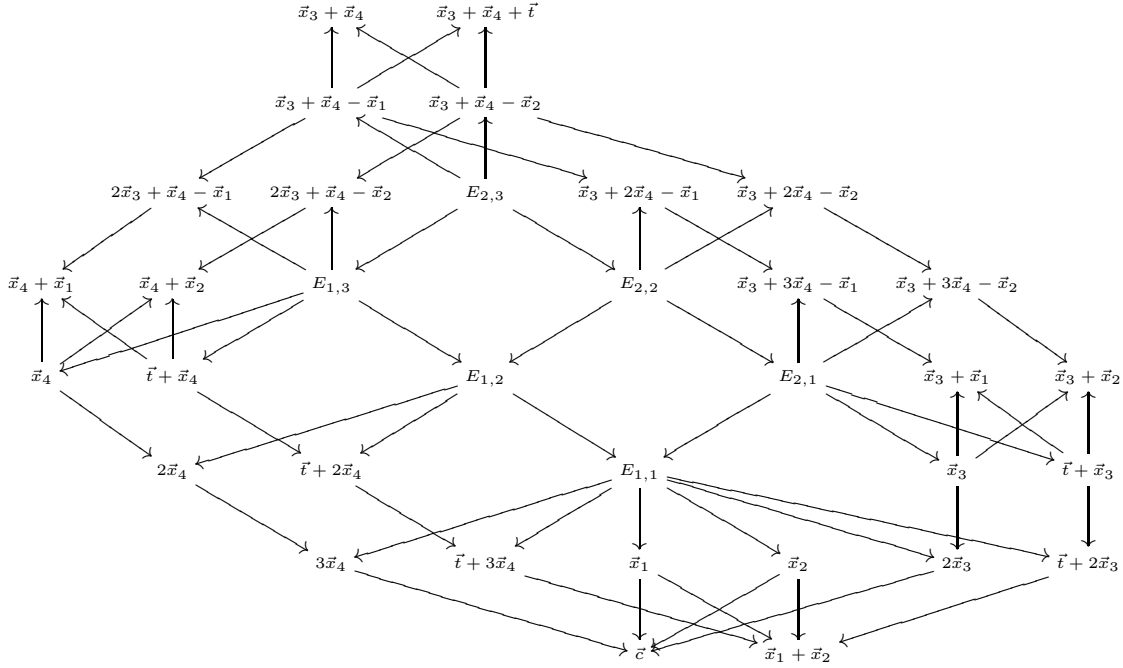
\[\xymatrix@R=3em@C=1em{
&& \x_3+\x_4 & \x_3+\x_4+\ttt \\
&& \x_3+\x_4-\x_1 \ar[u]\ar[ru]\ar[ld]\ar[rrd] & \x_3+\x_4-\x_2 \ar[lu]\ar[u]\ar[ld]\ar[rrd] \\
& 2\x_3+\x_4-\x_1 \ar[ld] & 2\x_3+\x_4-\x_2 \ar[ld] & E_{2,3} \ar[ld]\ar[rd]\ar[lu]\ar[u] & \x_3+2\x_4-\x_1 \ar[rd] & \x_3+2\x_4-\x_2 \ar[rd] \\
\x_4+\x_1 & \x_4+\x_2 & E_{1,3}\ar[rd]\ar[lu]\ar[u]\ar[lld]\ar[ld]&&E_{2,2} \ar[ld]\ar[rd]\ar[u]\ar[ru] & \x_3+3\x_4-\x_1 \ar[rd] & \x_3+3\x_4-\x_2 \ar[rd] \\
\x_4\ar[u]\ar[ru]\ar[rd]& \ttt+\x_4 \ar[lu]\ar[u]\ar[rd]&&E_{1,2}\ar[rd]\ar[lld]\ar[ld]&&E_{2,1}\ar[ld]\ar[u]\ar[ru]\ar[rd]\ar[rrd] & \x_3+\x_1 & \x_3+\x_2 \\
&2\x_4 \ar[rd]&\ttt+2\x_4\ar[rd]&&E_{1,1}\ar[lld]\ar[ld]\ar[d]\ar[rd]\ar[rrd]\ar[rrrd]&&\x_3 \ar[u]\ar[ru]\ar[d]&\ttt+\x_3 \ar[lu]\ar[u]\ar[d]\\
&&3\x_4\ar[rrd]&\ttt+3\x_4\ar[rrd]&\x_1\ar[d]\ar[rd]&\x_2\ar[ld]\ar[d]&2\x_3\ar[lld]&\ttt+2\x_3\ar[lld]\\
&&&&\c&\x_1+\x_2
}\]}
\end{example}

We have a nicer choice of a set of line bundles with certain symmetry given by the following result.

\begin{theorem}\label{Main2}
Let $\X$ be a GL projective space of dimension $d$ with weights $p_1,\ldots,p_n$ such that $n=d+2$ and $p_1=p_2=2$.
Then $\X$ has a $d$-tilting bundle
\[\pi(U^{\rm CM})\oplus(\bigoplus_{\x\in S}\OO(\x)),\]
where $U^{\rm CM}\in\CM^{\L}R$ is given in
Theorem \ref{thm.mf_for_E} and $S$ is the subset of $\L$ given by
\[S:=\left\{\begin{array}{cc}
\bigcup_{i=1,2}([-\frac{d-1}{2}\c,\frac{d-1}{2}\c+\x_i]\cup[-\frac{d-3}{2}\c-\x_1-\x_2,\frac{d-1}{2}\c+\x_i])&\mbox{if $d$ is odd,}\\
\bigcup_{i=1,2}([-\frac{d}{2}\c+\x_i,\frac{d}{2}\c]\cup[-\frac{d}{2}\c+\x_i,\frac{d-2}{2}\c+\x_1+\x_2])&\mbox{if $d$ is even.}
\end{array}\right.\]
\end{theorem}

The set $S$ has the following natural interpretation.

\begin{observation}
We consider an abelian group
\[\overline{\L}:=\langle\y_1=\y_2,\y_3,\cdots,\y_n\rangle/
\langle 2\y_1-p_i\y_i\mid 3\le i\le n\rangle.\]
and its submonoid $\overline{\L}_+$ generated by all $\y_i$'s.
We regard $\overline{\L}$ as a partially ordered set: $\x\le\y$ if and only if $\y-\x\in\overline{\L}_+$.
Then we have an exact sequence
\[0\xrightarrow{}\langle\x_1-\x_2\rangle\xrightarrow{}\L
\xrightarrow{q}\overline{\L}\xrightarrow{}0\]
given by $q(\x_i):=\y_i$, where $\langle\x_1-\x_2\rangle$ is a subgroup
of $\L$ of order $2$.

Then the equality
\begin{equation}\label{S and q}
S=q^{-1}([(1-d)\y_1,d\y_1])
\end{equation}
holds. In fact, the map $q:\L\to\overline{\L}$ is a morphism of partially ordered sets.
Then $q^{-1}((1-d)\y_1)$ (respectively, $q^{-1}(d\y_1)$) consists of
two elements which give the lower (respectively, upper) bounds of
the intervals defining $S$.
On the other hand, for $\x,\y\in\L$, it is easy to check that
$q(\x)\le q(\y)$ holds if and only if either $\x\le \y$ or
$\x\le\y+\x_1-\x_2$ holds.
\end{observation}

\begin{proof}[Proof of Theorem~\ref{Main2}]
Since $p_1=p_2=2$, the algebra $\underline{\End}^{\L}_R(U^{\rm CM})
=\bigotimes_{i=1}^nk\A_{p_i-1}$ has global dimension at most $n-2=d$.
Theorem~\ref{construct dCT} shows that $\CM^{\L}R$ has
a $d$-cluster tilting subcategory
\[\UU:=\add\{U^{\rm CM}(\ell\w),\ R(\x)\mid\ell\in\Z,\ \x\in\L\}.\]
By Theorem~\ref{tilting-cluster tilting 2}, it suffices to show that
\[V:=U^{\rm CM}\oplus(\bigoplus_{\x\in S}R(\x))\]
is a slice in $\UU$.
We start with proving the following observations.

\begin{lemma}\label{S is a slice}
\begin{itemize}
\item[(a)] $S$ is a complete set of representatives of $\L/\Z\w$.
\item[(b)] We have $\Hom^{\L}_R(R(\x),R(\y+\ell\w))=0$ for
any $\x,\y\in S$ and $\ell>0$.
\end{itemize}
\end{lemma}

\begin{proof}
(a) 
By using a similar argument as in the proof of
Proposition~\ref{compare [0,dc] with L/w}(b),
one can easily check that the interval $[0,(2d-1)\y_1]$ in $\overline{\L}$
gives a complete set of representatives of $\overline{\L}/\Z\v$ for
\[\v:=-\sum_{i=3}^n\y_i.\]
Shifting by $(1-d)\y_1$, the interval $[(1-d)\y_1,d\y_1]$
also gives a complete set of representatives of $\overline{\L}/\Z\v$.
Since $q(\w)=\v$ holds, we have an exact sequence
\[0\to\langle\x_1-\x_2\rangle\to\L/\Z\w\xrightarrow{q}\overline{\L}/\Z\v\to0.\]
Therefore \eqref{S and q} implies that $S$ gives a complete set of
representatives of $\L/\Z\w$.

(b) We have $\Hom_{\X}(\OO(\x),\OO(\y+\ell\w))=R_{\y+\ell\w-\x}$
by Proposition~\ref{vanishing}.
If $\y+\ell\w-\x\ge 0$ for some $\ell>0$, then we have
\[0\le q(\y+\ell\w-\x)\le d\y_1+\ell\v-(1-d)\y_1=
-\y_1+\sum_{i=3}^n(p_i-\ell)\y_i,\]
in $\overline{\L}$, a contradiction.
Thus $R_{\y+\ell\w-\x}=0$ holds by Observation~\ref{basic results on L}(c).
\end{proof}

Finally we need the following technical observations.
 
\begin{lemma}\label{connection S and P}
\begin{itemize}
\item[(a)] For any $i$ with $1\le i\le n$, let $\ell_i$ be an integer
satisfying $1\le\ell_i\le p_i-1$. Then
$\frac{|I|}{2}\c-\sum_{i\in I}\ell_i\x_i$ 
(respectively, $\frac{|I|+1}{2}\c-\sum_{i\in I}\ell_i\x_i$) belongs to
$S$ for any even (respectively, odd) subset $I$ of $\{1,\ldots,n\}$
such that $\{3,\ldots,n\}\not\subset I$.
\item[(b)] There exist monomorphisms $M\to U^{\rm CM}$ and $U^{\rm CM}\to M'$ in $\mod^{\L}R$ with $M,M'\in\add\bigoplus_{\x\in S}R(\x)$ such that their cokernels have rank $0$.
\end{itemize}
\end{lemma}

\begin{proof}
(a) Let $J:=I\setminus\{1,2\}$. Then $|J|\le d-1$ holds.
Since $\ell_1=\ell_2=1$, we have
\[q\left(\frac{|I|+a}{2}\c-\sum_{i\in I}\ell_i\x_i\right)=(|I|+a)\y_1-\sum_{i\in I}\ell_i\y_i
=(|J|+a)\y_1-\sum_{i\in J}\ell_i\y_i,\]
where $a=0$ if $I$ is even, and $a=1$ if $I$ is odd.
Therefore it is enough to show
\[\begin{array}{ll}
(1-d)\y_1\le (|J|+a)\y_1-\sum_{i\in J}\ell_i\y_i\le d\y_1.
\end{array}\]
The right inequality is clear since $|J|\le d-1$.
The left inequality is equivalent to $(|J|-a+1-d)\y_1\le \sum_{i\in J}(p_i-\ell_i)\y_i$,
which holds since $|J|\le d-1$.

(b) We only have to consider the direct summand $\rho(E^{\l})$ of $U^{\rm CM}$. The morphisms $P^{\not\ni n}\to\rho(E^{\l})$ and $\rho(E^{\l})\to Q^{\not\ni n}(\c)$ given in Proposition~\ref{rank 2^d} give the desired morphisms.
\end{proof}

%

Now we are ready to prove Theorem~\ref{Main2}.

It is enough to show that $V$ is a slice in $\UU$.
By definition of $\UU$ and Lemma~\ref{S is a slice}(a), we have $\UU=\add\{V(\ell\w)\mid\ell\in\Z\}$.

It remains to show $\Hom^{\L}_R(V,V(\ell\w))=0$ for any $\ell>0$.
By Lemma~\ref{connection S and P}(b), there exist monomorphisms $a\colon M\to V$ and $b\colon V\to M'$ such that $M,M'\in\add\bigoplus_{\x\in S}R(\x)$ and their cokernels have rank $0$.

Now we take a morphism $f\colon V\to V(\ell\w)$ with $\ell>0$.
By Lemma~\ref{S is a slice}(b), $afb\colon M\to M'(\ell\w)$ is zero.
Thus $fb:V\to M'(\ell\w)$ factors through $\Cokernel a$, which has rank zero. Thus $fb=0$ holds by Proposition~\ref{non-zero is injective}(b), and hence
$f=0$ as desired.
\end{proof}

\begin{example}
Let $\X$ be a GL projective space with $d=2$ and $n=4$ weights $(2,2,3,4)$.
Then $V$ given in Theorem~\ref{Main2} is the following, where $\ttt:=\x_1-\x_2$.
{\tiny\[\xymatrix@R=3em@C=1.5em{
&&&-\x_1\ar[lld]\ar[d]\ar[rd]\ar[rrd]&-\x_2\ar[lld]\ar[ld]\ar[d]\ar[rrd]\\
&\x_3-\x_1\ar[ld]\ar[rrd]&\x_3-\x_2\ar[ld]\ar[rd]&0\ar[d]&\ttt\ar[ld]&\x_4-\x_1\ar[rd]\ar[lld]&\x_4-\x_2\ar[rd]\ar[llld]\\
2\x_3-\x_1\ar[rrd]&2\x_3-\x_2\ar[rd]&&E_{2,3}\ar[ld]\ar[rd]&&&2\x_4-\x_1\ar[rd]\ar[lld]&2\x_4-\x_2\ar[rd]\ar[llld]\\
&&E_{1,3}\ar[rd]\ar[lld]\ar[ld]&&E_{2,2}\ar[ld]\ar[rd]&&&3\x_4-\x_1\ar[lld]&3\x_4-\x_2\ar[llld]\\
\x_4\ar[rd]&\ttt+\x_4\ar[rd]&&E_{1,2}\ar[rd]\ar[lld]\ar[ld]&&E_{2,1}\ar[ld]\ar[rrd]\ar[rrrd]\\
&2\x_4\ar[rd]&\ttt+2\x_4\ar[rd]&&E_{1,1}\ar[lld]\ar[ld]\ar[d]\ar[rd]\ar[rrd]\ar[rrrd]&&&\x_3\ar[ld]&\ttt+\x_3\ar[ld]\\
&&3\x_4\ar[rrd]&\ttt+3\x_4\ar[rrd]&\x_1\ar[d]\ar[rd]&\x_2\ar[ld]\ar[d]&2\x_3\ar[lld]&\ttt+2\x_3\ar[lld]\\
&&&&\c&\x_1+\x_2
}\]}
\end{example}


\backmatter

\printindex

\end{document}